\documentclass[a4paper,11pt]{elsarticle}
\usepackage[utf8]{inputenc}

\usepackage{bm}
\usepackage{amsmath,amsthm,amssymb}
\usepackage{mathrsfs}
\usepackage{graphicx}
\usepackage{epsfig, setspace}
\usepackage{url}
\usepackage{graphicx, transparent, color}
\usepackage[ruled]{algorithm2e}
\usepackage{xcolor}

\DeclareSymbolFont{rsfs}{U}{rsfs}{m}{n}
\DeclareSymbolFontAlphabet{\mathcal}{rsfs}
\makeatletter
\def\ps@pprintTitle{
  \let\@oddhead\@empty
  \let\@evenhead\@empty
  \let\@oddfoot\@empty
  \let\@evenfoot\@oddfoot
}
\makeatother

\usepackage[margin=2.5cm]{geometry}
\usepackage{tikz}
\usetikzlibrary{shapes,calc}
\usepackage{verbatim}
\usepackage{listings}
\usepackage{textcomp}
\usepackage[normalem]{ulem}
\usepackage{todonotes}

\usepackage{epsfig}
\usepackage{pst-plot}
\usepackage{amsfonts}
\usepackage{lmodern}
\usepackage{euscript}
\usepackage{oldgerm}

\numberwithin{equation}{section}

\newcommand{\iph}{{i + \frac{1}{2}}}

\newtheorem{example}{Example}[section]

\newtheorem{theorem}{Theorem}[section]

\begin{document}
\begin{frontmatter}
 \title{Efficient Hybrid WENO Schemes for Special Relativistic Hydrodynamics with Adaptive Characteristic Reconstruction}
\date{}
\author[1]{Rakesh Kumar\corref{cor1}}\ead{rakesh.kumar@mahindrauniversity.edu.in, rakeshiitb21@gmail.com}
\author[1]{Biswarup Biswas}\ead{biswarupb7@gmail.com}
\author[2]{Asha Kumari Meena}\ead{ashameena01@gmail.com }
\author[3]{Harish Kumar}\ead{harishkumarkaushik511@gmail.com}

\address[1]{Department  of Mathematics,
Ecole Centrale School of Engineering \\
Mahindra University,
Hyderabad-500043, India.}
\address[2]{Department of Mathematics, Malaviya National Institute of Technology,\\ Jaipur, Rajasthan, India}
 \address[3]{Department of Mathematics, Indian Institute of Technology, Delhi, India}

\begin{abstract}
Special relativistic hydrodynamics (SRHD) equations arise in the modeling of high-speed fluid flows encountered in astrophysical phenomena such as jets, supernova explosions, and gamma-ray bursts. Owing to their highly nonlinear hyperbolic nature, solutions often develop strong discontinuities, making the design of stable and accurate numerical schemes challenging. Although Weighted Essentially Non-Oscillatory (WENO) schemes are widely used for such problems, component-wise WENO reconstruction may produce spurious oscillations near discontinuities. On the other hand, characteristic-wise WENO reconstruction provides accurate non-oscillatory solutions for systems of conservation laws, but it involves the computation of eigenvectors in each cell, which leads to high computational cost.
In this work, we intend to develop hybrid schemes which maintain the non-oscillatory feature of characteristic-wise WENO while being less costly. We propose three hybrid schemes, namely the H1-WENO, H2-WENO, and H3-WENO schemes, based on a new troubled-cell indicator constructed from the smoothness indicators of the WENO scheme. The proposed troubled-cell indicator effectively distinguishes smooth and discontinuous regions, allowing the hybrid schemes to employ inexpensive reconstructions in smooth regions and the characteristic-wise WENO reconstruction only near discontinuities. Numerical experiments demonstrate that the proposed schemes retain the accuracy and robustness of characteristic-wise WENO methods while significantly reducing the computational cost. In particular, the H1-WENO scheme achieves an approximately 30--40\% improvement in computational efficiency compared to the standard WENO scheme in 2D test cases.

\end{abstract}
\end{frontmatter}

\section{Introduction}
The special relativistic hydrodynamics equations describe the motion of fluids moving at velocities close to the speed of light. These equations arise from the conservation of mass, momentum, and energy in the framework of Einstein’s theory of special relativity \cite{anile}. SRHD models play an important role in several applications such as astrophysical jets, supernova explosions, gamma-ray bursts, accretion disks around black holes, and high-energy plasma flows \cite{landau1987,begelman1984theory,mirabel1999sources,zensus1997parsec,bottcher2012relativistic,anile}. The governing equations form a highly nonlinear hyperbolic system of conservation laws \cite{marti1991numerical} whose solutions may develop discontinuities such as shock waves and contact discontinuities even from smooth initial data. This makes the construction of accurate and stable numerical schemes very challenging. The main difficulties in developing numerical methods for SRHD include  avoiding spurious oscillations near discontinuities and accurately capturing complex wave interactions. Therefore, robust high-resolution shock-capturing schemes are required to obtain stable and physically reliable solutions for relativistic flows.

In recent years, a wide variety of numerical schemes have been developed for solving the equations of relativistic hydrodynamics (RHD), particularly because of their highly nonlinear and hyperbolic nature. 
One of the initial attempt to develop numerical method for solving the RHD equations was by Wilson in 1972 \cite{wilson1972numerical} which was based on an explicit finite difference scheme. The method suffered from accuracy issues for higher Lorentz factors~\cite{centrella1984planar}. Since then, several shock capturing schemes have been proposed using exact or approximate Riemann solvers \cite{marti1991numerical,marti1994analytical,dai1997iterative,ibanez1999riemann}. Piecewise parabolic  reconstruction method (PPM) is one of the interesting methods \cite{aloy1999genesis, marti1996extension, mignone2005piecewise} which showed promising results for solving the RHD equations.
High-resolution shock-capturing methods have become one of the most successful approaches for resolving discontinuities such as shock waves and contact discontinuities while maintaining high-order accuracy in smooth regions. Among these methods, the Weighted Essentially Non-Oscillatory schemes have received significant attention for solving systems of hyperbolic conservation laws due to their ability to achieve high-order accuracy and suppress spurious oscillations near discontinuities.

WENO schemes are among the most successful high-order numerical methods for solving systems of hyperbolic conservation laws, particularly in the presence of shocks and discontinuities. These schemes are designed to achieve high-order accuracy in smooth regions while maintaining non-oscillatory behavior near discontinuities. The WENO methodology can be viewed as an extension of the Essentially Non-Oscillatory (ENO) schemes introduced in \cite{shu-osh_88a, shu-osh_89a}. The first WENO scheme was proposed in the finite volume framework by Liu et al.~\cite{liu-etal_94a} and was later extended to the finite difference framework by Jiang and Shu~\cite{jia-shu_96a}. Owing to their excellent balance between accuracy, robustness, and shock-capturing capability, WENO schemes have become one of the standard tools for the numerical approximation of nonlinear hyperbolic partial differential equations. Building on the proven success of WENO schemes for nonlinear hyperbolic problems, they have also been successfully extended to the RHD equations \cite{weno2007Tchekhovskoy,wu2015high,chen2022physical,bhoriya2026physical}.

During the past three decades, WENO schemes have undergone significant developments aimed at improving their accuracy, robustness, and computational efficiency. Several higher-order formulations and variants of the WENO methodology have been proposed in the literature; see, for example, \cite{bal-shu_00a, bor-etal_08a, cas-etal_11a, ger-etal_09a, don-bor_13a}. Among the notable variants are the WENO-Z scheme \cite{bor-etal_08a, don-bor_13a}, the Central WENO scheme \cite{kol_14a, cra-sem_16a, cra-etal_18a, cas-sem_19a, sem-vis_20a}, and the WENO-AO scheme \cite{bal-etal_16a, kum-cha_18a, kum-cha_19a, bal-etal_20a, Arb-etal_17a, sem-vis_20a}. These variants have been developed to further improve resolution in smooth regions while preserving stability and non-oscillatory behavior near discontinuities.

The extension of WENO reconstruction from scalar conservation laws to systems of hyperbolic conservation laws is nontrivial. A straightforward component-wise reconstruction may produce spurious oscillations near discontinuities due to the nonlinear coupling among different characteristic fields \cite{jia-shu_96a, har-etal_97a, ren-etal_03a}. To address this issue, characteristic-wise WENO reconstruction was introduced \cite{jia-shu_96a}, in which the variables are projected onto the local characteristic fields using the eigenvectors of the Jacobian matrix. The reconstruction is then performed in the characteristic space before projecting the solution back to the physical space. Although this approach significantly improves stability and reduces nonphysical oscillations near discontinuities, it substantially increases the computational cost due to the repeated evaluation of eigenvalues and eigenvectors.
To reduce this computational cost, several hybrid WENO schemes have been proposed \cite{ren-etal_03a, pup_03a,hil-pul_04a,vis-gai_05a, mov-joh_13a,kum_18a, jun-etal_19a,zhao-etal_20a,kum_cha_22a}. The main idea of these schemes is to employ inexpensive component-wise reconstruction in smooth regions and use characteristic-wise reconstruction only near discontinuities. Such approaches rely on efficient troubled-cell indicators to distinguish smooth and nonsmooth regions. By restricting the costly characteristic decomposition to troubled cells, hybrid schemes significantly improve computational efficiency while preserving the robustness and non-oscillatory properties of characteristic-wise WENO reconstruction.

In this article, we propose a new class of hybrid WENO schemes, namely H1-WENO, H2-WENO, and H3-WENO, with the aim of improving computational efficiency while preserving the accuracy and non-oscillatory properties of characteristic-wise WENO methods. The proposed hybrid schemes are constructed using the fifth-order WENO-AO(5,3) scheme \cite{bal-etal_16a} as the underlying reconstruction procedure. To distinguish smooth and discontinuous regions of the solution, we develop a new troubled-cell indicator based on the smoothness indicators of the WENO scheme. The proposed indicator effectively detects troubled cells and enables the adaptive use of different reconstruction strategies in smooth and non-smooth regions.
In the H1-WENO scheme, component-wise WENO reconstruction is employed in smooth regions, whereas characteristic-wise WENO reconstruction is used near discontinuities. In the H2-WENO scheme, a fifth-order upwind component-wise finite difference method is applied in smooth regions, while characteristic-wise WENO reconstruction is retained near discontinuities. The H3-WENO scheme employs a characteristic-wise upwind finite difference method in smooth regions together with characteristic-wise WENO reconstruction in troubled regions.
This work also investigates an important computational aspect of WENO-type methods, namely whether the dominant computational cost arises from the evaluation of nonlinear WENO weights or from the computation of eigenvectors required for characteristic decomposition. Through extensive numerical experiments, we demonstrate that the characteristic decomposition is considerably more expensive than the evaluation of nonlinear WENO weights. By restricting characteristic decomposition only to troubled regions, the proposed hybrid schemes significantly reduce the computational cost. Numerical experiments show that the H1-WENO scheme achieves approximately 40--50\% improvement in computational efficiency compared with the standard WENO-AO scheme, while the H2-WENO and H3-WENO schemes provide improvements of about 30--40\% and 5--10\%, respectively.

The organization of the paper is as follows. In Section \ref{sec:rhd}, we discuss the SRHD equations along with their eigenvector structure. Section \ref{sec:fdm} presents the construction of the numerical scheme for the SRHD equations in the finite difference framework. In Section \ref{sec:weno}, we briefly review the WENO-AO(5,3) reconstruction procedure. Section \ref{sec:hweno} is devoted to the development of the proposed hybrid schemes and the construction of the new troubled-cell indicator. In Section \ref{sec:num}, we validate the proposed schemes through a set of one- and two-dimensional test problems and provide a comparative study of their performance. Finally, conclusions are presented in Section \ref{sec:con}.

\section{Special Relativistic Hydrodynamics Equations}\label{sec:rhd}
The special relativistic hydrodynamics equations describe the evolution of fluids
moving at velocities comparable to the speed of light. In the laboratory frame, they are
expressed in divergence form as
\begin{equation}
\frac{\partial \mathbf{U}}{\partial t} + 
\sum_{\ell=1}^{d} \frac{\partial \mathbf{F}_{\ell}(\mathbf{U})}{\partial x_{\ell}} = 0,
\qquad (\mathbf{x},t) \in \Omega\times (0, T],
\label{eq:RHD}
\end{equation}
where $d\in\{1,2\}$ denotes the spatial dimension. The conserved state vector and
corresponding flux functions are
\begin{align}
\mathbf{U} &= \big(D, m_1, \dots, m_d, E\big)^{\mathsf{T}}, \label{eq:state}\\[1mm]
\mathbf{F}_{\ell}(\mathbf{U}) 
&= \big(Du_{\ell},\; m_1u_{\ell} + p\delta_{1\ell},\; \dots,\; 
      m_du_{\ell} + p\delta_{d\ell},\; m_{\ell}\big)^{\mathsf{T}},
\qquad \ell = 1,\dots,d,
\label{eq:flux}
\end{align}
where $D=\rho W$ is the relativistic mass density, 
$\mathbf{m} = \rho hW^2 \mathbf{u}$ is the momentum vector, and
\[
E = \rho h W^2 - p
\]
is the total energy density. Here $\rho$ is the rest-mass density, $p$ the pressure, 
$W = (1-|\mathbf{u}|^2)^{-1/2}$ the Lorentz factor, and 
$h = 1 + e + p/\rho$ the specific enthalpy ($e$ is internal energy per unit mass).
Units are chosen such that the speed of light $c=1$. To close the system, we adopt the ideal relativistic equation of state
\begin{equation}
p = (\Gamma - 1)\rho e, 
\qquad \Gamma \in (1,2].
\label{eq:eos}
\end{equation}
Due to the nonlinear coupling of variables in \eqref{eq:state}--\eqref{eq:eos}, 
recovery of primitive variables typically requires the solution of a nonlinear
algebraic equation for the pressure.

Relativistic effects play a crucial role in modeling high-energy astrophysical 
phenomena such as relativistic jets, accretion disks, and neutron-star mergers~\cite{landau1987,begelman1984theory,mirabel1999sources,zensus1997parsec,bottcher2012relativistic,anile}.
The strong nonlinearity of \eqref{eq:RHD} prevents analytical solutions in most
practical settings, making robust numerical schemes indispensable. In most numerical simulations, the information about eigenvalues and eigen-structures plays a crucial role. In the following subsection, we discuss the eigen-structure of the RHD equations in both one and two spatial dimensions.
\subsection{1D RHD Equations}
The one–dimensional special relativistic hydrodynamics equations can be written in conservative form as
\begin{equation}
  \mathbf{U}_t + \mathbf{F}(\mathbf{U})_x = 0 ,
\end{equation}
where $\mathbf{U}$ is the vector of conserved variables and $\mathbf{F}(\mathbf{U})$ is the corresponding flux. The above system is strictly hyperbolic~\cite{anile,ryu2006equation}. The flux Jacobian
\[
  \mathcal{A}(\mathbf{U}) := \frac{\partial \mathbf{F}(\mathbf{U})}{\partial \mathbf{U}}
\]
admits three distinct real eigenvalues. Let $u$ denote the fluid velocity in the $x$-direction (with $c=1$), $W=(1-u^2)^{-1/2}$ the Lorentz factor, $h$ the specific enthalpy, and $c_s$ the relativistic sound speed. Then the eigenvalues of $\mathcal{A}$ are given by~\cite{ryu2006equation,deepak2019entropy}
\begin{equation}
\lambda_1=\frac{u-c_s}{1-uc_s}, \qquad
\lambda_2=u, \qquad
\lambda_3=\frac{u+c_s}{1+uc_s}.
\end{equation}
The relativistic sound speed satisfies
\begin{equation}
c_s^2=\left(\frac{\partial p}{\partial e}\right)_s,
\end{equation}
which, for an ideal-gas equation of state, becomes
\begin{equation}
c_s^2=\frac{\gamma p}{\rho h}.
\end{equation}
The Jacobian $\mathcal{A}$ is diagonalizable and can be decomposed as
\begin{equation}
  \mathcal{A} = \mathbf{R}\,\mathbf{D}\,\mathbf{L},
\end{equation}
where $\mathbf{D} = \operatorname{diag}(\lambda_1,\lambda_2,\lambda_3)$ is the diagonal matrix of
eigenvalues, $\mathbf{R}$ is the matrix of right eigenvectors, and 
$\mathbf{L} = \mathbf{R}^{-1}$ is the matrix of left eigenvectors, so that $\mathbf{R}\mathbf{L}=I$.
A convenient choice for the right eigenvector matrix $\mathbf{R}$ is~\cite{Duan2020}
\begin{equation}
  \mathbf{R}
  =
  \begin{bmatrix}
    1 & 1 & 1 \\[0.3em]
    (u - c_s)\,W h & u W & (u + c_s)\,W h \\[0.3em]
    (1 - u c_s)\,W h & W & (1 + u c_s)\,W h
  \end{bmatrix},
\end{equation}
whose columns correspond, respectively, to the eigenvalues 
$\lambda_1, \lambda_2$ and $\lambda_3$.
\subsection{2D RHD Equations }
The two–dimensional special relativistic hydrodynamics equations in conservative form read
\begin{equation}
  \mathbf{U}_t + \mathbf{F}(\mathbf{U})_x + \mathbf{G}(\mathbf{U})_y = 0,
\end{equation}
where $\mathbf{U}$ is the vector of conserved variables and $\mathbf{F}(\mathbf{U})$, $\mathbf{G}(\mathbf{U})$
are the fluxes in the $x$– and $y$–directions, respectively. The system is hyperbolic, and the flux Jacobians
\[
  \mathcal{A}^x(\mathbf{U}) := \frac{\partial \mathbf{F}(\mathbf{U})}{\partial \mathbf{U}},
  \qquad
  \mathcal{A}^y(\mathbf{U}) := \frac{\partial \mathbf{G}(\mathbf{U})}{\partial \mathbf{U}}
\]
admit real eigenvalues and a complete set of eigenvectors. In the $x$–direction, the four eigenvalues of
$\mathcal{A}^x$ are~\cite{Duan2020}
\begin{equation}
  \lambda_1^x = u, \qquad
  \lambda_2^x = u, \qquad
  \lambda_{\pm}^x
  = \frac{u(1 - c_s^2) \pm \dfrac{c_s}{W}\,
      \sqrt{1 - u^2 - v^2 c_s^2}}{1 - (u^2 + v^2)c_s^2},
\end{equation}
while in the $y$–direction, the four eigenvalues of $\mathcal{A}^y$ are
\begin{equation}
  \lambda_1^y = v, \qquad
  \lambda_2^y = v, \qquad
  \lambda_{\pm}^y
  = \frac{v(1 - c_s^2) \pm \dfrac{c_s}{W}\,
      \sqrt{1 - v^2 - u^2 c_s^2}}{1 - (u^2 + v^2)c_s^2}.
\end{equation}
Here, $v$ denote the fluid velocity in the $y$-direction. In the limit $v = 0$ (or $u = 0$), these reduce to the one–dimensional eigenvalues given earlier. For convenience, we introduce the auxiliary factors
\begin{equation}
  \mathcal{A}_{\pm}^x = \frac{1 - u^2}{1 - u\,\lambda_{\pm}^x},
  \qquad
  \mathcal{A}_{\pm}^y = \frac{1 - v^2}{1 - v\,\lambda_{\pm}^y}.
\end{equation}
A suitable choice of right–eigenvector matrices $\mathbf{R}^x(\mathbf{U})$ and $\mathbf{R}^y(\mathbf{U})$
for the $x$– and $y$–direction Jacobians is
\begin{equation}
\mathbf{R}^x(\mathbf{U}) =
\begin{bmatrix}
1 & \dfrac{1}{W} & u\,W & 1 \\[6pt]
\mathcal{A}_{-}^x\,\lambda_{-}^x\,h\,W & u & 2h\,u^{2}\,W^2 & \mathcal{A}_{+}^x\,\lambda_{+}^x\,h\,W \\[6pt]
h\,u\,W & v & h\bigl(1 + 2u^{2} W^2\bigr) & h\,v\,W \\[6pt]
\mathcal{A}_{-}^x\,h\,W & 1 & 2h\,v\,W^2 & \mathcal{A}_{+}^x\,h\,W
\end{bmatrix},
\end{equation}
and
\begin{equation}
\mathbf{R}^y(\mathbf{U}) =
\begin{bmatrix}
1 & W u & \dfrac{1}{W} & 1 \\[8pt]
h W u & h\bigl(1 + 2u^{2} W^2\bigr) & u & h W u \\[8pt]
h W \mathcal{A}^y_{-}\,\lambda^y_{-} & 2h\,u\,v\,W^2 & v & h W \mathcal{A}^y_{+}\,\lambda^y_{+} \\[8pt]
h W \mathcal{A}^y_{-} & 2h\,u\,W^2 & 1 & h W \mathcal{A}^y_{+}
\end{bmatrix}.
\end{equation}
The corresponding left–eigenvector matrices are
\[
  \mathbf{L}^x(\mathbf{U}) = \bigl(\mathbf{R}^x(\mathbf{U})\bigr)^{-1},
  \qquad
  \mathbf{L}^y(\mathbf{U}) = \bigl(\mathbf{R}^y(\mathbf{U})\bigr)^{-1}.
\]
Thus, the Jacobians admit the eigen-decompositions
\begin{equation}
  \mathcal{A}^x = \mathbf{R}^x \,\mathbf{D}^x\, \mathbf{L}^x,
  \qquad
  \mathcal{A}^y = \mathbf{R}^y \,\mathbf{D}^y\, \mathbf{L}^y,
\end{equation}
where
\[
  \mathbf{D}^x = \operatorname{diag}(\lambda_1^x, \lambda_2^x, \lambda_{-}^x, \lambda_{+}^x),
  \qquad
  \mathbf{D}^y = \operatorname{diag}(\lambda_1^y, \lambda_2^y, \lambda_{-}^y, \lambda_{+}^y).
\]

\section{Finite Difference Scheme (FDM)}\label{sec:fdm}
In the previous section, we have discussed the RHD equations and its eigen-value structure. 
This section outlines the flux reconstruction method for
solving systems of conservation laws in finite difference framework. We begin with the one-dimensional system of RHD equations,
expressed as
\begin{equation}\label{main.ex}
  \mathbf{U}_t + \mathbf{F}(\mathbf{U})_x = 0,
  \qquad (x,t) \in [a,b] \times (0,T].
\end{equation}
subject to appropriate boundary condition. The procedure is then generalized to higher dimensions using a dimension–by–dimension approach
(see \cite{shu-97d, jia-shu_96a} for further details).
To numerically approximate the governing conservation law \eqref{main.ex}, 
the physical domain is divided into a collection of non-overlapping control 
volumes (cells). Each computational cell is denoted by
\[
\mathbb{I}_i = [x_{i-\frac{1}{2}},\, x_{i+\frac{1}{2}}],
\]
where \(x_{i-\frac{1}{2}}\) and \(x_{i+\frac{1}{2}}\) represent the left and 
right interfaces of the \(i\)-th cell, respectively. For simplicity, we 
assume a uniform mesh throughout the domain, so that the cell width$
\Delta x = x_{i+\frac{1}{2}} - x_{i-\frac{1}{2}}
$
remains constant for all \(i\). The geometric center of the cell is defined as
\[
x_i = \frac{1}{2}\left(x_{i-\frac{1}{2}} + x_{i+\frac{1}{2}}\right),~~~~ i=0, 1,2,\ldots, N.
\]

In the finite difference formulation, approximating  conservation law
\eqref{main.ex} over each cell \(\mathbb{I}_i\), we obtain the following 
semi-discrete representation:
\begin{equation}
  \frac{d \mathbf{U}_i(t)}{dt} 
  = -\frac{1}{\Delta x}
  \left(
      \mathbb{F}_{i+\frac{1}{2}} 
      - \mathbb{F}_{i-\frac{1}{2}}
  \right),
  \qquad i=0,1,\ldots,N,
  \label{eq:semids_FD}
\end{equation}
where $\mathbf{U}_i(t)$ denotes the discrete solution vector at the grid point 
$x_i$, and $\mathbb{F}_{i\pm\frac{1}{2}}$ are the numerical fluxes evaluated at 
the cell interfaces. The accuracy and stability of the scheme strongly depend 
on the construction of the numerical flux.
To ensure numerical robustness, particularly in the presence of discontinuities, 
the physical flux $\mathbf{F}(\mathbf{U})$ is decomposed into right-going and 
left-going components:
\begin{equation}
 \mathbf{F}(\mathbf{U})
 = \mathbf{F}^{+}(\mathbf{U}) 
 + \mathbf{F}^{-}(\mathbf{U}),
 \label{eq:flux_split_general}
\end{equation}
where $\mathbf{F}^{+}$ is associated with non-negative characteristic speeds and 
$\mathbf{F}^{-}$ corresponds to non-positive characteristic speeds, i.e.,
\[
\frac{d\mathbf{F}^{+}}{d\mathbf{U}} \geq 0, 
\qquad
\frac{d\mathbf{F}^{-}}{d\mathbf{U}} \leq 0.
\]

Different flux splitting strategies are available in the literature 
(e.g., see \cite{shu-97d}). In this work, we adopt the local 
Lax--Friedrichs (LF) splitting due to its robustness and simplicity:
\begin{equation}
\mathbf{F}^{\pm}(\mathbf{U})
= \frac{1}{2}
\left(
    \mathbf{F}(\mathbf{U})
    \pm \lambda \mathbf{U}
\right),
\label{eq:LF_split}
\end{equation}
where $\lambda$ is a suitably chosen maximum absolute of the 
eigenvalues of the flux Jacobian $\partial \mathbf{F}/\partial \mathbf{U}$ 
computed locally in the $x$-direction. The corresponding split form of the 
numerical flux becomes
\begin{equation}
 \mathbb{F}_{i+\frac{1}{2}}
 = \mathbb{F}^{+}_{i+\frac{1}{2}}
 + \mathbb{F}^{-}_{i+\frac{1}{2}}.
 \label{eq:numerical_flux_split}
\end{equation}

High-order accuracy in smooth regions and non-oscillatory behavior near shocks 
are achieved by reconstructing $\mathbb{F}^{+}_{i+\frac{1}{2}}$ and 
\(\mathbb{F}^{-}_{i+\frac{1}{2}}\) separately using nonlinear WENO 
reconstruction techniques. The upcoming subsections summarize both the 
component-wise and characteristic-wise WENO reconstructions for the RHD 
equations. For a detailed and comprehensive discussion, the reader is referred 
to \cite{shu-97d}.

\subsubsection{Component-Wise Reconstruction}
To compute the flux at the interface $x_{i+\frac{1}{2}}$, the positive and negative flux components are calculated using the Lax-Friedrichs splitting:
\begin{equation*}
 \mathbb{H}_i^{\pm} = \frac{1}{2}\left(\mathbf{F}(\mathbf{U}_i) \pm \lambda \mathbf{U}_i\right), \quad i=0,1,2,\ldots,N.
\end{equation*}
For each component $p = 0, 1, 2$, the WENO reconstruction is applied as follows:
\[
    (\mathbb{F}_{i+\frac{1}{2}}^{+})^p = \text{WENO-Reconstruction}\left((\mathbb{H}_{i-2}^{+})^p, (\mathbb{H}_{i-1}^{+})^p, (\mathbb{H}_{i}^{+})^p, (\mathbb{H}_{i+1}^{+})^p, (\mathbb{H}_{i+2}^{+})^p\right),
\]
\[
    (\mathbb{F}_{i+\frac{1}{2}}^{-})^p = \text{WENO-Reconstruction}\left((\mathbb{H}_{i+3}^{-})^p, (\mathbb{H}_{i+2}^{-})^p, (\mathbb{H}_{i+1}^{-})^p, (\mathbb{H}_{i}^{-})^p, (\mathbb{H}_{i-1}^{-})^p\right).
\]
The reconstructed flux is then given by:
\[
    \mathbb{F}_{i+\frac{1}{2}} = \mathbb{F}_{i+\frac{1}{2}}^{+} + \mathbb{F}_{i+\frac{1}{2}}^{-}.
\]

\subsubsection{Characteristic-Wise Reconstruction}
In the characteristic-wise approach, the left and right eigenvectors of the Jacobian matrix $\mathbf{F}'(\mathbf{U})$ are computed at the interface $x_{i+\frac{1}{2}}$, denoted as $\mathbf{L}_{i+\frac{1}{2}} = \mathbf{L}(\mathbf{V}_{i+\frac{1}{2}})$ and $\mathbf{R}_{i+\frac{1}{2}} = \mathbf{R}(\mathbf{V}_{i+\frac{1}{2}})$, where $\mathbf{V}_{i+\frac{1}{2}}$ is the average of $\mathbf{V}_i$ and $\mathbf{V}_{i+1}$ primitive variable. The flux splitting is performed as:
\[
    \mathbb{H}_k^{+} = \mathbf{L}_{i+\frac{1}{2}} \left[\frac{1}{2}\left(\mathbf{F}(\mathbf{U}_k) + \lambda \mathbf{U}_k\right)\right], \quad k = i-2, i-1, i, i+1, i+2,
\]
\[
    \mathbb{H}_k^{-} = \mathbf{L}_{i+\frac{1}{2}} \left[\frac{1}{2}\left(\mathbf{F}(\mathbf{U}_k) - \lambda \mathbf{U}_k\right)\right], \quad k = i-1, i, i+1, i+2, i+3.
\]
For each flux component $p = 0, 1, 2$, the WENO reconstruction is carried out as:
\[
    (\mathbb{F}_{i+\frac{1}{2}}^{+})^p = \text{WENO-Reconstruction}\left((\mathbb{H}_{i-2}^{+})^p, (\mathbb{H}_{i-1}^{+})^p, (\mathbb{H}_{i}^{+})^p, (\mathbb{H}_{i+1}^{+})^p, (\mathbb{H}_{i+2}^{+})^p\right),
\]
\[
    (\mathbb{F}_{i+\frac{1}{2}}^{-})^p = \text{WENO-Reconstruction}\left((\mathbb{H}_{i+3}^{-})^p, (\mathbb{H}_{i+2}^{-})^p, (\mathbb{H}_{i+1}^{-})^p, (\mathbb{H}_{i}^{-})^p, (\mathbb{H}_{i-1}^{-})^p\right).
\]
Finally, the reconstructed flux is combined as:
\[
    \mathbb{F}_{i+\frac{1}{2}} = \mathbf{R}_{i+\frac{1}{2}} \left(\mathbb{F}_{i+\frac{1}{2}}^{+} + \mathbb{F}_{i+\frac{1}{2}}^{-}\right).
\]

Similarly, the flux at $x_{i-\frac{1}{2}}$ is computed.

\subsection{Extension to Multiple Dimensions}
The WENO reconstruction procedure can be extended to multi-dimensional systems using a dimension-by-dimension approach \cite{jia-shu_96a, bal-shu_00a}. Consider the two-dimensional RHD equations:
\begin{equation}\label{main.ex2d}
  \mathbf{U}_t + \mathbf{F(U)}_x + \mathbf{G(U)}_y = 0, \quad (x,y,t)\in \Omega\times (0,T],
\end{equation}
in the computational domain $\Omega = [a,b] \times [c,d]$. The domain is discretized into uniform rectangular cells $\mathbb{I}_{ij} = [x_{i-\frac{1}{2}}, x_{i+\frac{1}{2}}] \times [y_{j-\frac{1}{2}}, y_{j+\frac{1}{2}}]$ with mesh widths $\Delta x$ and $\Delta y$. The cell center is located at
\[
(x_i, y_j) = \left(\frac{1}{2}(x_{i-\frac{1}{2}} + x_{i+\frac{1}{2}}), \frac{1}{2}(y_{j-\frac{1}{2}} + y_{j+\frac{1}{2}})\right),~i=0,1,\ldots,N_x~~j=0,1,\ldots,N_y.
\]
where $N_x$ and $N_y$ denote the number of cells along the $x$ and $y$-direction, respectively. 

The semi-discrete form of equation \eqref{main.ex2d} over cell $\mathbb{I}_{ij}$ is:
\begin{equation}\label{semi2d}
  \frac{d \mathbf{U}_{ij}(t)}{dt} = -\frac{1}{\Delta x}\left(\mathbb{F}_{i+\frac{1}{2},j} - \mathbb{F}_{i-\frac{1}{2},j}\right) - \frac{1}{\Delta y}\left(\mathbb{G}_{i,j+\frac{1}{2}} - \mathbb{G}_{i,j-\frac{1}{2}}\right),
\end{equation}
where $\mathbf{U}_{ij}(t)$ is the solution value at the cell center, and $\mathbb{F}_{i+\frac{1}{2},j}$ and $\mathbb{G}_{i,j+\frac{1}{2}}$ are the numerical fluxes at the interface.

\subsubsection{Dimension-by-Dimension Reconstruction}
The numerical fluxes in each direction are computed independently using the one-dimensional WENO reconstruction procedure. For the flux of the $x$-direction $\mathbb{F}_{i+\frac{1}{2},j}$, we consider a one-dimensional reconstruction along the $x$-direction while holding the $y$-index $j$ fixed. Similarly, for the flux of the $y$ -direction $\mathbb{G}_{i,j+\frac{1}{2}}$, we perform a reconstruction along the $y$-direction with a fixed $x$ -index $i$.

For the flux in the $x$-direction:
\begin{itemize}
    \item For fixed value of $j$, perform flux splitting in the $x$-direction:
    \begin{align*}
        \mathbb{H}_{k,j}^{x,+} &= \frac{1}{2}\left(\mathbf{F}(\mathbf{U}_{k,j}) + \lambda_x \mathbf{U}_{k,j}\right), ~~k=~i-2,~i-~1,i, ~i+1, ~i+2,\\
        \mathbb{H}_{k,j}^{x,-} &= \frac{1}{2}\left(\mathbf{F}(\mathbf{U}_{k,j}) - \lambda_x \mathbf{U}_{k,j}\right), ~~k=~i-1,~ i,~ i+1,~ i+2,~ i+3,
    \end{align*}
    where $\lambda_x$ is the maximum absolute eigenvalue of the Jacobian $\partial \mathbf{F}/\partial \mathbf{U}$ over the domain.
    \item Apply WENO reconstruction in component- or characteristic-wise to obtain $\mathbb{F}_{i+\frac{1}{2},j}^{\pm}$.
    \item Combine to get $\mathbb{F}_{i+\frac{1}{2},j} = \mathbb{F}_{i+\frac{1}{2},j}^{+} + \mathbb{F}_{i+\frac{1}{2},j}^{-}$.
\end{itemize}

For the flux in $y$-direction:
\begin{itemize}
    \item For fixed value of $i$, perform flux splitting in the $y$-direction:
   \begin{align*}
        \mathbb{H}_{i,k}^{y,+} &= \frac{1}{2}\left(\mathbf{G}(\mathbf{U}_{i,k}) + \lambda_y \mathbf{U}_{i,k}\right), ~~k=~j-2,~j-~1,j, ~j+1, ~j+2,\\
        \mathbb{H}_{i,k}^{y,-} &= \frac{1}{2}\left(\mathbf{G}(\mathbf{U}_{i,k}) - \lambda_y \mathbf{U}_{i,k}\right), ~~k=~j-1,~ j,~ j+1,~ j+2,~ j+3,
    \end{align*}
    where $\lambda_y$ is the maximum absolute eigenvalue of the Jacobian $\partial \mathbf{G}/\partial \mathbf{U}$ over the domain.
    \item Apply WENO reconstruction in component- or characteristic-wise to obtain $\mathbb{G}_{i,j+\frac{1}{2}}^{\pm}$.
    \item Combine to get $\mathbb{G}_{i,j+\frac{1}{2}} = \mathbb{G}_{i,j+\frac{1}{2}}^{+} + \mathbb{G}_{i,j+\frac{1}{2}}^{-}$.
\end{itemize}

\section{Reconstruction Methodology: WENO-AO(5,3)}\label{sec:weno}
In the previous section, we discussed the finite-difference formulation for systems of the hyperbolic conservation laws. In this section, we briefly review the  WENO reconstruction employed in the present work. Several variants of WENO reconstruction have been proposed in the literature, including WENO-JS~\cite{jia-shu_96a}, WENO-Z \cite{bor-etal_08a, don-bor_13a}, CWENO ~\cite{pup_03a,cra-etal_18a,cra-etal_19a,dum-etal_17a}, and WENO-AO~\cite{bal-etal_16a, bal-etal_20a, kum-cha_18a} schemes. Among these, we adopt the WENO-AO(5,3) scheme, as it does not require the existence of positive linear weights and has been shown to provide improved accuracy and resolution compared to classical WENO reconstructions.
This section details the implementation of the WENO-AO(5,3) scheme~\cite{bal-etal_16a}, which is used for high-order spatial reconstruction within the finite-difference framework described earlier. 
Specifically, the WENO-AO(5,3) scheme~\cite{bal-etal_16a} nonlinearly combines one fourth-degree polynomial constructed on a five-point stencil with three quadratic polynomials defined on compact three-point sub-stencils. This adaptive hybridization yields fifth-order accuracy in smooth regions of the solution and automatically reduces to third-order accuracy in the vicinity of discontinuities, thereby enhancing numerical stability while maintaining high resolution. The brief detail of WENO-AO(5,3) reconstruction is describe as below.

\subsection{Reconstruction formulation.}
Given the cell-averaged data, the reconstructed flux value at the interface \( x_{i+\frac{1}{2}} \) is computed as
\[
\mathbb{F}_{i+\frac{1}{2}} = 
\frac{\omega_0^5}{\gamma_0^5}
\Bigg[
  \mathbb{P}_0^5(x_{i+\frac{1}{2}}) 
  - \sum_{k=-1}^{1} \gamma_k^3 \, \mathbb{P}_k^3(x_{i+\frac{1}{2}})
\Bigg]
+ \sum_{k=-1}^{1} \omega_k^3 \, \mathbb{P}_k^3(x_{i+\frac{1}{2}}),
\]
where \( \mathbb{P}_k^m \) denotes the polynomial of degree \( m \) reconstructed on stencil \( \mathbb{S}_k^m \) (see \cite{bal-etal_16a, kum-cha_18a} for more details and notations). The parameter \( \gamma_k^m \) are known as linear weights and can be chosen as \cite{bal-etal_16a}
\[
\begin{aligned}
\gamma_{-1}^3 &= \tfrac{1}{2}(1-\gamma_{\mathrm{Hi}})(1-\gamma_{\mathrm{Lo}}), \\[2pt]
\gamma_{0}^3 &= (1-\gamma_{\mathrm{Hi}})\gamma_{\mathrm{Lo}}, \\[2pt]
\gamma_{1}^3 &= \tfrac{1}{2}(1-\gamma_{\mathrm{Hi}})(1-\gamma_{\mathrm{Lo}}), \\[2pt]
\gamma_{0}^5 &= \gamma_{\mathrm{Hi}},
\end{aligned}
\]
with \( \gamma_{\mathrm{Hi}} = \gamma_{\mathrm{Lo}} = 0.85 \) used in all computations presented herein. These weights satisfy the convexity condition \( \sum_{k=-1}^{1} \gamma_k^3 + \gamma_0^5 = 1 \). The $\omega_k^m$ are the non-linear weight corresponding to linear weights $\gamma_k^m$ and are constructed using smoothness indicator \cite{jia-shu_96a, bal-etal_16a}. The  smoothness indicators \( \beta_k^m \) are computed for each stencil \( \mathbb{S}_k^m \) and for a polynomial of degree \( m \) it is given by \cite{jia-shu_96a, bal-etal_16a}
\[
\beta_k^m = \sum_{\ell=1}^{m}
\int_{x_{i-\frac12}}^{x_{i+\frac12}}
\Delta x^{2\ell-1}
\Bigl( \frac{d^\ell}{dx^\ell} \mathbb{P}_k^m(x) \Bigr)^2 dx,
\]
which measures the integrated square of its derivatives. In \cite{bal-etal_16a}, author shown that smoothness indicator can be written as sum of square of co-efficient of polynomials.  The un-normalized nonlinear weights are then constructed via
\[
\widetilde{\omega}_k^m = 
\gamma_k^m \Bigl( 1 + \frac{\tau^2}{(\epsilon + \beta_k^m)^2} \Bigr),
\]
with parameters \( \epsilon = 10^{-12} \) (to avoid division by zero) and $\tau$
is known as {global} smoothness indicator and it is given by: 
\[
\tau = \frac{1}{3}\Bigl(| 
\beta_0^5 - \beta_{-1}^3| + |\beta_0^5-\beta_0^3| + |\beta_0^5-\beta_1^3| \bigr),
\]
The final normalized weights are
\[
\omega_k^m = 
\frac{\widetilde{\omega}_k^m}{
   \widetilde{\omega}_0^5 +\displaystyle\sum_{l=-1}^{1} \widetilde{\omega}_l^3 
},
\]
ensuring \( \sum_{k=-1}^{1} \omega_k^3 + \omega_0^5 = 1 \).
In smooth regions, \( \tau = \mathcal{O}(\Delta x^6) \) and \( \beta_k^m = \mathcal{O}(\Delta x^2) \), causing the nonlinear weights to converge to the optimal linear weights and recovering the full fifth-order accuracy of the central polynomial. Near discontinuities, stencils containing the discontinuity yield large \( \beta_k^m \), and the corresponding weights are suppressed. The reconstruction then falls back to a convex combination of the smooth quadratic polynomials, ensuring an essentially non-oscillatory third-order profile.

\section{Hybrid FDM-WENO Reconstruction}\label{sec:hweno}
In the previous sections, we discussed both the component-wise and characteristic-wise WENO reconstruction procedures in a finite-difference framework for the system of hyperbolic conservation laws. Component-wise WENO reconstruction is computationally efficient for systems of hyperbolic conservation laws; however, it has been reported in the literature (see for instance \cite{kum_cha_22a, jun-etal_19a, pup_03a}) that this approach may introduce spurious small-scale oscillations when compared to linear reconstructions, contrary to the fundamental purpose of WENO schemes. On the other hand, characteristic-wise WENO reconstruction effectively suppresses such oscillations but is computationally expensive, as it requires the evaluation of left and right eigenvectors at every cell interface. Moreover, in smooth regions of the solution, the nonlinear weights of the WENO method naturally approach the linear weights, reducing the scheme to a linear reconstruction \cite{jun-etal_19a, aru_etal_22a}. This observation motivates the development of a hybrid strategy that blends linear and nonlinear reconstructions, as well as component-wise and characteristic-wise formulations, depending on the local smoothness of the solution.

In designing an efficient hybrid scheme, it is crucial to develop a reliable troubled-cell indicator that separates the computational domain based on the regularity of the solution or flux. An ideal troubled-cell indicator must satisfy two key requirements: (i) it should accurately detect discontinuities or shock regions, and (ii) it should avoid falsely marking smooth regions as troubled cells to ensure computational efficiency.
Several troubled-cell indicators have been proposed in the literature for hybrid schemes developed for systems of hyperbolic conservation laws (see for instance \cite{pup_03a,zhao-etal_20a,kum_18a}). In many cases, the indicator is constructed using the WENO smoothness indicators \cite{pup_03a, kum_20a, jun-etal_19a, hil-pul_04a,vis-gai_05a,mov-joh_13a,fu-etal_14a,kum_cha_22a, aru_etal_22a}, which has been shown to improve both robustness and efficiency. Motivated by these developments, we also aim to design a troubled-cell indicator based on the WENO smoothness indicators.

In this section, we have proposed a new  troubled-cell indicator for a system of hyperbolic conservation laws using the information of smoothness indicators of the WENO-AO(5,3) scheme defined over the stencils $\mathbb{S}_i=\{i-2,i-1,i,i+1,i+2\}$ and $\mathbb{S}_{i+1}=\{i-1,i, i+1,i+2,i+3\}$. Here, we discuss the idea for the 1D scalar equation with flux $f(u)$ and then extend it to the system of equation as it is done in \cite{kum_cha_22a, aru_etal_22a}. 
Consider the lower order smoothness indicator of WENO-AO(5,3) scheme defined over the smaller stencils and define the average of smoothness indicator as follows:
\begin{equation}
    \beta^-_{i+\frac{1}{2}} = \frac{1}{3} \sum_{k=0}^2 \left(\beta_{i+\frac{1}{2}}^3\right)_k^{-}, ~~~ \beta^+_{i+\frac{1}{2}} = \frac{1}{3} \sum_{k=0}^2 \left(\beta_{i+\frac{1}{2}}^3\right)_k^{+},~~
\end{equation}
where $\left(\beta^3_{i+\frac{1}{2}}\right)^{+}_k$ and $\left(\beta^3_{i+\frac{1}{2}}\right)^{-}_k$  denote the smoothness indicator defined over the sub-stencil $\mathbb{S}_{i,k}^3 =\{i+k-2, i+k-1, i+k\}$ and  $\mathbb{S}_{i+1,k}^3 =\{i+k-1, i+k, i+k+1\}$ for $k=0,~1,~2$, respectively. 
 The smoothness indicators for a scalar flux $f(u)$ is given by \cite{jia-shu_96a}
\begin{align*}
    \left(\beta^3_{i+\frac{1}{2}}\right)^{+}_0 &= \frac{1}{4}(f_{i-2}^+-4f_{i-1}^++ 3f_i^+)^2 + \frac{13}{12} (f_{i-2}^+ -2 f_{i-1}^++ f_i^+)^2,\\
    \left(\beta^3_{i+\frac{1}{2}}\right)^{+}_1 &= \frac{1}{4}(f_{i+1}^+-f_{i-2}^+)^2 + \frac{13}{12} (f_{i-1}^+-2f_i^++f_{i+1}^+)^2,\\
    \left(\beta^3_{i+\frac{1}{2}}\right)^{+}_2 &= \frac{1}{4}(-3f_i^+ + 4f_{i+1}^+ -f_{i+2}^+)^2 + \frac{13}{12}(f_i^+ - 2f_{i+1}^+ +f_{i+2}^+)^2.
 \end{align*}
Similarly, we can obtain for negative part also. We have following the following estimate about the smoothness indicator and proof is similar to \cite{kum_cha_22a, aru_etal_22a}. 

\begin{theorem}{\rm If the flux $f(u)$ is smooth over the stencil $\mathbb{S}_{i}$ and $\mathbb{S}_{i+1}$, then  we have
\begin{equation}
\sqrt{(\beta_\iph)^-} + \sqrt{(\beta_\iph)^+} = O(\lambda \Delta x)
\label{eq:smest1}
\end{equation}
where $\lambda= \max|f'(u)|$ and maximum is taken over suitable range of $u$. 
}
\end{theorem}
\begin{proof}
Using Taylor series expansion of smoothness indicators $\beta^+_{i+\frac{1}{2}}$ at $x_i$, we obtain
\begin{equation*}\label{si.1}
 \beta^+_{i+\frac{1}{2}} = \left( \frac{d f^+}{d u} \right)^2 u_x^2\Delta x^2  +O(\Delta x^3)
\end{equation*}
Similarly for $\beta^-_{i+\frac{1}{2}}$, we have
\begin{equation*}\label{si.2}
 \beta^-_{i+\frac{1}{2}} = \left( \frac{d f^-}{d u} \right)^2 u_x^2 \Delta x^2 + O(\Delta x^3)
\end{equation*}
But
\[
\frac{d f^\pm}{d u} = \frac{1}{2}[ f'(u_i) \pm \lambda]
\]
Since $\lambda \ge |f'(u)|$, we get
\[
\sqrt{\beta^+} \approx \frac{1}{2}[f'(u_i) + \lambda] |u_x|\Delta x, \qquad \sqrt{\beta^-} \approx -\frac{1}{2}[f'(u_i) - \lambda] |u_x|\Delta x
\]
and hence adding the two we obtain the desired result.
\end{proof}

We will build a smoothness detector using the result in the above theorem. Firstly, we note that
\[
\sqrt{\beta^-_{i+\frac{1}{2}}} + \sqrt{\beta^+_{i+\frac{1}{2}}} \le \sqrt{2} \sqrt{\beta^-_{i+\frac{1}{2}} + \beta^+_{i+\frac{1}{2}}}
\]
so that
\[
\beta^-_{i+\frac{1}{2}} + \beta^+_{i+\frac{1}{2}} \approx \frac{1}{2} \left( \sqrt{\beta^-_{i+\frac{1}{2}}} + \sqrt{\beta^+_{i+\frac{1}{2}}} \right)^2 \approx \frac{1}{2} \lambda^2 u_x^2 \Delta x^2
\]
While we could use~\eqref{eq:smest1}, the above estimate avoids the computation of the square root. If the solution is smooth around $x_\iph$, then the quantity $\beta^-_{i+\frac{1}{2}} + \beta^+_{i+\frac{1}{2}} = O(\lambda^2 \Delta x^2)$, and otherwise it will be larger than this quantity. Based on this idea, we propose to use the following troubled cell indicator,
\[
 \mathbb{T}_{i+\frac{1}{2}} = \begin{cases}
      1, & \mbox{if}~~ \frac{1}{2}\left[(\beta_{i+\frac{1}{2}}^+) +(\beta_{i+\frac{1}{2}}^-)\right]> \mathbb{K} \\
      0, & \mbox{otherwise}
     \end{cases}
\]
where
\[
 \mathbb{K}=\mathcal{K} \Delta x^2\max_{i}(|f'(u)|^2).
\]
$\mathcal{K}$ is constant introduced to avoid unwanted marking of smooth cells. 
This is valid for scalar problem, we have extended this to RHD equation as follows:
\[
 \mathbb{T}_{i+\frac{1}{2}} = \begin{cases}
      1, & \mbox{if}~~ \frac{1}{2r}\displaystyle{\sum_{q=0}^{r-1}\left[(\beta_{i+\frac{1}{2}}^+)^q +(\beta_{i+\frac{1}{2}}^-)^q\right]> \mathbb{K}} \\
      0, & \mbox{otherwise}
     \end{cases}
\]
where $r$ is the number of equations in the system and is equal to $r=3$ for 1D problems and $r=4$ for 2D problems. The value of $\mathbb{K}$ for 1D RHD problems in the domain $[a,b]$ is
\[
 \mathbb{K}=\frac{\mathcal{K} \Delta x^2}{(b-a)^2} \max(\lambda_1^x, \lambda_2^x,\lambda_3^x)^2.
\]
In case of 2D problems, over the domain $\Omega =[a,b]\times [c,d]$, along the $x$-direction, $\mathbb{K}$ is defined as
\[
 \mathbb{K}=\frac{\mathcal{K} \Delta x^2}{\max\{(b-a)^2,(c-d)^2\}} \max(\lambda_1^x, \lambda_2^x, \lambda_{-}^x, \lambda_{+}^x)^2,
\]
with a similar expression for the $y$ component of the flux.

\begin{algorithm}\label{alog2D}
\small
\SetAlgoLined
\begin{enumerate}
\item \textbf{Flux reconstruction in the $x$–direction (interfaces $x_{i+\frac{1}{2},j}$).}
\begin{enumerate}
\item
Compute the positive and negative parts of the $x$–flux
\[
  \mathbb{H}_{i+r,j}^{x,\pm}
  = \frac{1}{2}\bigl({\bf F}({\bf U}_{i+r,j}) \pm \lambda_x {\bf U}_{i+r,j}\bigr),
  \qquad r=-2,-1,0,1,2,3.
\]
Compute 
\((\beta_{i+\frac{1}{2},j}^+)^{p}\), \((\beta_{i+\frac{1}{2},j}^{-})^{p}\), for $p=0,\ldots 3$ and troubled-cell indicator
\(\mathbb{T}_{i+\frac{1}{2},j}^{x}\).

\item
\eIf{$\mathbb{T}_{i+\frac{1}{2},j}^{x} == 0$}
{
\[
  \mathbb{F}_{i+\frac{1}{2},j}^{+}
  = \mbox{WENO-AO}(
  \mathbb{H}_{i-2,j}^{x,+}, \mathbb{H}_{i-1,j}^{x,+}, \mathbb{H}_{i,j}^{x,+}, \mathbb{H}_{i+1,j}^{x,+},\mathbb{H}_{i+2,j}^{x,+}
  ),
\]
\[
  \mathbb{F}_{i+\frac{1}{2},j}^{-}
  = \mbox{WENO-AO}(
  \mathbb{H}_{i+3,j}^{x,-},\mathbb{H}_{i+2,j}^{x,-}, \mathbb{H}_{i+1,j}^{x,-},\mathbb{H}_{i,j}^{x,-}, \mathbb{H}_{i-1,j}^{x,-}
  ),
\]
\[
  \mathbb{F}_{i+\frac{1}{2},j}
  = \mathbb{F}_{i+\frac{1}{2},j}^{+}
  + \mathbb{F}_{i+\frac{1}{2},j}^{-}.
\]
}
{
Compute the left and right eigenvectors of the Jacobian
\({\bf F}'({\bf U})\) at the interface \(x_{i+\frac{1}{2},j}\),
denoted by \({\bf L}_{i+\frac{1}{2},j}^{x}\) and
\({\bf R}_{i+\frac{1}{2},j}^{x}\).
\[
  \mathbb{M}_{i+r,j}^{x,\pm}
  = {\bf L}_{i+\frac{1}{2},j}^{x}\,\mathbb{H}_{i+r,j}^{x,\pm},
  \qquad r=-2,-1,0,1,2,3,
\]
\[
  \mathbb{F}_{i+\frac{1}{2},j}^{+}
  = \mbox{WENO}\bigl(
  \mathbb{M}_{i-2,j}^{x,+},
  \mathbb{M}_{i-1,j}^{x,+},
  \mathbb{M}_{i,j}^{x,+},
  \mathbb{M}_{i+1,j}^{x,+},
  \mathbb{M}_{i+2,j}^{x,+}
  \bigr),
\]
\[
  \mathbb{F}_{i+\frac{1}{2},j}^{-}
  = \mbox{WENO}\bigl(
  \mathbb{M}_{i+3,j}^{x,-},
  \mathbb{M}_{i+2,j}^{x,-},
  \mathbb{M}_{i+1,j}^{x,-},
  \mathbb{M}_{i,j}^{x,-},
  \mathbb{M}_{i-1,j}^{x,-}
  \bigr),
\]
\[
  \mathbb{F}_{i+\frac{1}{2},j}
  = {\bf R}_{i+\frac{1}{2},j}
  \bigl(
    \mathbb{F}_{i+\frac{1}{2},j}^{+}
    + \mathbb{F}_{i+\frac{1}{2},j}^{-}
  \bigr).
\]
}
\end{enumerate}

\item \textbf{Flux reconstruction in the $y$–direction (interfaces $y_{i,j+\frac{1}{2}}$).}
\begin{enumerate}
\item
Compute the positive and negative parts of the $y$–flux
\[
  \mathbb{H}_{i,j+r}^{y,\pm}
  = \frac{1}{2}\bigl({\bf G}({\bf U}_{i,j+r}) \pm \lambda_y {\bf U}_{i,j+r}\bigr),
  \qquad r=-2,-1,0,1,2,3.
\]
Compute  
\((\beta_{i,j+\frac{1}{2}}^+)^{p}\), \((\beta_{i,j+\frac{1}{2}}^{-})^{p}\), for $p=0,\ldots 3$ and troubled-cell indicator
\(\mathbb{T}_{i,j+\frac{1}{2}}^{y}\).

\item
\eIf{$\mathbb{T}_{i,j+\frac{1}{2}}^{y} = 0$}
{
\[
  \mathbb{G}_{i,j+\frac{1}{2}}^{+}
  = \mbox{WENO-AO}(
  \mathbb{H}_{i,j-2}^{y,+},
  \mathbb{H}_{i,j-1}^{y,+},
  \mathbb{H}_{i,j}^{y,+},
  \mathbb{H}_{i,j+1}^{y,+},
  \mathbb{H}_{i,j+2}^{y,+}
  ),
\]
\[
  \mathbb{G}_{i,j+\frac{1}{2}}^{-}
  = \mbox{WENO-AO}(
  \mathbb{H}_{i,j+3}^{y,-},
  \mathbb{H}_{i,j+2}^{y,-},
  \mathbb{H}_{i,j+1}^{y,-},
  \mathbb{H}_{i,j}^{y,-},
  \mathbb{H}_{i,j-1}^{y,-}),
\]
\[
  \mathbb{G}_{i,j+\frac{1}{2}}
  = \mathbb{G}_{i,j+\frac{1}{2}}^{+}
  + \mathbb{G}_{i,j+\frac{1}{2}}^{-}.
\]
}
{
Compute the left and right eigenvectors of the Jacobian
\({\bf G}'({\bf U})\) at the interface \(y_{i,j+\frac{1}{2}}\),
denoted by \({\bf L}_{i,j+\frac{1}{2}}^{y}\) and
\({\bf R}_{i,j+\frac{1}{2}}^{y}\).
\[
  \mathbb{M}_{i,j+r}^{y,\pm}
  = {\bf L}_{i,j+\frac{1}{2}}^{y}\,\mathbb{H}_{i,j+r}^{y,\pm},
  \qquad r=-2,-1,0,1,2,3,
\]
\[
  \mathbb{G}_{i,j+\frac{1}{2}}^{+}
  = \mbox{WENO}\bigl(
  \mathbb{M}_{i,j-2}^{y,+},
  \mathbb{M}_{i,j-1}^{y,+},
  \mathbb{M}_{i,j}^{y,+},
  \mathbb{M}_{i,j+1}^{y,+},
  \mathbb{M}_{i,j+2}^{y,+}
  \bigr),
\]
\[
  \mathbb{G}_{i,j+\frac{1}{2}}^{-}
  = \mbox{WENO}\bigl(
  \mathbb{M}_{i,j+3}^{y,-},
  \mathbb{M}_{i,j+2}^{y,-},
  \mathbb{M}_{i,j+1}^{y,-},
  \mathbb{M}_{i,j}^{y,-},
  \mathbb{M}_{i,j-1}^{y,-}
  \bigr),
\]
\[
  \mathbb{G}_{i,j+\frac{1}{2}}
  = {\bf R}_{i,j+\frac{1}{2}}^{y}
  \bigl(
    \mathbb{G}_{i,j+\frac{1}{2}}^{+}
    + \mathbb{G}_{i,j+\frac{1}{2}}^{-}
  \bigr).
\]
}
\end{enumerate}
\end{enumerate}
\caption{H1-WENO scheme for RHD equations over the cell $(i,j)$.}
\label{algo2Dh1}
\end{algorithm}

\begin{algorithm}\label{alog2D}
\small
\SetAlgoLined
\begin{enumerate}
\item \textbf{Flux reconstruction in the $x$–direction (interfaces $x_{i+\frac{1}{2},j}$).}
\begin{enumerate}
\item
Compute the positive and negative parts of the $x$–flux
\[
  \mathbb{H}_{i+r,j}^{x,\pm}
  = \frac{1}{2}\bigl({\bf F}({\bf U}_{i+r,j}) \pm \lambda_x {\bf U}_{i+r,j}\bigr),
  \qquad r=-2,-1,0,1,2,3.
\]
Compute 
\((\beta_{i+\frac{1}{2},j}^+)^{p}\), \((\beta_{i+\frac{1}{2},j}^{-})^{p}\), for $p=0,\ldots 3$ and troubled-cell indicator
\(\mathbb{T}_{i+\frac{1}{2},j}^{x}\).

\item
\eIf{$\mathbb{T}_{i+\frac{1}{2},j}^{x} == 0$}
{
\[
  \mathbb{F}_{i+\frac{1}{2},j}^{+}
  = \frac{1}{60}\Bigl(
  2\mathbb{H}_{i-2,j}^{x,+}
  -13\mathbb{H}_{i-1,j}^{x,+}
  +47\mathbb{H}_{i,j}^{x,+}
  +27\mathbb{H}_{i+1,j}^{x,+}
  -3\mathbb{H}_{i+2,j}^{x,+}
  \Bigr),
\]
\[
  \mathbb{F}_{i+\frac{1}{2},j}^{-}
  = \frac{1}{60}\Bigl(
  2\mathbb{H}_{i+3,j}^{x,-}
  -13\mathbb{H}_{i+2,j}^{x,-}
  +47\mathbb{H}_{i+1,j}^{x,-}
  +27\mathbb{H}_{i,j}^{x,-}
  -3\mathbb{H}_{i-1,j}^{x,-}
  \Bigr),
\]
\[
  \mathbb{F}_{i+\frac{1}{2},j}^{x}
  = \mathbb{F}_{i+\frac{1}{2},j}^{+}
  + \mathbb{F}_{i+\frac{1}{2},j}^{-}.
\]
}
{
Compute the left and right eigenvectors of the Jacobian
\({\bf F}'({\bf U})\) at the interface \(x_{i+\frac{1}{2},j}\),
denoted by \({\bf L}_{i+\frac{1}{2},j}^{x}\) and
\({\bf R}_{i+\frac{1}{2},j}^{x}\).
\[
  \mathbb{M}_{i+r,j}^{x,\pm}
  = {\bf L}_{i+\frac{1}{2},j}^{x}\,\mathbb{H}_{i+r,j}^{x,\pm},
  \qquad r=-2,-1,0,1,2,3,
\]
\[
  \mathbb{F}_{i+\frac{1}{2},j}^{+}
  = \mbox{WENO}\bigl(
  \mathbb{M}_{i-2,j}^{x,+},
  \mathbb{M}_{i-1,j}^{x,+},
  \mathbb{M}_{i,j}^{x,+},
  \mathbb{M}_{i+1,j}^{x,+},
  \mathbb{M}_{i+2,j}^{x,+}
  \bigr),
\]
\[
  \mathbb{F}_{i+\frac{1}{2},j}^{-}
  = \mbox{WENO}\bigl(
  \mathbb{M}_{i+3,j}^{x,-},
  \mathbb{M}_{i+2,j}^{x,-},
  \mathbb{M}_{i+1,j}^{x,-},
  \mathbb{M}_{i,j}^{x,-},
  \mathbb{M}_{i-1,j}^{x,-}
  \bigr),
\]
\[
  \mathbb{F}_{i+\frac{1}{2},j}
  = {\bf R}_{i+\frac{1}{2},j}
  \bigl(
    \mathbb{F}_{i+\frac{1}{2},j}^{+}
    + \mathbb{F}_{i+\frac{1}{2},j}^{-}
  \bigr).
\]
}
\end{enumerate}

\item \textbf{Flux reconstruction in the $y$–direction (interfaces $y_{i,j+\frac{1}{2}}$).}
\begin{enumerate}
\item
Compute the positive and negative parts of the $y$–flux
\[
  \mathbb{H}_{i,j+r}^{y,\pm}
  = \frac{1}{2}\bigl({\bf G}({\bf U}_{i,j+r}) \pm \lambda_y {\bf U}_{i,j+r}\bigr),
  \qquad r=-2,-1,0,1,2,3.
\]
Compute the smoothness indicators
\((\beta_{i,j+\frac{1}{2}}^+)^{p}\), \((\beta_{i,j+\frac{1}{2}}^{-})^{p}\), for $p=0,\ldots 3$ and troubled-cell indicator
\(\mathbb{T}_{i,j+\frac{1}{2}}^{y}\).

\item
\eIf{$\mathbb{T}_{i,j+\frac{1}{2}}^{y} = 0$}
{
\[
  \mathbb{G}_{i,j+\frac{1}{2}}^{+}
  = \frac{1}{60}\Bigl(
  2\mathbb{H}_{i,j-2}^{y,+}
  -13\mathbb{H}_{i,j-1}^{y,+}
  +47\mathbb{H}_{i,j}^{y,+}
  +27\mathbb{H}_{i,j+1}^{y,+}
  -3\mathbb{H}_{i,j+2}^{y,+}
  \Bigr),
\]
\[
  \mathbb{G}_{i,j+\frac{1}{2}}^{-}
  = \frac{1}{60}\Bigl(
  2\mathbb{H}_{i,j+3}^{y,-}
  -13\mathbb{H}_{i,j+2}^{y,-}
  +47\mathbb{H}_{i,j+1}^{y,-}
  +27\mathbb{H}_{i,j}^{y,-}
  -3\mathbb{H}_{i,j-1}^{y,-}
  \Bigr),
\]
\[
  \mathbb{G}_{i,j+\frac{1}{2}}
  = \mathbb{G}_{i,j+\frac{1}{2}}^{+}
  + \mathbb{G}_{i,j+\frac{1}{2}}^{-}.
\]
}
{
Compute the left and right eigenvectors of the Jacobian
\({\bf G}'({\bf U})\) at the interface \(y_{i,j+\frac{1}{2}}\),
denoted by \({\bf L}_{i,j+\frac{1}{2}}^{y}\) and
\({\bf R}_{i,j+\frac{1}{2}}^{y}\).
\[
  \mathbb{M}_{i,j+r}^{y,\pm}
  = {\bf L}_{i,j+\frac{1}{2}}^{y}\,\mathbb{H}_{i,j+r}^{y,\pm},
  \qquad r=-2,-1,0,1,2,3,
\]
\[
  \mathbb{G}_{i,j+\frac{1}{2}}^{+}
  = \mbox{WENO}\bigl(
  \mathbb{M}_{i,j-2}^{y,+},
  \mathbb{M}_{i,j-1}^{y,+},
  \mathbb{M}_{i,j}^{y,+},
  \mathbb{M}_{i,j+1}^{y,+},
  \mathbb{M}_{i,j+2}^{y,+}
  \bigr),
\]
\[
  \mathbb{G}_{i,j+\frac{1}{2}}^{-}
  = \mbox{WENO}\bigl(
  \mathbb{M}_{i,j+3}^{y,-},
  \mathbb{M}_{i,j+2}^{y,-},
  \mathbb{M}_{i,j+1}^{y,-},
  \mathbb{M}_{i,j}^{y,-},
  \mathbb{M}_{i,j-1}^{y,-}
  \bigr),
\]
\[
  \mathbb{G}_{i,j+\frac{1}{2}}
  = {\bf R}_{i,j+\frac{1}{2}}^{y}
  \bigl(
    \mathbb{G}_{i,j+\frac{1}{2}}^{+}
    + \mathbb{G}_{i,j+\frac{1}{2}}^{-}
  \bigr).
\]
}
\end{enumerate}
\end{enumerate}
\caption{H2-WENO scheme for RHD equations over the cell $(i,j)$.}
\label{algo2Dh2}
\end{algorithm}

\begin{algorithm}\label{alog2D}
\small
\SetAlgoLined
\begin{enumerate}
\item \textbf{Flux reconstruction in the $x$–direction (interfaces $x_{i+\frac{1}{2},j}$).}
\begin{enumerate}
\item
Compute the positive and negative parts of the $x$–flux
\[
  \mathbb{H}_{i+r,j}^{x,\pm}
  = \frac{1}{2}\bigl({\bf F}({\bf U}_{i+r,j}) \pm \lambda_x {\bf U}_{i+r,j}\bigr),
  \qquad r=-2,-1,0,1,2,3.
\]
Compute the left and right eigenvectors of the Jacobian
\({\bf F}'({\bf U})\) at the interface \(x_{i+\frac{1}{2},j}\),
\[
  \mathbb{M}_{i+r,j}^{x,\pm}
  = {\bf L}_{i+\frac{1}{2},j}^{x}\,\mathbb{H}_{i+r,j}^{x,\pm},
  \qquad r=-2,-1,0,1,2,3,
\]
Compute 
\((\beta_{i+\frac{1}{2},j}^+)^{p}\), \((\beta_{i+\frac{1}{2},j}^{-})^{p}\), for $p=0,\ldots 3$ and troubled-cell indicator
\(\mathbb{T}_{i+\frac{1}{2},j}^{x}\).

\item
\eIf{$\mathbb{T}_{i+\frac{1}{2},j}^{x} == 0$}
{
\[
  \mathbb{F}_{i+\frac{1}{2},j}^{+}
  = \frac{1}{60}\Bigl(
  2\mathbb{M}_{i-2,j}^{x,+}
  -13\mathbb{M}_{i-1,j}^{x,+}
  +47\mathbb{M}_{i,j}^{x,+}
  +27\mathbb{M}_{i+1,j}^{x,+}
  -3\mathbb{M}_{i+2,j}^{x,+}
  \Bigr),
\]
\[
  \mathbb{F}_{i+\frac{1}{2},j}^{-}
  = \frac{1}{60}\Bigl(
  2\mathbb{M}_{i+3,j}^{x,-}
  -13\mathbb{M}_{i+2,j}^{x,-}
  +47\mathbb{M}_{i+1,j}^{x,-}
  +27\mathbb{M}_{i,j}^{x,-}
  -3\mathbb{M}_{i-1,j}^{x,-}
  \Bigr),
\]
\[
  \mathbb{F}_{i+\frac{1}{2},j}
  = {\bf R}_{i+\frac{1}{2},j}^{x}(\mathbb{F}_{i+\frac{1}{2},j}^{+}
  + \mathbb{F}_{i+\frac{1}{2},j}^{-}).
\]
}
{

\[
  \mathbb{F}_{i+\frac{1}{2},j}^{+}
  = \mbox{WENO}\bigl(
  \mathbb{M}_{i-2,j}^{x,+},
  \mathbb{M}_{i-1,j}^{x,+},
  \mathbb{M}_{i,j}^{x,+},
  \mathbb{M}_{i+1,j}^{x,+},
  \mathbb{M}_{i+2,j}^{x,+}
  \bigr),
\]
\[
  \mathbb{F}_{i+\frac{1}{2},j}^{-}
  = \mbox{WENO}\bigl(
  \mathbb{M}_{i+3,j}^{x,-},
  \mathbb{M}_{i+2,j}^{x,-},
  \mathbb{M}_{i+1,j}^{x,-},
  \mathbb{M}_{i,j}^{x,-},
  \mathbb{M}_{i-1,j}^{x,-}
  \bigr),
\]
\[
  \mathbb{F}_{i+\frac{1}{2},j}
  = {\bf R}_{i+\frac{1}{2},j}^{x}
  \bigl(
    \mathbb{F}_{i+\frac{1}{2},j}^{+}
    + \mathbb{F}_{i+\frac{1}{2},j}^{-}
  \bigr).
\]
}
\end{enumerate}

\item \textbf{Flux reconstruction in the $y$–direction (interfaces $y_{i,j+\frac{1}{2}}$).}
\begin{enumerate}
\item
Compute the positive and negative parts of the $y$–flux
\[
  \mathbb{H}_{i,j+r}^{y,\pm}
  = \frac{1}{2}\bigl({\bf G}({\bf U}_{i,j+r}) \pm \lambda_y {\bf U}_{i,j+r}\bigr),
  \qquad r=-2,-1,0,1,2,3.
\]
Compute the left and right eigenvectors of the Jacobian
\({\bf G}'({\bf U})\) at the interface \(y_{i,j+\frac{1}{2}}\),

\[
  \mathbb{M}_{i,j+r}^{y,\pm}
  = {\bf L}_{i,j+\frac{1}{2}}^{y}\,\mathbb{H}_{i,j+r}^{y,\pm},
  \qquad r=-2,-1,0,1,2,3,
\]
Compute the smoothness indicators
\((\beta_{i,j+\frac{1}{2}}^+)^{p}\), \((\beta_{i,j+\frac{1}{2}}^{-})^{p}\), for $p=0,\ldots 3$ and troubled-cell indicator
\(\mathbb{T}_{i,j+\frac{1}{2}}^{y}\).

\item
\eIf{$\mathbb{T}_{i,j+\frac{1}{2}}^{y} = 0$}
{
\[
  \mathbb{G}_{i,j+\frac{1}{2}}^{+}
  = \frac{1}{60}\Bigl(
  2\mathbb{M}_{i,j-2}^{y,+}
  -13\mathbb{M}_{i,j-1}^{y,+}
  +47\mathbb{M}_{i,j}^{y,+}
  +27\mathbb{M}_{i,j+1}^{y,+}
  -3\mathbb{M}_{i,j+2}^{y,+}
  \Bigr),
\]
\[
  \mathbb{G}_{i,j+\frac{1}{2}}^{-}
  = \frac{1}{60}\Bigl(
  2\mathbb{M}_{i,j+3}^{y,-}
  -13\mathbb{M}_{i,j+2}^{y,-}
  +47\mathbb{M}_{i,j+1}^{y,-}
  +27\mathbb{M}_{i,j}^{y,-}
  -3\mathbb{M}_{i,j-1}^{y,-}
  \Bigr),
\]
\[
  \mathbb{G}_{i,j+\frac{1}{2}}
  = {\bf R}_{i,j+\frac{1}{2}}^{y}(\mathbb{G}_{i,j+\frac{1}{2}}^{+}
  + \mathbb{G}_{i,j+\frac{1}{2}}^{-}).
\]
}
{

\[
  \mathbb{G}_{i,j+\frac{1}{2}}^{+}
  = \mbox{WENO}\bigl(
  \mathbb{M}_{i,j-2}^{y,+},
  \mathbb{M}_{i,j-1}^{y,+},
  \mathbb{M}_{i,j}^{y,+},
  \mathbb{M}_{i,j+1}^{y,+},
  \mathbb{M}_{i,j+2}^{y,+}
  \bigr),
\]
\[
  \mathbb{G}_{i,j+\frac{1}{2}}^{-}
  = \mbox{WENO}\bigl(
  \mathbb{M}_{i,j+3}^{y,-},
  \mathbb{M}_{i,j+2}^{y,-},
  \mathbb{M}_{i,j+1}^{y,-},
  \mathbb{M}_{i,j}^{y,-},
  \mathbb{M}_{i,j-1}^{y,-}
  \bigr),
\]
\[
  \mathbb{G}_{i,j+\frac{1}{2}}
  = {\bf R}_{i,j+\frac{1}{2}}^{y}
  \bigl(
    \mathbb{G}_{i,j+\frac{1}{2}}^{+}
    + \mathbb{G}_{i,j+\frac{1}{2}}^{-}
  \bigr).
\]
}
\end{enumerate}
\end{enumerate}
\caption{H3-WENO scheme for RHD equations over the cell $(i,j)$.}
\label{algo2Dh3}
\end{algorithm}
The estimate $\beta^-_{i+\frac{1}{2}} + \beta^+_{i+\frac{1}{2}} = O(\lambda^2 \Delta x^2)$ is sufficient to identify the troubled-cells but it may also mark some smooth-cells as troubled-cell, as found in our numerical experiments. To avoid the marking of smooth cells, we multiply the estimate of $\lambda^2 \Delta x^2$ by a fixed constant $\mathcal{K}$ which leads to an efficient troubled-cell indicator in the sense that it avoids marking smooth cells. 

Based on this, we propose three hybrid WENO schemes, named {H1-WENO}, {H2-WENO}, and {H3-WENO}, which adaptively switch between different reconstruction strategies based on a troubled-cell indicator that classifies the computational domain into smooth and non-smooth regions. The working of hybrid flux can be defined as follows:
\[
 \mathbb{F}_{i+\frac{1}{2}}^{\mbox{hybdrid}} = \begin{cases}
      \mbox{Higher order low-cost reconstruction}, & \mathbb{T}_{i+\frac{1}{2}}==0, \\
      \mbox{Non-Oscillatory reconstruction}, & \mbox{otherwise}.
     \end{cases}
\]
Here are details of three hybrid schemes:
\begin{itemize}
    \item {H1-WENO:} Low cost component-wise WENO reconstruction is employed in smooth regions, while characteristic-wise WENO reconstruction is used in non-smooth regions.
    \item {H2-WENO:} A fifth-order component-wise linear reconstruction is adopted in smooth regions, and characteristic-wise WENO reconstruction is applied in non-smooth regions.
    \item {H3-WENO:} A fifth-order characteristic-wise linear reconstruction is used in smooth regions, while characteristic-wise WENO reconstruction is employed in non-smooth regions.
\end{itemize}

In the numerical computations, we take $\mathcal{K}=20$ and $\mathcal{K}=40$ for the 1D and 2D problems, respectively, for the H1-WENO and H2-WENO schemes, whereas $\mathcal{K}=1$ is used for both the 1D and 2D problems in the H3-WENO scheme. Among these, {H2-WENO} is the most computationally efficient scheme, since it avoids both the eigenvector transformation and nonlinear weights in smooth regions. The algorithms for H1-WENO, H2-WENO, and H3-WENO for 2D system of hyperbolic system are presented in Algorithms~\ref{algo2Dh1}, \ref{algo2Dh2}, and \ref{algo2Dh3}, respectively.

\section{Numerical results}
\label{sec:num}
In the previous section, we presented the construction of three new hybrid algorithms for the spatial approximation of flux derivatives for the RHD equations in both one- and two-dimensional settings. These methods are based on a proposed troubled-cell indicator that accurately identifies non-smooth regions of the solution, enabling the selective application of non-oscillatory schemes near discontinuities while preserving high-order accuracy in smooth regions.
In this section, we compare the proposed algorithms with the {WENO-AO(5,3)} scheme \cite{bal-etal_16a} in terms of both accuracy and computational efficiency for smooth problems as well as problems containing discontinuities. Furthermore, we investigate the influence of the parameter involved in the troubled-cell indicator by varying it over a suitable range of values. In particular, we study how the variation in the value parameter affects the percentage of detected troubled cells and the overall computational efficiency of the schemes.
After the spatial discretization of the flux derivatives, the governing conservation laws reduce to a system of ordinary differential equations of the form
\[
\frac{du}{dt} = \mathcal{R}(u),
\]
where $\mathcal{R}(u)$ denotes the spatial discretization operator. The resulting semi-discrete system is advanced in time using the third-order TVD Runge--Kutta method \cite{shu-osh_88a}, given by
\begin{align*}
u^1 &= u^n + \Delta t \mathcal{R}(u^n), \\
u^2 &= \frac{3}{4}u^n + \frac{1}{4}u^1 + \frac{1}{4}\Delta t \mathcal{R}(u^1), \\
u^{n+1} &= \frac{1}{3}u^n + \frac{2}{3}u^2 + \frac{2}{3}\Delta t \mathcal{R}(u^2).
\end{align*}
For all numerical experiments, we use a CFL number of $0.5$ for one-dimensional problems and $0.4$ for two-dimensional problems to ensure consistency across the simulations. Unless otherwise specified, the ratio of specific heats is taken as $\gamma = \frac{5}{3}$.
\begin{figure}[ht]
 \centering
 \begin{tabular}{cc}
  \includegraphics[width = 7.2cm]{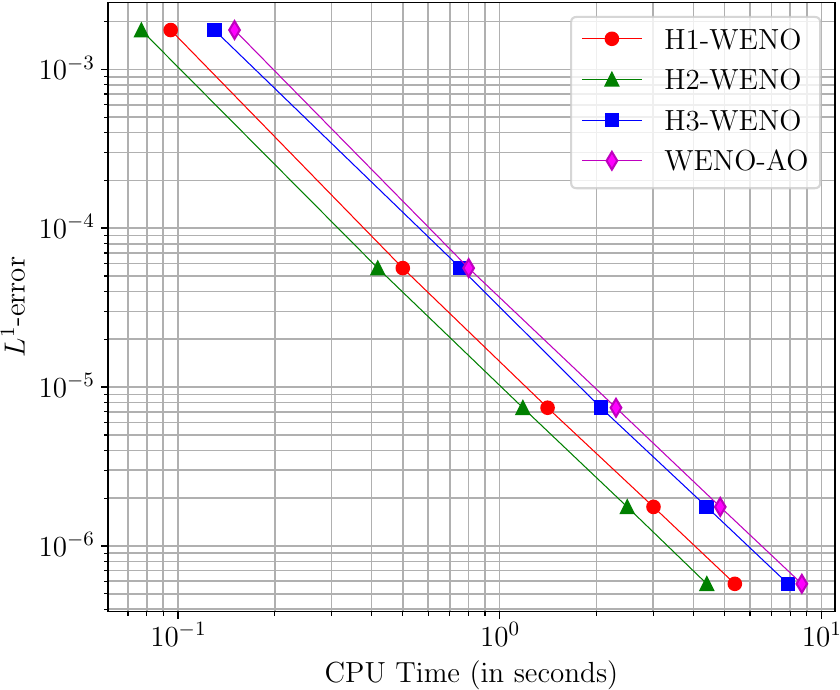}&
  \includegraphics[width = 7.2cm]{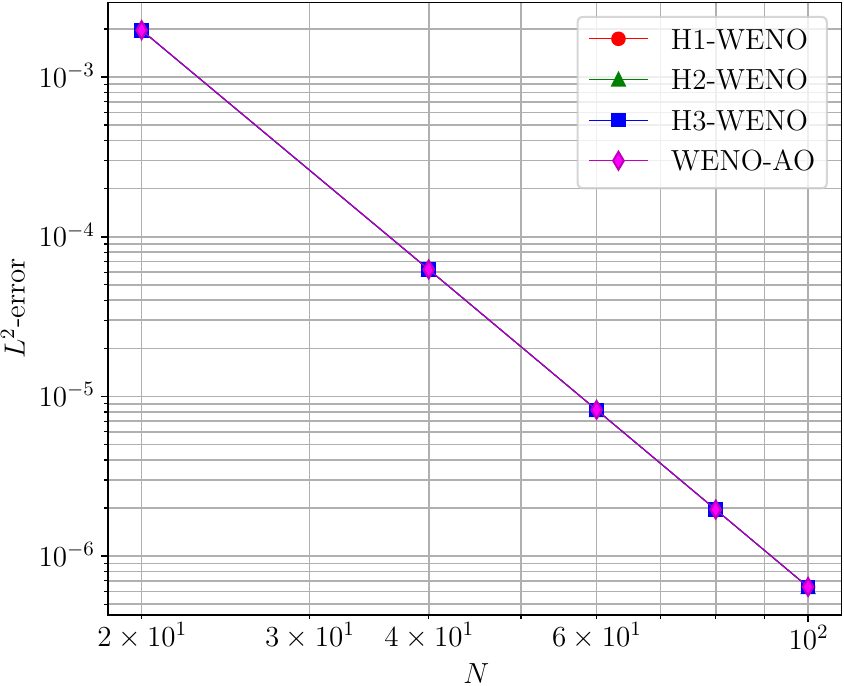}\\
  (a) & (b) 
 \end{tabular}
\caption{Comparison of WENO-AO, H1-WENO, H2-WENO, and H3-WENO in term of (a) CPU time taken to obtain $L^1$-error  (b) $L^2$-error vs number of mesh points for Example \ref{test1}.}
\label{fig:Eg5}
\end{figure}
\begin{figure}[ht]
 \centering
 \begin{tabular}{cc}
  \includegraphics[width = 7.4cm]{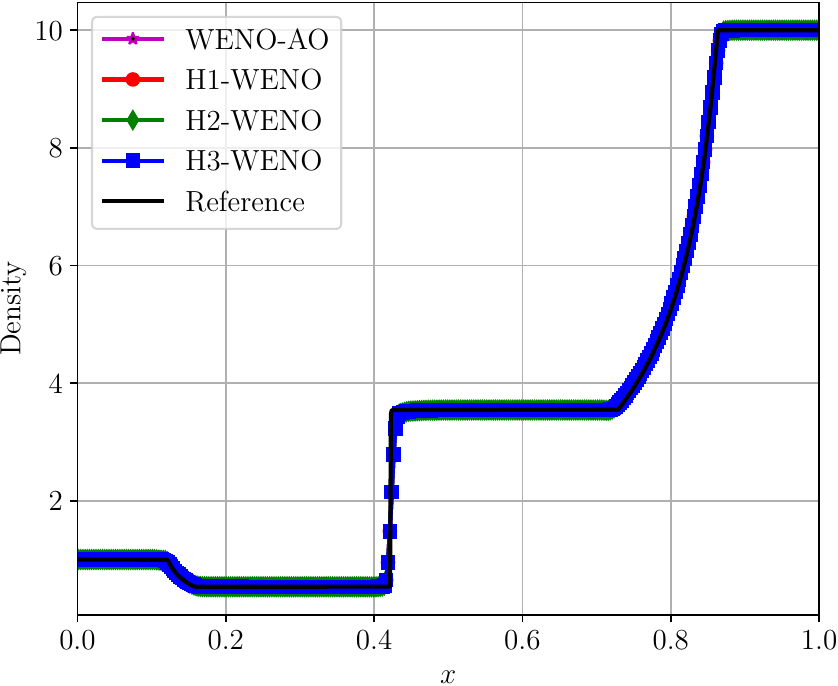}&
  \includegraphics[width = 7.4cm]{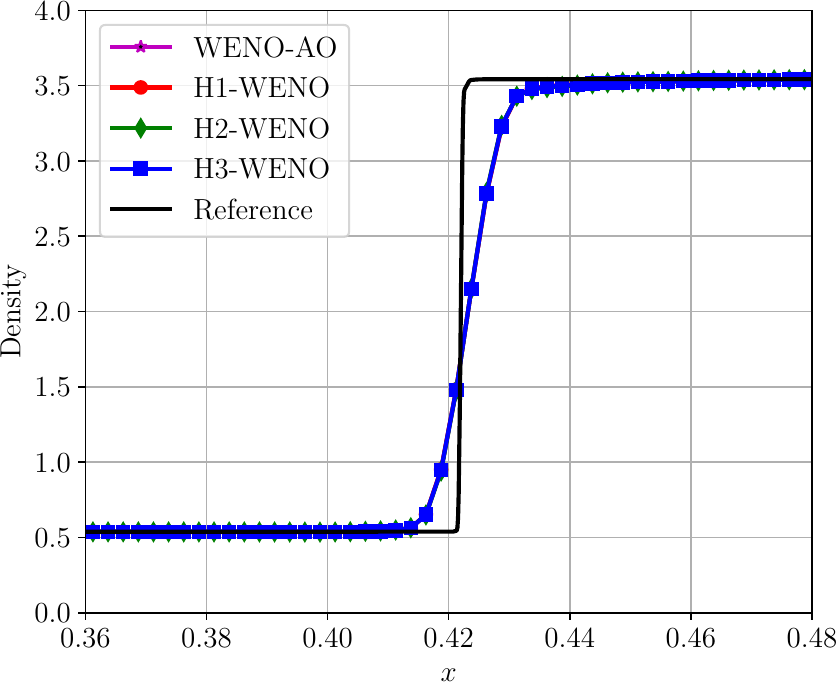}\\
  (a) & (b) \\
   \includegraphics[width = 7.4cm]{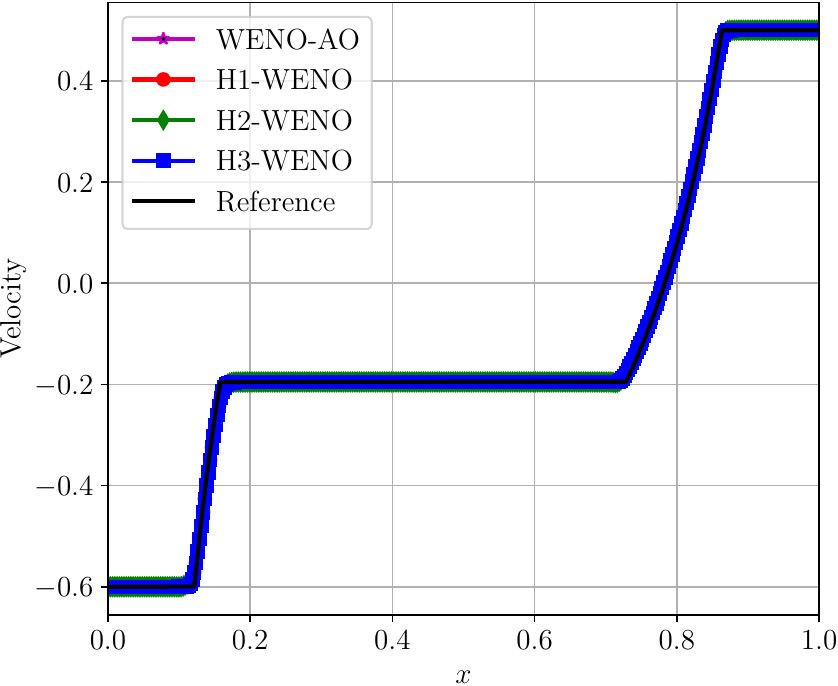}&
  \includegraphics[width = 7.4cm]{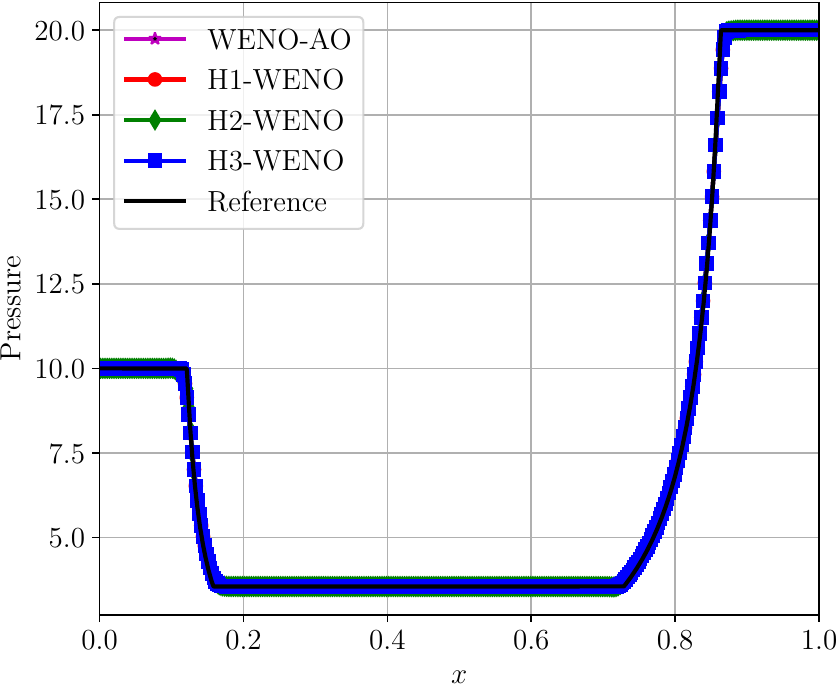}\\
  (c)  & (d) \\
 \end{tabular}
\caption{Comparison of numerical solutions obtained using H1-WENO, H2-WENO, H3-WENO, and WENO-AO schemes at time $T=0.4$ for Example \ref{test3} on a mesh with 500 grid points, compared with the reference solution.}
\label{Fig:test2}
\end{figure}
\begin{figure}[ht]
    \centering
 \begin{tabular}{ccc}
     \includegraphics[width = 4.6cm]{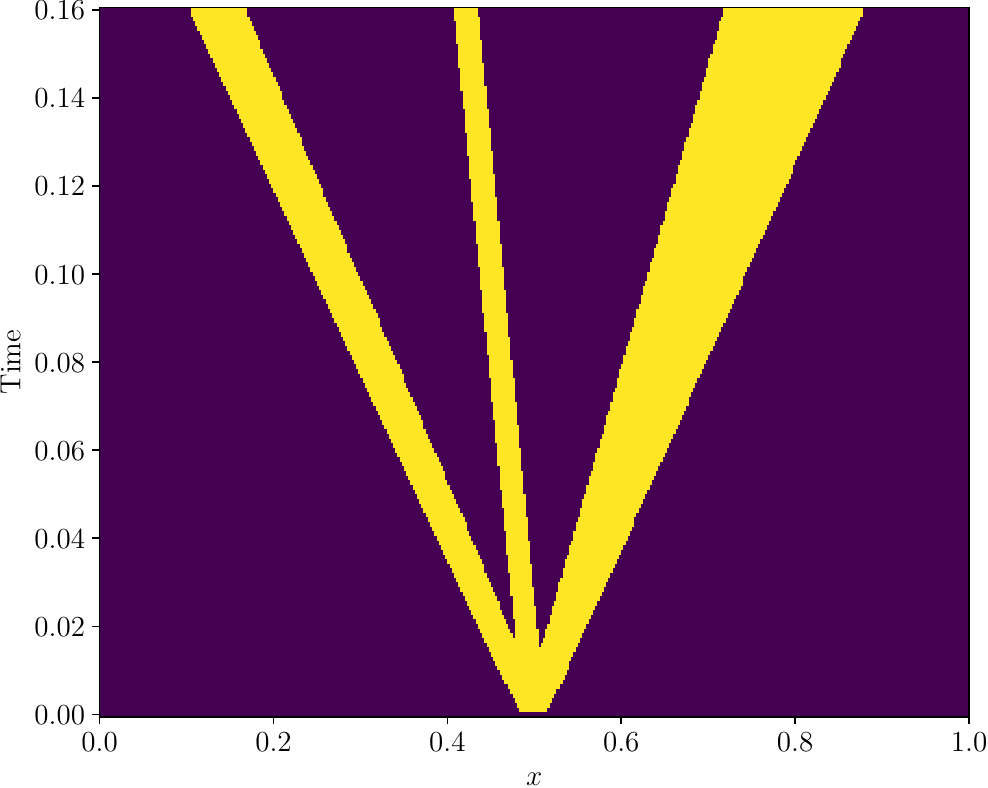}& 
      \includegraphics[width = 4.6cm]{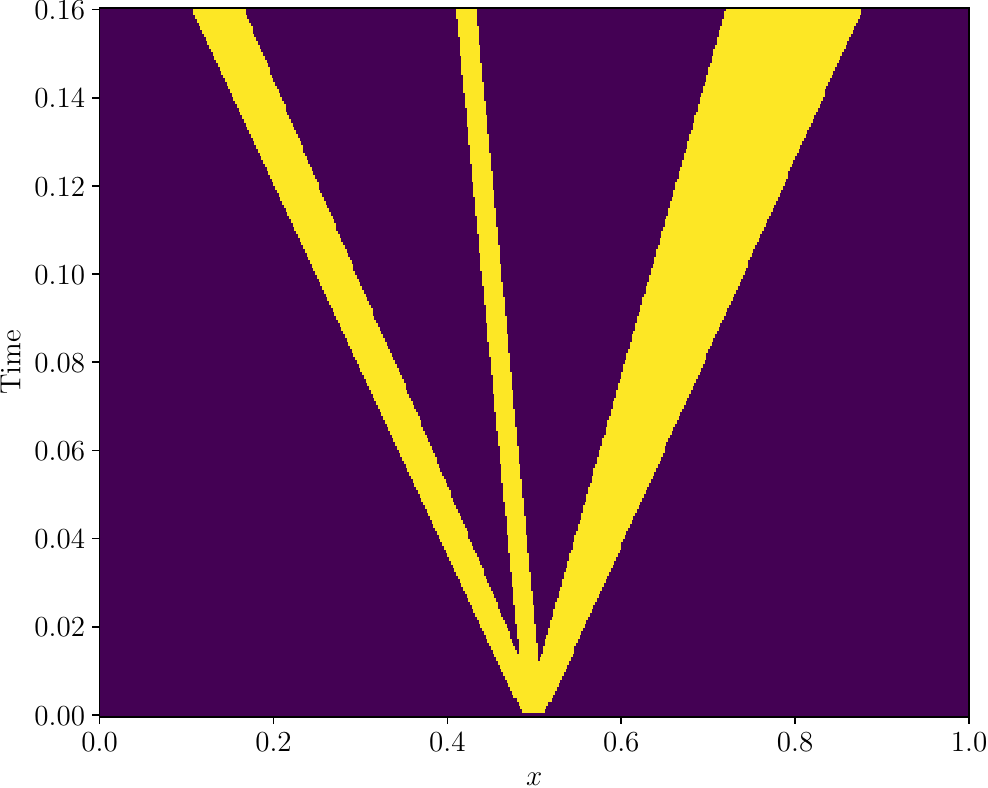}& \includegraphics[width = 4.6cm]{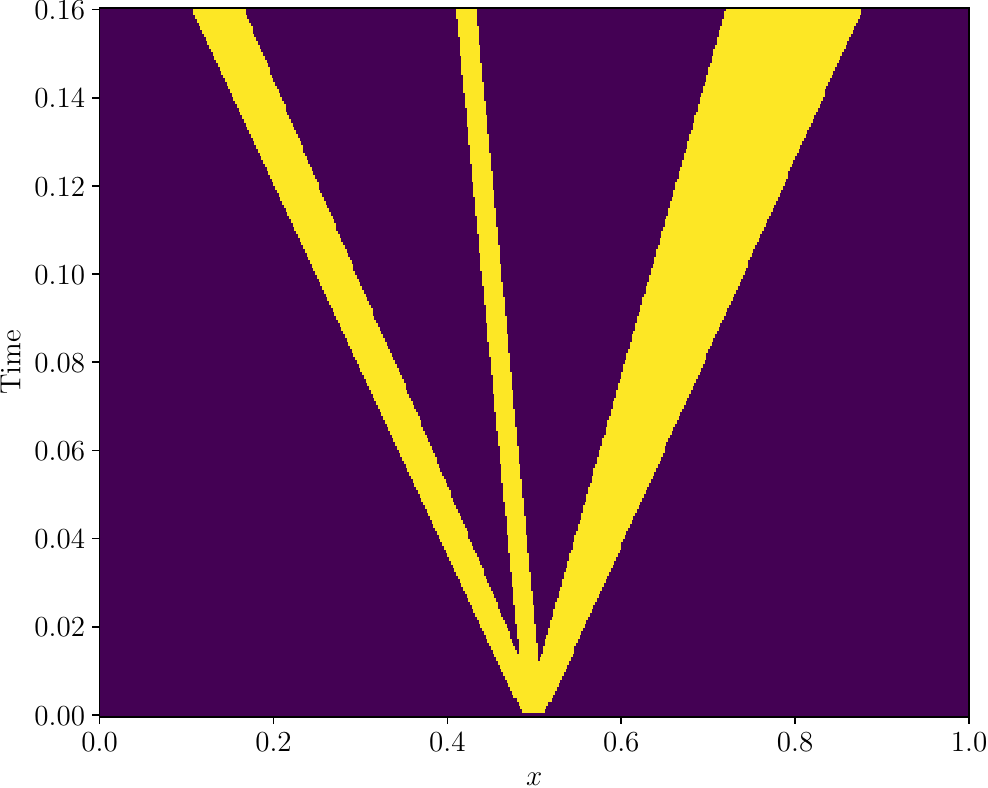}\\
      (a) H1-WENO & (b) H2-WENO &(c) H3-WENO
      \end{tabular}
    \caption{Plot of troubled-cell indicator over $x-t$ plane obtained using the proposed hybrid schemes for Example \ref{test3} using 500 mesh points.}
     \label{Fig:test2.tc}
\end{figure}
\begin{figure}[ht!]
 \centering
 \begin{tabular}{cc}
  \includegraphics[width = 7.4cm]{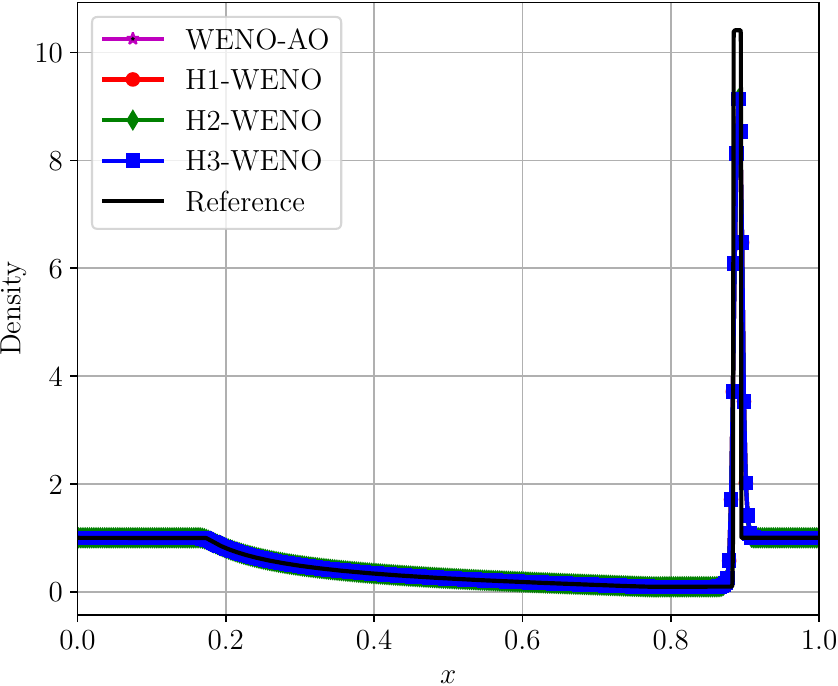}&
  \includegraphics[width = 7.4cm]{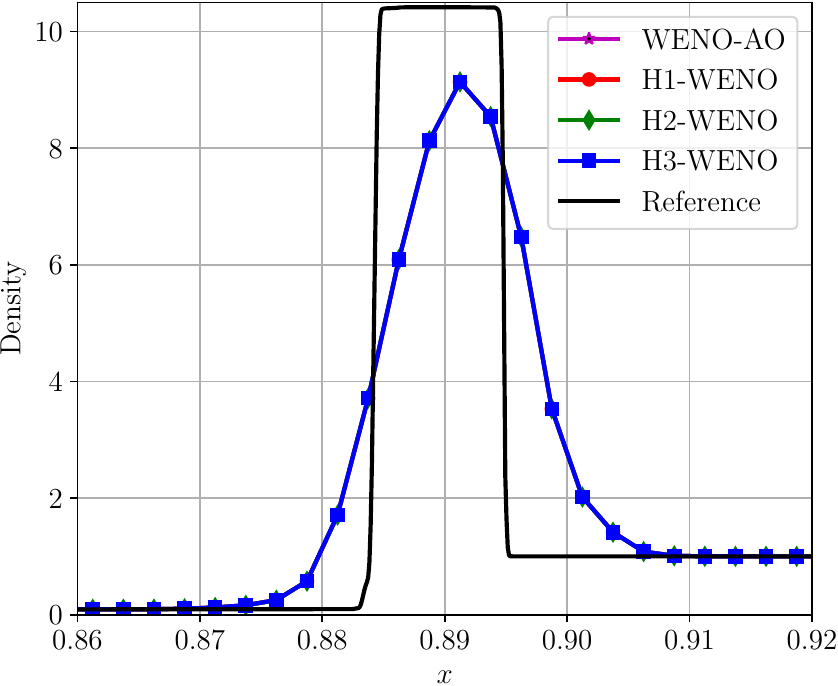}\\
  (a) & (b) \\
   \includegraphics[width = 7.4cm]{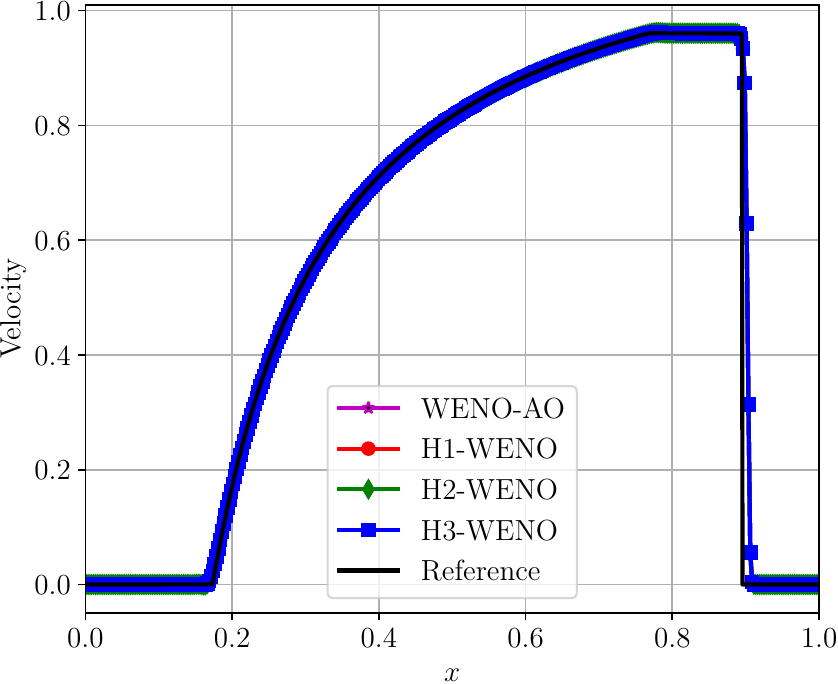}&
  \includegraphics[width = 7.4cm]{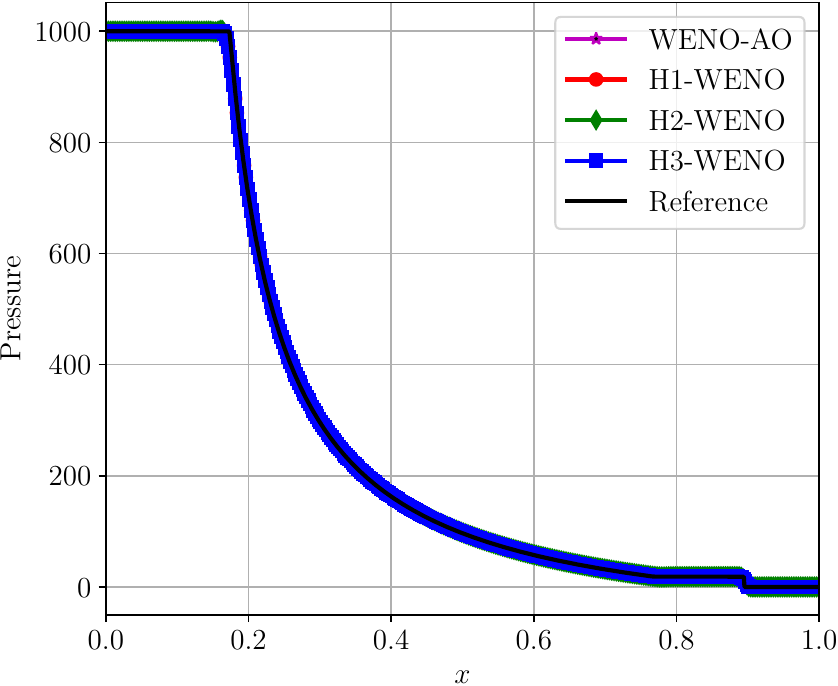}\\
  (c)  & (d) \\
 \end{tabular}
\caption{Comparison of numerical solutions obtained using H1-WENO, H2-WENO, H3-WENO, and WENO-AO schemes at time $T=0.4$ for Example \ref{test3} on a mesh with 500 grid points, compared with the reference solution. }
\label{Fig:test3}
\end{figure}
\subsection{One dimensional numerical tests}
We first proceed with one dimensional test cases. First we test accuracy of the schemes. We then test the proposed schemes on various Riemann problems.

\begin{table}[ht!]
\small
\begin{center}
\begin{tabular}{|c|c|}
\hline
H1-WENO & H2-WENO\\
\hline 
   \begin{tabular}{ c|c| c| c| c }
 $N$ & $L^{\infty}$-error & Order & $L^{1}$-error & Order\\
  \hline
   20 & 2.76672e-03 & -- & 1.77114e-03 & -- \\
   40 & 8.81214e-05 &       5.00 & 5.61679e-05 &       5.00 \\
   60 & 1.16451e-05 &       4.97 & 7.41736e-06 &       4.98 \\
   80 & 2.76683e-06 &       4.99 & 1.76192e-06 &       4.99 \\
  100 & 9.07151e-07 &       5.00 & 5.77612e-07 &       5.00  \\
\end{tabular} &  
 \begin{tabular}{c| c| c| c }
  $L^{\infty}$-error & Order & $L^{1}$-error & Order\\
 \hline
 2.76302e-03 & -- & 1.76690e-03 & -- \\
 8.81164e-05 &       5.00 & 5.61598e-05 &       5.00 \\
1.16450e-05 &       4.97 & 7.41715e-06 &       4.98 \\
 2.76682e-06 &       4.99 & 1.76191e-06 &       4.99 \\
   9.07149e-07 &       5.00 & 5.77610e-07 &       5.00 \\
\end{tabular}
\end{tabular}
\begin{tabular}{|c|c|}
\hline
H3-WENO & WENO-AO\\
\hline 
   \begin{tabular}{ c|c| c| c| c }
 $N$ & $L^{\infty}$-error & Order & $L^{1}$-error & Order\\
\hline
   20 & 2.76302e-03 & --         & 1.76690e-03 & -- \\
   40 & 8.81164e-05 &       5.00 & 5.61598e-05 &       5.00 \\
   60 & 1.16450e-05 &       4.97 & 7.41715e-06 &       4.98\\
   80 & 2.76682e-06 &       4.99 & 1.76191e-06 &       4.99\\
  100 & 9.07149e-07 &       5.00 & 5.77610e-07 &       5.00\\
\hline

\end{tabular} &  
 \begin{tabular}{c| c| c| c }
  $L^{\infty}$-error & Order & $L^{1}$-error & Order\\
\hline
2.76672e-03 & -- & 1.77114e-03 & -- \\
8.81214e-05 &       5.00 & 5.61679e-05 &       5.00 \\
1.16451e-05 &       4.97 & 7.41736e-06 &       4.98 \\
2.76683e-06 &       4.99 & 1.76192e-06 &       4.99 \\
9.07151e-07 &       5.00 & 5.77612e-07 &       5.00 \\
\hline
\end{tabular}\\
\hline
\end{tabular}
\end{center}
\caption{Comparison of H1-WENO, H2-WENO, and H3-WENO schemes with WENO-AO scheme in terms of $L^{\infty}$, $L^1$-errors for Example \ref{test3} at time $T=10$.
}
\label{Tab:test1}
\end{table}
\begin{example}[Accuracy test:]
	\label{test1}{\rm
In order to assess the accuracy of the proposed schemes in comparison with the WENO-AO scheme, we first consider a test problem with the smooth exact solution.  The initial conditions are given by
\[
\left(\rho,\,u,\,p\right)=\left(2+\sin(2\pi x),\,0.5,\,1\right).
\]
The computational domain is taken as $[0,1]$ with periodic boundary conditions.
The corresponding exact solution represents the advection of the rest-mass density and is given by
\[
\rho(x,t)=2+\sin\left(2\pi(x-0.5t)\right),
\]
while the remaining variables remain unchanged throughout the simulation. The solution evolved until the final time $t=10$. To obtain the desired fifth-order spatial convergence, the time step is chosen as
\[
\Delta t = \mathrm{CFL}\,\Delta x^{5/3}.
\]
The numerical errors in the $L^\infty$ and $L^1$ norms for the density variable $\rho$ are reported in Table~\ref{Tab:test1}. The results demonstrate that all proposed schemes accurately capture the smooth solution and produce errors comparable to those obtained with the WENO-AO scheme. Moreover, all schemes achieve the expected fifth-order convergence rate.
In Figure~\ref{fig:Eg5}~(a), we compare the $L^1$-error against CPU time on a log-log scale, while Figure~\ref{fig:Eg5}~(b) presents the variation of the $L^2$-error with respect to the number of mesh points. From Figure~\ref{fig:Eg5}~(a), it can be observed that the H2-WENO scheme requires significantly less computational time to achieve a comparable $L^1$-error, providing nearly a $50\%$ improvement in efficiency over the WENO-AO scheme. Similarly, the H1-WENO scheme achieves an improvement of approximately $40\%$, while the H3-WENO scheme offers only a modest gain of approximately $10\%$. Figure~\ref{fig:Eg5}~(b) further shows that the $L^2$-errors produced by all schemes are very close to each other and exhibit the expected fifth-order convergence behavior. Among all the schemes considered, the H2-WENO scheme is found to be the most computationally efficient.
Furthermore, no troubled cells were detected during the entire simulation, indicating that the proposed troubled cell indicator correctly identifies the solution as smooth and does not introduce unnecessary limiting.
	}
\end{example}

\begin{figure}[ht!]
 \centering
 \begin{tabular}{cc}
  \includegraphics[width = 7.4cm]{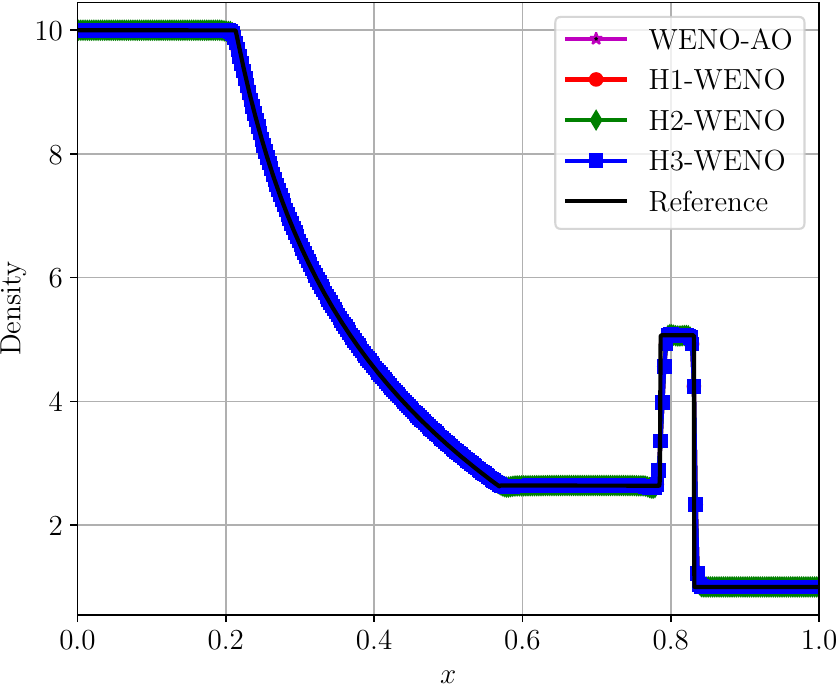}&
  \includegraphics[width = 7.4cm]{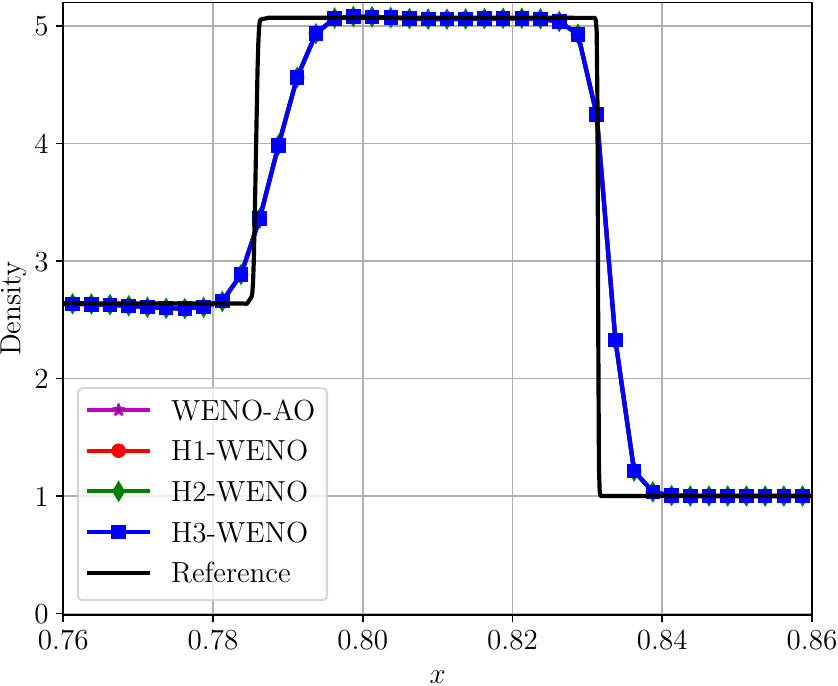}\\
  (a) & (b) \\
   \includegraphics[width = 7.4cm]{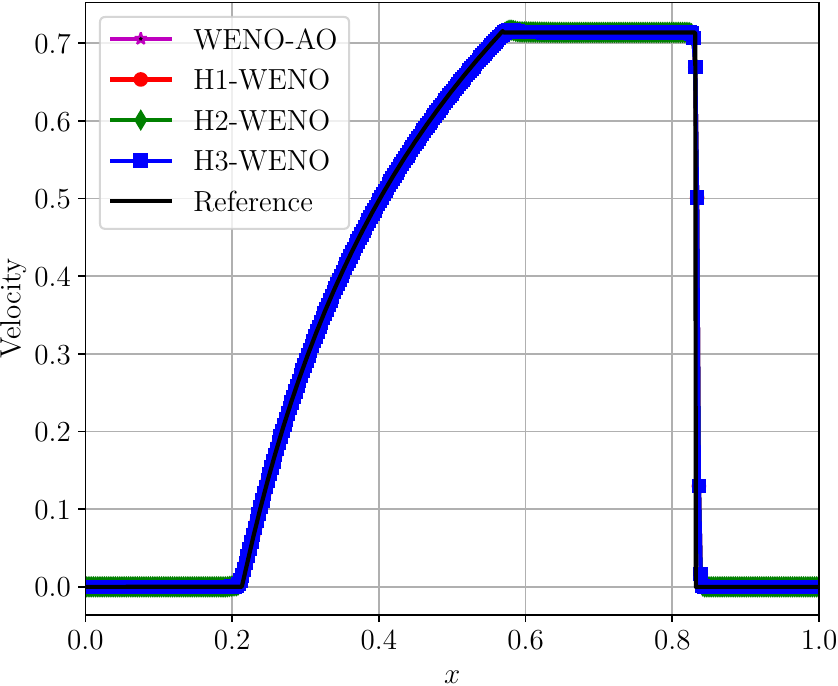}&
  \includegraphics[width = 7.4cm]{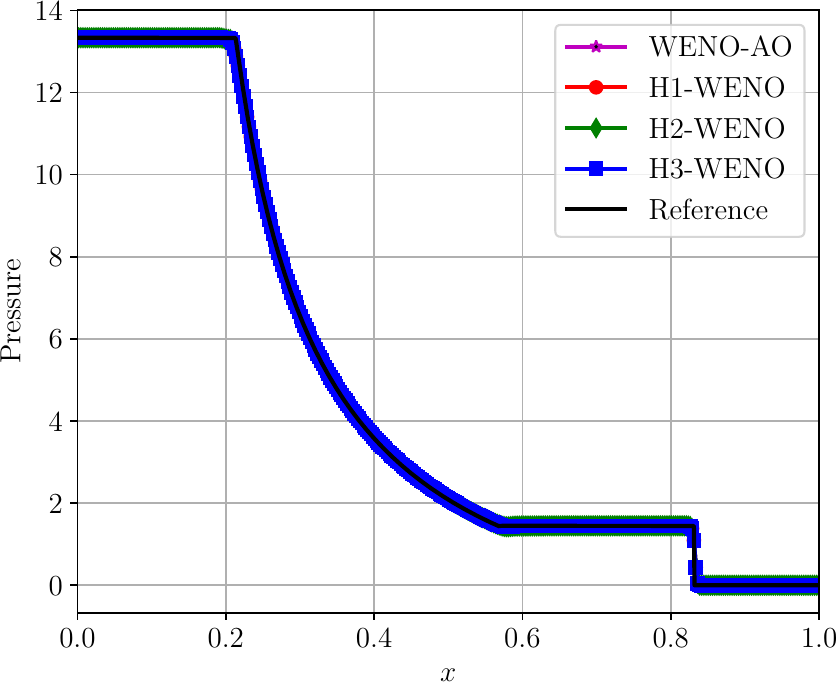}\\
  (c)  & (d) \\
 \end{tabular}
\caption{Comparison of numerical solutions obtained using H1-WENO, H2-WENO, H3-WENO, and WENO-AO schemes at time $T=0.4$ for Example \ref{test5} on a mesh with 500 grid points, compared with the reference solution. }
\label{Fig:test4}
\end{figure}

\begin{figure}[ht!]
 \centering
 \begin{tabular}{cc}
  \includegraphics[width = 7.4cm]{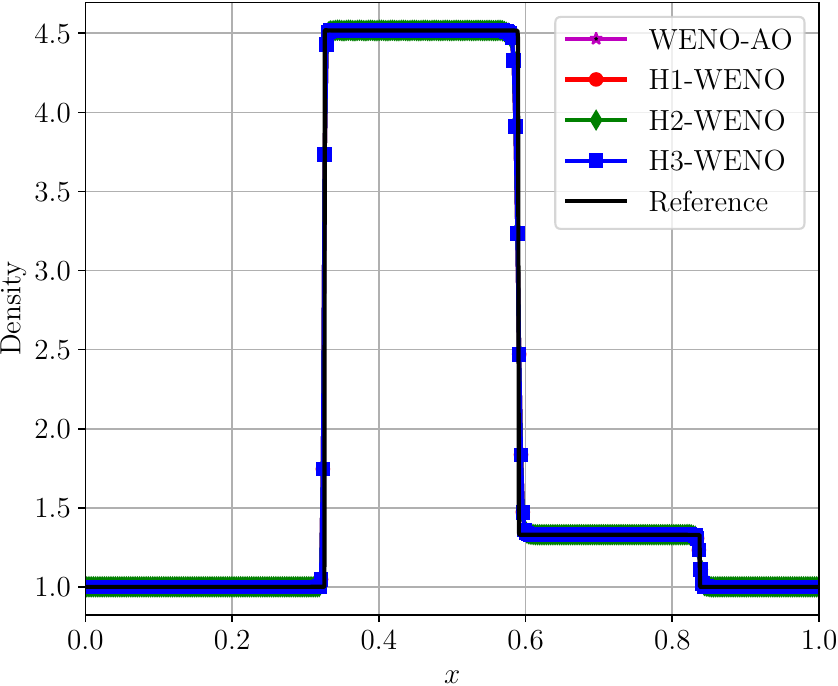}&
  \includegraphics[width = 7.4cm]{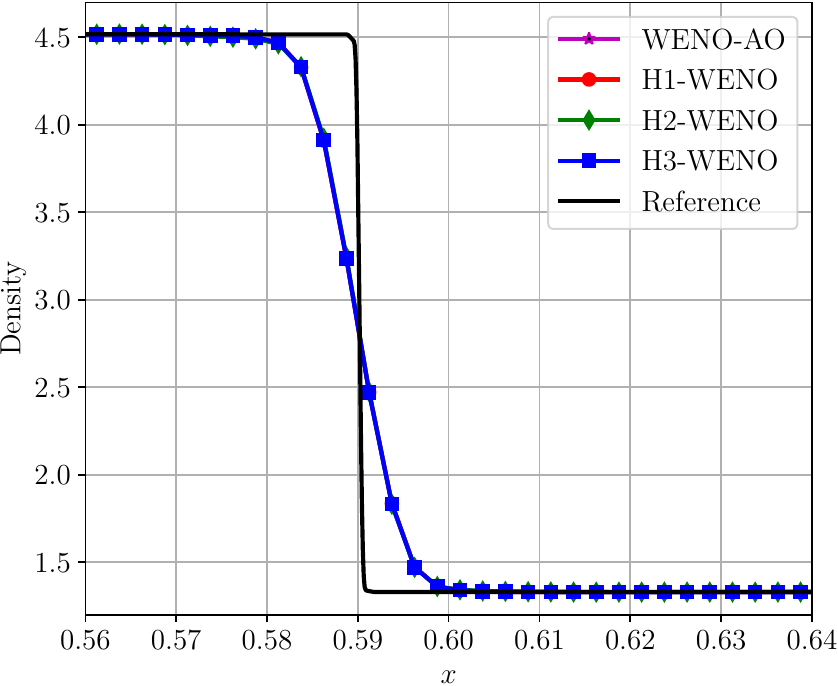}\\
  (a) & (b) \\
   \includegraphics[width = 7.4cm]{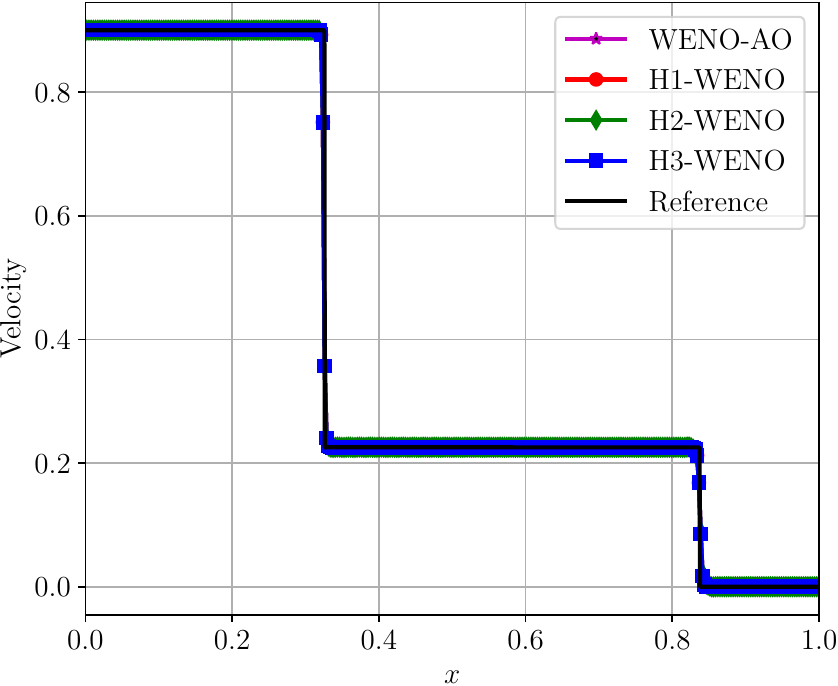}&
  \includegraphics[width = 7.4cm]{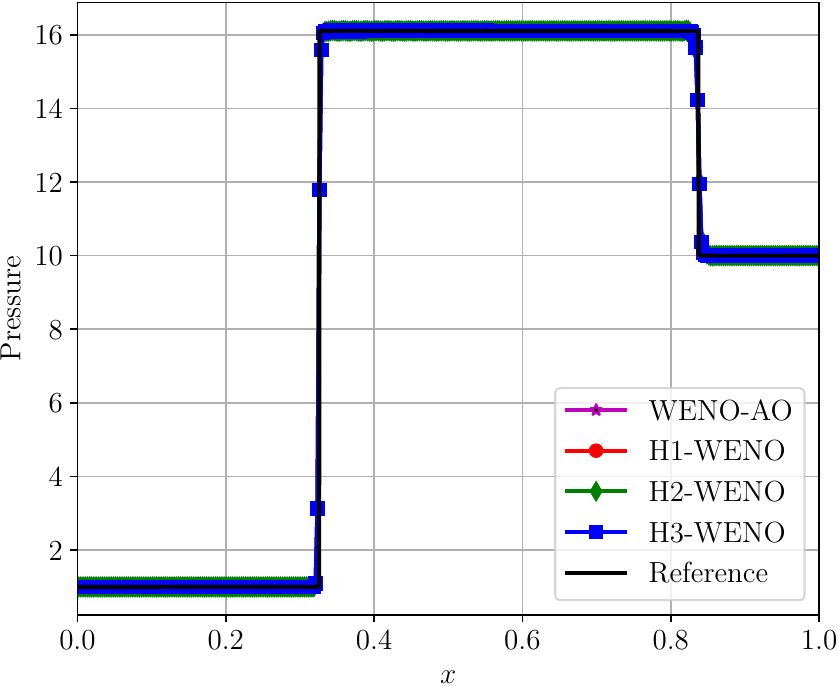}\\
  (c)  & (d) \\
 \end{tabular}
\caption{Comparison of numerical solutions obtained using H1-WENO, H2-WENO, H3-WENO, and WENO-AO schemes at time $T=0.4$ for Example \ref{test6} on a mesh with 500 grid points, compared with the reference solution.}
\label{Fig:test6}
\end{figure}

\begin{figure}[ht!]
 \centering
 \begin{tabular}{cc}
  \includegraphics[width = 7.4cm]{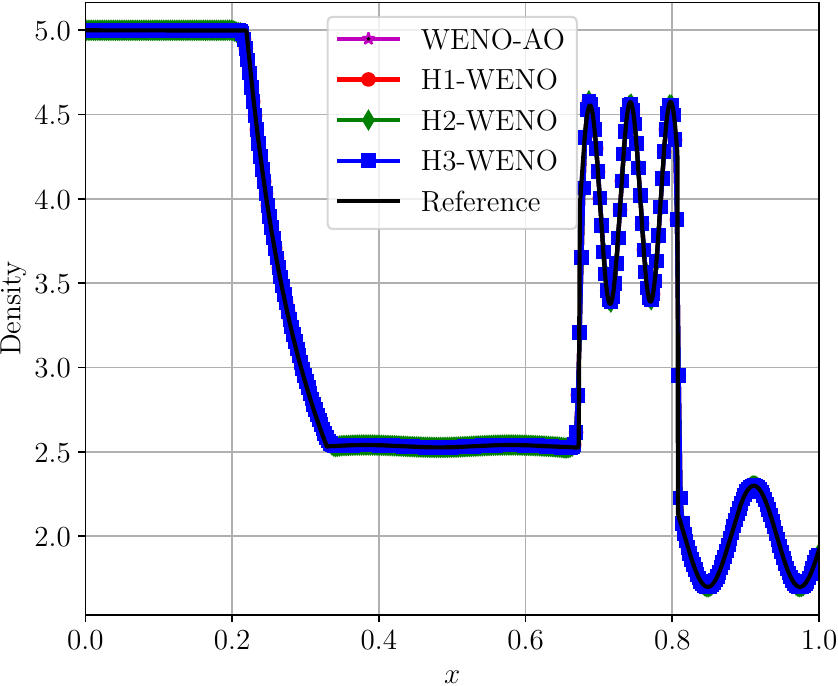}&
  \includegraphics[width = 7.4cm]{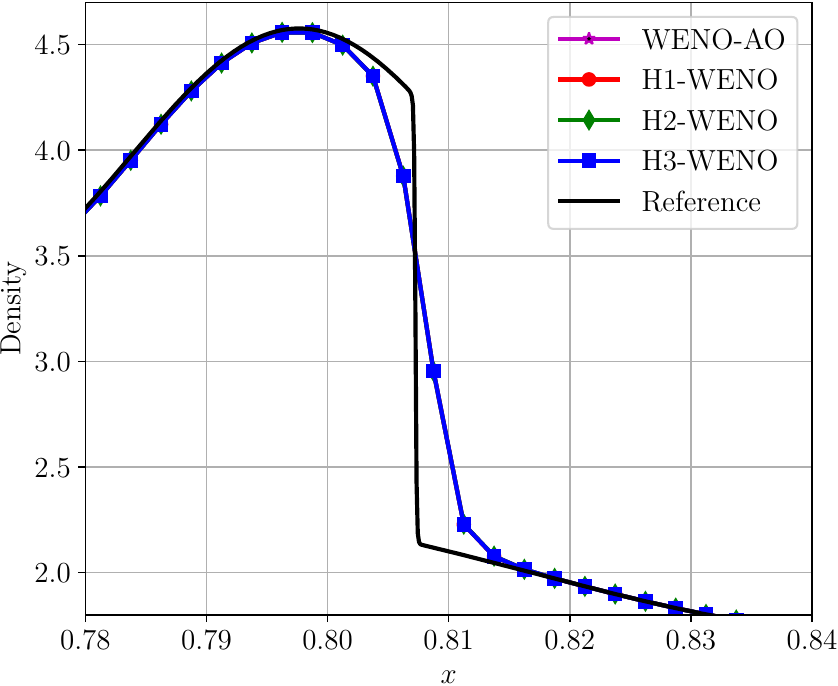}\\
  (a) & (b) \\
   \includegraphics[width = 7.4cm]{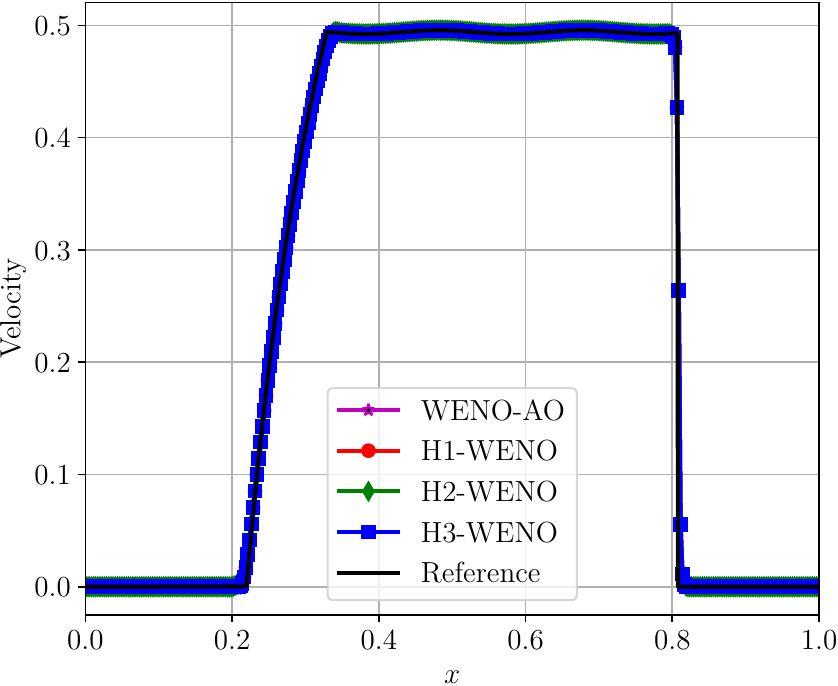}&
  \includegraphics[width = 7.4cm]{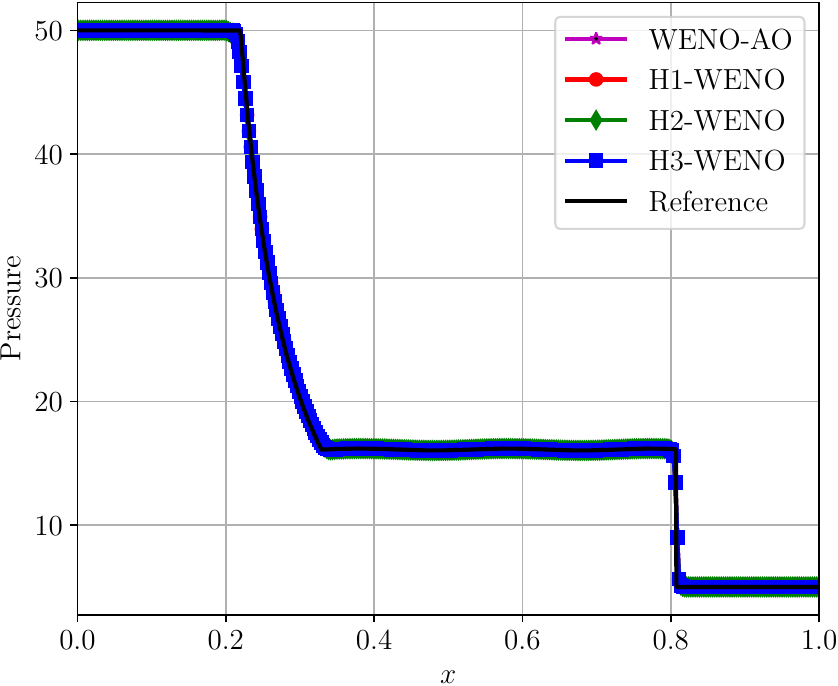}\\
  (c)  & (d) \\
 \end{tabular}
\caption{Comparison of numerical solutions obtained using H1-WENO, H2-WENO, H3-WENO, and WENO-AO schemes at time $T=0.35$ for Example \ref{test7} on a mesh with 500 grid points, compared with the reference solution.}
\label{Fig:test7}
\end{figure}
\begin{example}[Riemann Problem 1:]
	\label{test3}{\rm
In this test case, we consider a Riemann problem from \cite{Mignone2005HLLC}. We use the computational domain of $[0,1]$ with initial discontinuity at $x=0.5$. The initial states are given by,
	\begin{equation*}
	\left(\rho,\,u,\,p\right)=\begin{cases}
	\left(1,-0.6,10\right) & \text{if $x<0.5$}\\
	\left(10,0.5,20\right) &  \text{if $x>0.5$}
	\end{cases}.
	\end{equation*}
The exact solution consists of a wave structure containing two rarefaction waves propagating in opposite directions, separated by a contact discontinuity. The rarefaction waves are generated due to the pressure drop across the initial discontinuity, while the contact discontinuity corresponds to a jump in density and tangential velocity, with the pressure and normal velocity remaining continuous across it. Outflow boundary conditions are imposed at both boundaries.
In Figure~\ref{Fig:test2}, we compare the solutions obtained using the proposed schemes with those computed using the WENO-AO scheme, along with the reference solution on a grid consisting of $500$ mesh points at final time $T=0.4$. The reference solution is generated using the WENO-AO scheme on a highly refined mesh of $10000$ grid points. It can be observed that the solutions produced by the proposed schemes are in very good agreement with both the WENO-AO scheme and the reference solution for all physical variables.
Figure~\ref{Fig:test2.tc} presents the distribution of the troubled-cell indicator used in the proposed schemes over the $x$--$t$ plane. In the figure, the yellow regions correspond to the detected discontinuities, whereas the blue regions represent smooth portions of the solution. The results show that the troubled-cell indicator successfully identifies the discontinuous regions of the solution while leaving the smooth regions unaffected. 
	
}\end{example}

\begin{example}[Riemann Problem 2:]
	\label{test4} {\rm 
We consider a shock-tube Riemann problem from \cite{marti2003numerical}. The computational domain is taken as $[0,1]$ with outflow boundary conditions imposed at both ends. The initial conditions are given by
\begin{equation}\label{rp2}
\left(\rho,\,u,\,p\right)=
\begin{cases}
\left(1,0,10^3\right), & \text{if } x < 0.5,\\
\left(1,0,10^{-2}\right), & \text{if } x > 0.5.
\end{cases}
\end{equation}
The exact solution of this problem contains all fundamental wave structures, including shock waves, rarefaction waves, and contact discontinuities. Therefore, it serves as a suitable benchmark for assessing the wave-capturing capability of the proposed schemes.
The numerical solution is computed using the H1-WENO, H2-WENO, H3-WENO, and WENO-AO schemes on a grid consisting of $500$ mesh points. In Figure~\ref{Fig:test3}, we compare the numerical solutions obtained from all the considered schemes with the reference solution at the final time $T=0.4$. The reference solution is generated on a highly refined mesh using the WENO-AO scheme. From the figure, it can be observed that all the schemes successfully capture the essential wave structures of the problem and produce results that are in close agreement with the WENO-AO scheme and the reference solution.
	}
\end{example}

\begin{figure}[ht!]
 \centering
 \begin{tabular}{cc}
  \includegraphics[width = 7.4cm]{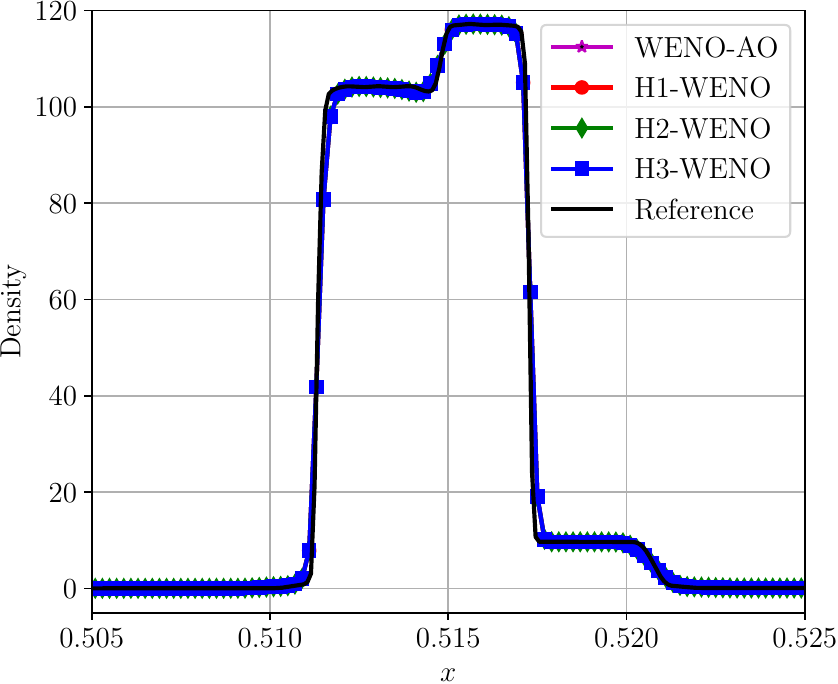}&
  \includegraphics[width = 7.4cm]{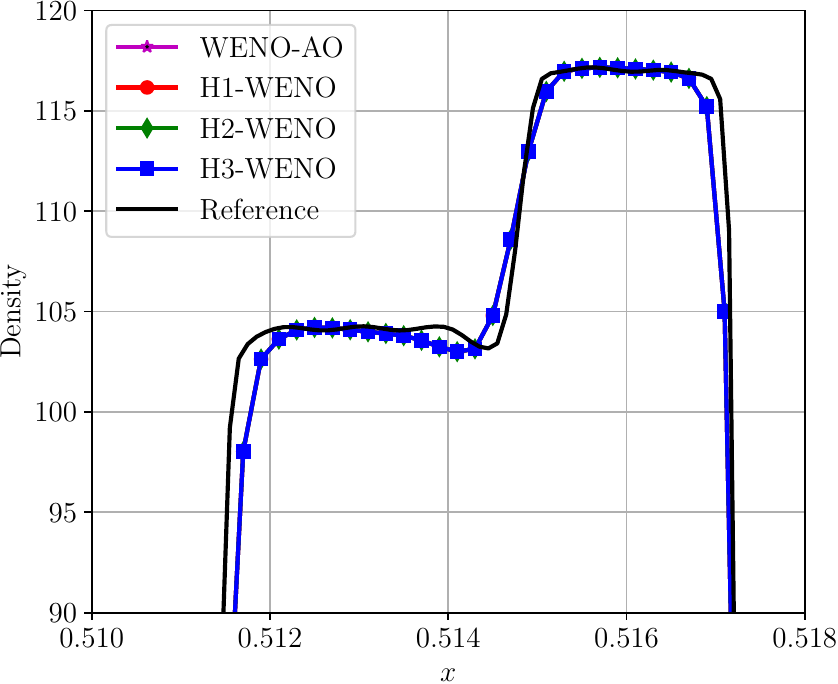}\\
  (a) & (b) \\
   \includegraphics[width = 7.4cm]{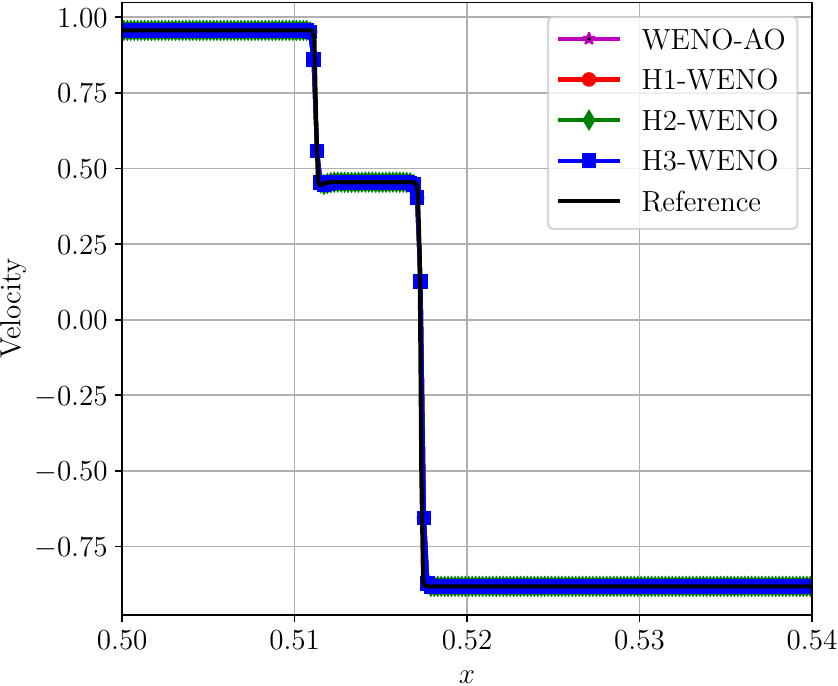}&
  \includegraphics[width = 7.4cm]{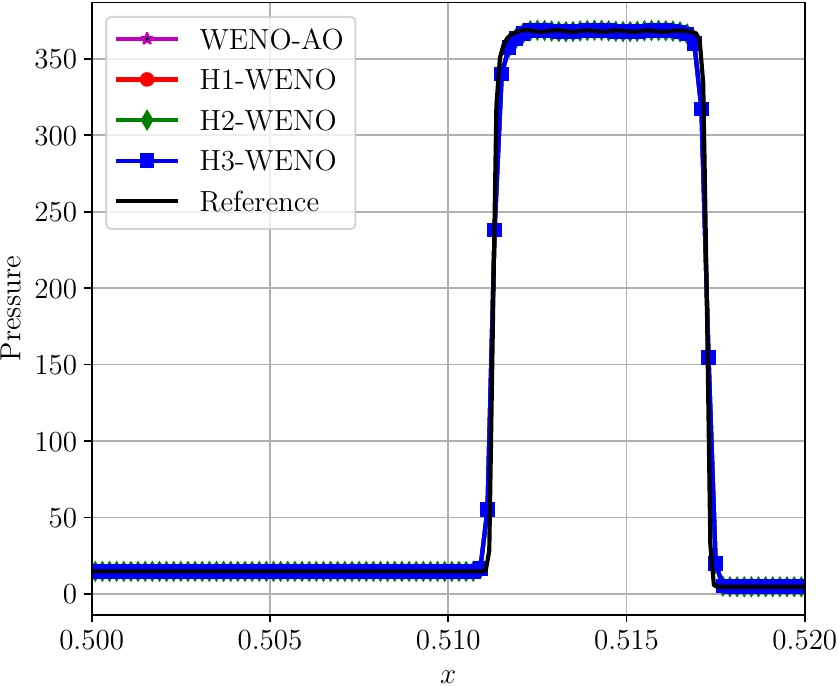}\\
  (c)  & (d) \\
 \end{tabular}
\caption{Comparison of numerical solutions obtained using H1-WENO, H2-WENO, H3-WENO, and WENO-AO schemes at time $T=0.43$ for Example \ref{test8} on a mesh with 2000 grid points, compared with the reference solution. }
\label{Fig:test8}
\end{figure}

\begin{example}[Riemann Problem 3:]
\label{test5}{\rm
We consider another Riemann problem taken from \cite{Mignone2005HLLC}. The computational domain is chosen as $[0,1]$, and the initial data are prescribed as
\begin{equation}\label{rp3}
\left(\rho,\,u,\,p\right)=
\begin{cases}
\left(10,0,\dfrac{40}{3}\right), & \text{if } x < 0.5,\\
\left(1,0,\dfrac{2}{3}\times10^{-6}\right), & \text{if } x > 0.5.
\end{cases}
\end{equation}
Outflow boundary conditions are applied at both ends of the domain. Due to the presence of an extremely low-pressure state on the right side, this test problem provides a stringent assessment of the stability and robustness of the numerical methods.
The numerical solutions computed using the H1-WENO, H2-WENO, H3-WENO, and WENO-AO schemes are displayed in Figure~\ref{Fig:test4} at the final time $t=0.4$ on a mesh consisting of $500$ cells. It is evident from the results that all the schemes are capable of accurately resolving the wave structures arising in the solution while maintaining stable and non-oscillatory behavior near discontinuities. Moreover, the solutions obtained from the proposed hybrid schemes closely match those produced by the WENO-AO scheme, indicating that the proposed methods retain high resolution and robustness even in the presence of severe pressure variations.
}
\end{example}

\begin{example}[Riemann Problem 4:]
\label{test6}{\rm
In this example, we consider another Riemann problem taken from \cite{Mignone2005HLLC}. The computational domain is defined as $[0,1]$, and the initial conditions are specified by
\begin{equation}
\left(\rho,\,u,\,p\right)=
\begin{cases}
\left(1,0.9,1\right), & \text{if } x < 0.5,\\
\left(1,0,10\right), & \text{if } x > 0.5.
\end{cases}
\end{equation}
Outflow boundary conditions are imposed at both boundaries. The exact solution of this problem consists of two shock waves separated by a contact discontinuity, making it a useful benchmark for evaluating the shock-capturing capability of the numerical schemes.
The numerical solutions obtained using the considered schemes are presented in Figure~\ref{Fig:test6} at the final time $t=0.4$. From the results, it can be seen that all the schemes accurately resolve the shock waves and contact discontinuity without generating noticeable spurious oscillations near discontinuous regions. Furthermore, the solutions produced by the proposed schemes remain in close agreement with each other, indicating only minor differences in their overall performance and resolution quality.}
\end{example}

\begin{example}[Density perturbation test case:]
	\label{test7}{\rm
	In this test case, we consider a benchmark problem from \cite{del2002efficient}. The computational domain is taken as $[0,1]$, and outflow boundary conditions are applied at both boundaries. The initial conditions are given by
\[
\left(\rho,\,u,\,p\right)=
\begin{cases}
\left(5,0,50\right), & \text{if } x < 0.5,\\
\left(2 + 0.3\sin(50x),0,5\right), & \text{if } x > 0.5.
\end{cases}
\]
This problem generates highly oscillatory yet smooth density waves interacting with discontinuous structures, making it a demanding test for evaluating the resolution capability of numerical schemes. In particular, accurately resolving the fine-scale smooth oscillations without introducing excessive numerical dissipation is a challenging task.
The numerical results obtained using the considered schemes are shown in Figure~\ref{Fig:test7} at the final time $t=0.35$ on a mesh consisting of $500$ cells. From the figure, it can be observed that all the schemes successfully capture the oscillatory wave structures while maintaining good resolution of the small-scale features. The solutions produced by the proposed schemes remain comparable to those obtained using the WENO-AO scheme, demonstrating their ability to preserve complex smooth structures with high accuracy and minimal numerical diffusion.
}
\end{example}

\begin{example}[Blast waves test case:]
	\label{test8}{\rm
For this test problem, we consider the blast wave interaction problem from \cite{marti1996extension}. The computational domain is chosen as $[0,1]$, and outflow boundary conditions are imposed at both boundaries. The initial conditions are prescribed as
\[
\left(\rho,\,u,\,p\right)=
\begin{cases}
\left(1,0,1000\right), & \text{if } x < 0.1,\\
\left(1,0,0.01\right), & \text{if } 0.1 < x < 0.9,\\
\left(1,0,100\right), & \text{if } x > 0.9.
\end{cases}
\]
This problem produces strong blast waves that interact with each other, resulting in extremely sharp solution structures concentrated within a very narrow region. Accurately resolving these fine-scale features requires highly refined meshes. Therefore, the simulations are carried out using  grid resolution consisting of $4000$ cells.
The numerical solutions obtained using the considered schemes are presented in Figure~\ref{Fig:test8} at the final time $t=0.43$. To better illustrate the performance of the schemes in resolving the intricate wave interactions, a zoomed view of the solution around $x=0.52$ is also provided. From the results, it can be observed that all the schemes are capable of capturing complex shock interactions and narrow structures with good accuracy. Moreover, the proposed schemes produce results comparable to those obtained using the WENO-AO scheme while maintaining sharp resolution of the discontinuities and minimizing numerical oscillations.
	}
\end{example}
\subsection{Two dimensional numerical tests}\label{sec:num2d}
We now present a set of two-dimensional test cases, including several two-dimensional Riemann problems.
\begin{figure}[ht]
 \centering
 \begin{tabular}{cc}
  \includegraphics[width = 7.4cm]{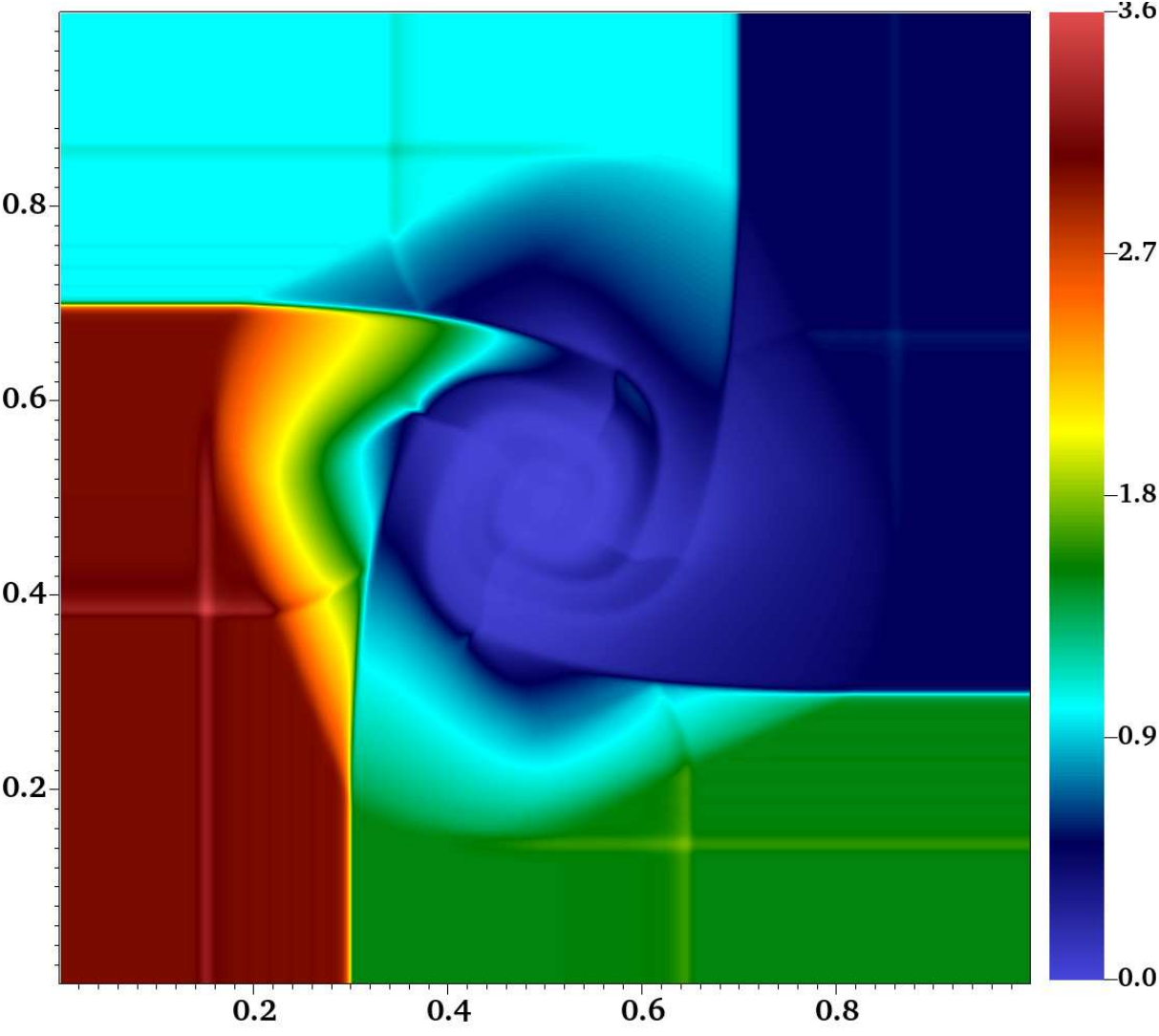}&
  \includegraphics[width = 7.4cm]{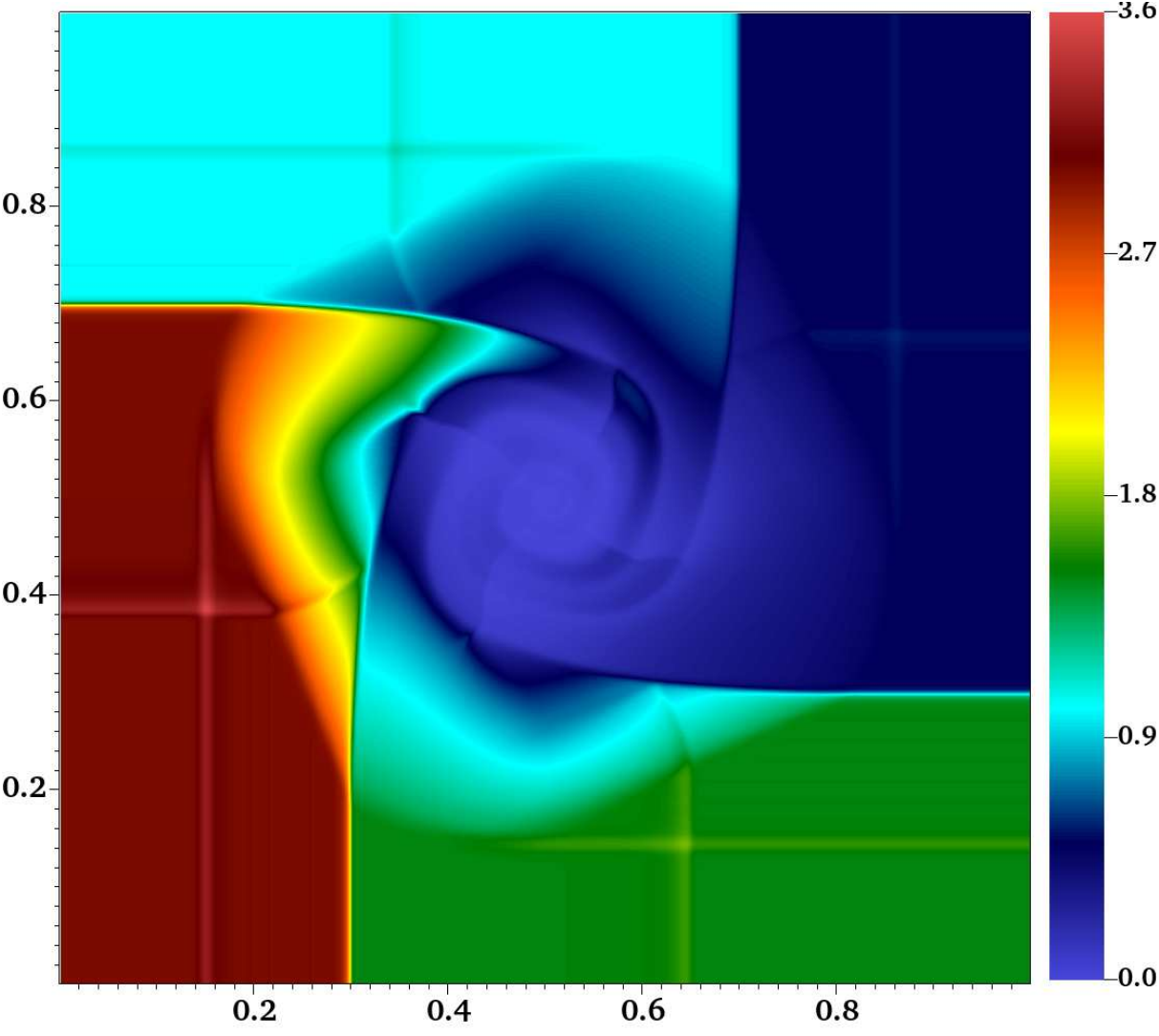}\\
  (a) WENO-AO & (b) H1-WENO \\
   \includegraphics[width = 7.4cm]{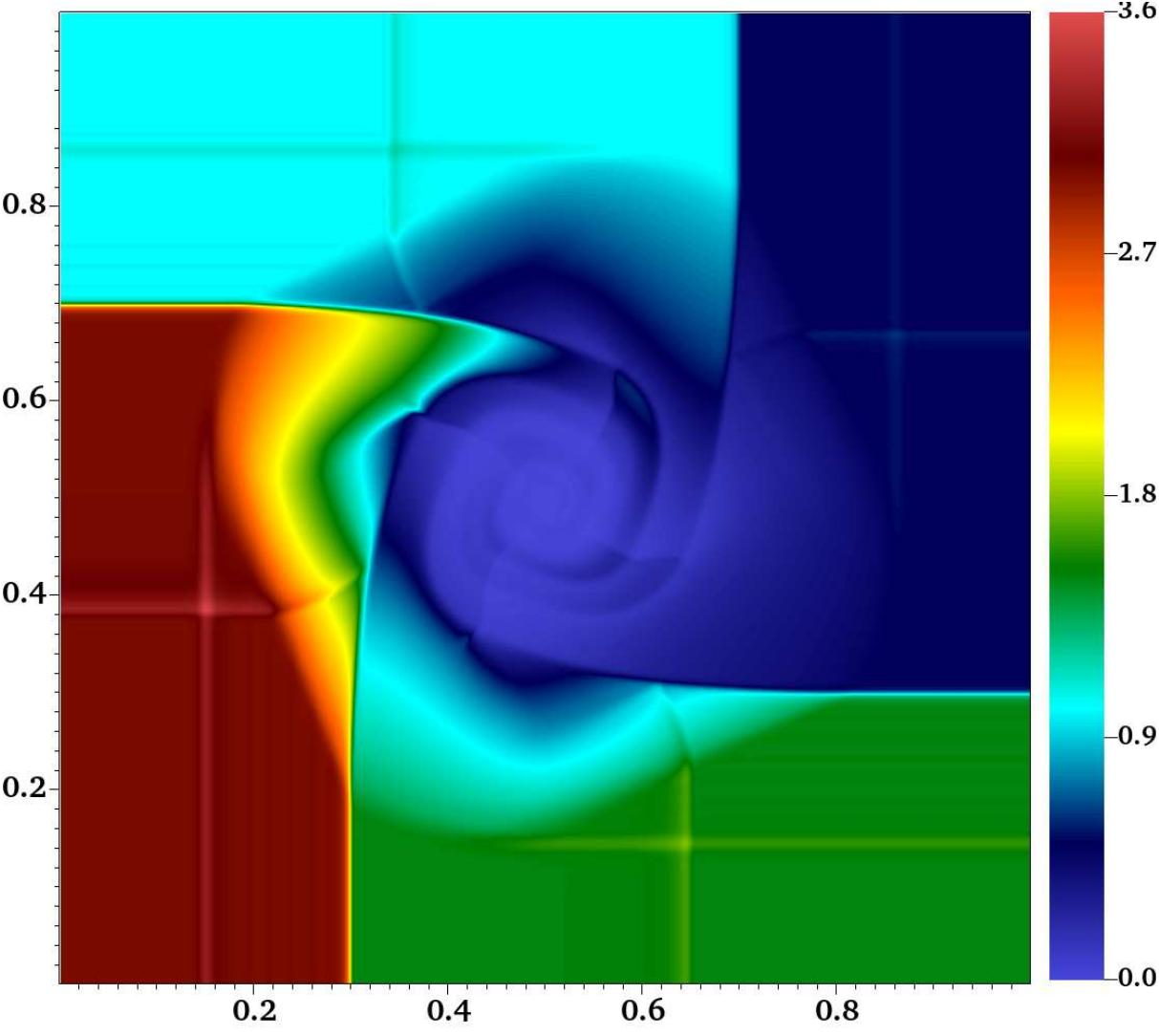}&
  \includegraphics[width = 7.4cm]{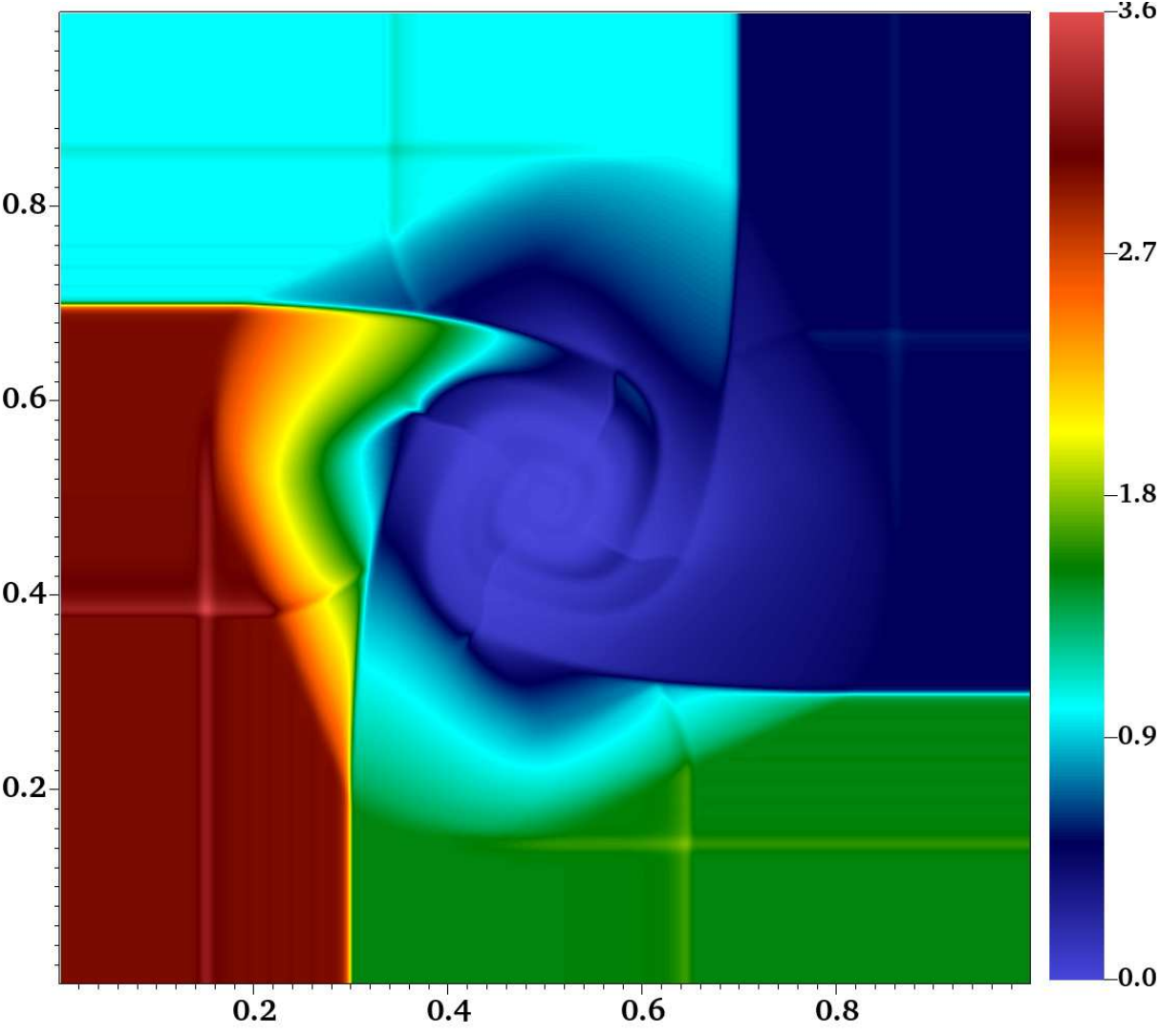}\\
  (c) H2-WENO & (d) H3-WENO \\
 \end{tabular}
 \caption{Comparison of the H1-WENO, H2-WENO, and H3-WENO schemes with the WENO-AO scheme for Example \ref{rp1} in terms of the density variable on a $400 \times 400$ computational grid at time $T=0.4$.}
 \label{Fig:rp1}
 \end{figure}

 \begin{figure}
 \centering
 \begin{tabular}{ccc}
   \includegraphics[width = 4.6cm]{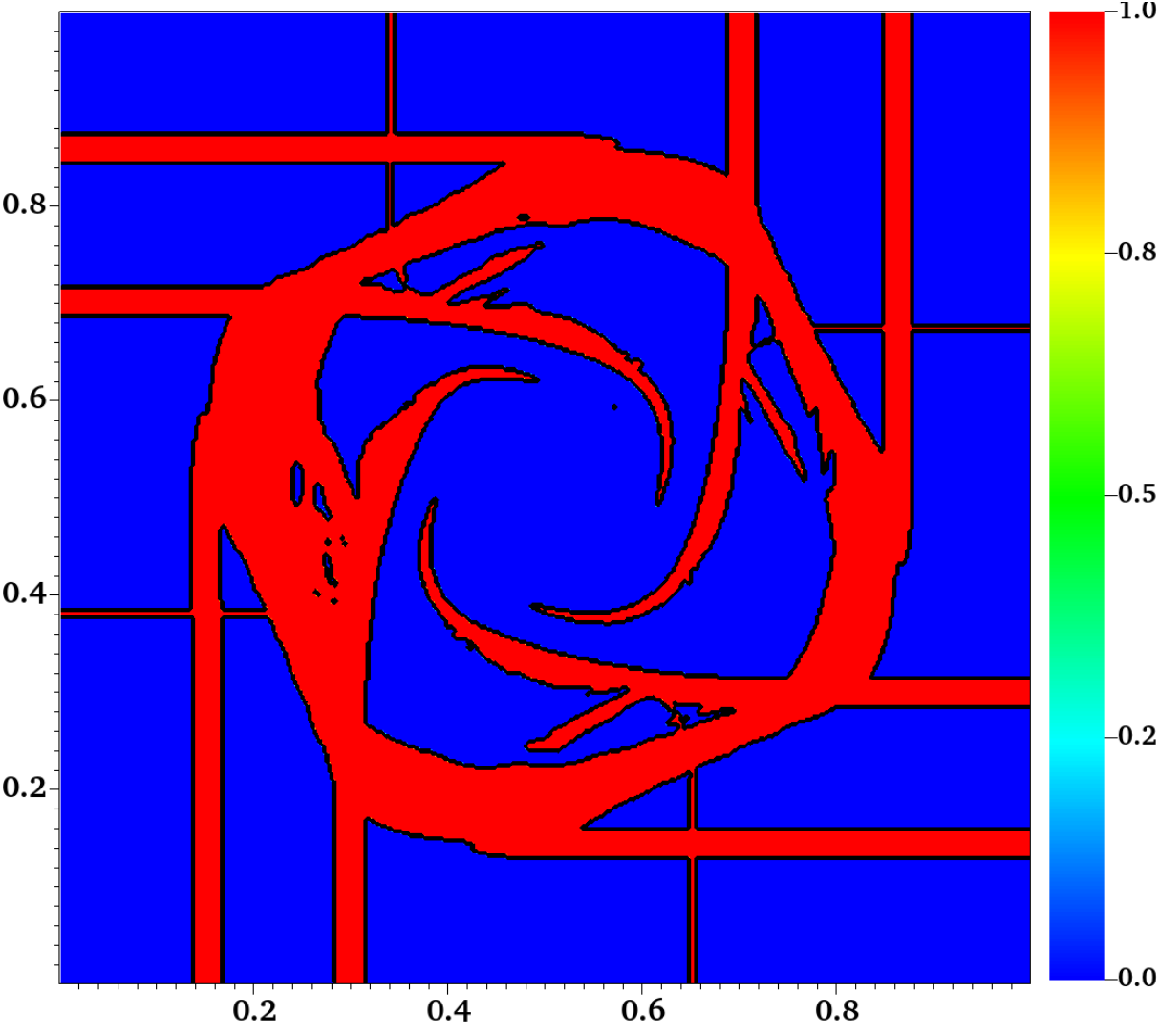} &
  \includegraphics[width = 4.6cm]{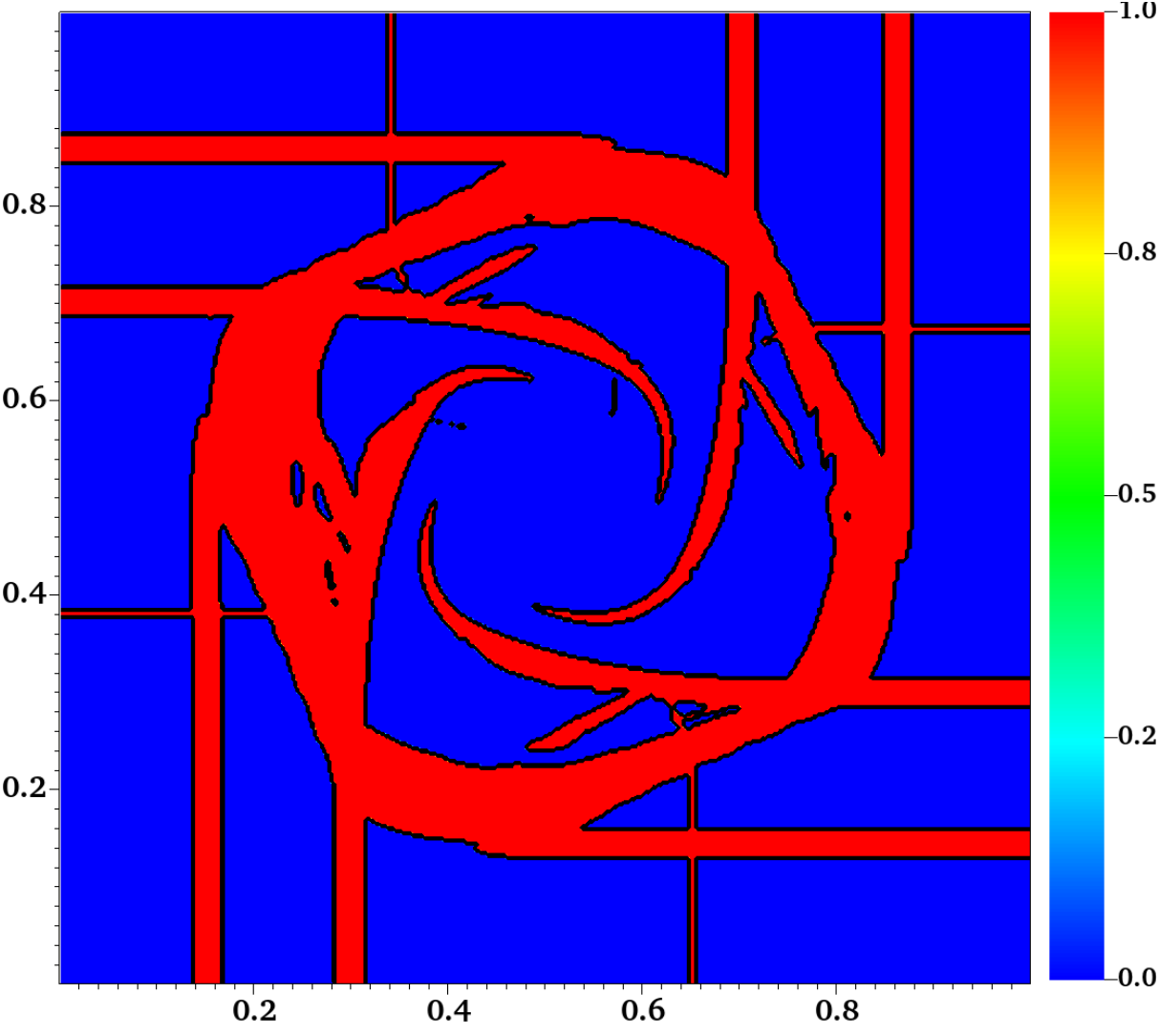}&
   \includegraphics[width = 4.60cm]{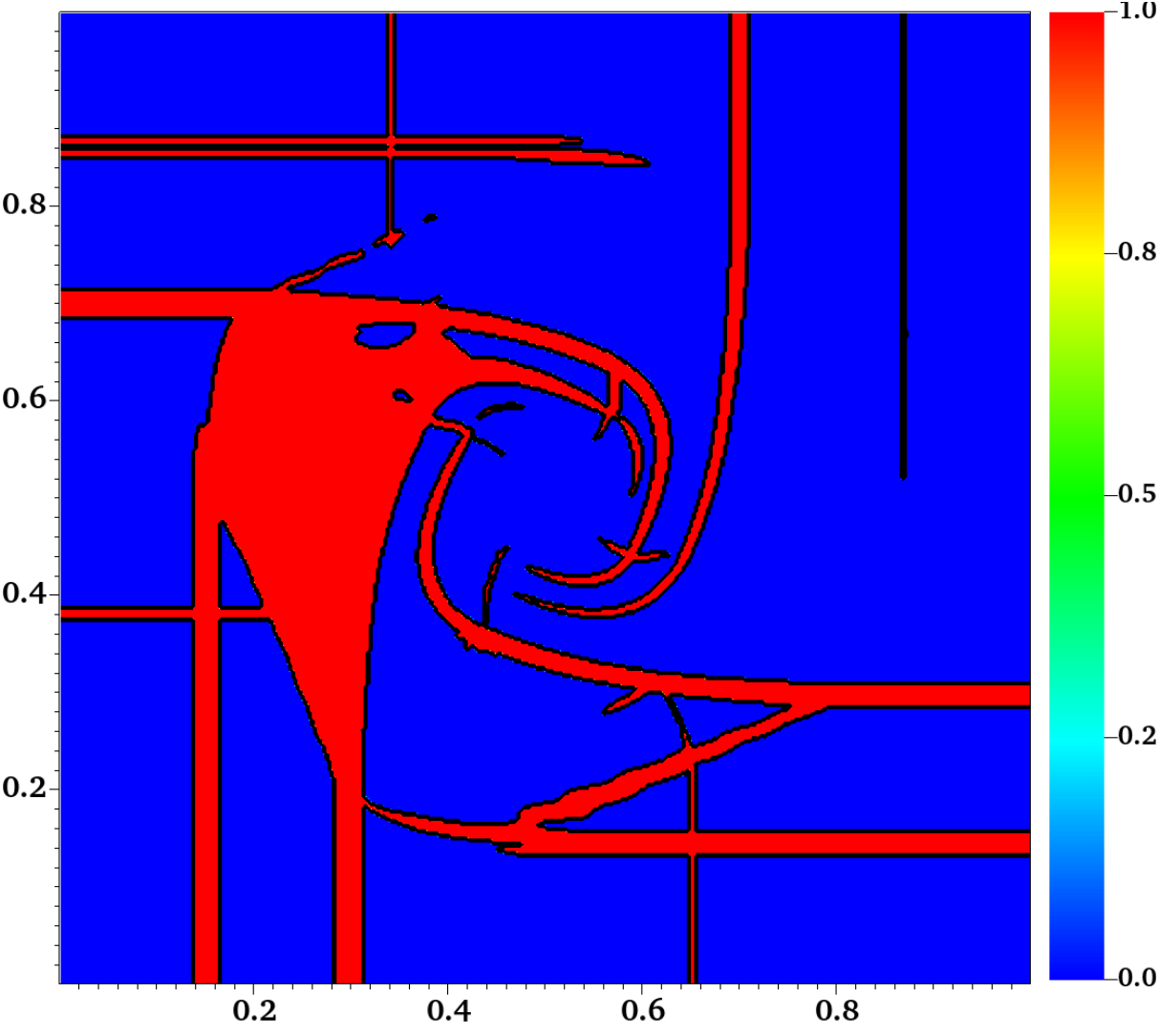}\\
  (a)  H1-WENOZ & (b) H2-WENOZ& (c) H3-WENOZ\\
 \end{tabular}
\caption{Plot of the troubled-cell indicators obtained using the hybrid schemes for Example \ref{rp1} at time $T=0.4$ over a grid size $400\times 400$.}
\label{Fig:rp1.tc}
\end{figure}

\begin{example}[Two-dimensional Riemann problem 1:]
\label{rp1}{\rm
We consider a two-dimensional Riemann problem from \cite{nunez2016xtroem}. The computational domain is taken as $[0,1]\times[0,1]$ with outflow boundary conditions. The initial Riemann data are prescribed as
\begin{equation*}
\left(\rho,\,u_x,\,u_y,\,p\right)=
\begin{cases}
\left(0.5,\,0.5,\,-0.5,\,5\right), & \text{if } x>0.5 \text{ and } y>0.5,\\
\left(1,\,0.5,\,0.5,\,5\right), & \text{if } x<0.5 \text{ and } y>0.5,\\
\left(3,\,-0.5,\,0.5,\,5\right), & \text{if } x<0.5 \text{ and } y<0.5,\\
\left(1.5,\,-0.5,\,-0.5,\,5\right), & \text{if } x>0.5 \text{ and } y<0.5.
\end{cases}
\end{equation*}

The solution of this problem consists of four interacting vortex sheets that produce a low-density region near the center of the computational domain. The solution is computed using the H1-WENO, H2-WENO, H3-WENO, and WENO-AO schemes up to the final time $T=0.4$ on a uniform mesh of $400\times 400$ grid points.
In Figure \ref{Fig:rp1}, we present the density contours obtained using the different numerical schemes at the final time. The proposed hybrid schemes successfully capture the complex flow structures and produce results that are in good agreement with those obtained using the WENO-AO scheme.
Figure \ref{Fig:rp1.tc} shows the corresponding troubled-cell indicators at the final time. In these plots, the red regions represent the detected discontinuities, whereas the blue regions correspond to smooth parts of the solution. It can be seen that the proposed troubled-cell indicator accurately identifies non-smooth regions of the flow. Among the three hybrid schemes, the H3-WENO scheme performs better than the H1-WENO and H2-WENO schemes, as it marks fewer unnecessary cells in smooth regions while still effectively capturing the discontinuities.
	}
	
\end{example}

\begin{figure}[h!]
 \centering
 \begin{tabular}{cc}
  \includegraphics[width = 7.4cm]{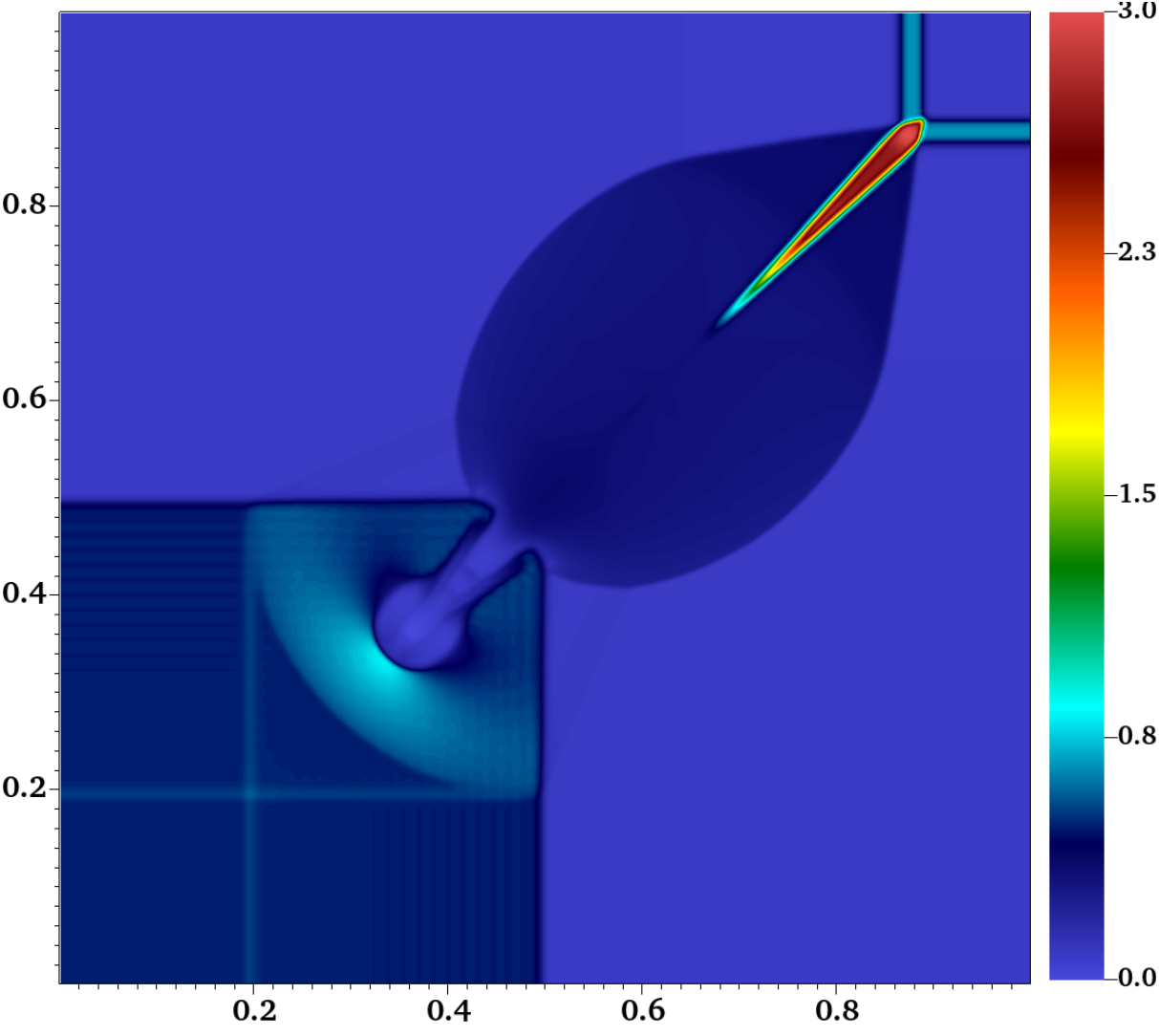}&
  \includegraphics[width = 7.4cm]{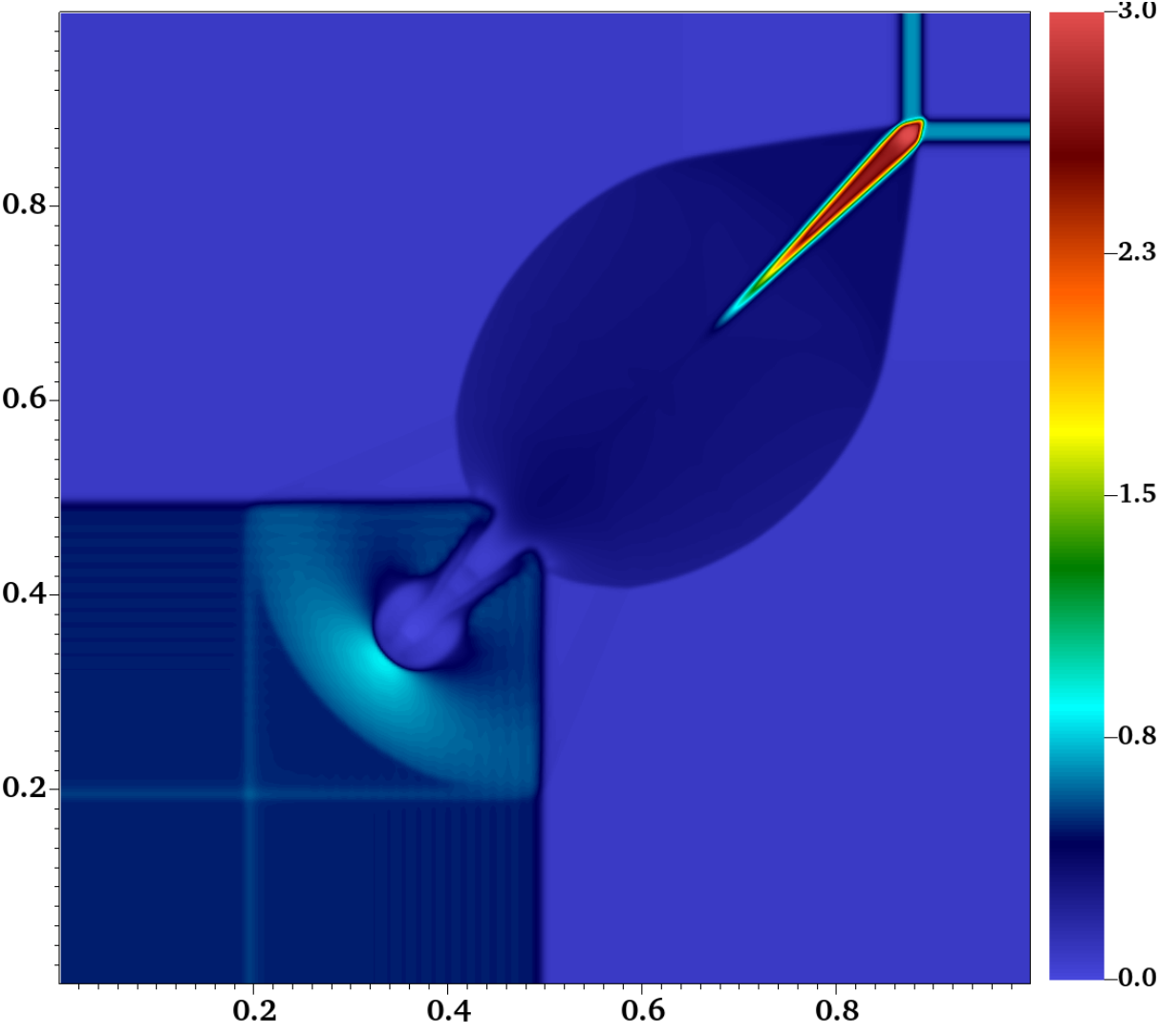}\\
  (a) WENO-AO & (b) H1-WENO \\
   \includegraphics[width = 7.4cm]{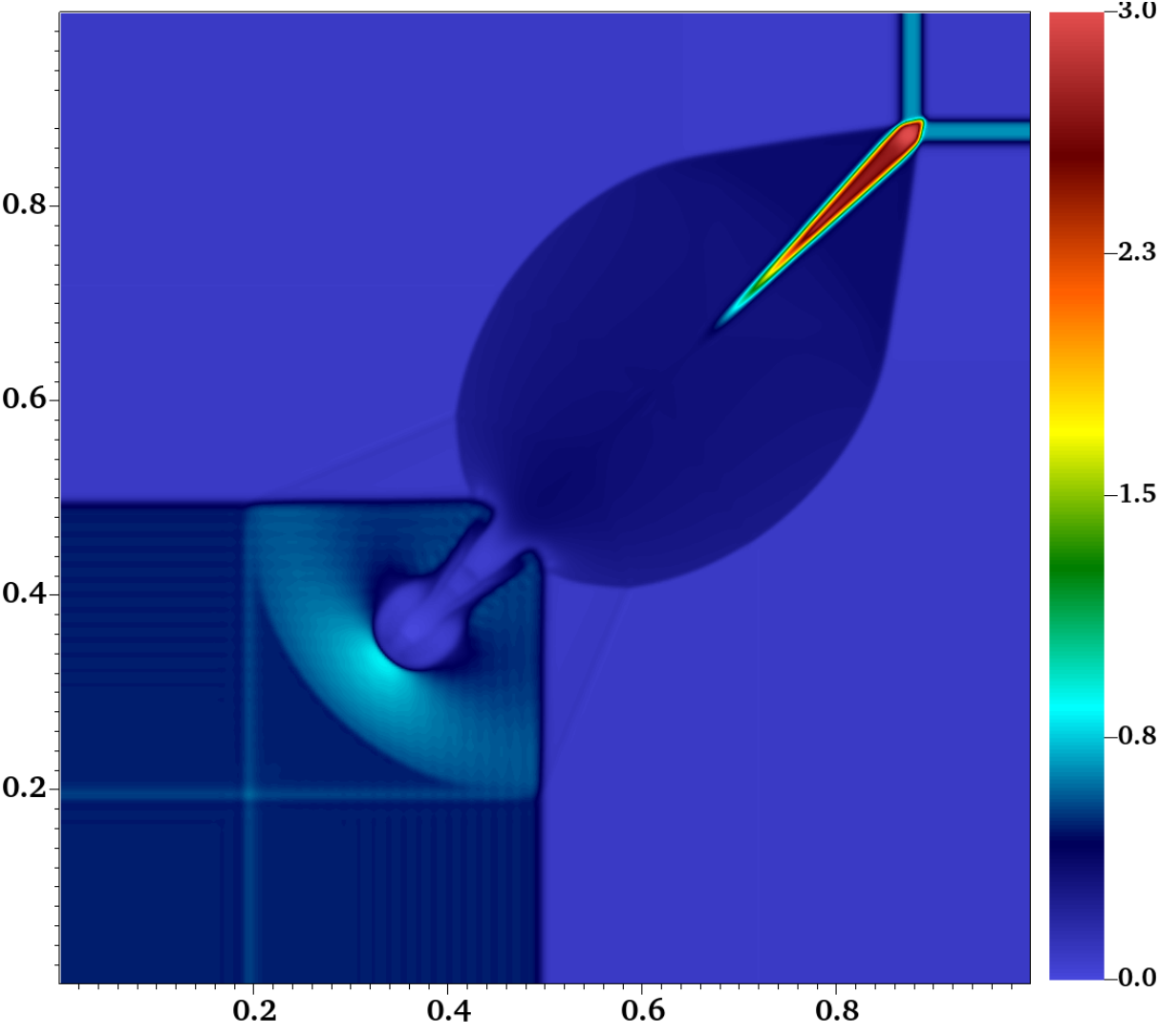}&
  \includegraphics[width = 7.4cm]{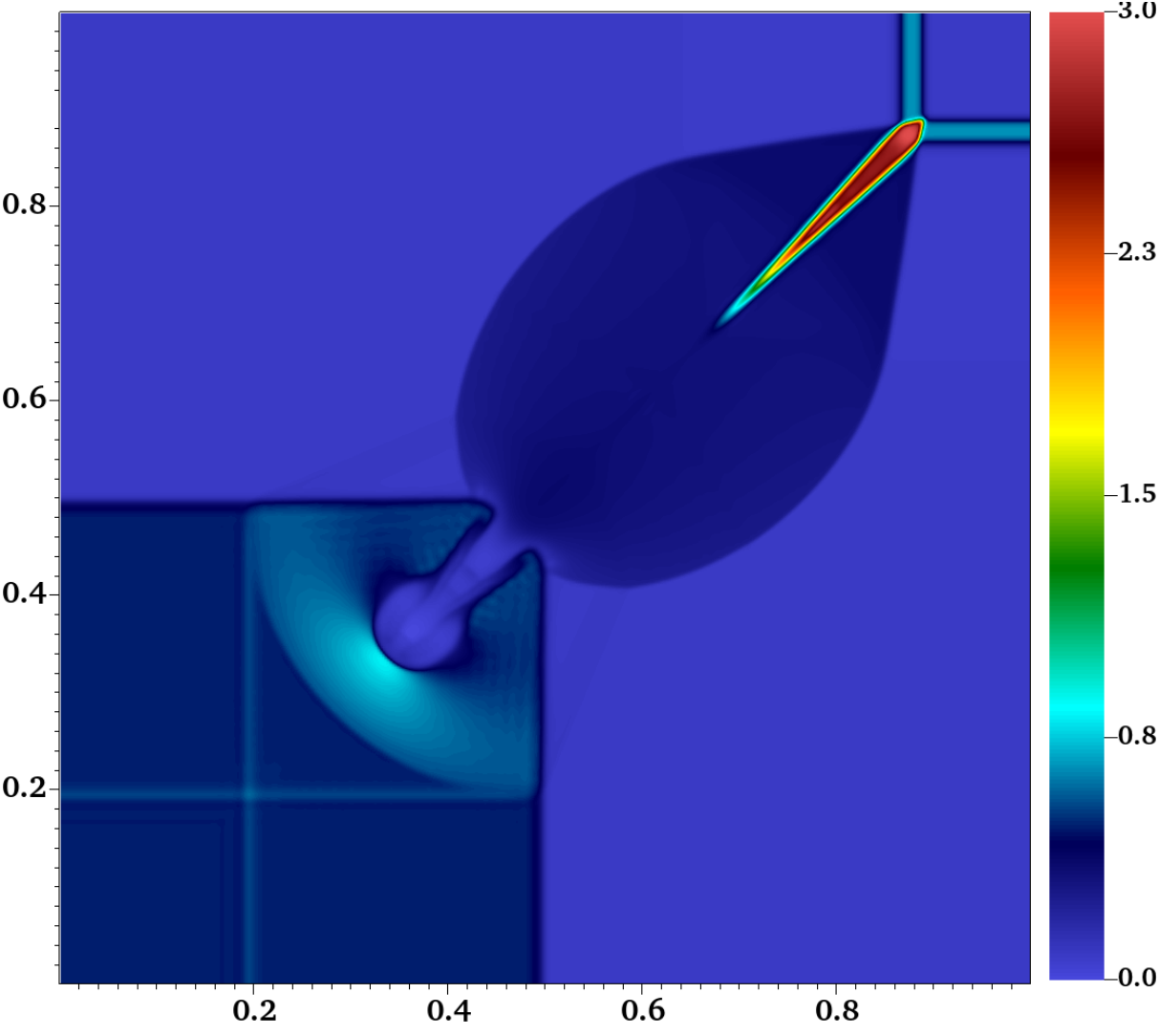}\\
  (c) H2-WENO & (d) H3-WENO \\
 \end{tabular}
 \caption{Comparison of the H1-WENO, H2-WENO, and H3-WENO schemes with the WENO-AO scheme for Example \ref{rp2} in terms of the density variable on a $400 \times 400$ computational grid at time $T=0.4$.}
 \label{Fig:rp2}
 \end{figure}
  \begin{figure}
 \centering
 \begin{tabular}{ccc}
   \includegraphics[width = 4.6cm]{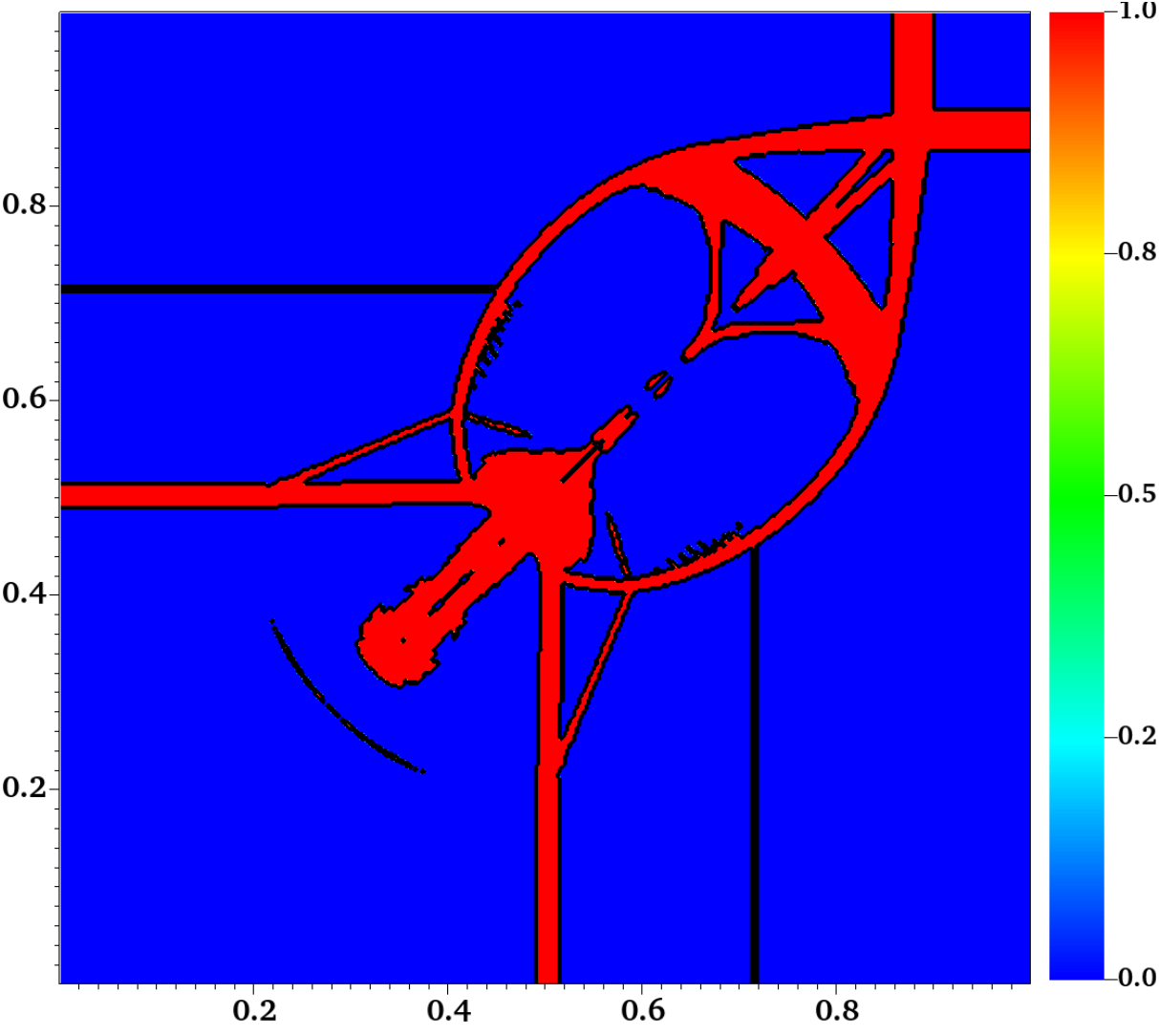} &
  \includegraphics[width = 4.6cm]{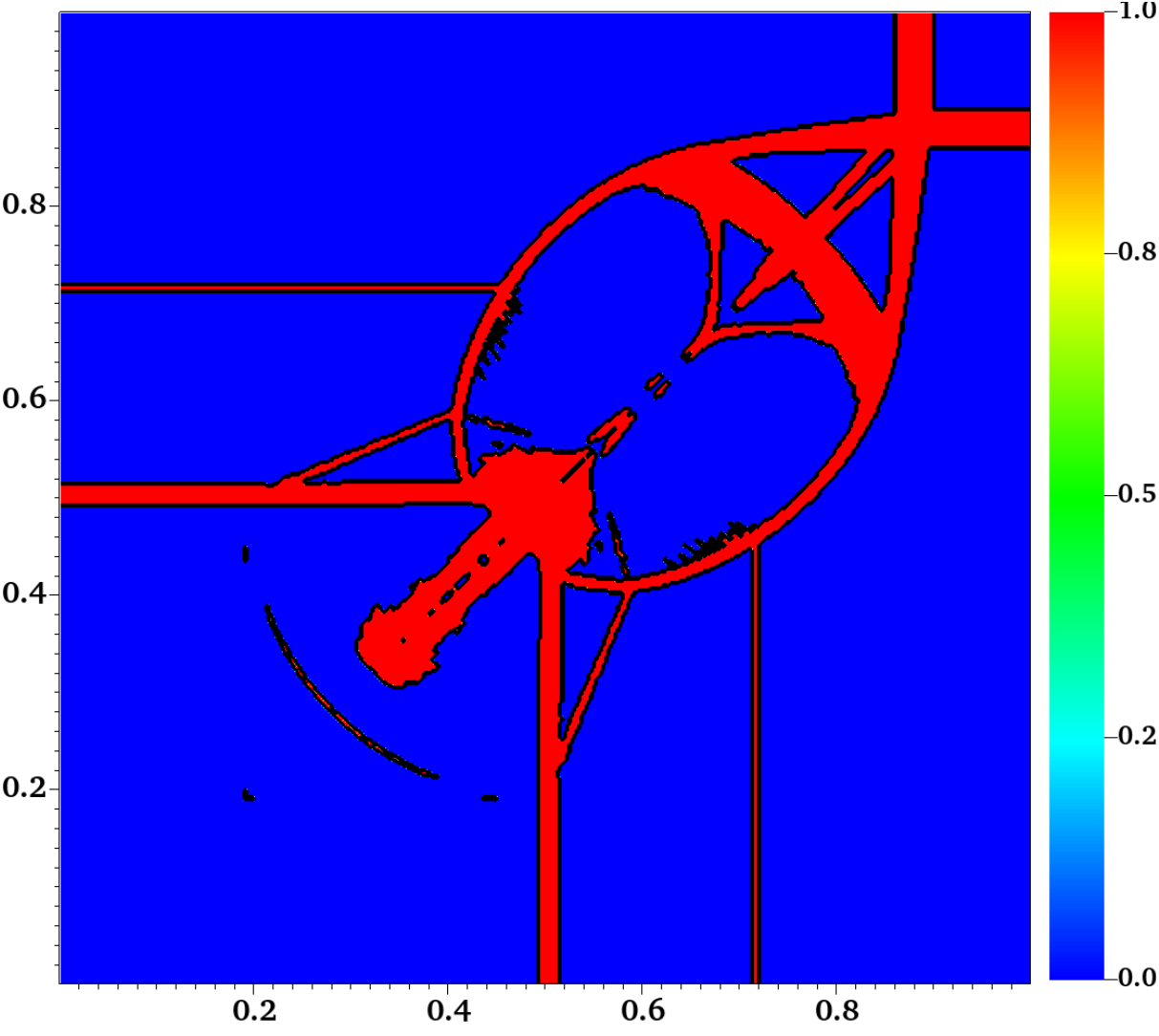}&
   \includegraphics[width = 4.60cm]{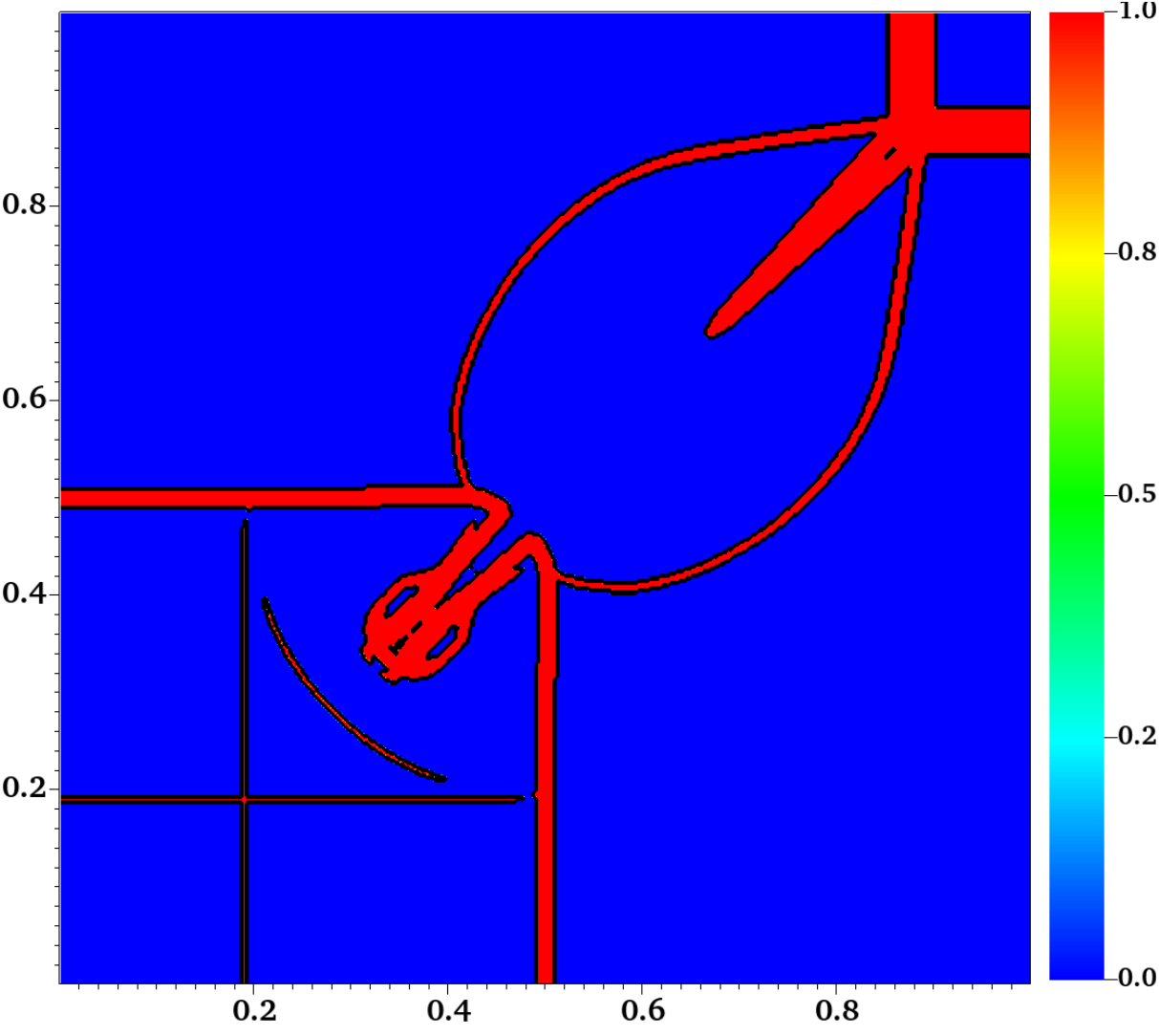}\\
  (a)  H1-WENO & (b) H2-WENO& (c) H3-WENO\\
 \end{tabular}
\caption{Plot of the troubled-cell indicators obtained using the hybrid schemes for Example \ref{rp2} at time $T=0.4$ over a grid size $400\times 400$.}
\label{Fig:rp2.tc}
\end{figure}
\begin{figure}[ht!]
 \centering
 \begin{tabular}{cc}
  \includegraphics[width = 7.4cm]{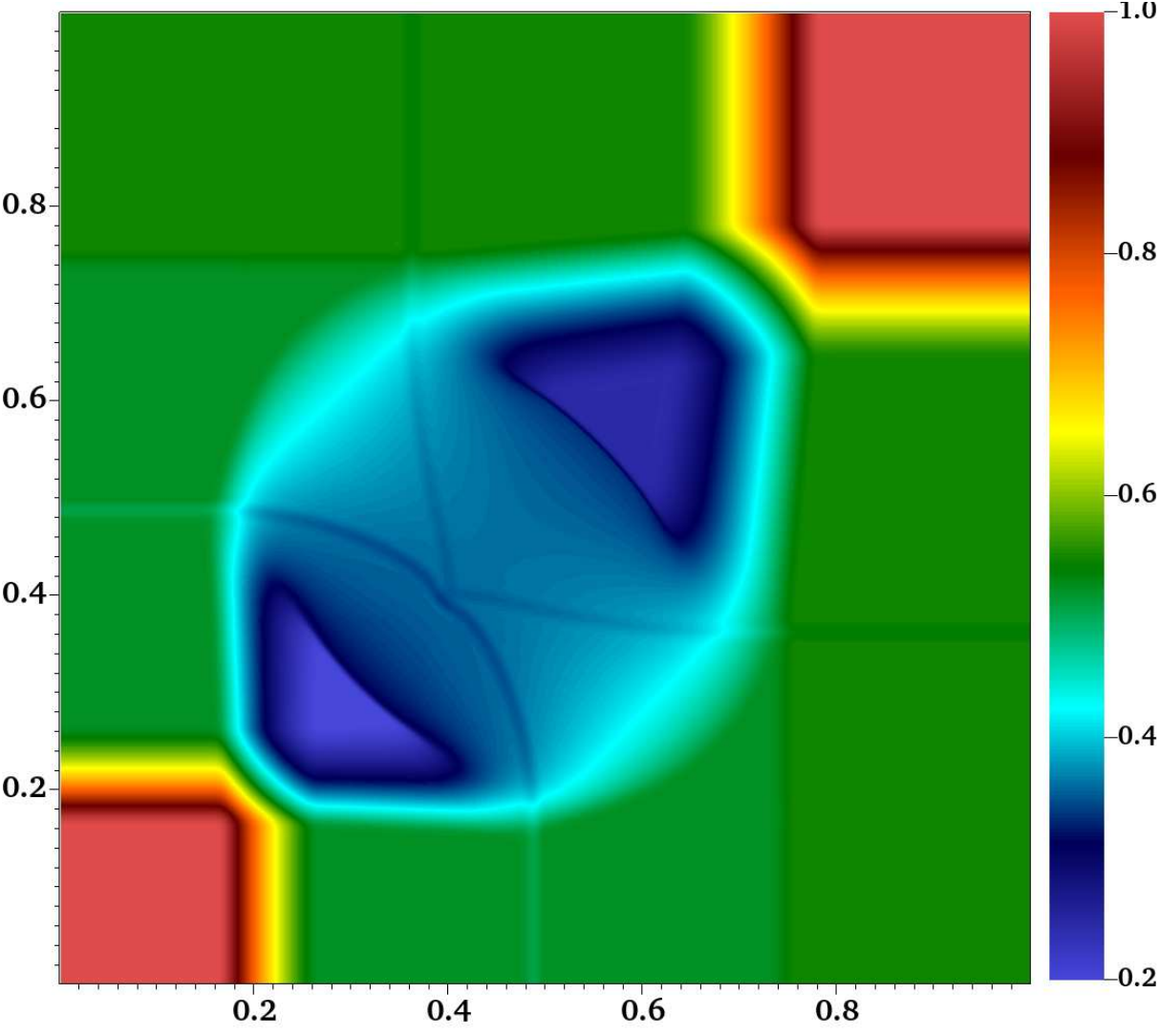}&
  \includegraphics[width = 7.4cm]{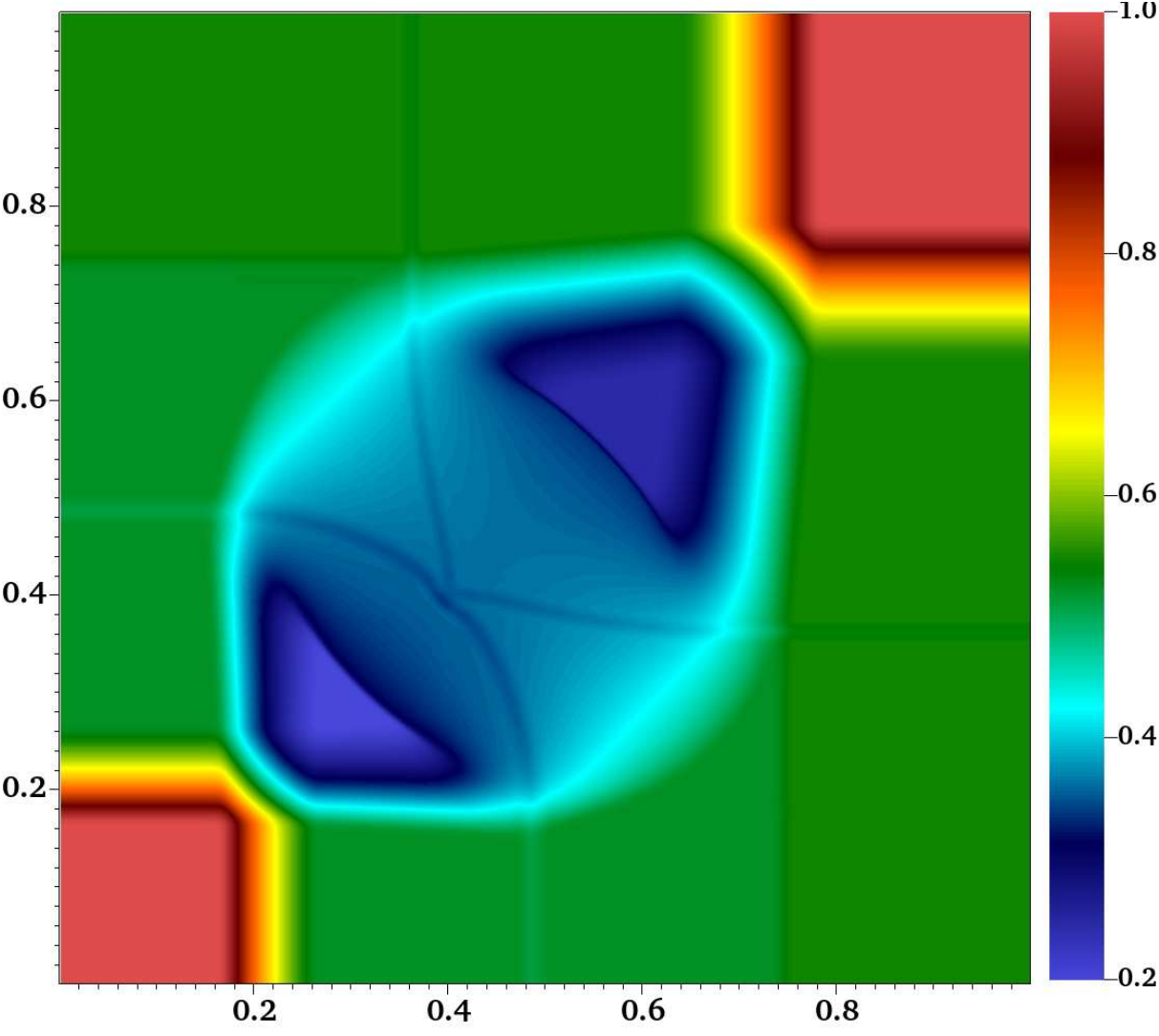}\\
  (a) WENO-AO & (b) H1-WENO \\
   \includegraphics[width = 7.4cm]{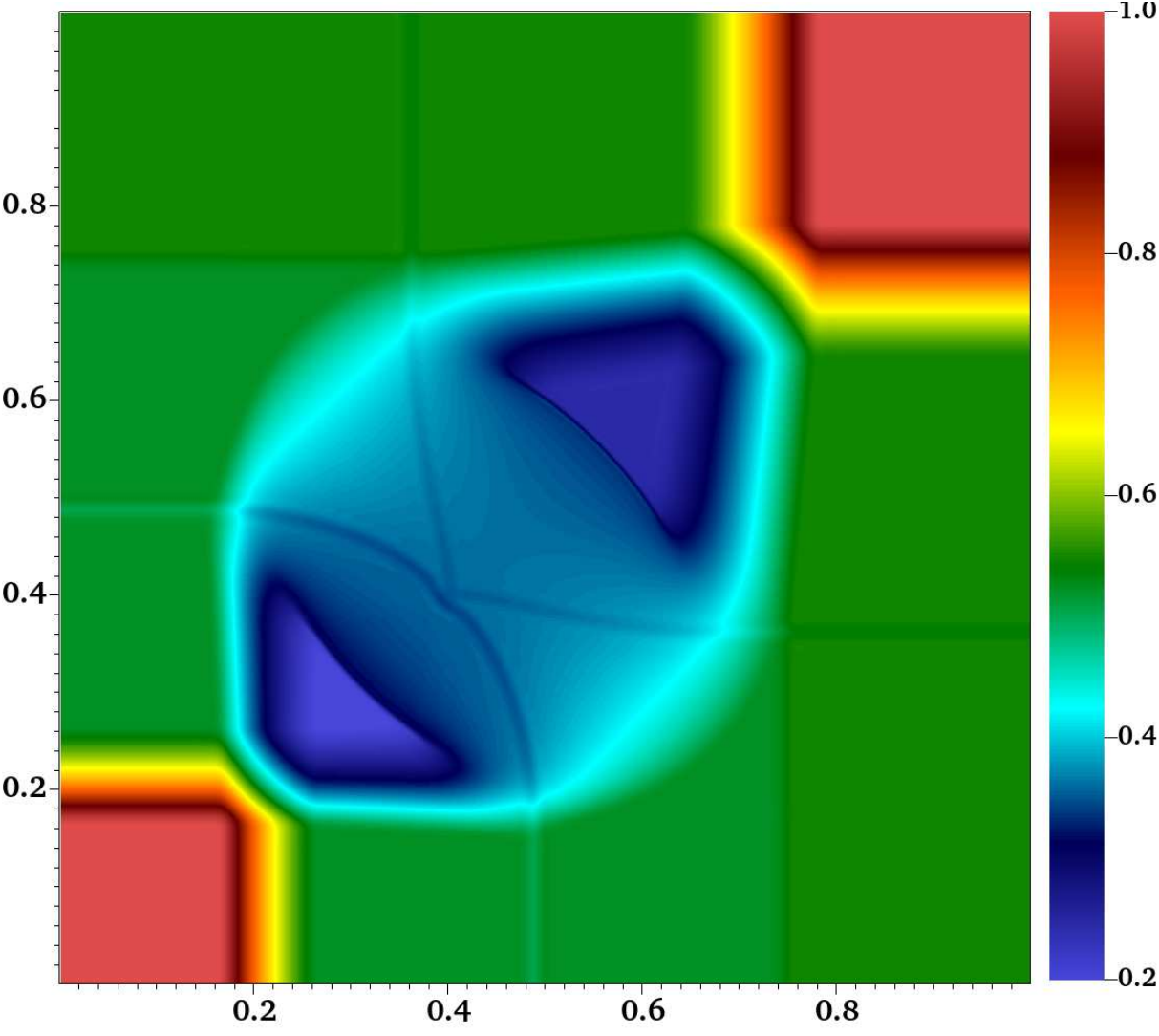}&
  \includegraphics[width = 7.4cm]{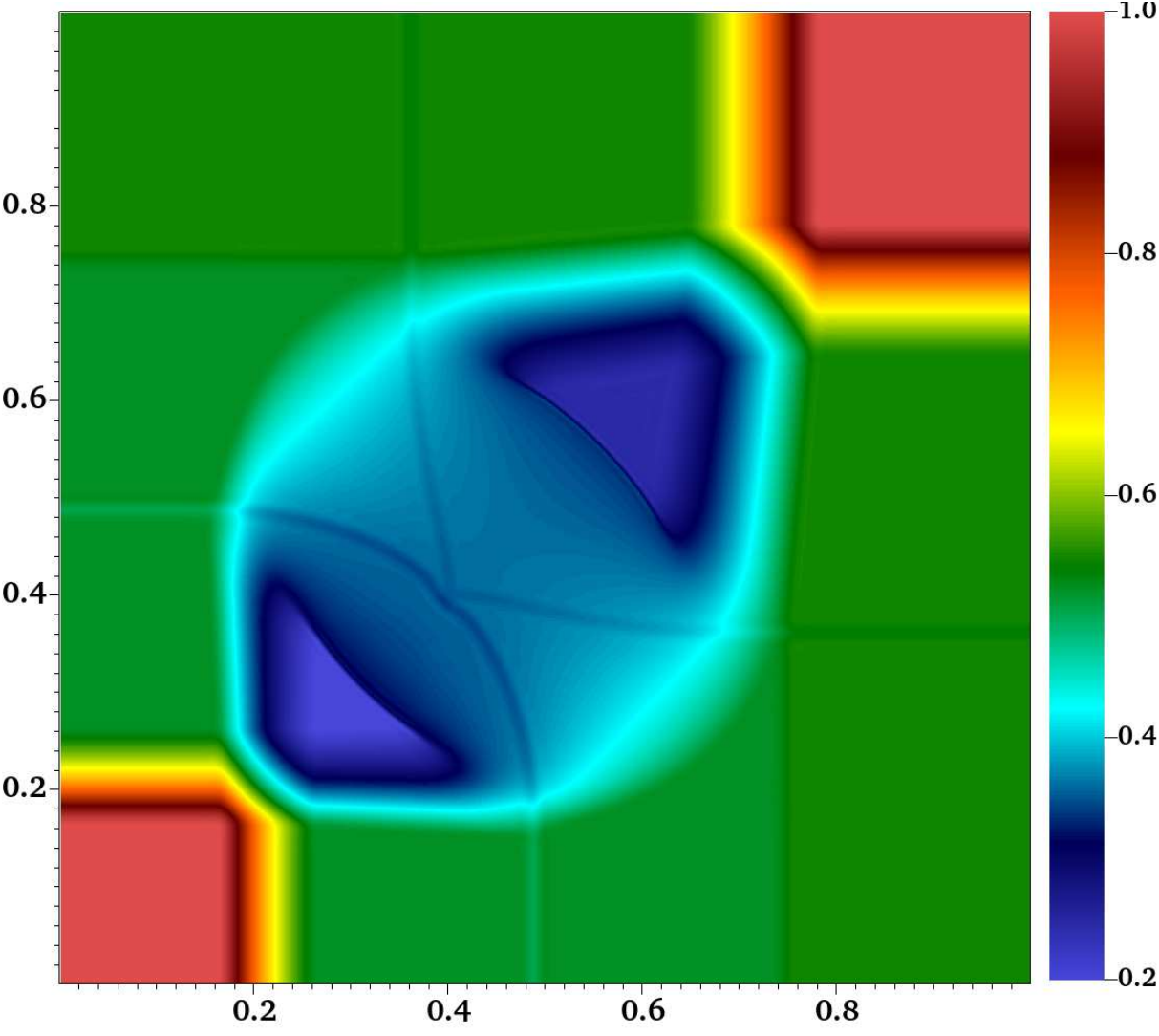}\\
  (c) H2-WENO & (d) H3-WENO \\
 \end{tabular}
 \caption{Comparison of the H1-WENO, H2-WENO, and H3-WENO schemes with the WENO-AO scheme for Example \ref{rp3} in terms of the density variable on a $400 \times 400$ computational grid at time $T=0.4$.}
 \label{Fig:rp3}
 \end{figure}

  \begin{figure}[ht!]
 \centering
 \begin{tabular}{ccc}
   \includegraphics[width = 4.6cm]{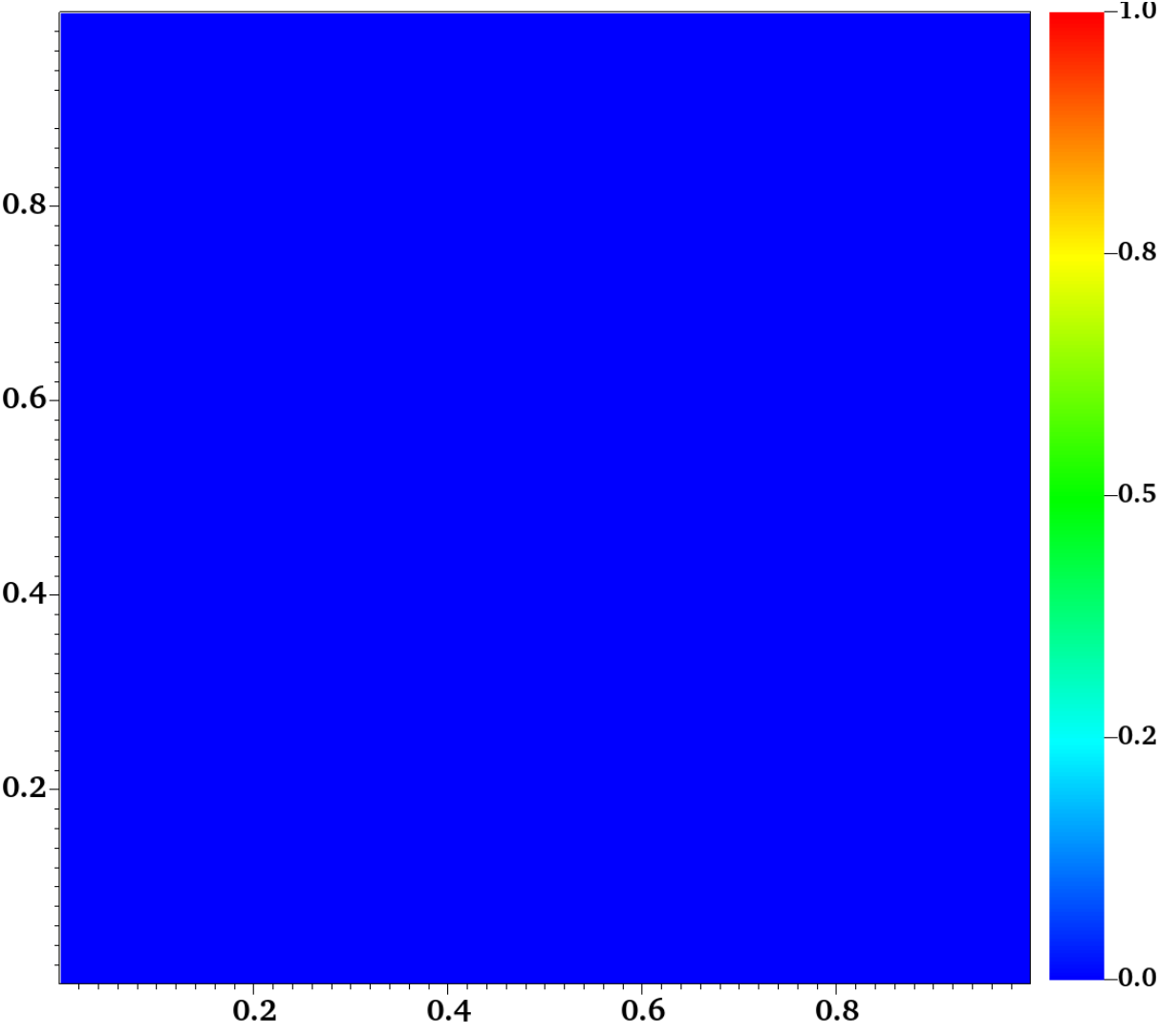} &
  \includegraphics[width = 4.6cm]{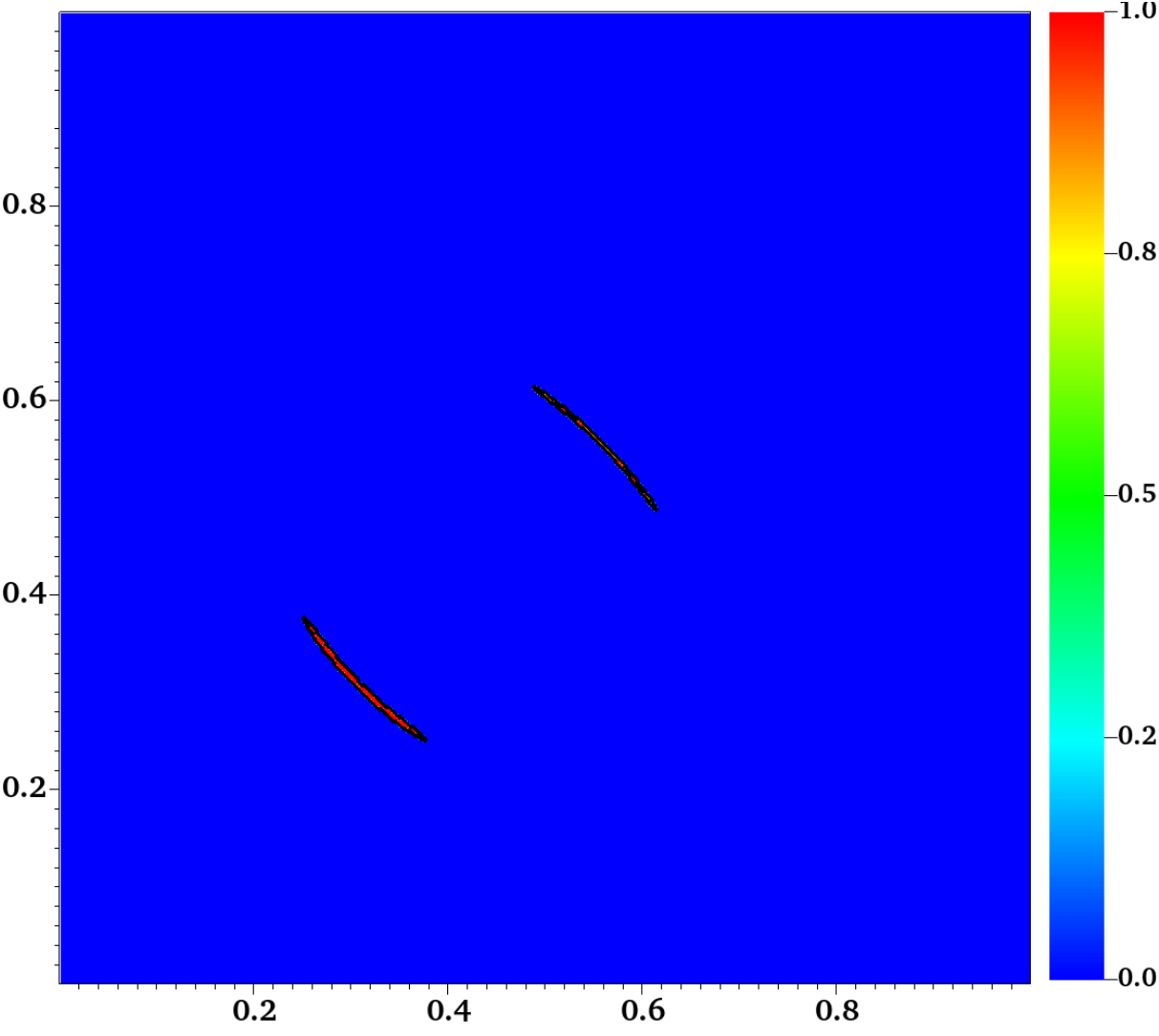}&
   \includegraphics[width = 4.60cm]{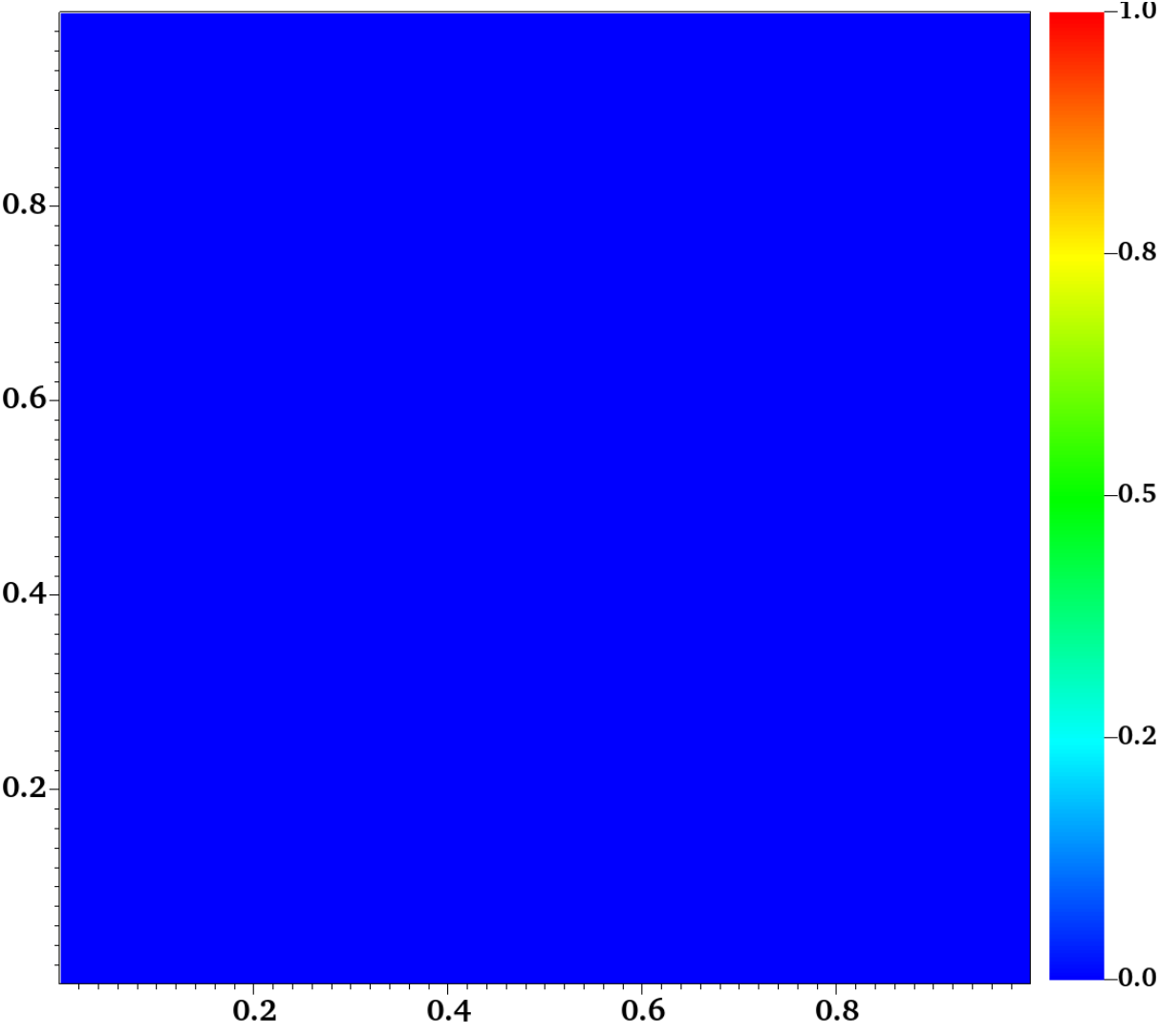}\\
  (a)  H1-WENO & (b) H2-WENO & (c) H3-WENO\\
 \end{tabular}
\caption{Plot of the troubled-cell indicators obtained using the hybrid schemes for Example \ref{rp3} at time $T=0.4$ over a grid size $400\times 400$.}
\label{Fig:rp3.tc}
\end{figure}
\begin{example}[Two-dimensional Riemann problem 2:]
	\label{rp2}
	{\rm
    We next consider a two-dimensional Riemann problem similar to the one studied in \cite{del2002efficient}. The computational domain is chosen as $[0,1]\times[0,1]$ with outflow boundary conditions imposed on all sides. The initial conditions are specified as
\begin{equation*}
\left(\rho,\,u_x,\,u_y,\,p\right)=
\begin{cases}
\left(0.1,\,0,\,0,\,0.01\right), & \text{if } x>0.5 \text{ and } y>0.5,\\
\left(0.1,\,0.9,\,0,\,1\right), & \text{if } x<0.5 \text{ and } y>0.5,\\
\left(0.5,\,0,\,0,\,1\right), & \text{if } x<0.5 \text{ and } y<0.5,\\
\left(0.1,\,0,\,0.9,\,1\right), & \text{if } x>0.5 \text{ and } y<0.5.
\end{cases}
\end{equation*}
This configuration generates complex multidimensional wave interactions involving shock waves, contact discontinuities, and rarefaction structures. The numerical solution is computed using the H1-WENO, H2-WENO, H3-WENO, and WENO-AO schemes on a uniform mesh consisting of $400\times4 00$ cells.

The computed density contours at the final time are shown in Figure \ref{Fig:rp2}. All the considered schemes successfully capture the major flow structures and discontinuities without introducing significant spurious oscillations. The proposed hybrid schemes produce results comparable to the WENO-AO scheme while preserving sharp resolution of discontinuous features. Figure \ref{Fig:rp2.tc} presents the corresponding troubled-cell indicator distributions at the final time. It can be observed that the indicators effectively detect discontinuities, with the H3-WENO troubled-cell indicator showing better localization compared to the H1-WENO and H2-WENO indicators.}
\end{example}

\begin{example}[Two-dimensional Riemann problem 3:]
	\label{rp3}{\rm
 Next, we consider a two-dimensional Riemann problem from \cite{nunez2016xtroem}. The computational domain is chosen as $[0,1]\times[0,1]$ with the following initial states:
\begin{equation*}
\left(\rho,\,u_x,\,u_y,\,p\right)=
\begin{cases}
\left(1,\,0,\,0,\,1\right), & \text{if } x>0.5 \text{ and } y>0.5,\\
\left(0.5771,\,-0.3529,\,0,\,0.4\right), & \text{if } x<0.5 \text{ and } y>0.5,\\
\left(1,\,-0.3529,\,-0.3529,\,1\right), & \text{if } x<0.5 \text{ and } y<0.5,\\
\left(0.5771,\,0,\,-0.3529,\,0.4\right), & \text{if } x>0.5 \text{ and } y<0.5.
\end{cases}
\end{equation*}
The evolution of the solution is governed by the interaction of two rarefaction waves that subsequently generate two symmetric shock structures. This multidimensional wave interaction provides a challenging test for assessing the ability of numerical schemes to accurately resolve discontinuities while preserving the symmetry of the solution.

The numerical simulations are carried out using the H1-WENO, H2-WENO, H3-WENO, and WENO-AO schemes on a uniform mesh of $400\times400$ cells up to the final time $t=0.4$. In Figure \ref{Fig:rp3}, we present the density contours obtained using all four schemes. It can be observed that the proposed hybrid schemes are capable of accurately capture the complex flow structures and produce results that are in close agreement with those obtained using the WENO-AO scheme.

Figure \ref{Fig:rp3.tc} shows the corresponding troubled-cell indicator distributions for the H1-WENO, H2-WENO, and H3-WENO schemes at the final time. It can be observed that, at the final time, no cells are identified as troubled cells in the case of the H1-WENO and H2-WENO schemes, while only a small number of cells are marked in the H3-WENO scheme. This behavior indicates that the numerical solution becomes sufficiently smooth at later times due to the inherent numerical diffusion of the schemes, and the solutions are obtained without noticeable spurious oscillations.
	}
\end{example}

\begin{figure}[ht!]
 \centering
 \begin{tabular}{cc}
  \includegraphics[width = 7.4cm]{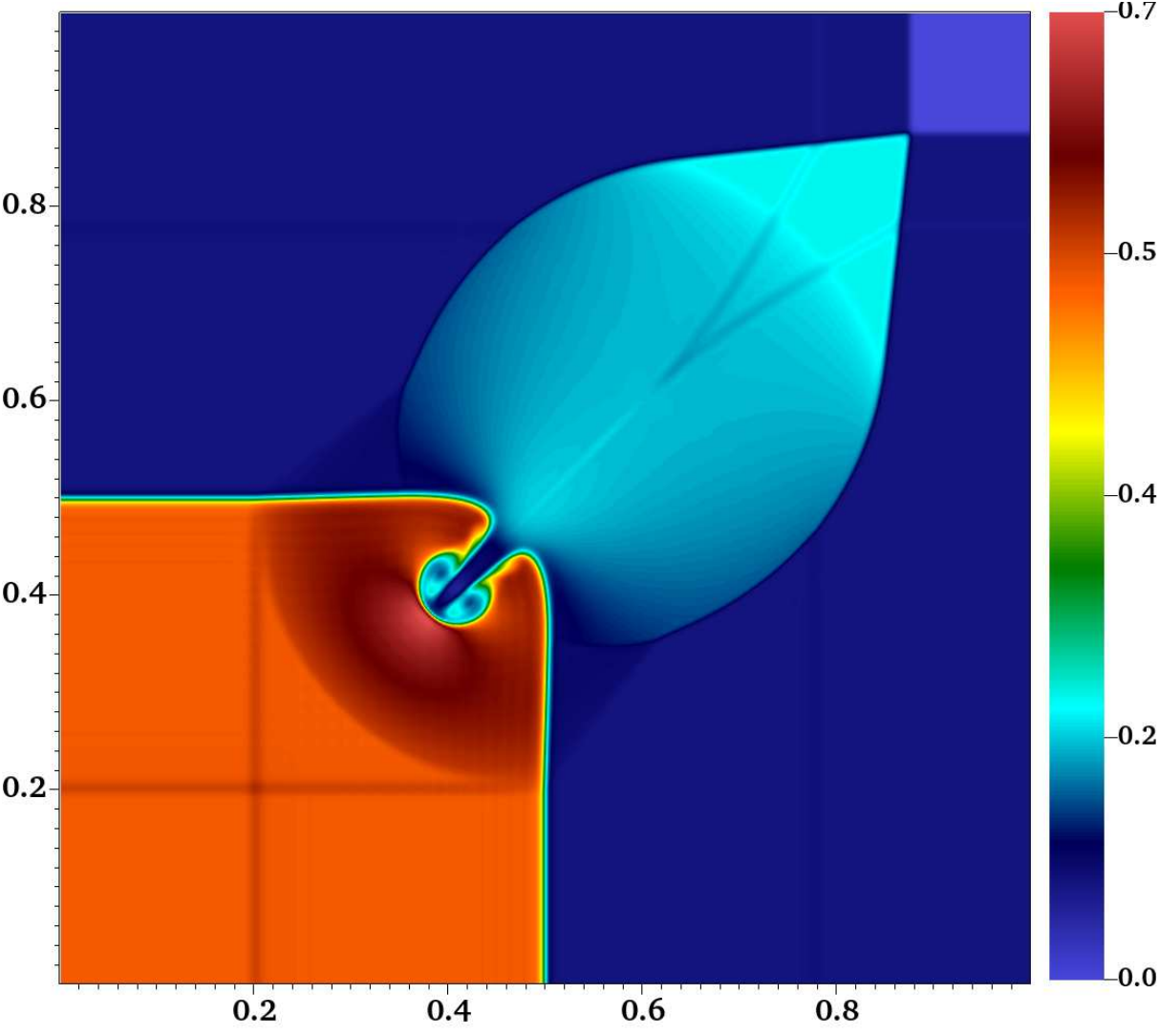}&
  \includegraphics[width = 7.4cm]{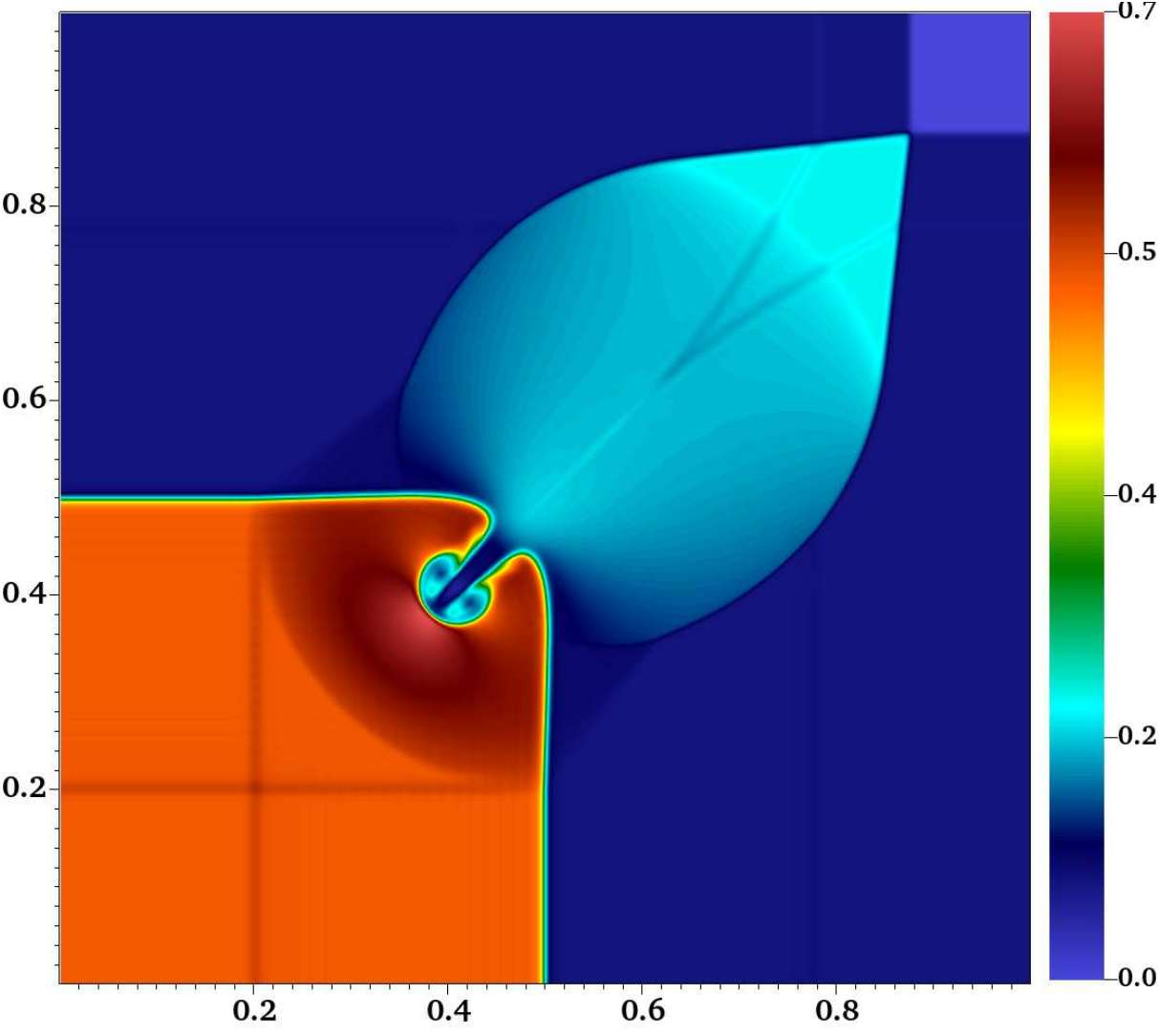}\\
  (a) WENO-AO & (b) H1-WENO \\
   \includegraphics[width = 7.4cm]{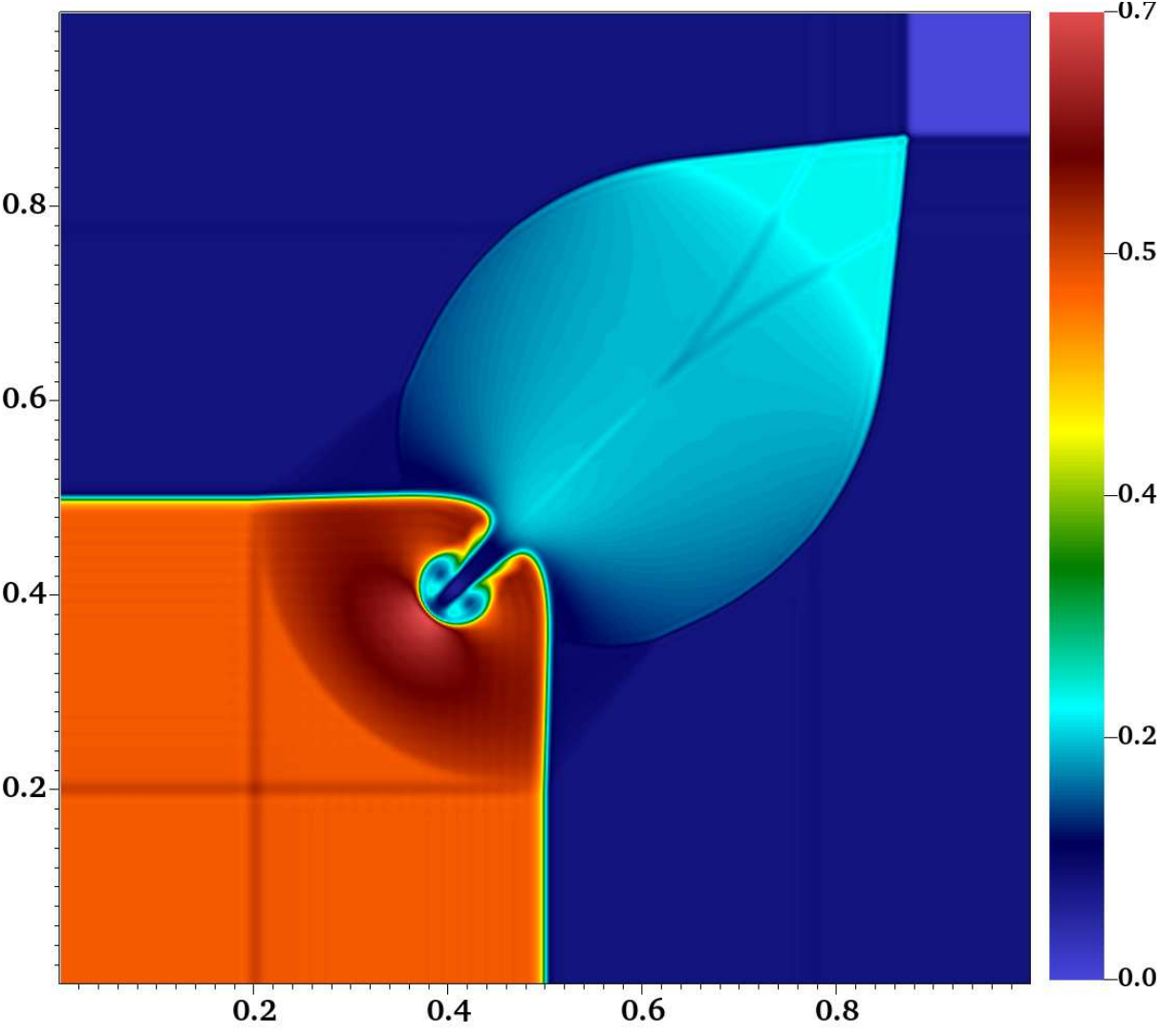}&
  \includegraphics[width = 7.4cm]{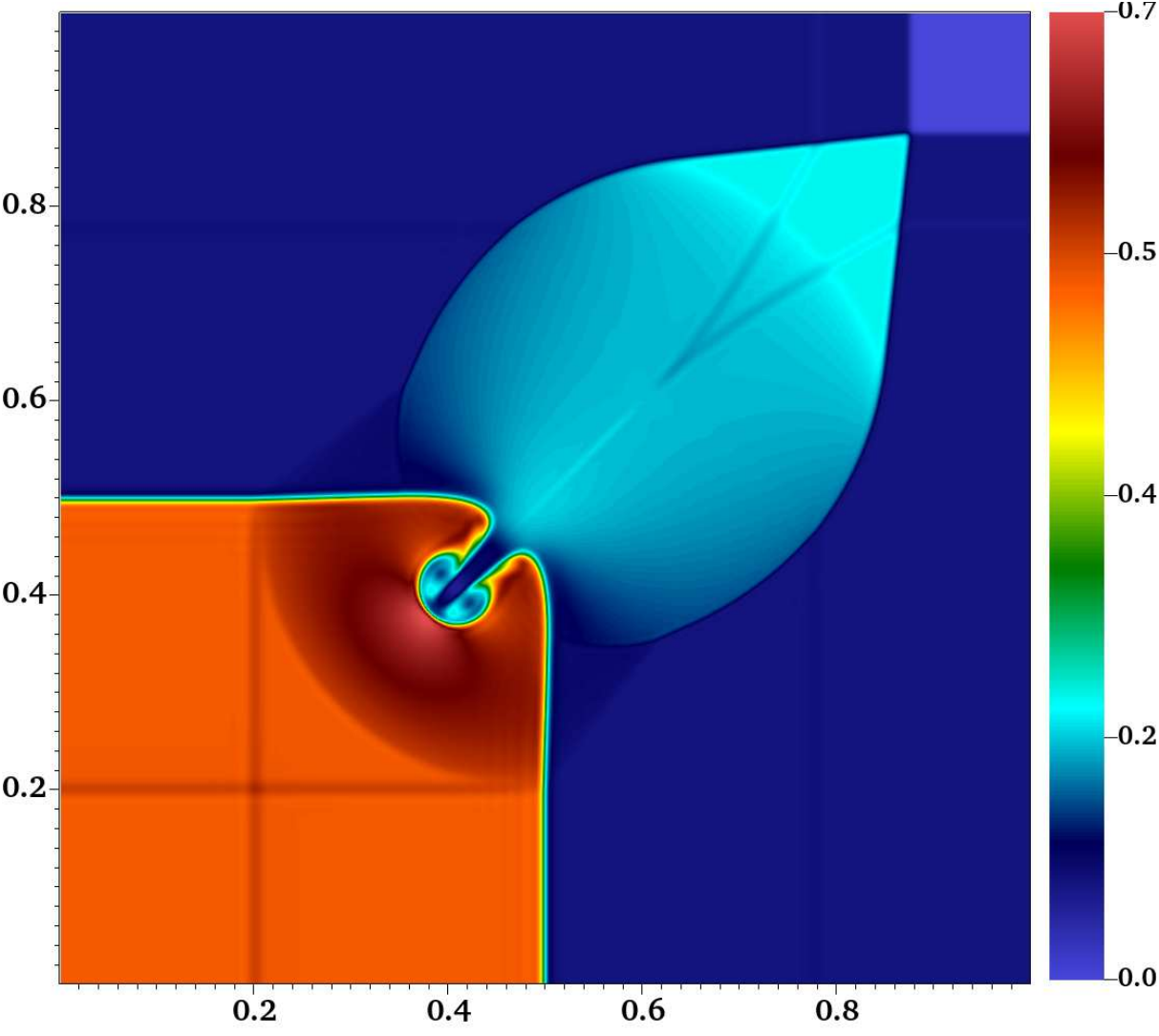}\\
  (c) H2-WENO & (d) H3-WENO \\
 \end{tabular}
 \caption{Comparison of the H1-WENO, H2-WENO, and H3-WENO schemes with the WENO-AO scheme for Example \ref{rp4} in terms of the density variable on a $400 \times 400$ computational grid at time $T=0.4$.}
 \label{Fig:rp4}
 \end{figure}
 
 \begin{figure}
 \centering
 \begin{tabular}{ccc}
   \includegraphics[width = 4.6cm]{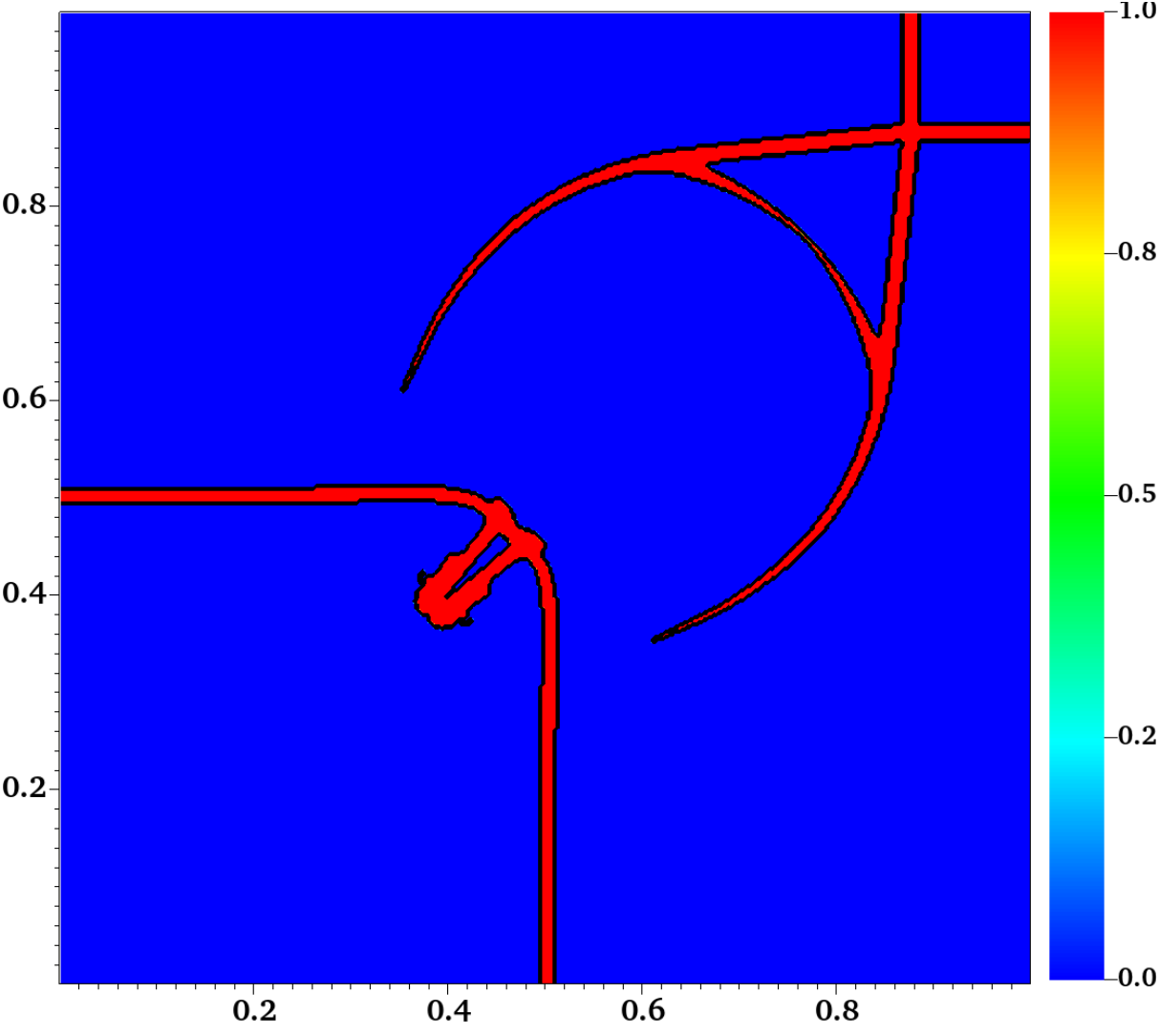} &
  \includegraphics[width = 4.6cm]{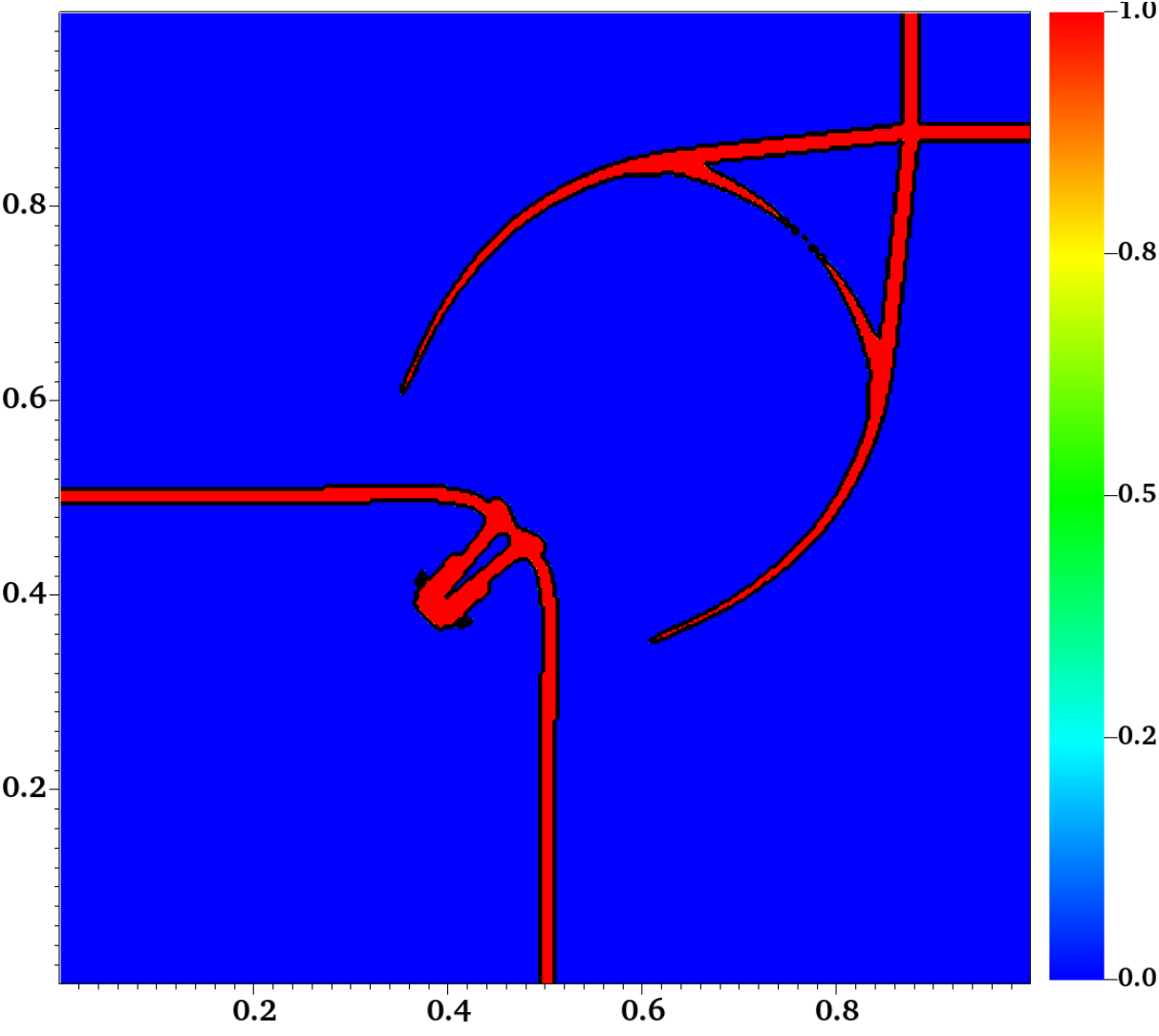}&
   \includegraphics[width = 4.60cm]{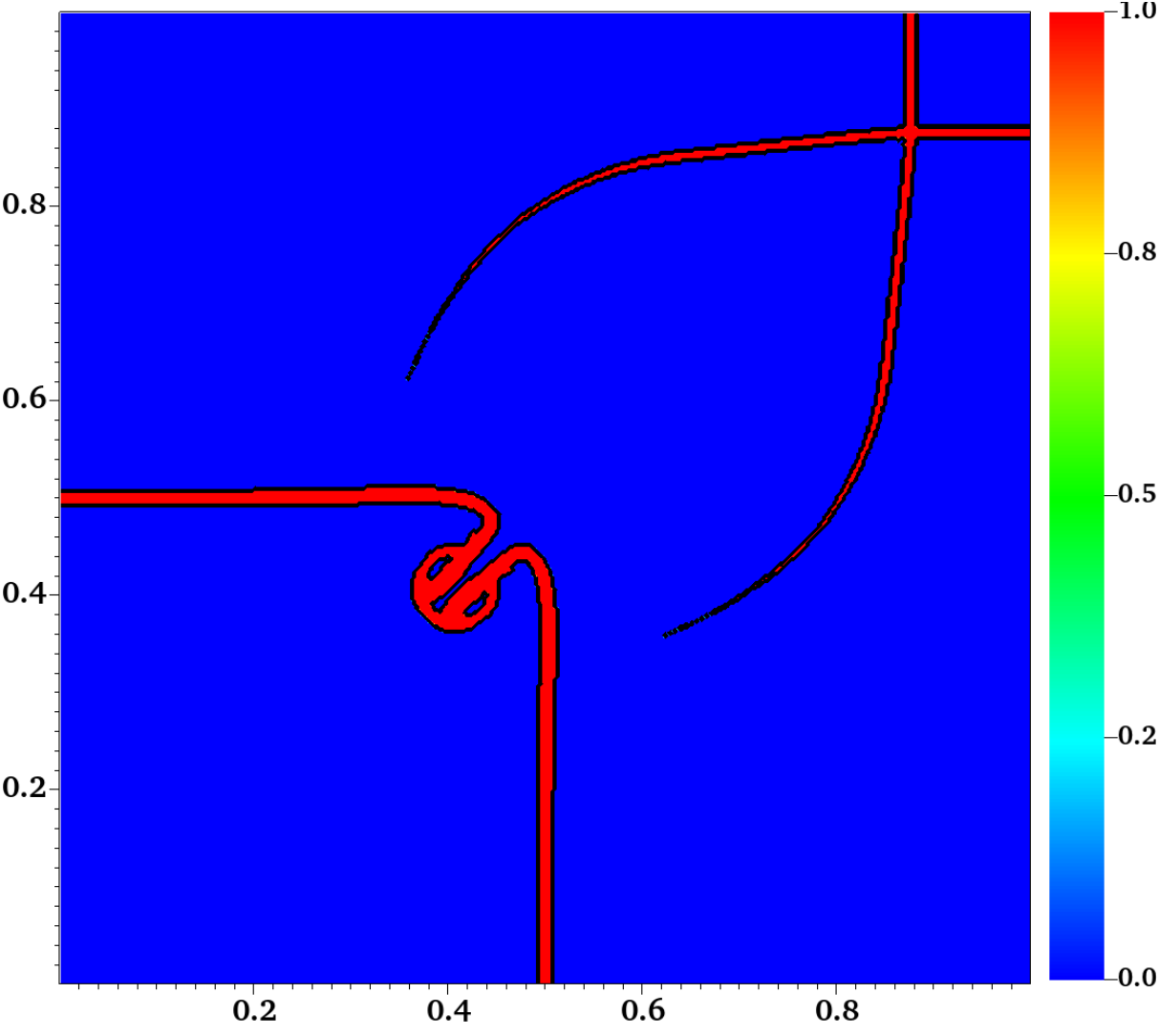}\\
  (a)  H1-WENO & (b) H2-WENO & (c) H3-WENO\\
 \end{tabular}
\caption{Plot of the troubled-cell indicators obtained using the hybrid schemes for Example \ref{rp4} at time $T=0.4$ over a grid size $400\times 400$.}
\label{Fig:rp4.tc}
\end{figure}

\begin{example}[Two-dimensional Riemann problem 4:]
	\label{rp4}	{\rm
We consider another two-dimensional Riemann problem from \cite{nunez2016xtroem} on the computational domain $[0,1]\times[0,1]$ with outflow boundary conditions imposed along all boundaries. The initial conditions are prescribed as
\begin{equation*}
\left(\rho,\,u_x,\,u_y,\,p\right)=
\begin{cases}
\left(0.035145216124503,\,0,\,0,\,0.162931056509027\right), 
& \text{if } x>0.5 \text{ and } y>0.5,\\
\left(0.1,\,0.7,\,0,\,1\right), 
& \text{if } x<0.5 \text{ and } y>0.5,\\
\left(0.5,\,0,\,0,\,1\right), 
& \text{if } x<0.5 \text{ and } y<0.5,\\
\left(0.1,\,0,\,0.7,\,1\right), 
& \text{if } x>0.5 \text{ and } y<0.5.
\end{cases}
\end{equation*}

This problem produces complicated multidimensional interactions involving shock waves and contact discontinuities, making it a demanding benchmark for evaluating the robustness and resolution capability of numerical methods.
The numerical solution is computed using the H1-WENO, H2-WENO, H3-WENO, and WENO-AO schemes on a uniform mesh consisting of $400\times400$ cells. The density plot obtained at the final time are presented in Figure \ref{Fig:rp4}. From the figure, it can be observed that all the proposed hybrid schemes are able to accurately resolve the essential flow structures and provide results that closely match those obtained using the WENO-AO scheme.
In Figure \ref{Fig:rp4.tc}, we display the corresponding troubled-cell indicator distributions for the H1-WENO, H2-WENO, and H3-WENO schemes at the final time.  The results demonstrate that the troubled-cell indicator successfully captures the non-smooth region generated during the evolution of the solution. Moreover, the H3-WENO scheme exhibits improved selectivity by detecting fewer unnecessary troubled cells in smooth regions compared to the H1-WENO and H2-WENO schemes.
}
\end{example}

\begin{figure}
 \centering
 \begin{tabular}{cc}
  \includegraphics[width = 7.2cm]{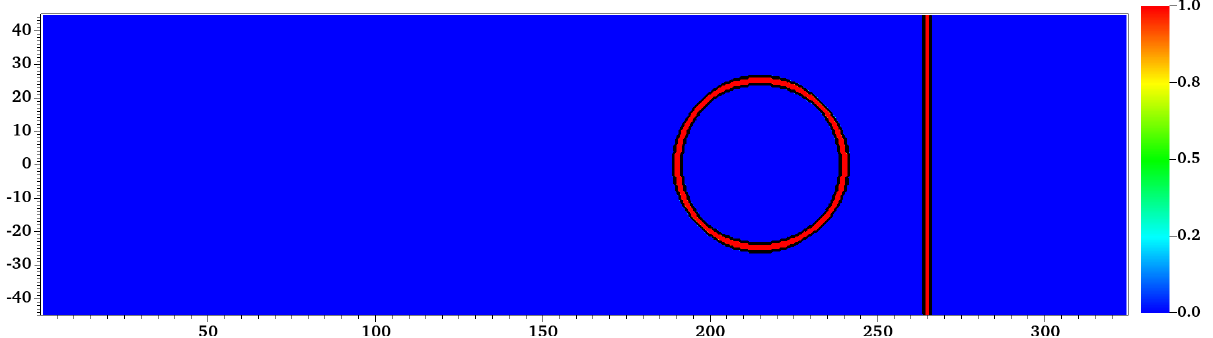} &
  \includegraphics[width = 7.2cm]{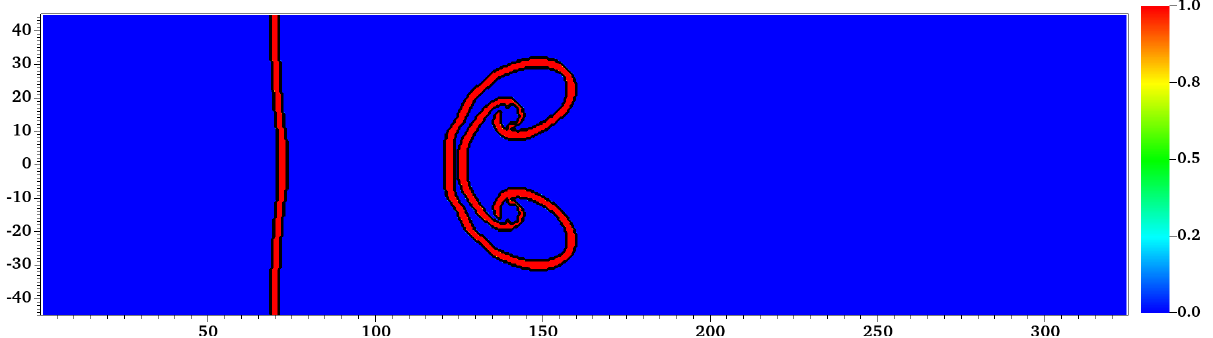}\\
  (a) Troubled cell indicator at $T=0$ & (b) H1-WENO\\
  \includegraphics[width = 7.2cm]{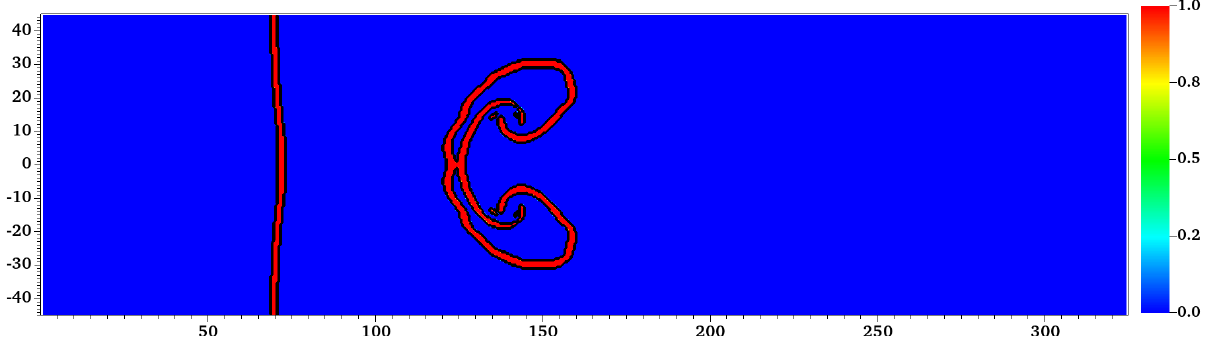} &
  \includegraphics[width = 7.2cm]{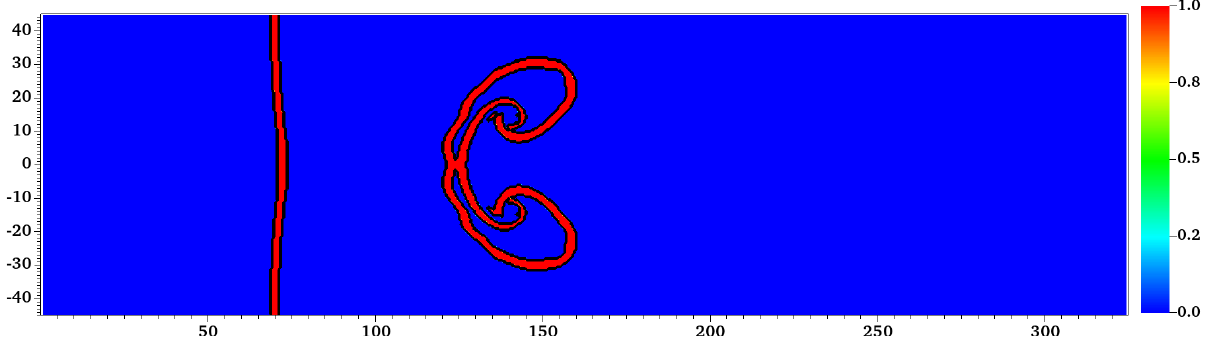}\\
  (c) H2-WENO & (d) H3-WENO\\
    \includegraphics[width = 7.2cm]{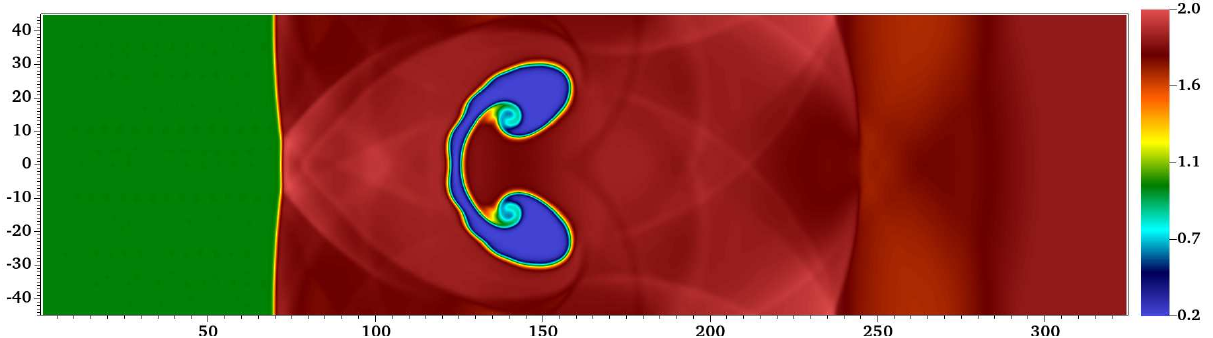} &
  \includegraphics[width = 7.2cm]{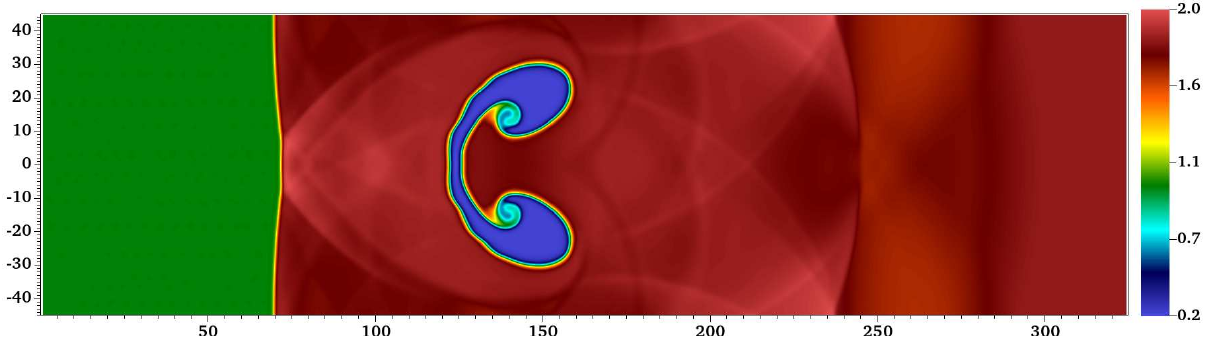}\\
  (e) WENO-AO(5,3) & (f) H1-WENO\\
  \includegraphics[width = 7.2cm]{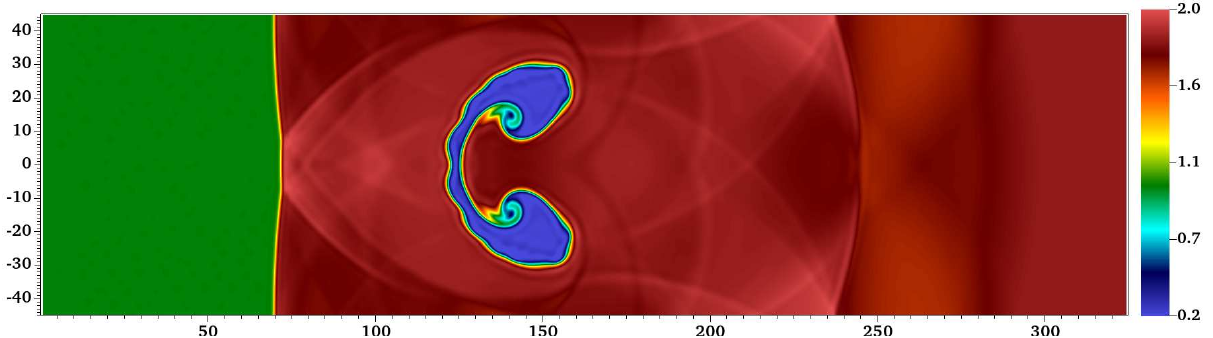} &
  \includegraphics[width = 7.2cm]{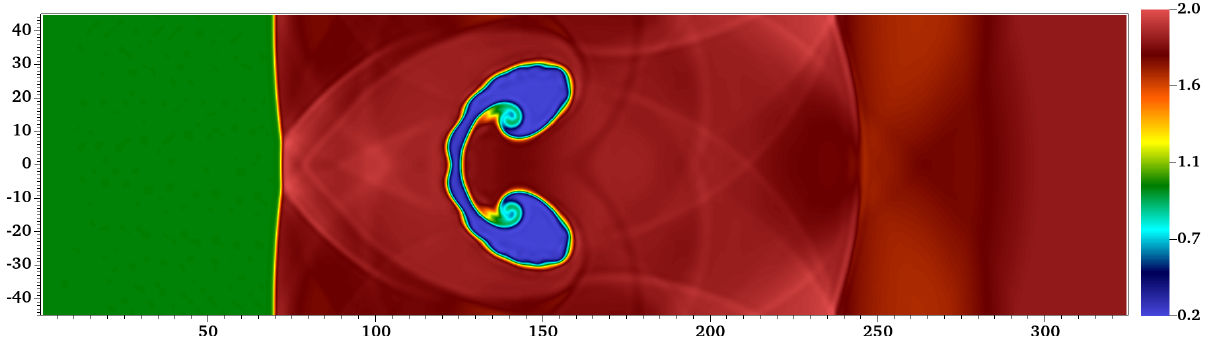}\\
  (g) H2-WENO & (h) H3-WENO
 \end{tabular}
\caption{Comparison of the solutions obtained using the H1-WENO, H2-WENO, H3-WENO, and WENO-AO(5,3) schemes at time T=450 for Example \ref{sb1} on a grid of size 180×650. The corresponding distributions of the troubled-cell indicator at the initial and final times are also shown.}
\label{Fig:sb1}
\end{figure}
\begin{figure}
 \centering
 \begin{tabular}{cc}
  \includegraphics[width = 7.2cm]{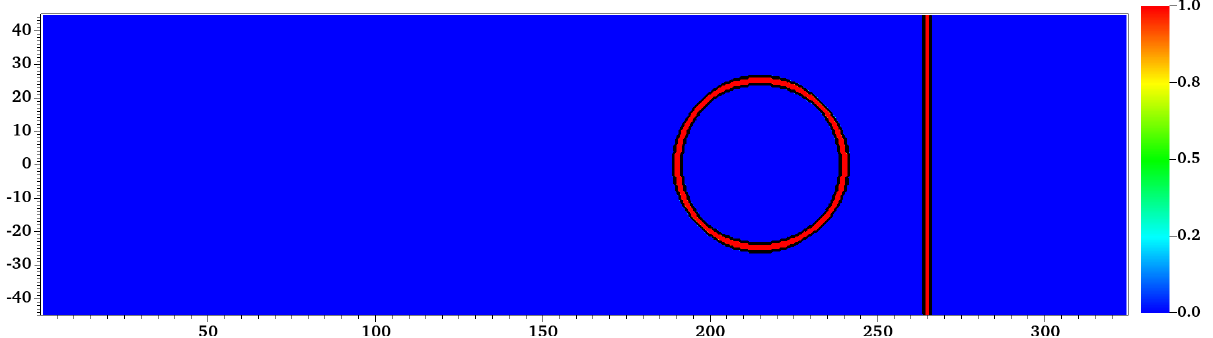} &
  \includegraphics[width = 7.2cm]{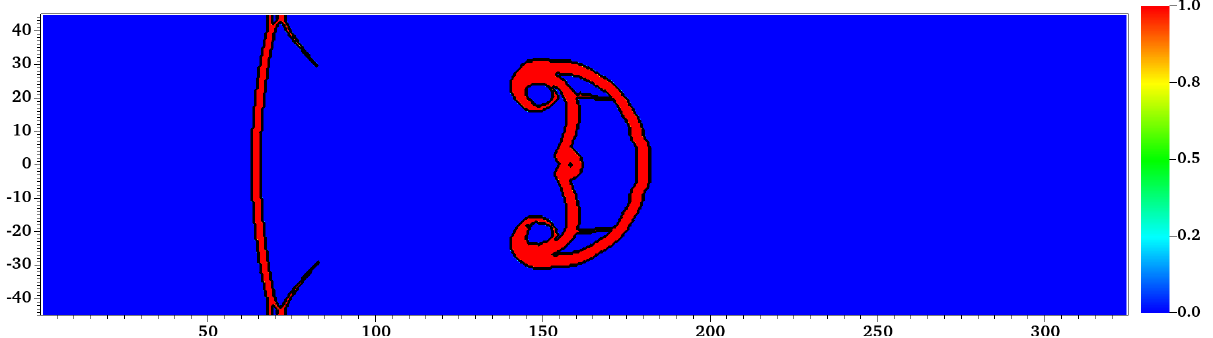}\\
  (a) Troubled cell indicator at $T=0$ & (b) H1-WENO\\
  \includegraphics[width = 7.2cm]{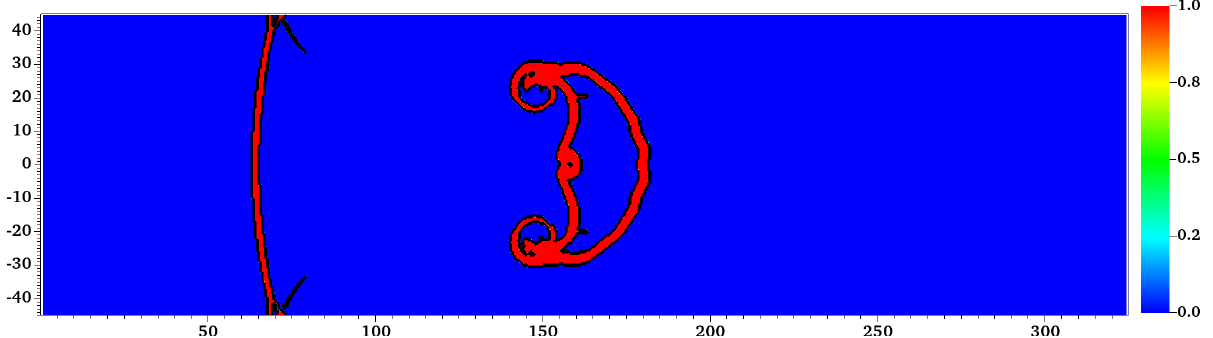} &
  \includegraphics[width = 7.2cm]{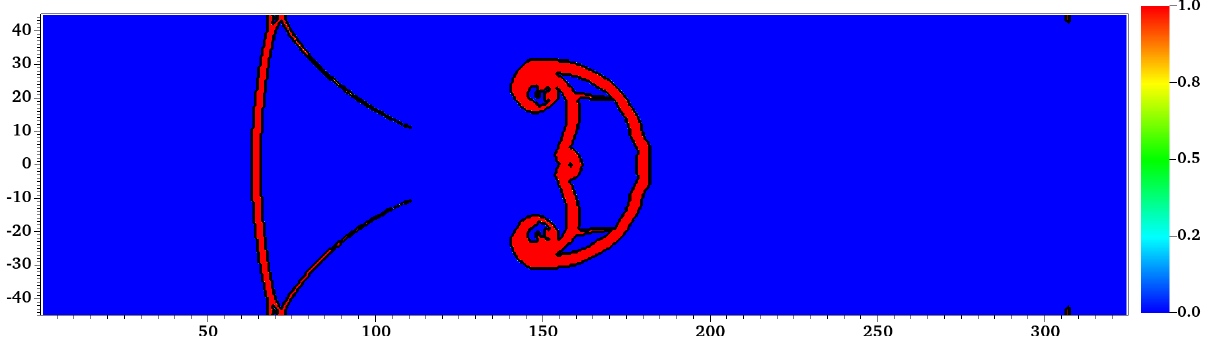}\\
 (c) H2-WENO & (d) H3-WENO\\
    \includegraphics[width = 7.2cm]{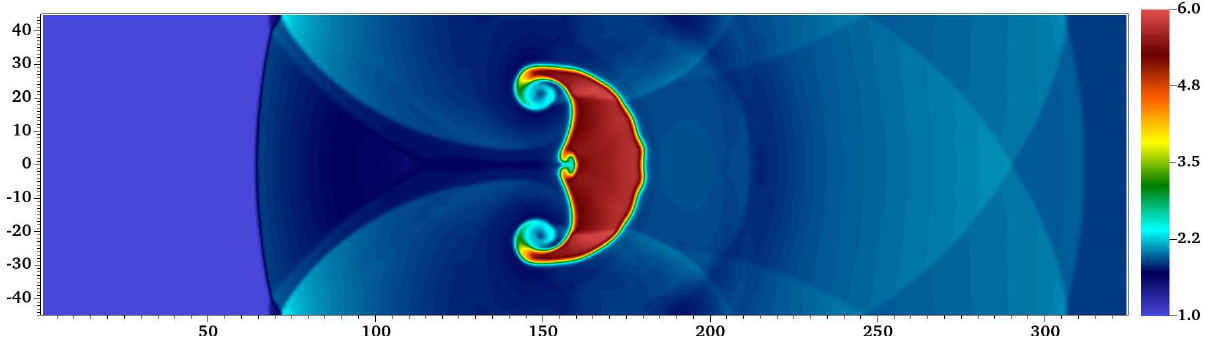} &
  \includegraphics[width = 7.2cm]{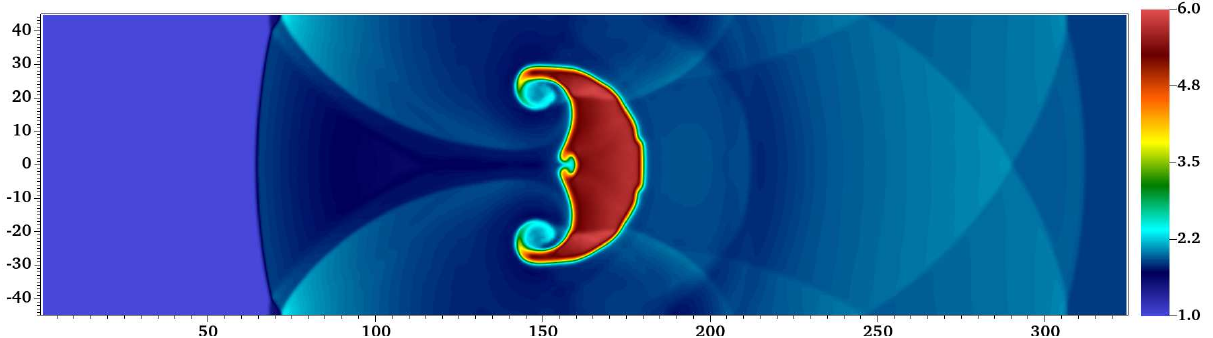}\\
  (e) WENO-AO(5,3) & (f) H1-WENO\\
  \includegraphics[width = 7.2cm]{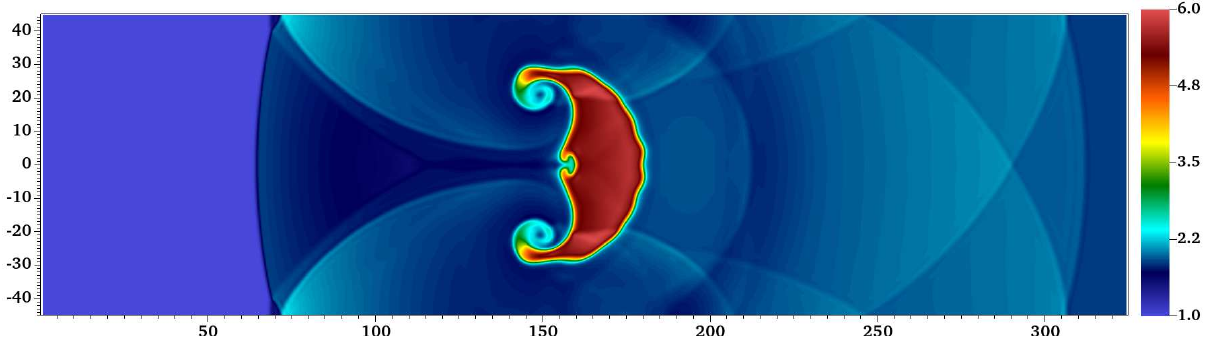} &
  \includegraphics[width = 7.2cm]{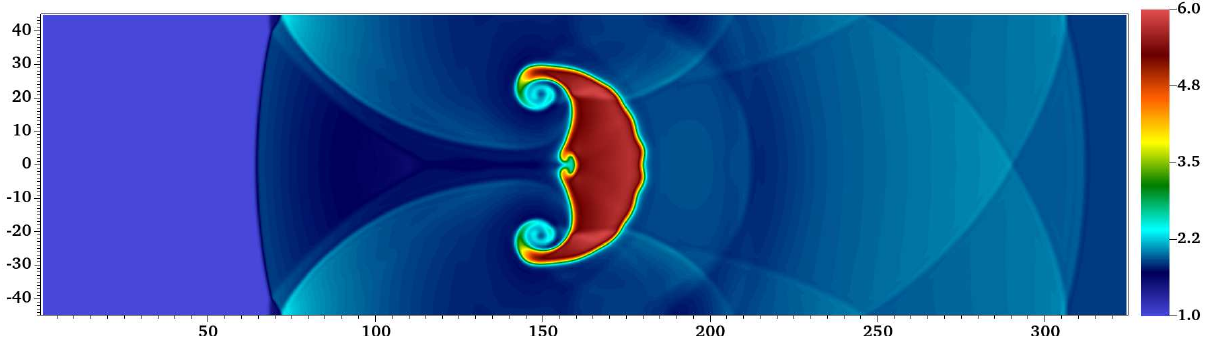}\\
 (g) H2-WENO & (h) H3-WENO
 \end{tabular}
\caption{Comparison of the solutions obtained using the H1-WENO, H2-WENO, H3-WENO, and WENO-AO(5,3) schemes at time T=450 for Example \ref{sb2} on a grid of size 180×650. The corresponding distributions of the troubled-cell indicator at the initial and final times are also shown.}
\label{Fig:sb2}
\end{figure}
\begin{example}[Shock-Bubble Interaction Problem I]
	\label{sb1}	{\rm
	In this example, we investigate the interaction between a left-moving shock wave and a low-density bubble, a standard benchmark problem discussed in \cite{he2012}. The computational domain is taken as $[0,325]\times[-45,45]$. Reflective boundary conditions are imposed along the top and bottom boundaries at $y=\pm 45$, while constant inflow states corresponding to the left and right shock states are prescribed at the boundaries $x=0$ and $x=325$, respectively.

The initial shock configuration is given by
\begin{equation*}
\left(\rho,\,u_x,\,u_y,\,p\right)=
\begin{cases}
\left(1,\,0,\,0,\,0.05\right), 
& \text{if } x<265,\\
\left(1.865225080631180,\,-0.196781107378299,\,0,\,0.15\right), 
& \text{if } x>265.
\end{cases}
\end{equation*}

A cylindrical bubble of radius $25$, centered at $(215,0)$, is initialized with the state
\begin{equation*}
\left(\rho,\,u_x,\,u_y,\,p\right)
=
\left(0.1358,\,0,\,0,\,0.05\right).
\end{equation*}

The interaction between the shock wave and the bubble generates complex flow structures involving strong compression waves, interface deformation, and vortex formation, making this problem particularly challenging for high-resolution numerical schemes.

The numerical simulations are carried out using the H1-WENO, H2-WENO, H3-WENO, and WENO-AO schemes on a computational mesh of size $650 \times 180$. In Figure \ref{Fig:sb1}, the density profiles obtained with the proposed hybrid schemes are compared with those computed using the WENO-AO scheme at the final time $t=450$. The figure also presents the corresponding troubled-cell indicator distributions at both the initial and final times for the H1-WENO, H2-WENO, and H3-WENO schemes.
The numerical results demonstrate that all the proposed hybrid schemes are able to accurately capture the major flow features and show good agreement with the high-resolution WENO-AO solution. Furthermore, the troubled-cell indicator effectively detects regions containing shocks and sharp gradients, while avoiding unnecessary detection in smooth regions, thereby reducing excessive limiting.
	}
\end{example}

\begin{example}[Shock-Bubble Interaction Problem II]
	\label{sb2}{\rm
We next consider the second shock--bubble interaction problem presented in \cite{he2012}. The overall computational setup remains identical to that of the previous example, including the computational domain, boundary conditions, shock configuration, and numerical parameters. The primary difference in this case is that the bubble is initialized with a heavier density state, leading to a substantially different interaction pattern between the incident shock wave and the bubble interface.

The initial state inside the bubble is prescribed as
\begin{equation*}
\left(\rho,\,u_x,\,u_y,\,p\right)
=
\left(3.1538,\,0,\,0,\,0.05\right).
\end{equation*}
Due to the higher density ratio between the surrounding medium and the bubble, the shock interaction generates more pronounced interface deformation and complex wave structures. This problem therefore provides an additional test for assessing the robustness and shock-capturing capability of the proposed numerical schemes.

The numerical simulations are carried out using the H1-WENO, H2-WENO, H3-WENO, and WENO-AO schemes on the same computational mesh as used in the previous example. The resulting density distributions at the final time are displayed in Figure \ref{Fig:sb2}. The obtained solutions show that all the schemes are capable of resolving the essential flow dynamics and capture the complicated shock--interface interactions with good agreement when compared with the  WENO-AO solution.
The corresponding troubled-cell indicator plots for the H1-WENO, H2-WENO, and H3-WENO schemes are shown in Figure \ref{Fig:sb2}. The indicator accurately identifies the regions containing strong discontinuities and sharp gradients produced during the interaction process.
	}
\end{example}

\subsection{Computational cost comparison}
In the previous section, we demonstrated that all the proposed schemes are capable of capturing and reproducing results consistent with the original WENO-AO scheme. Moreover, for smooth test cases, the proposed schemes maintain the desired accuracy while requiring lower computational cost. In particular, the hybrid schemes preserve fifth-order accuracy and achieve comparable error levels in reduced computational time compared to the original scheme, as illustrated in Figure~\ref{fig:Eg5}.
In this subsection, we analyze the computational efficiency of the hybrid schemes for test cases involving discontinuities and compare their performance with the WENO-AO scheme. The comparison is carried out in terms of the speed-up factor relative to the base scheme. The speed-up factor is defined as the ratio of the CPU time required to compute the solution using the proposed scheme to the CPU time required by the WENO-AO scheme. A speed-up factor less than one indicates that the proposed scheme is efficient than the WENO-AO scheme, whereas a value greater than one indicates slower performance. The unit speed up factor assigned to the WENO-AO scheme. 
In addition, we investigate the effect of varying the constant used in the hybridization criterion on the efficiency of the scheme and on the percentage of cells that are marked for the hybrid reconstruction. This study is carried out for representative one-dimensional and two-dimensional test problems in order to better understand the influence of this parameter on the overall performance of the proposed hybrid schemes.
\begin{table}[]
    \centering
 \begin{tabular}{|c|c|c|c|c|}
\hline
Test Cases & WENO-AO & H1-WENO & H2-WENO & H3-WENO \\ [2pt]
\hline
Example \ref{test3} & 1.0000 & 0.6676 & 0.5777 & 0.9307 \\
Example \ref{test4} & 1.0000 & 0.7919 & 0.7145 & 0.9001 \\
Example \ref{test5} & 1.0000 & 0.7029 & 0.6196 & 0.9178 \\
Example \ref{test6} & 1.0000 & 0.6548 & 0.5525 & 0.8956 \\
Example \ref{test7} & 1.0000 & 0.8643 & 0.8252 & 0.9225 \\
Example \ref{test8} & 1.0000 & 0.7363 & 0.6576 & 0.9170 \\
\hline
\end{tabular}

    \caption{Comparison of H1-WENO, H2-WENO, and H3-WENO schemes with WENO-AO scheme in term of speed-up factor for 1D test cases. }
    \label{tab:cpu1d}
\end{table}

\begin{table}[h!]
\centering
\begin{tabular}{|c|c|c|c|c|}
\hline
Test Cases & WENO-AO & H1-WENO & H2-WENO & H3-WENO \\ [2pt]
\hline
Example \ref{rp1} & 1.0000 & 0.6050 & 0.5166 & 0.9685 \\
Example \ref{rp2} & 1.0000 & 0.5644 & 0.4694 & 0.9788 \\
Example \ref{rp3} & 1.0000 & 0.5248 & 0.3985 & 0.9669 \\
Example \ref{rp4} & 1.0000 & 0.5382 & 0.4494 & 0.9817 \\
Example \ref{sb1} & 1.0000 & 0.5264 & 0.4145 & 0.9392 \\
Example \ref{sb2} & 1.0000 & 0.5488 & 0.4135 & 0.9775 \\
\hline
\end{tabular}
\caption{Comparison of H1-WENO, H2-WENO, and H3-WENO schemes with WENO-AO scheme in term of speed-up factor for 2D test cases. }
    \label{tab:cpu2d}
\end{table}
 \begin{figure}[ht]
 \centering
 \begin{tabular}{cc}
  \includegraphics[width = 7.4cm]{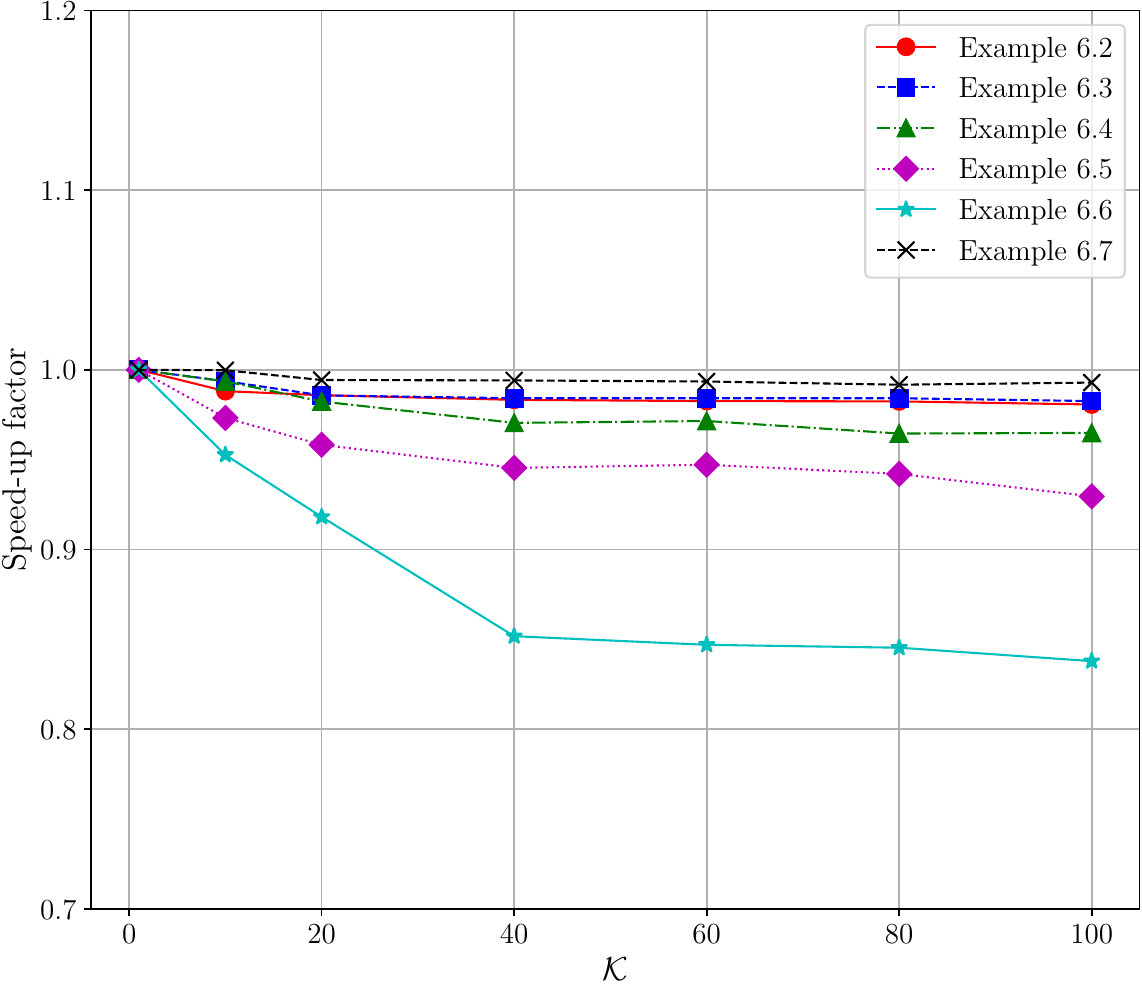}&
  \includegraphics[width = 7.4cm]{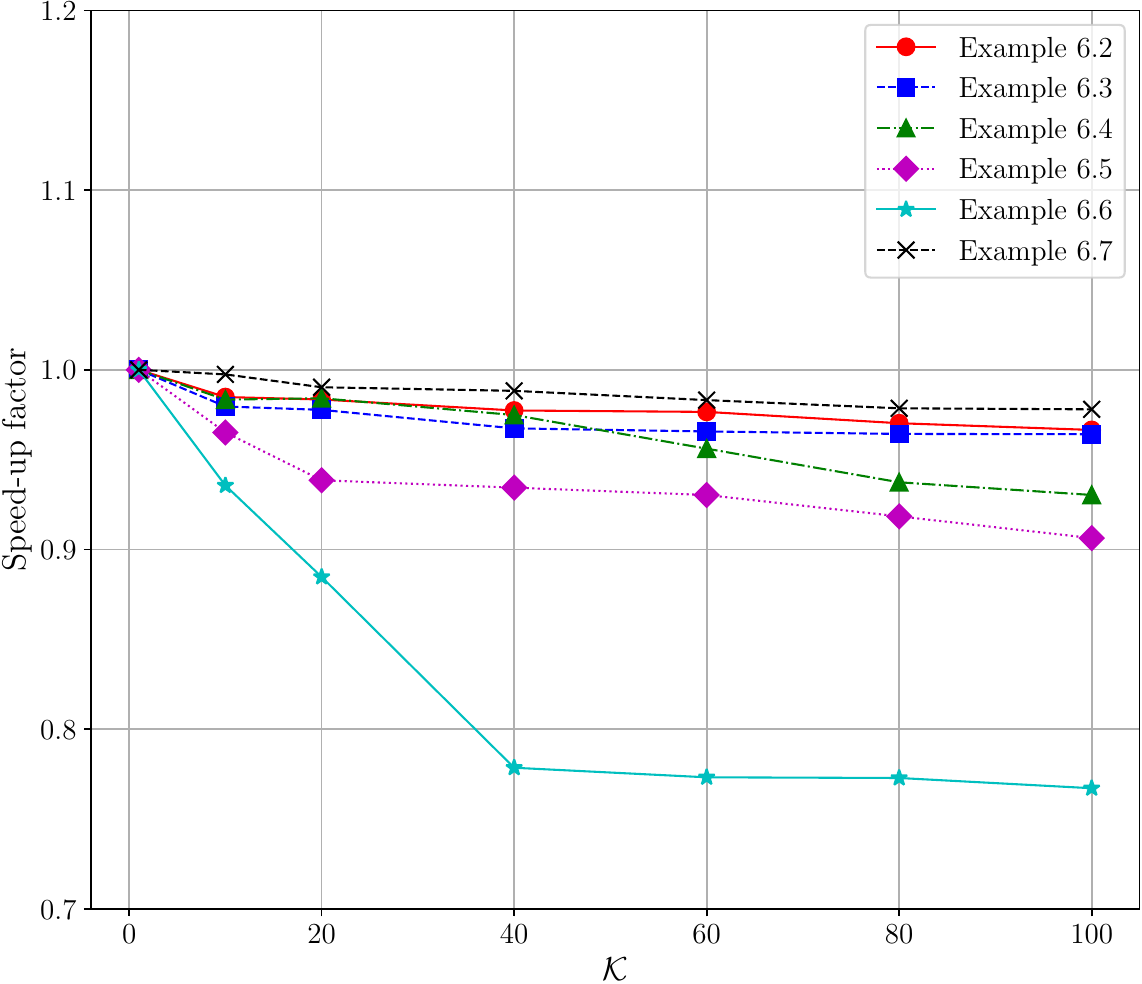}\\
  (a) H1-WENO & (b) H2-WENO \\
   \includegraphics[width = 7.4cm]{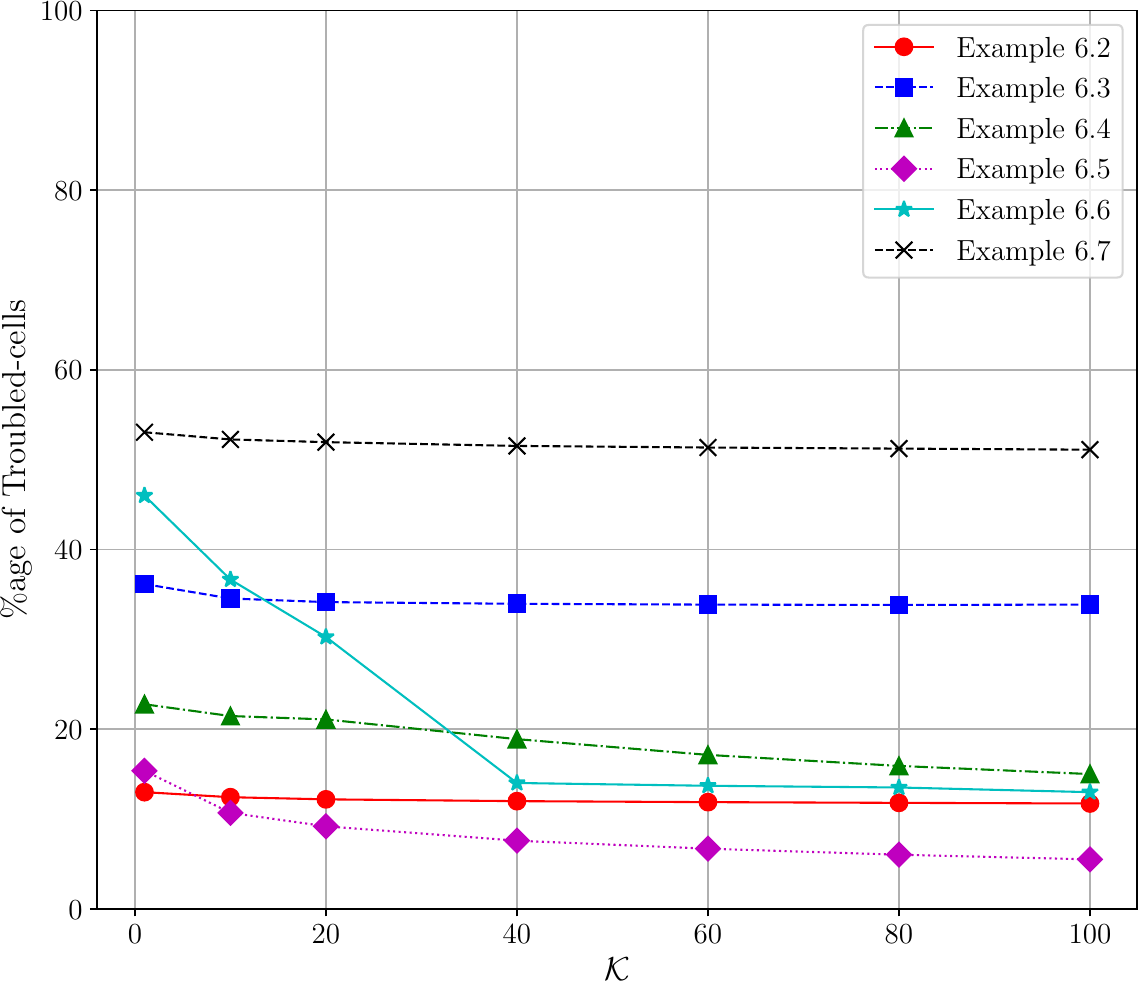}&
  \includegraphics[width = 7.4cm]{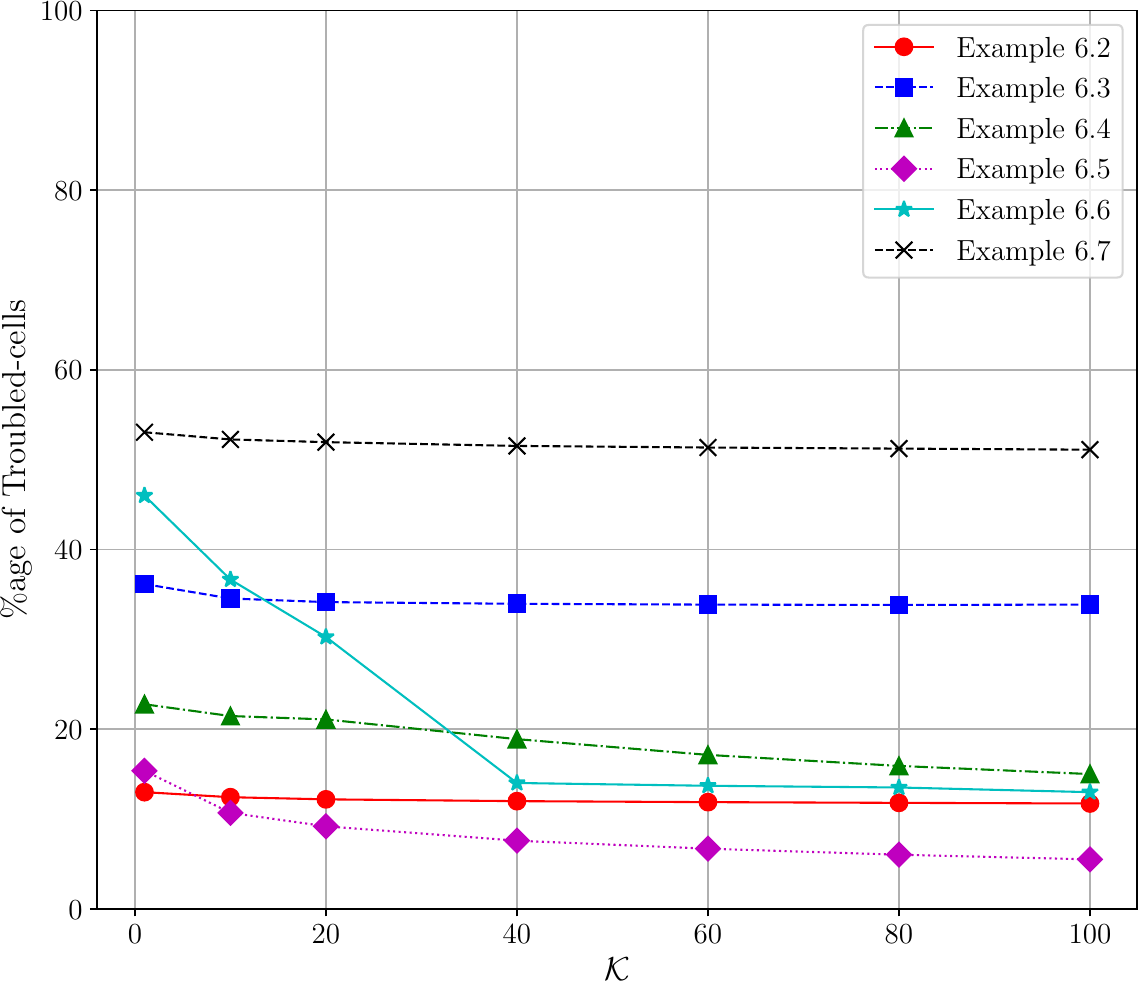}\\
  (c) H1-WENO & (d) H2-WENO \\
 \end{tabular}
 \caption{Comparison of the speed-up factor for different values of $\mathcal{K}$ in 1D test cases: (a) H1-WENO scheme and (b) H2-WENO scheme. Percentage of troubled cells identified during the simulations for different values of $\mathcal{K}$ in 1D test cases: (c) H1-WENO scheme and (d) H2-WENO scheme.}
 \label{Fig:CPU1D}
 \end{figure}
\begin{figure}[ht]
 \centering
 \begin{tabular}{cc}
  \includegraphics[width = 7.4cm]{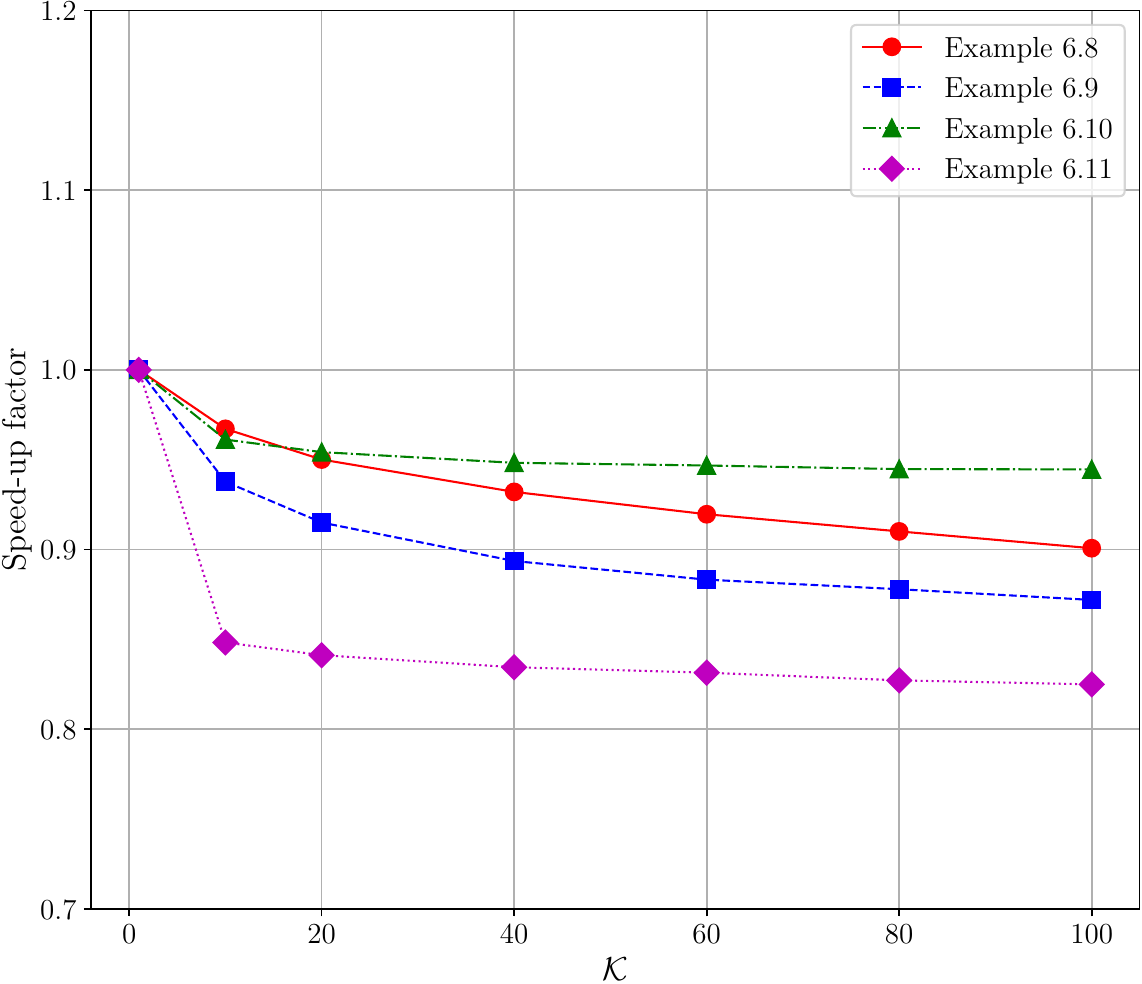}&
  \includegraphics[width = 7.4cm]{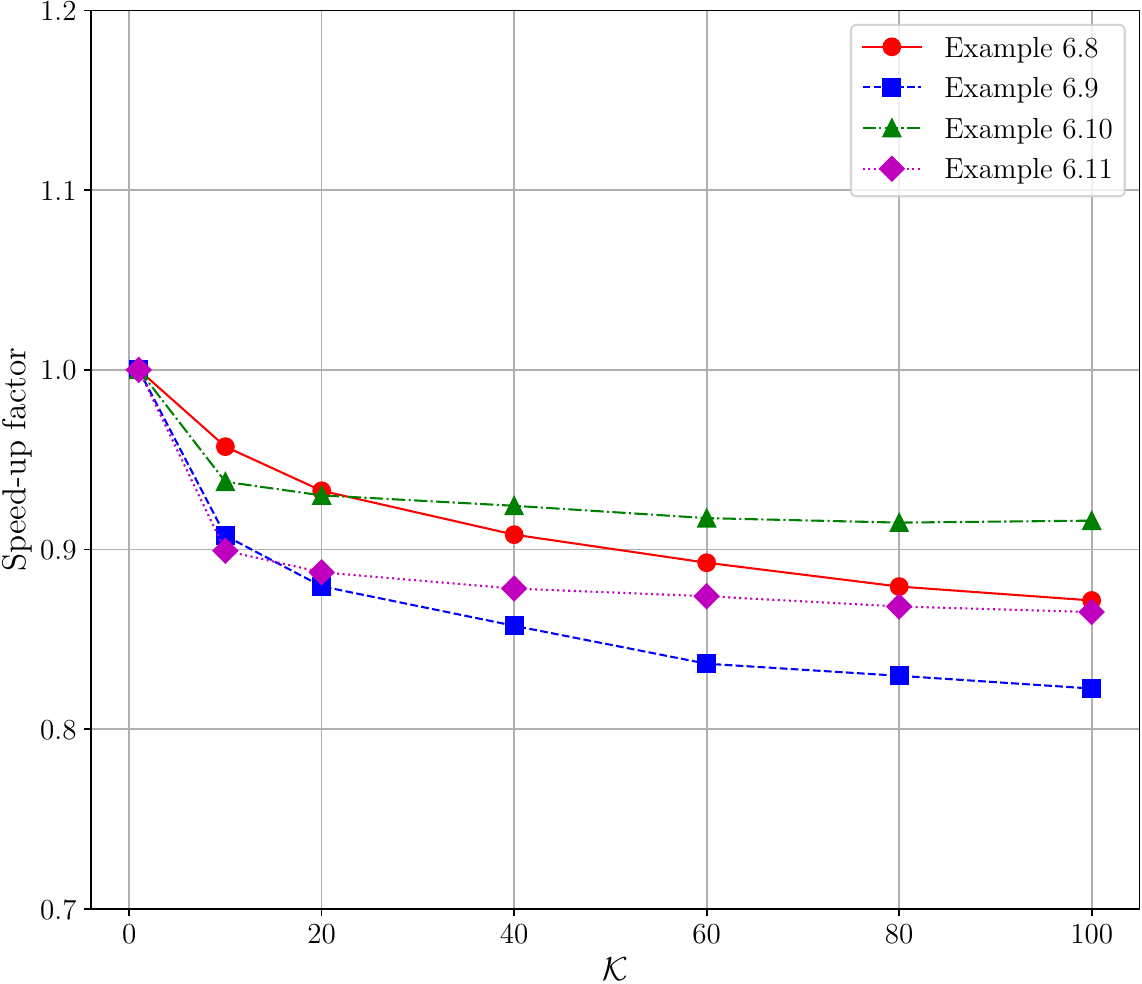}\\
  (a) H1-WENO & (b) H2-WENO \\
   \includegraphics[width = 7.4cm]{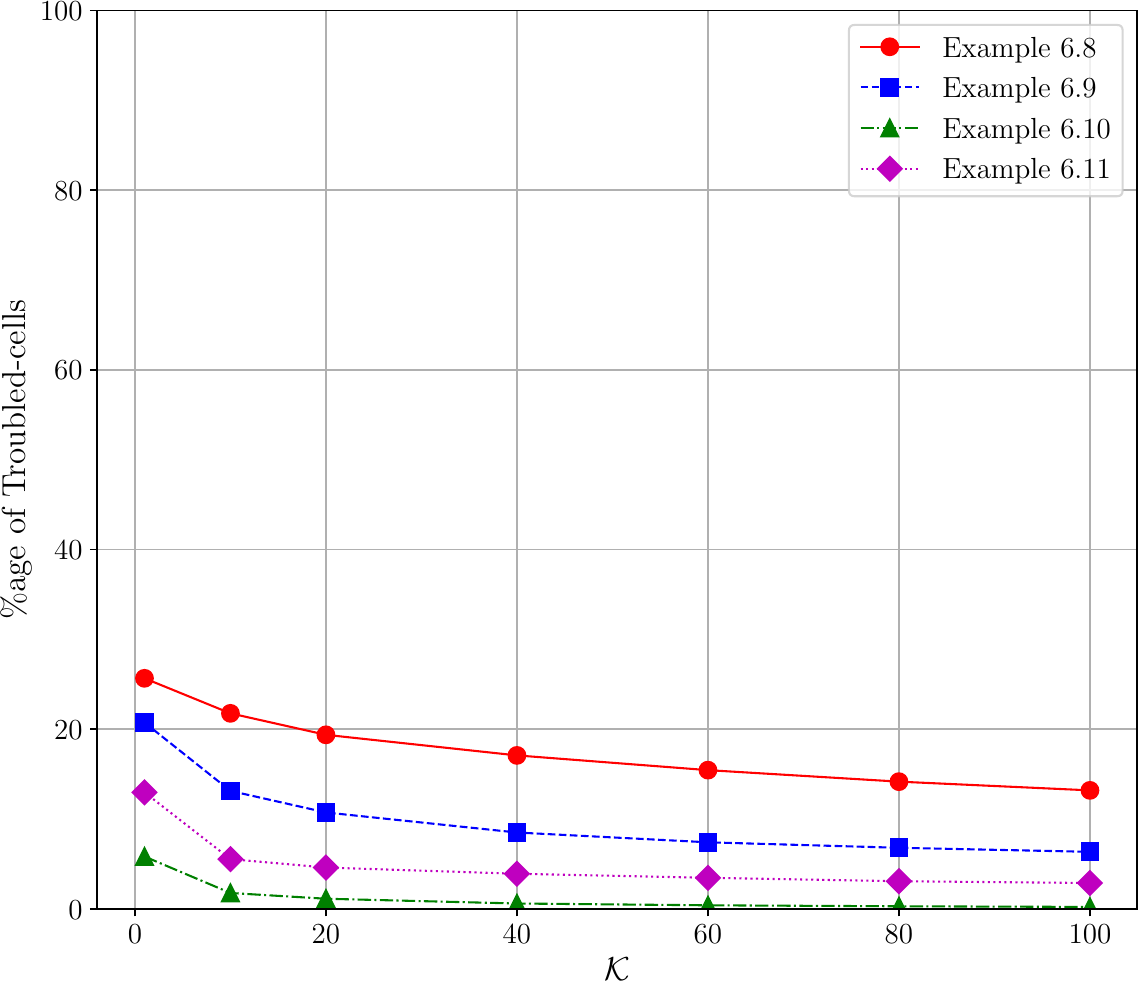}&
  \includegraphics[width = 7.4cm]{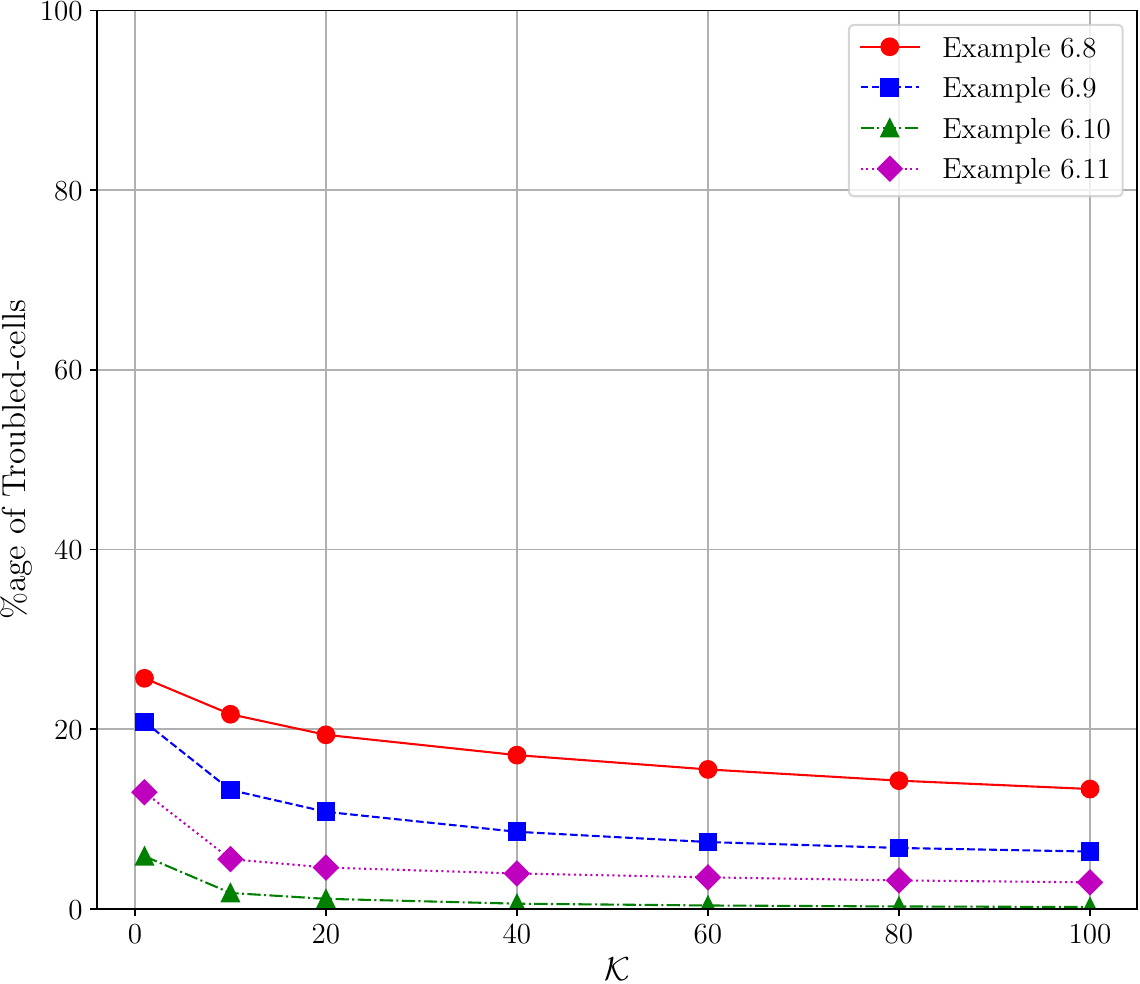}\\
  (c) H1-WENO & (d) H2-WENO \\
 \end{tabular}
 \caption{Comparison of the speed-up factor for different values of $\mathcal{K}$ in 2D test cases: (a) H1-WENO scheme and (b) H2-WENO scheme. Percentage of troubled cells identified during the simulations for different values of $\mathcal{K}$ in 2D test cases: (c) H1-WENO scheme and (d) H2-WENO scheme.}
 \label{Fig:CPU2D}
 \end{figure}

In Tables \ref{tab:cpu1d} and \ref{tab:cpu2d}, we present the speed-up factors obtained using various schemes for the 1D and 2D test cases, respectively. It can be observed that, in both cases, the H1-WENO and H2-WENO schemes, which avoid the computation of eigenvectors in smooth regions, are significantly more efficient than the H3-WENO and WENO-AO schemes. Specifically, the H1-WENO scheme achieves a speed-up of approximately 14--35\% in the 1D test case and 40--45\% in the 2D test case compared to the WENO-AO scheme. Similarly, the H2-WENO scheme shows an improvement of about 28--45\% in 1D and 50--60\% in 2D over WENO-AO. In contrast, the H3-WENO scheme provides a relatively modest gain of around 8--10\% in 1D and 3--4\% in 2D.
Furthermore, the WENO-AO scheme, which employs characteristic-wise WENO reconstruction, is computationally more expensive due to the repeated evaluation of eigenvectors and nonlinear weights. Notably, the computation of eigenvectors is significantly more costly than the evaluation of nonlinear weights, which further contributes to the reduced efficiency of schemes relying on characteristic decomposition.

 \begin{figure}[ht!]
 \centering
 \begin{tabular}{ccc}
   \includegraphics[width = 4.6cm]{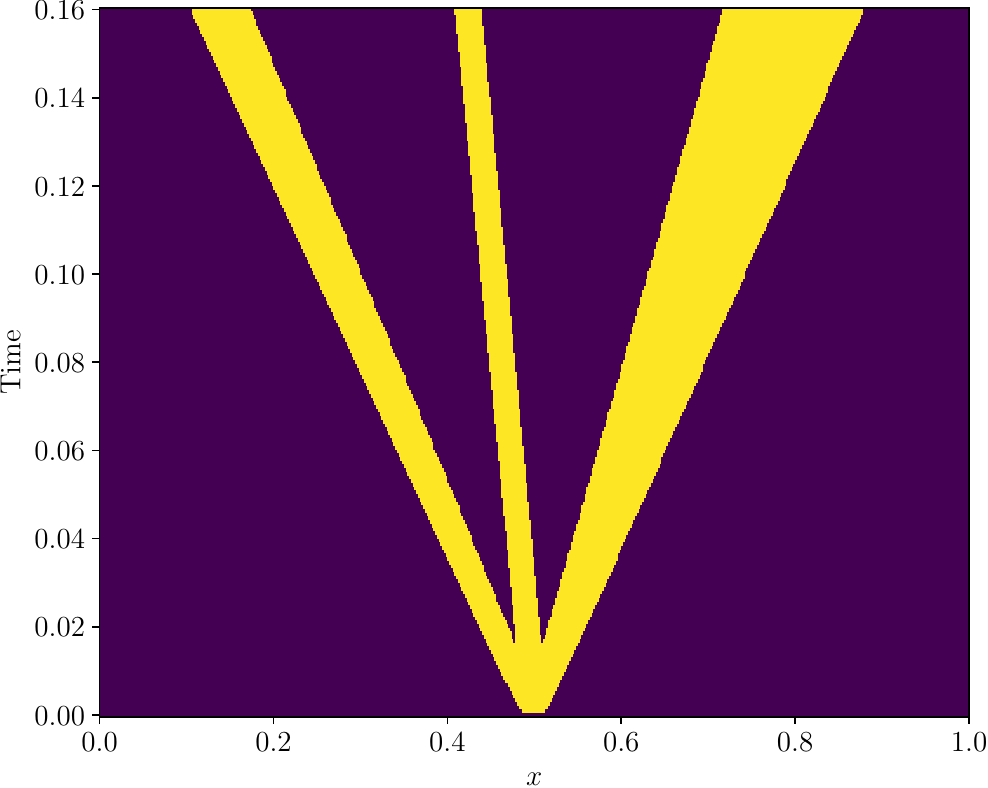} &
  \includegraphics[width = 4.6cm]{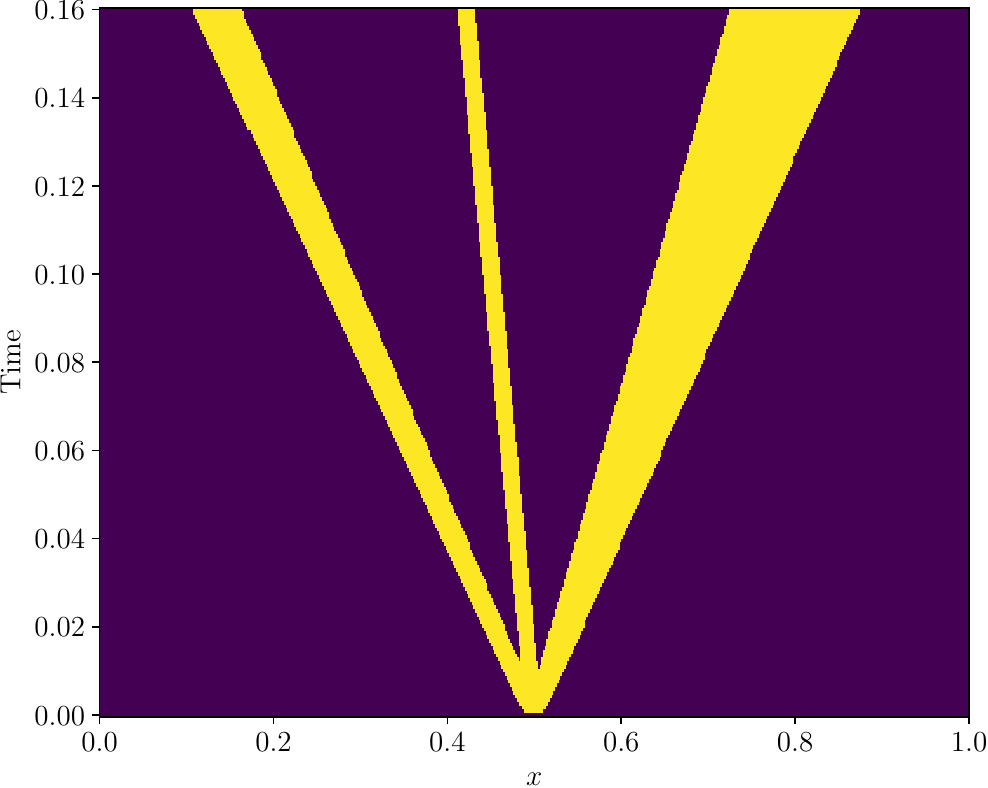}&
   \includegraphics[width = 4.60cm]{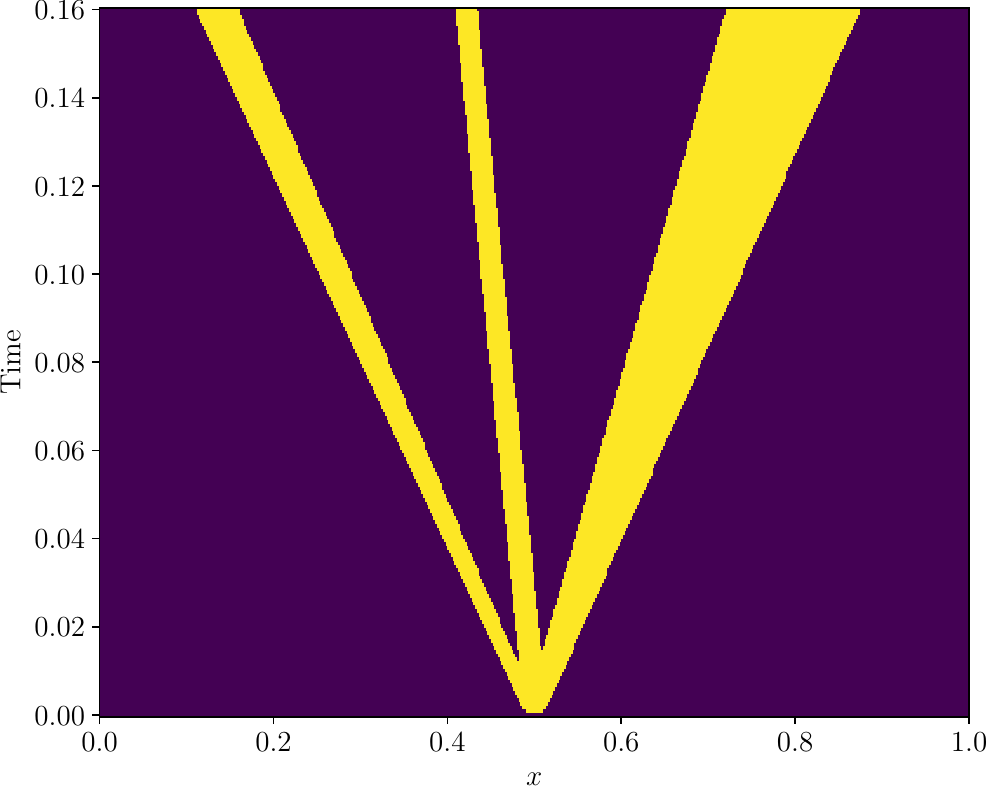}\\
  (a)  H1-WENO $(\mathcal{K}=1)$ & (b)  H1-WENO $(\mathcal{K}=100)$ &(c)   H2-WENO $(\mathcal{K}=1)$\\
     \includegraphics[width = 4.6cm]{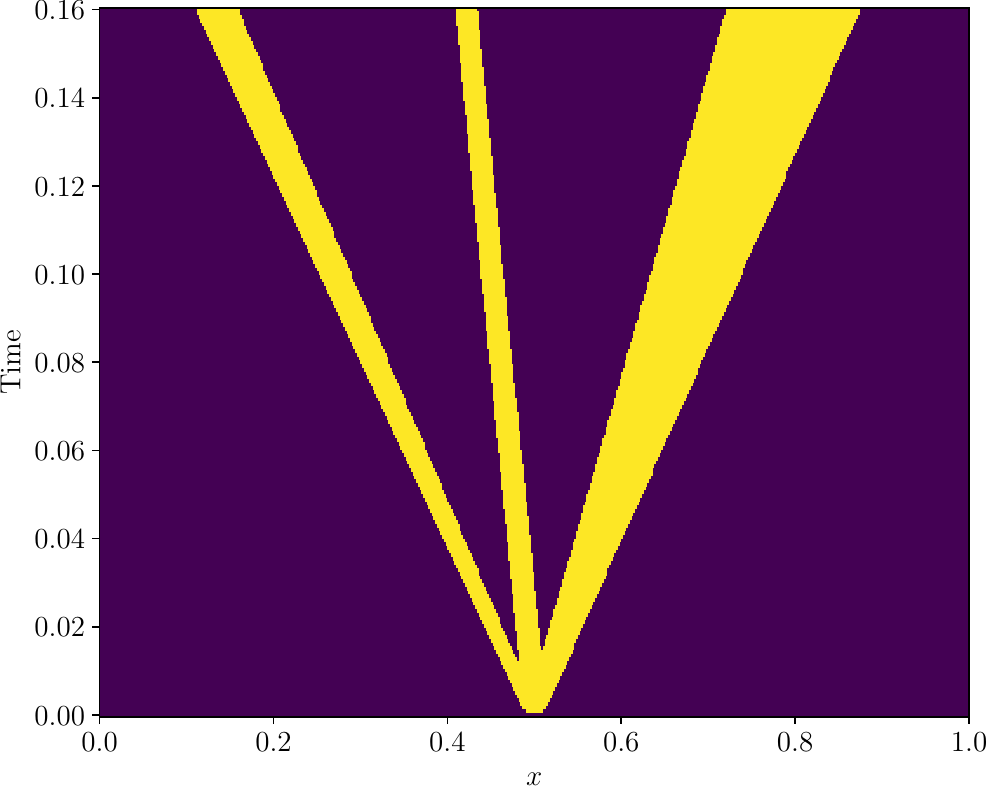} &
  \includegraphics[width = 4.6cm]{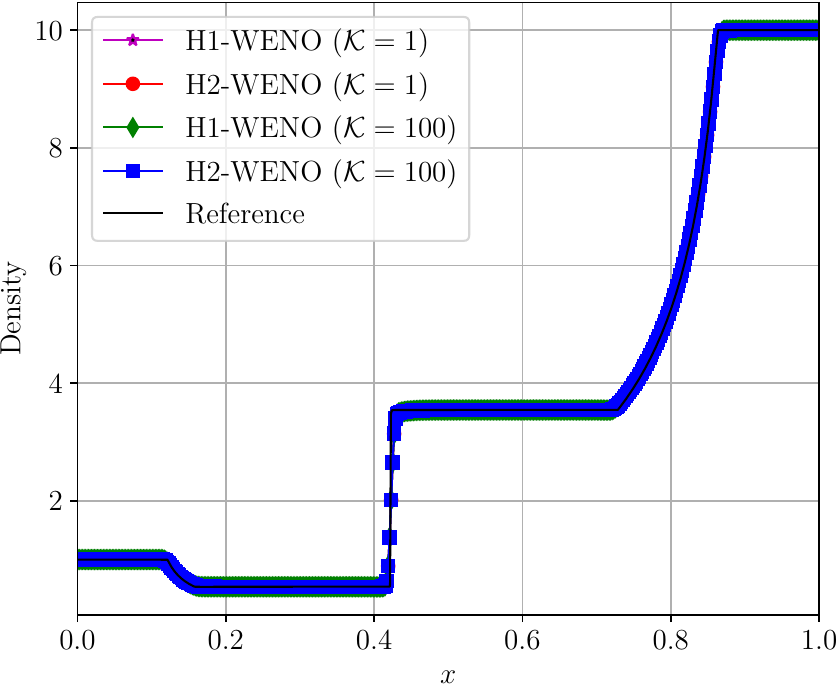}&
   \includegraphics[width = 4.60cm]{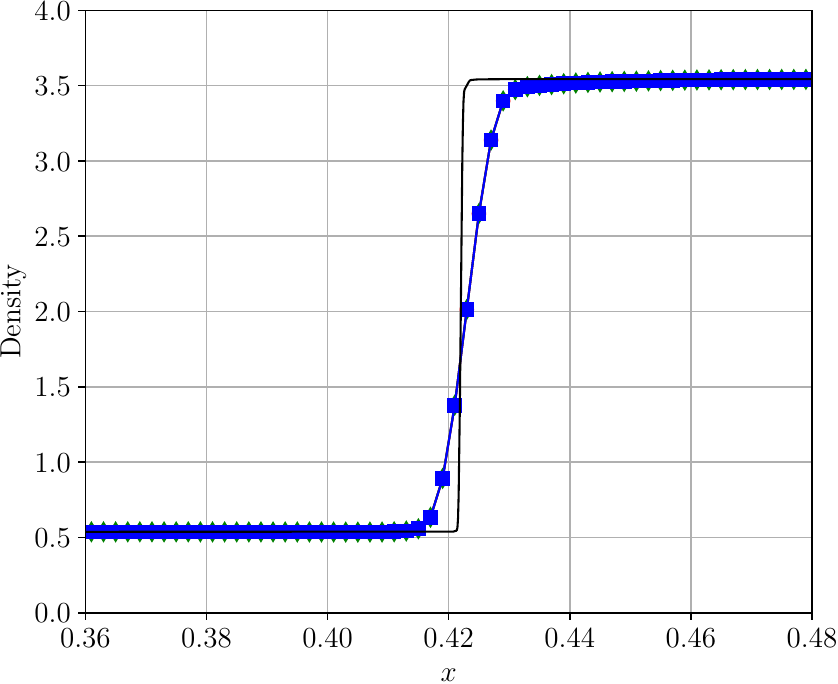}\\
  (d)  H2-WENO $(\mathcal{K}=100)$ & (e) Density& (f) Zoom of (e). 
 \end{tabular}
\caption{Comparison of troubled-cell indicator plots over $x-t$ plane for $\mathcal{K}=1$ and $100$ obtained using H1-WENO and H2-WENO schemes, along with a comparison of the corresponding density solution profiles computed on a 500-point grid. }
\label{Fig:RP1-H1H2}
\end{figure}

 \begin{figure}[ht!]
 \centering
 \begin{tabular}{ccc}
   \includegraphics[width = 4.6cm]{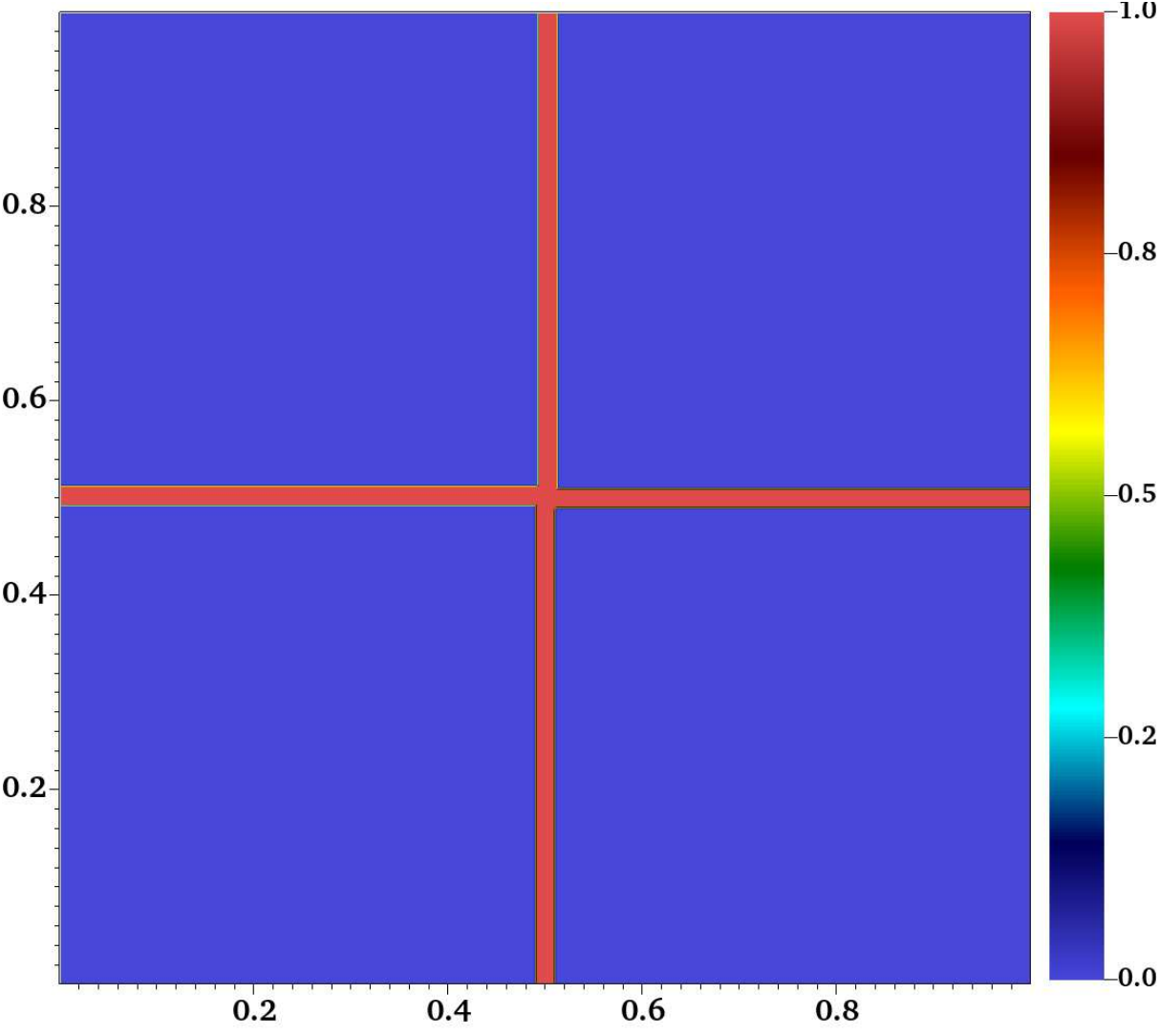} &
  \includegraphics[width = 4.6cm]{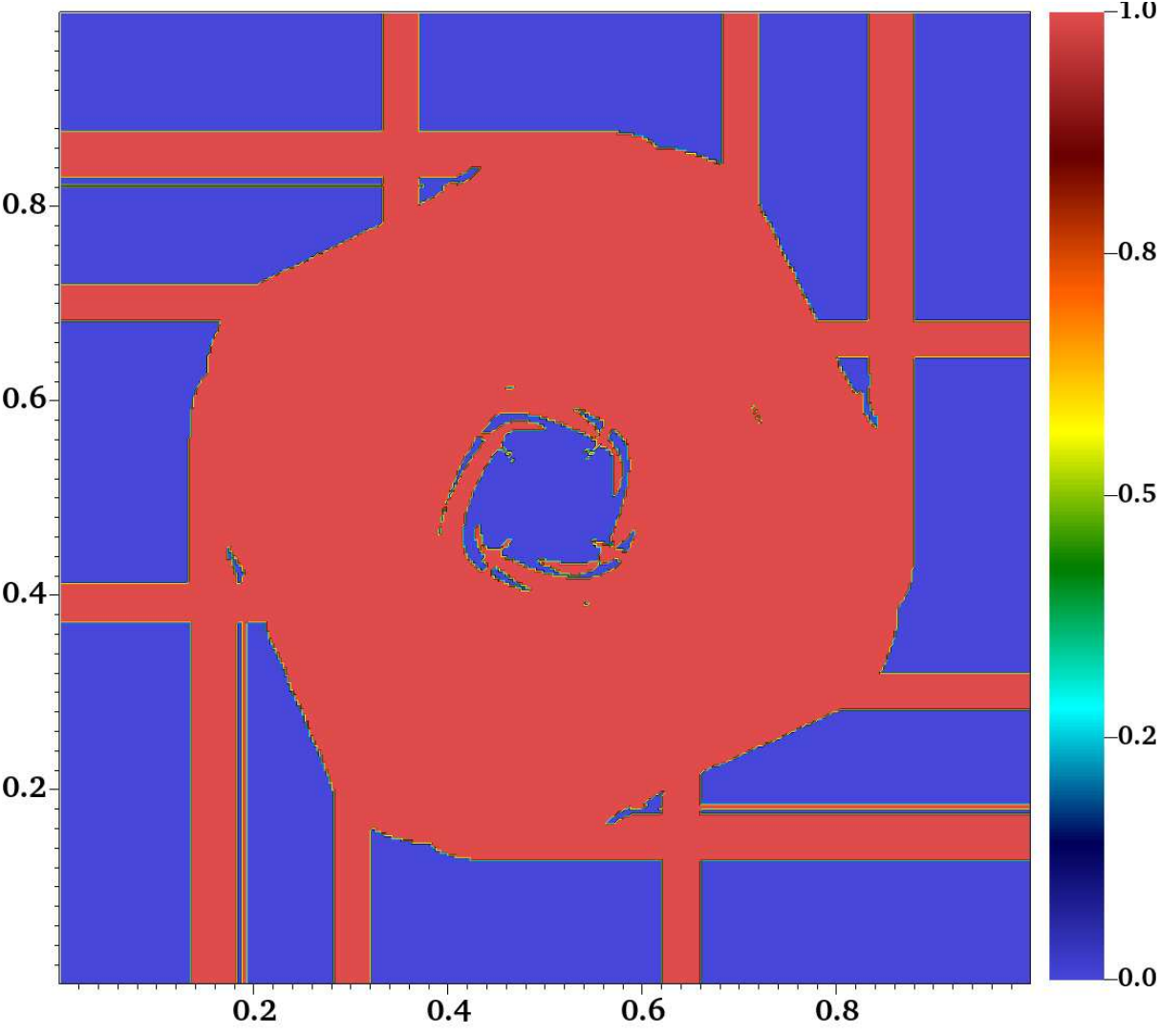}&
   \includegraphics[width = 4.60cm]{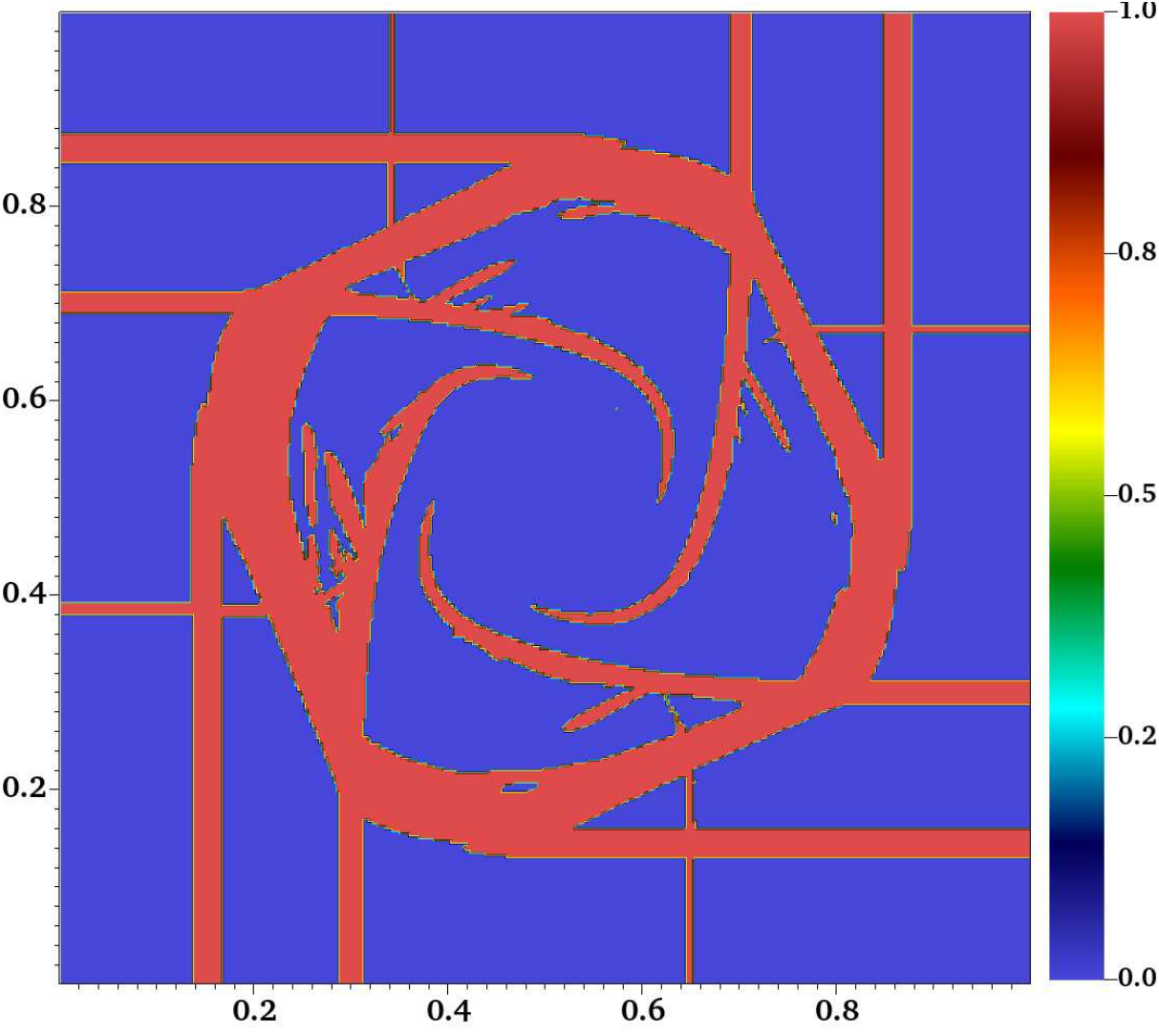}\\
  (a)  H1-WENO $(\mathcal{K}=1~~T=0)$ & (b)  H1-WENO $(\mathcal{K}=1)$ &(c)   H1-WENO $(\mathcal{K}=100)$\\
     \includegraphics[width = 4.6cm]{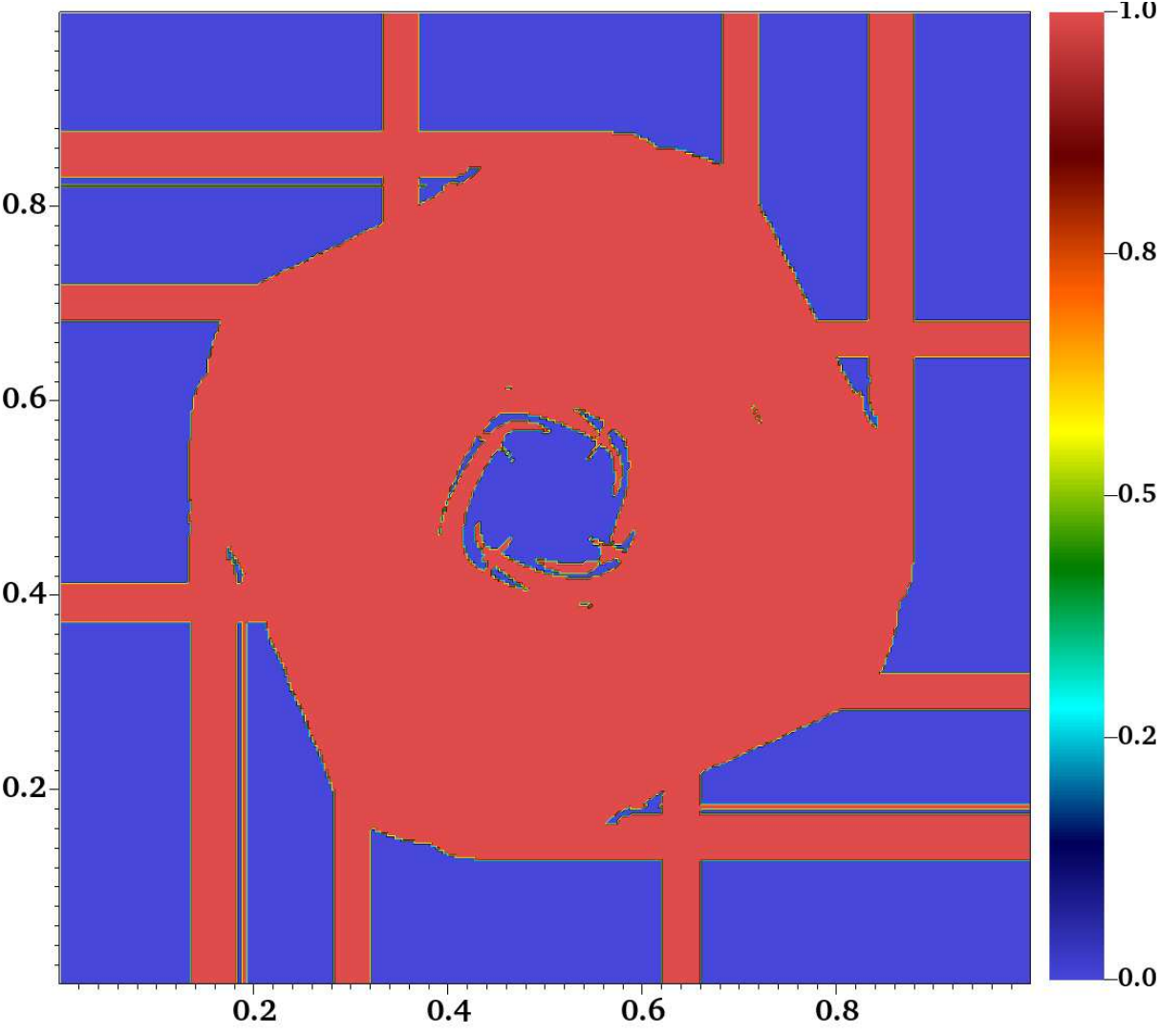} &
  \includegraphics[width = 4.6cm]{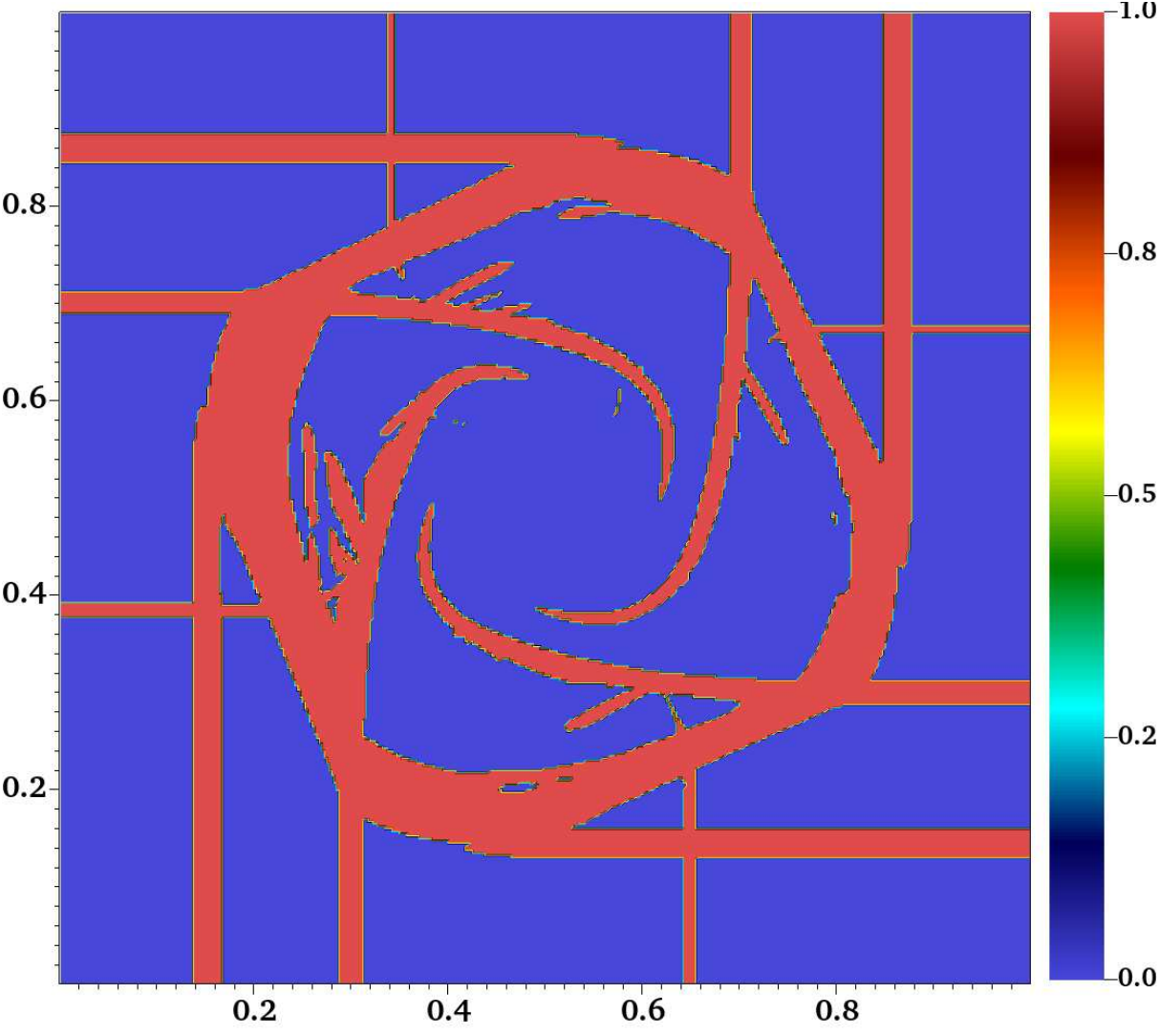}&
   \includegraphics[width = 4.60cm]{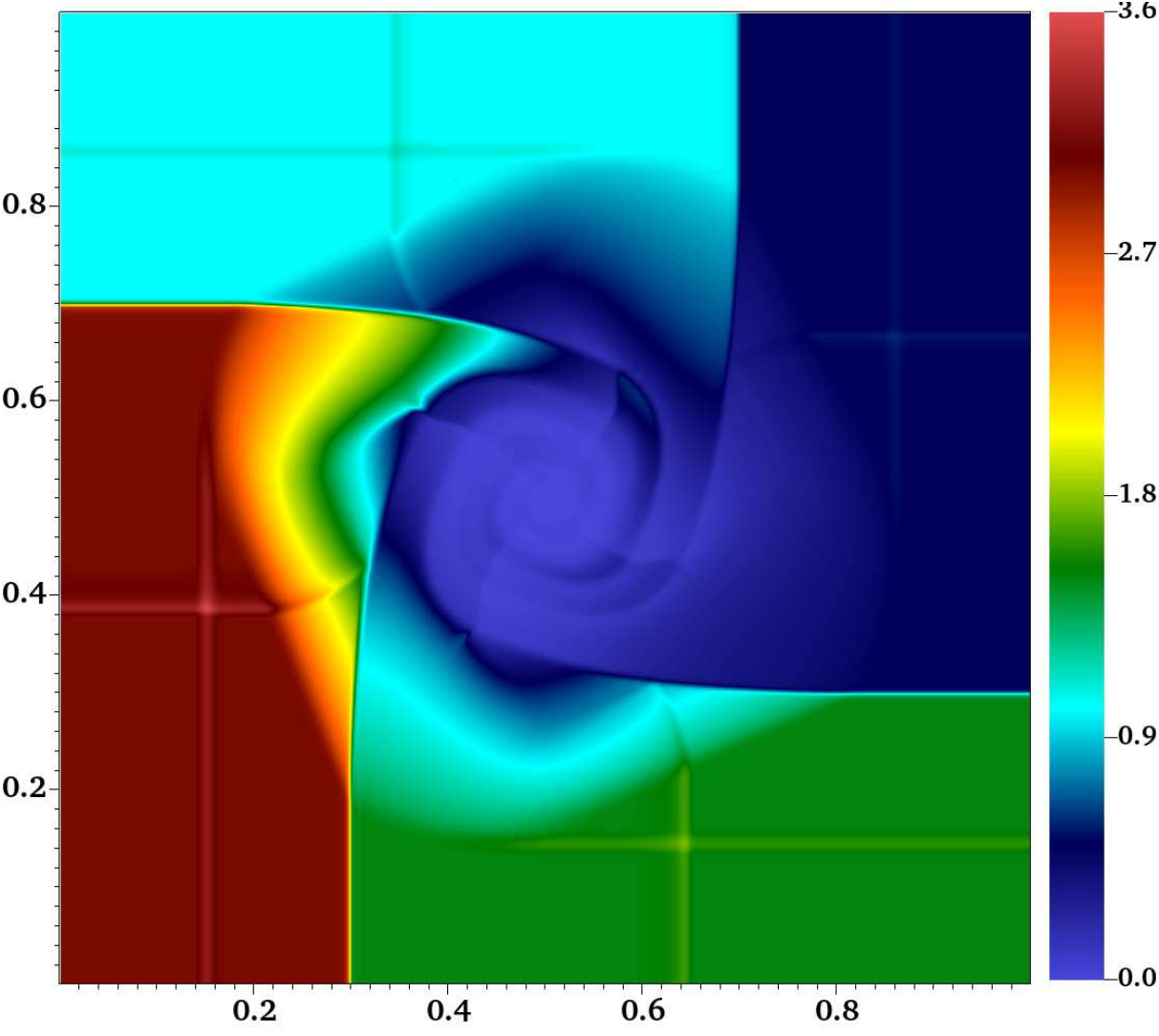}\\
  (d)  H2-WENO $(\mathcal{K}=1)$ & (e) H2-WENO $(\mathcal{K}=100)$ & (f) H1-WENO $(\mathcal{K}=1)$ \\
     \includegraphics[width = 4.6cm]{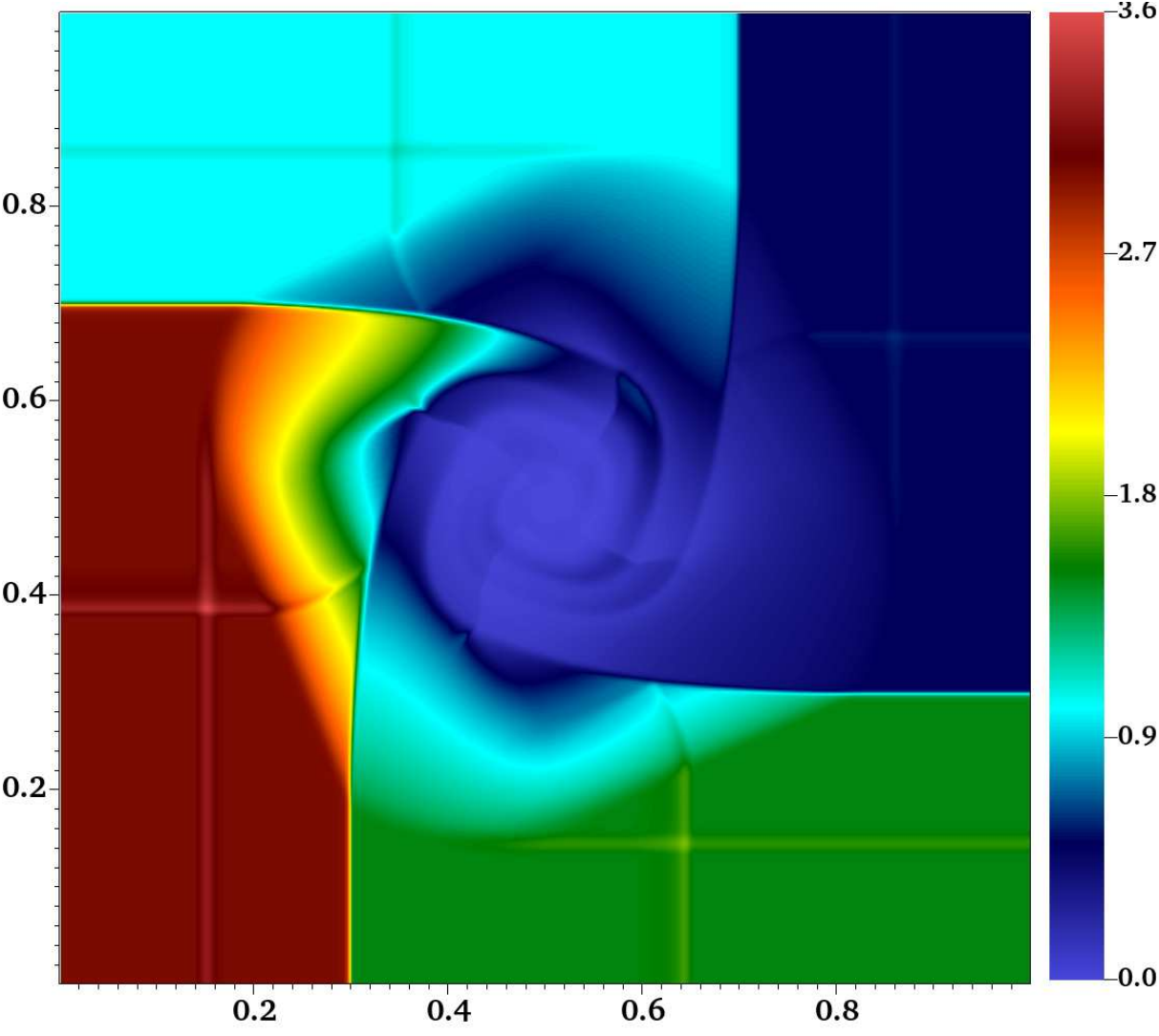} &
  \includegraphics[width = 4.6cm]{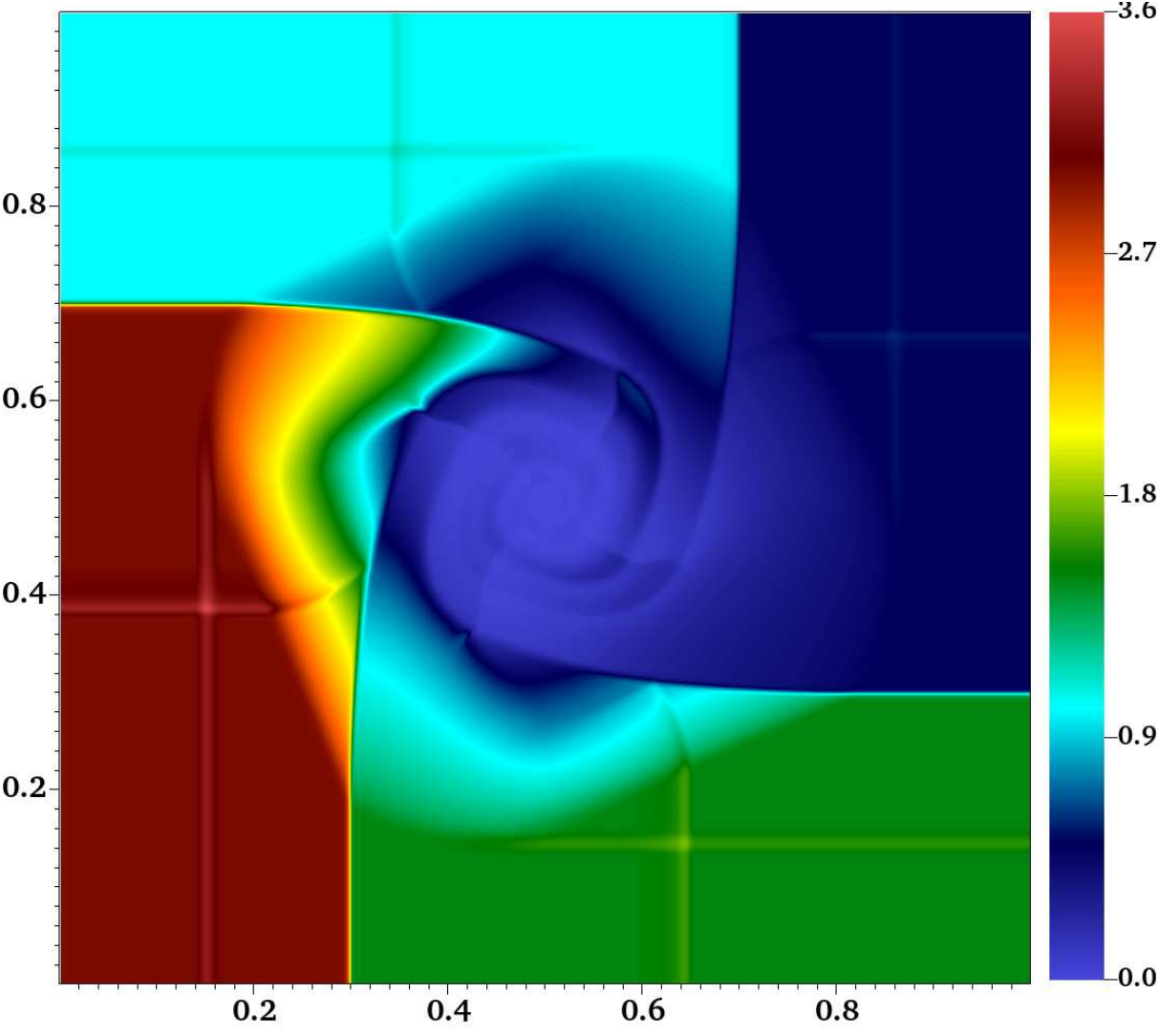}&
   \includegraphics[width = 4.60cm]{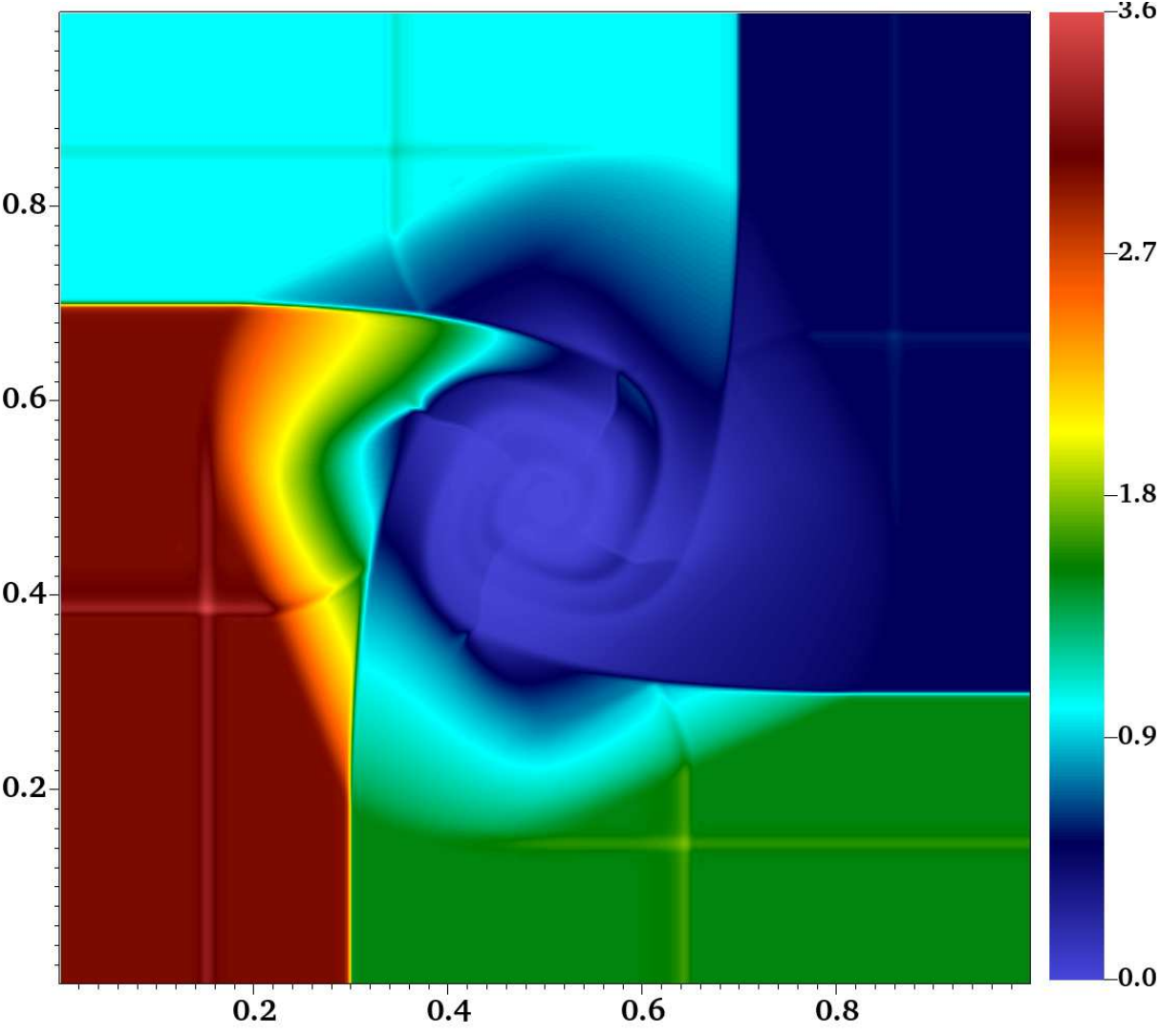}\\
  (g) H1-WENO $(\mathcal{K}=100)$  & (h) H2-WENO $(\mathcal{K}=1)$ & (i) H2-WENO $(\mathcal{K}=100)$ \\
 \end{tabular}
\caption{ Comparison of the troubled-cell indicator at initial and final times for Example \ref{rp1} using H1-WENO and H2-WENO schemes with $\mathcal{K}=1$ and $100$, together with the corresponding density profiles at final time computed on a $400 \times 400$ grid.}
\label{Fig:RP11-H1H2}
\end{figure}

\begin{figure}
 \centering
 \begin{tabular}{cc}
  \includegraphics[width = 7.6cm]{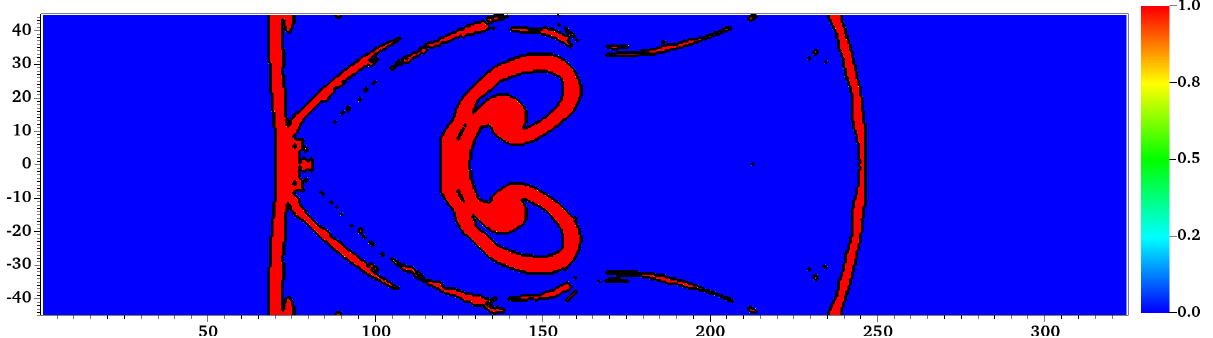} &
  \includegraphics[width = 7.6cm]{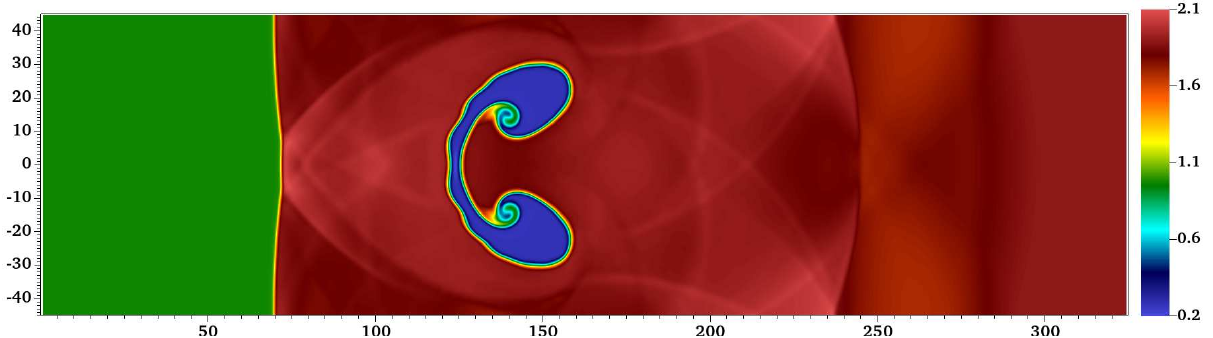}\\
  (a) H1-WENO ($\mathcal{K}=1$) & (b) H1-WENO(Density)\\
  \includegraphics[width = 7.6cm]{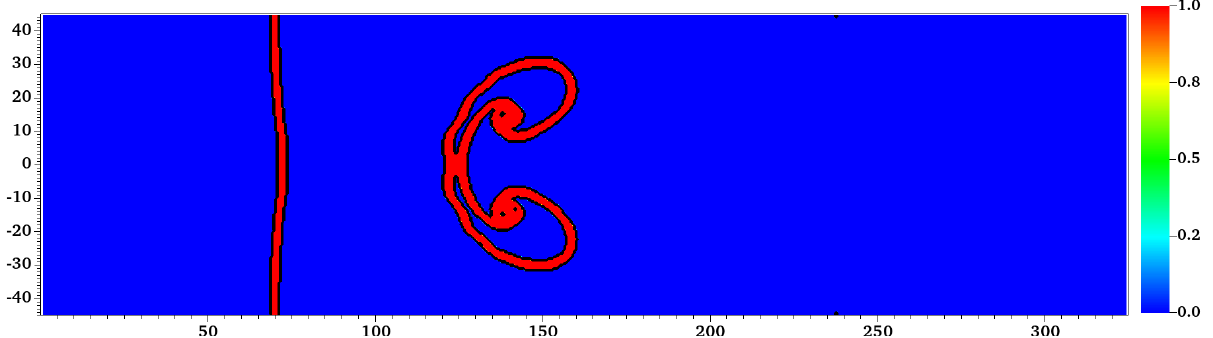} &
  \includegraphics[width = 7.6cm]{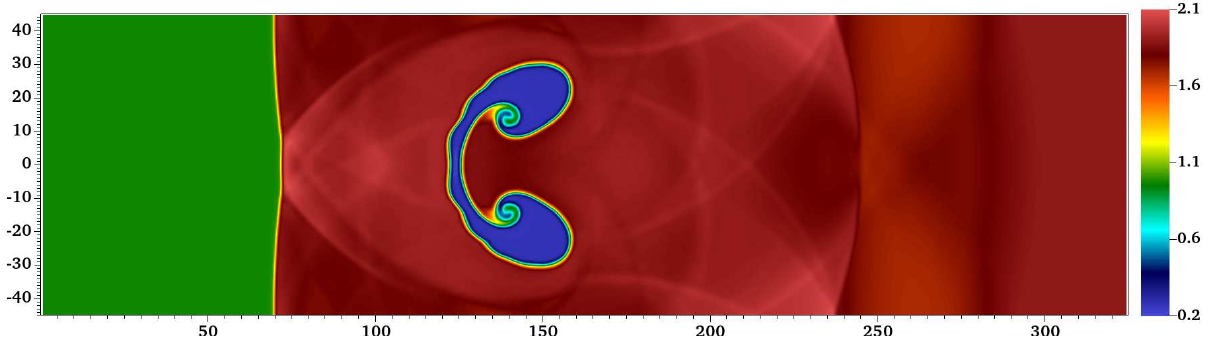}\\
  (c) H1-WENO ($\mathcal{K}=50$) & (d) H1-WENO (Density)\\
  \includegraphics[width = 7.6cm]{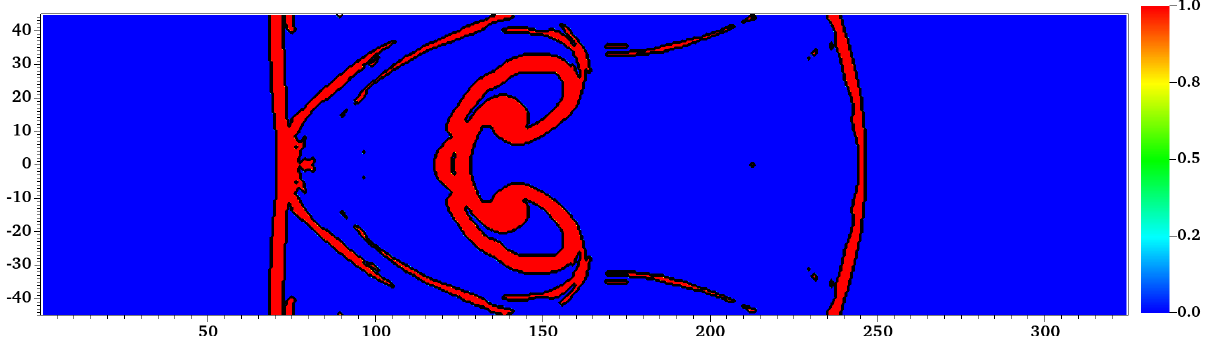} &
  \includegraphics[width = 7.6cm]{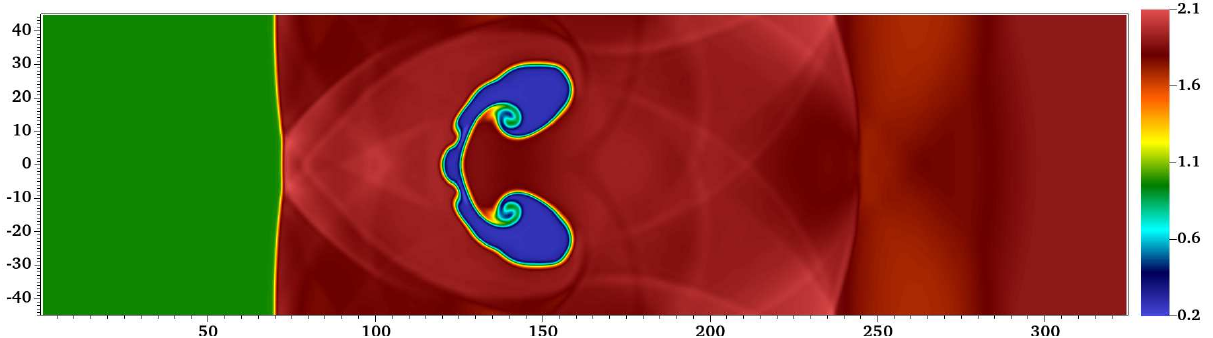}\\
  (e) H2-WENO ($\mathcal{K}=1$) & (f) H2-WENO (Density)\\
  \includegraphics[width = 7.6cm]{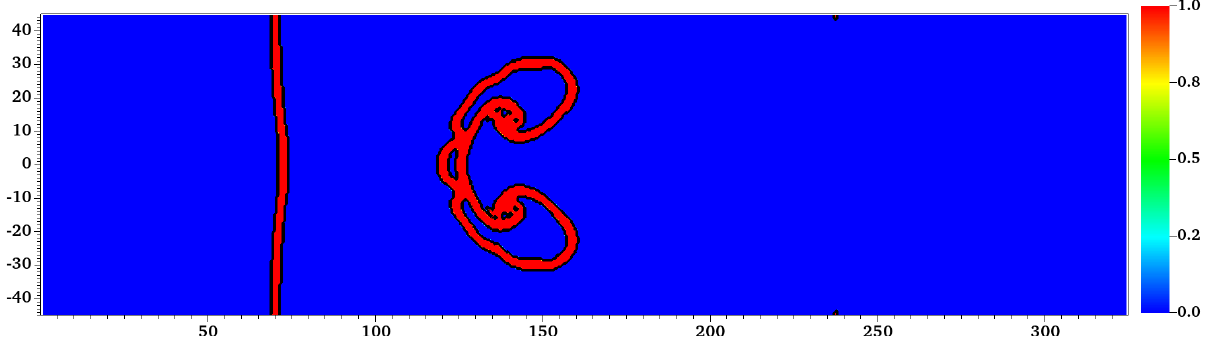} &
  \includegraphics[width = 7.6cm]{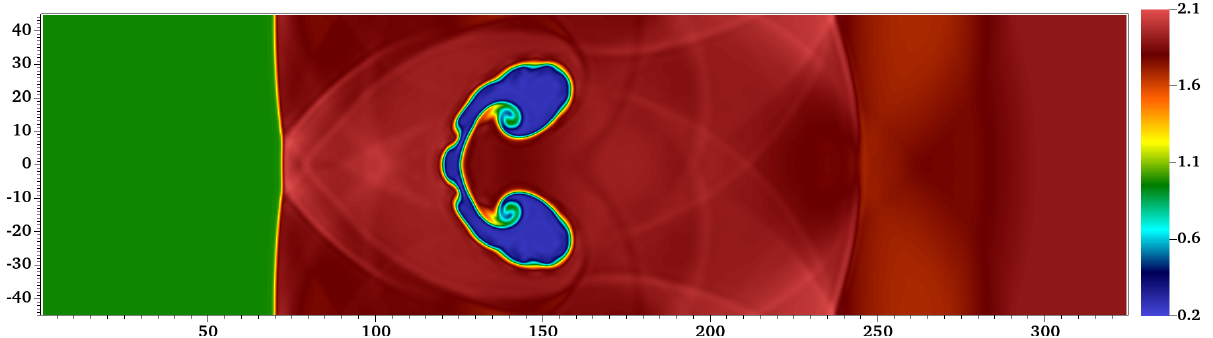}\\
  (g) H2-WENO ($\mathcal{K}=50$)& (h) H2-WENO (Density)\\
 \end{tabular}
\caption{Plots of the troubled-cell indicator and density solution for Example \ref{sb1} with $\mathcal{K}=1$ and $50$ on a $650 \times 180$ mesh, obtained using the H1-WENO and H2-WENO schemes at time $T=450$. }
\label{Fig:dmr}
\label{Fig:sb1_H1H2}
\end{figure}

As discussed in the previous section, for $\mathcal{K}=1$ the hybrid schemes are sufficient to identify troubled cells; however, they also tend to incorrectly flag some smooth cells, which can adversely affect the overall efficiency of the schemes. This necessitates increasing the value of $\mathcal{K}$ to reduce the unwanted marking of smooth regions. On the other hand, choosing a larger $\mathcal{K}$ may result in failure to detect certain discontinuities. Thus, there is a clear trade-off in selecting an appropriate value of the parameter $\mathcal{K}$. 
To investigate this balance, we perform numerical experiments for both H1-WENO and H2-WENO schemes. Furthermore, to analyze the effect of varying $\mathcal{K}$ on the percentage of marked troubled cells, we consider a range of values $\mathcal{K} \in \{1, \ldots, 100\}$. In Figures~\ref{Fig:CPU1D} and \ref{Fig:CPU2D}, we present the speed-up factors of the H1-WENO and H2-WENO schemes for different values of $\mathcal{K}$, relative to their corresponding baseline schemes with $\mathcal{K}=1$, for 1D and 2D test cases, respectively. In addition, we report the percentage of troubled cells identified during the simulations for both H1-WENO and H2-WENO schemes.

From Figures~\ref{Fig:CPU1D}~(a)-(b), we observe that the speed-up factor of both schemes becomes nearly constant for $\mathcal{K} \geq 40$. This behavior is further supported by Figures~\ref{Fig:CPU1D}~(c)--(d), where the percentage of troubled cells also stabilizes beyond $\mathcal{K} \geq 40$. Therefore, increasing $\mathcal{K}$ beyond this threshold does not lead to any significant improvement in the speed-up factor for either scheme. A similar trend is observed for the 2D test cases shown in Figure~\ref{Fig:CPU2D}, where results are presented only for the 2D Riemann problems. This is because Examples~\ref{sb1} and \ref{sb2} violate the positivity of variables when $\mathcal{K} > 50$. Moreover, we observe that for $\mathcal{K} \geq 20$, both the speed-up factor and the percentage of troubled cells marked remain nearly constant.

In order to compare the solutions for $\mathcal{K}=1$ and $\mathcal{K}=100$, we consider Example \ref{test3} and employ the H1-WENO and H2-WENO schemes with 500 mesh points. Figure \ref{Fig:RP1-H1H2} presents the distribution of troubled cells in the $x$–$t$ plane for both values of $\mathcal{K}$, along with the corresponding solution profiles shown in Figure \ref{Fig:RP1-H1H2}~(e)–(f).
It can be observed that the troubled-cell indicators effectively capture the discontinuities in both cases, with no significant difference between $\mathcal{K}=1$ and $\mathcal{K}=100$. Moreover, the solutions obtained using both schemes for these values of $\mathcal{K}$ are in close agreement. In Figure \ref{Fig:RP11-H1H2}, we compare the troubled-cell indicators at the initial and final times, along with the corresponding density profiles obtained using both schemes for $\mathcal{K}=1$ and $\mathcal{K}=100$ on a $400 \times 400$ mesh for Example \ref{rp1}.
It can be observed that the troubled-cell indicator effectively captures the discontinuities in all cases. However, for $\mathcal{K}=1$, some smooth cells are also identified as troubled at the final time, whereas this effect is significantly reduced for $\mathcal{K}=100$.
Furthermore, the density profiles obtained using both schemes for $\mathcal{K}=1$ and $\mathcal{K}=100$ are in close agreement.

In Figure \ref{Fig:sb1_H1H2}, we compare the proposed scheme for two different values of $\mathcal{K}$, namely $\mathcal{K}=1$ and $\mathcal{K}=50$, in terms of the troubled cell indicator and the density variable for Example \ref{sb1}291982. It is observed that, for both schemes, the case $\mathcal{K}=1$ captures the discontinuities along with a small portion of the smooth region. This undesired detection in smooth regions disappears when $\mathcal{K}=50$ is used. In both cases, the density profiles remain comparable. For the H1-WENO scheme, the percentage of troubled cells marked during the simulation is $4.87\%$ for $\mathcal{K} = 1$, which reduces to $1.6\%$ for $\mathcal{K} = 50$. A similar trend is observed for the H2-WENO scheme, where the percentage of troubled cells decreases from $5.05\%$ to $1.69\%$ as $\mathcal{K}$ increases from 1 to 50.
\section{Conclusion}\label{sec:con}
In this work, we proposed a new class of hybrid WENO schemes, namely H1-WENO, H2-WENO, and H3-WENO, for the efficient numerical approximation of SRHD equations. The proposed schemes are based on a new troubled-cell indicator constructed using the smoothness indicators of the WENO scheme. The indicator effectively distinguishes smooth and discontinuous regions, enabling the adaptive use of inexpensive reconstructions in smooth regions and characteristic-wise WENO reconstruction only near discontinuities.
The proposed hybrid strategies were developed using the WENO-AO(5,3) scheme as the underlying reconstruction framework. In particular, the H1-WENO scheme employs component-wise WENO reconstruction in smooth regions, the H2-WENO scheme uses a fifth-order upwind component-wise finite difference approximation, and the H3-WENO scheme adopts a characteristic-wise upwind finite difference method, while all three schemes retain characteristic-wise WENO reconstruction near discontinuities. Numerical experiments demonstrate that the proposed schemes preserve the accuracy, robustness, and non-oscillatory properties of characteristic-wise WENO methods.
We also investigated the relative computational cost associated with non-linear WENO weight evaluation and characteristic decomposition. The numerical results indicate that the computation of eigenvectors required for characteristic decomposition constitutes the major computational expense in WENO-type methods. By restricting characteristic decomposition only to troubled regions, the proposed hybrid schemes significantly improve computational efficiency. In particular, the H1-WENO scheme achieves approximately 40--50\% reduction in computational cost compared with the standard WENO-AO scheme, while H2-WENO and H3-WENO provide improvements of about 30--40\% and 5--10\%, respectively. We have also demonstrated that the parameter used in the troubled-cell indicator does not significantly affect the accuracy or computational efficiency of the proposed schemes, provided it is chosen within a suitable range. Overall, the proposed hybrid WENO schemes provide an effective balance between computational efficiency and numerical accuracy, making them promising candidates for large-scale simulations.
Overall, the proposed hybrid WENO schemes provide an effective balance between computational efficiency and numerical accuracy, making them promising candidates for large-scale simulations involving complex systems of hyperbolic conservation laws.
\section{Acknowledgment} 
The author Rakesh Kumar is supported by Prime Minister Early Career Research Grant (PMECRG) of the Anusandhan National Research Foundation (ANRF), India, under Grant No. ANRF/ECRG/2025/004846/PMS. Biswarup Biswas is supported by the State University Research Excellence (SURE) Scheme of the ANRF, India, under File No. SUR/2022/001786 and Asha Kumari Meena is funded by ANRF with File No: EEQ/2023/000453.
'

\bibliography{comp_bib}

@incollection {kum_20a,
    AUTHOR = {Kumar, Rakesh},
     TITLE = {Hybrid {FDM}-{WENO} method for the convection-diffusion
              problems},
 BOOKTITLE = {Hyperbolic problems: theory, numerics, applications},
    SERIES = {AIMS Ser. Appl. Math.},
    VOLUME = {10},
     PAGES = {515--523},
 PUBLISHER = {Am. Inst. Math. Sci. (AIMS), Springfield, MO},
      YEAR = {[2020] \copyright 2020},
   MRCLASS = {65M06 (35K57 65M22)},
  MRNUMBER = {4362552},
}

@article {kum_cha_22a,
    AUTHOR = {Kumar, Rakesh and Chandrashekar, Praveen},
     TITLE = {Multi-level {WENO} schemes with an adaptive
              characteristic-wise reconstruction for system of {E}uler
              equations},
   JOURNAL = {Comput. \& Fluids},
  FJOURNAL = {Computers \& Fluids. An International Journal},
    VOLUME = {239},
      YEAR = {2022},
     PAGES = {Paper No. 105386, 24},
      ISSN = {0045-7930},
   MRCLASS = {76M12 (76Nxx)},
  MRNUMBER = {4396537},
       DOI = {10.1016/j.compfluid.2022.105386},
       URL = {https://doi.org/10.1016/j.compfluid.2022.105386},
}

@article {aru_etal_22a,
    AUTHOR = {Arun, K R and Dond, Asha K. and Kumar, Rakesh},
     TITLE = {WENO Smoothness Indicator Based Troubled-cell Indicator for Hyperbolic Conservation Laws},
   JOURNAL = {},
  FJOURNAL = {Computers \& Fluids. An International Journal},
    VOLUME = {},
      YEAR = {2023},
     PAGES = {},
      ISSN = {0045-7930},
   MRCLASS = {76M12 (76Nxx)},
  MRNUMBER = {4396537},
       DOI = {},
       URL = {https://doi.org/10.1016/j.compfluid.2022.105386},
}

@article {zhao-etal_20a,
    AUTHOR = {Zhao, Guo-Yan and Sun, Ming-Bo and Pirozzoli, Sergio},
     TITLE = {On shock sensors for hybrid compact/{WENO} schemes},
   JOURNAL = {Comput. \& Fluids},
  FJOURNAL = {Computers \& Fluids. An International Journal},
    VOLUME = {199},
      YEAR = {2020},
     PAGES = {104439, 17},
      ISSN = {0045-7930},
   MRCLASS = {76M20 (76L05)},
  MRNUMBER = {4057020},
       DOI = {10.1016/j.compfluid.2020.104439},
       URL = {https://doi.org/10.1016/j.compfluid.2020.104439},
}

@article {fu-etal_14a,
    AUTHOR = {Fu, Huankun and Wang, Zhengjie and Yan, Yonghua and Liu,
              Chaoqun},
     TITLE = {Modified weighted compact scheme with global weights for shock
              capturing},
   JOURNAL = {Comput. \& Fluids},
  FJOURNAL = {Computers \& Fluids. An International Journal},
    VOLUME = {96},
      YEAR = {2014},
     PAGES = {165--176},
      ISSN = {0045-7930},
   MRCLASS = {65M06 (76L05)},
  MRNUMBER = {3198071},
       DOI = {10.1016/j.compfluid.2014.02.022},
       URL = {https://doi.org/10.1016/j.compfluid.2014.02.022},
}

@article{mov-joh_13a,
title = {A solution-adaptive method for efficient compressible multifluid simulations, with application to the Richtmyer–Meshkov instability},
journal = {Journal of Computational Physics},
volume = {239},
pages = {166-186},
year = {2013},
issn = {0021-9991},
doi = {https://doi.org/10.1016/j.jcp.2013.01.016},
url = {https://www.sciencedirect.com/science/article/pii/S0021999113000478},
author = {Pooya Movahed and Eric Johnsen},
}

@inbook{vis-gai_05a,
author = {Miguel Visbal and Datta Gaitonde},
title = {Shock Capturing Using Compact-Differencing-Based Methods},
booktitle = {43rd AIAA Aerospace Sciences Meeting and Exhibit},
chapter = {},
pages = {},
doi = {10.2514/6.2005-1265},
URL = {https://arc.aiaa.org/doi/abs/10.2514/6.2005-1265},
eprint = {https://arc.aiaa.org/doi/pdf/10.2514/6.2005-1265}
}

@article{hil-pul_04a,
title = {Hybrid tuned center-difference-WENO method for large eddy simulations in the presence of strong shocks},
journal = {Journal of Computational Physics},
volume = {194},
number = {2},
pages = {435-450},
year = {2004},
issn = {0021-9991},
doi = {https://doi.org/10.1016/j.jcp.2003.07.032},
url = {https://www.sciencedirect.com/science/article/pii/S002199910300490X},
author = {D.J Hill and D.I Pullin},
}

@article {cra-etal_18a,
    AUTHOR = {Cravero, I. and Puppo, G. and Semplice, M. and Visconti, G.},
     TITLE = {C{WENO}: uniformly accurate reconstructions for balance laws},
   JOURNAL = {Math. Comp.},
  FJOURNAL = {Mathematics of Computation},
    VOLUME = {87},
      YEAR = {2018},
    NUMBER = {312},
     PAGES = {1689--1719},
      ISSN = {0025-5718},
   MRCLASS = {65M08 (65M12)},
  MRNUMBER = {3787389},
MRREVIEWER = {Zhongming Wang},
       DOI = {10.1090/mcom/3273},
       URL = {https://doi.org/10.1090/mcom/3273},
}

@article {cra-etal_19a,
    AUTHOR = {Cravero, I. and Semplice, M. and Visconti, G.},
     TITLE = {Optimal definition of the nonlinear weights in
              multidimensional central {WENOZ} reconstructions},
   JOURNAL = {SIAM J. Numer. Anal.},
  FJOURNAL = {SIAM Journal on Numerical Analysis},
    VOLUME = {57},
      YEAR = {2019},
    NUMBER = {5},
     PAGES = {2328--2358},
      ISSN = {0036-1429},
   MRCLASS = {65M08 (65M20)},
  MRNUMBER = {4013927},
MRREVIEWER = {Huazhong Tang},
       DOI = {10.1137/18M1228232},
       URL = {https://doi.org/10.1137/18M1228232},
}

@article {dum-etal_17a,
    AUTHOR = {Dumbser, Michael and Boscheri, Walter and Semplice, Matteo and
              Russo, Giovanni},
     TITLE = {Central weighted {ENO} schemes for hyperbolic conservation
              laws on fixed and moving unstructured meshes},
   JOURNAL = {SIAM J. Sci. Comput.},
  FJOURNAL = {SIAM Journal on Scientific Computing},
    VOLUME = {39},
      YEAR = {2017},
    NUMBER = {6},
     PAGES = {A2564--A2591},
      ISSN = {1064-8275},
   MRCLASS = {65M08 (35L60 35L65 65Y05 76N15)},
  MRNUMBER = {3723333},
MRREVIEWER = {Benjamin Ivorra},
       DOI = {10.1137/17M1111036},
       URL = {https://doi.org/10.1137/17M1111036},
}

@article {kum_18a,
    AUTHOR = {Kumar, Rakesh},
     TITLE = {Adaptive semi-discrete formulation of {BSQI}-{WENO} scheme for
              the modified {B}urgers' equation},
   JOURNAL = {BIT},
  FJOURNAL = {BIT. Numerical Mathematics},
    VOLUME = {58},
      YEAR = {2018},
    NUMBER = {1},
     PAGES = {103--132},
      ISSN = {0006-3835},
   MRCLASS = {65M70 (65D07 65M22)},
  MRNUMBER = {3771438},
MRREVIEWER = {Gonzalo Rubio},
       DOI = {10.1007/s10543-017-0675-8},
       URL = {https://doi.org/10.1007/s10543-017-0675-8},
}

@article {kum-cha_19a,
    AUTHOR = {Kumar, Rakesh and Chandrashekar, Praveen},
     TITLE = {Efficient seventh order {WENO} schemes of adaptive order for
              hyperbolic conservation laws},
   JOURNAL = {Comput. \& Fluids},
  FJOURNAL = {Computers \& Fluids. An International Journal},
    VOLUME = {190},
      YEAR = {2019},
     PAGES = {49--76},
      ISSN = {0045-7930},
   MRCLASS = {65M06 (76M20 76Nxx)},
  MRNUMBER = {3961109},
       DOI = {10.1016/j.compfluid.2019.06.003},
       URL = {https://doi.org/10.1016/j.compfluid.2019.06.003},
}

@incollection {shu-97d,
    AUTHOR = {Shu, Chi-Wang},
     TITLE = {Essentially non-oscillatory and weighted essentially
              non-oscillatory schemes for hyperbolic conservation laws},
 BOOKTITLE = {Advanced numerical approximation of nonlinear hyperbolic
              equations ({C}etraro, 1997)},
    SERIES = {Lecture Notes in Math.},
    VOLUME = {1697},
     PAGES = {325--432},
 PUBLISHER = {Springer, Berlin},
      YEAR = {1998},
   MRCLASS = {65Mxx (65-02 76M20)},
  MRNUMBER = {1728856},
MRREVIEWER = {Shi Jin},
       DOI = {10.1007/BFb0096355},
       URL = {https://doi.org/10.1007/BFb0096355},
}

@article {shu-osh_88a,
    AUTHOR = {Shu, Chi-Wang and Osher, Stanley},
     TITLE = {Efficient implementation of essentially nonoscillatory
              shock-capturing schemes},
   JOURNAL = {J. Comput. Phys.},
  FJOURNAL = {Journal of Computational Physics},
    VOLUME = {77},
      YEAR = {1988},
    NUMBER = {2},
     PAGES = {439--471},
      ISSN = {0021-9991},
   MRCLASS = {65M05 (35L65)},
  MRNUMBER = {954915},
       DOI = {10.1016/0021-9991(88)90177-5},
       URL = {https://doi.org/10.1016/0021-9991(88)90177-5},
}

@article {jia-shu_96a,
    AUTHOR = {Jiang, Guang-Shan and Shu, Chi-Wang},
     TITLE = {Efficient implementation of weighted ENO schemes},
   JOURNAL = {J. Comput. Phys.},
  FJOURNAL = {Journal of Computational Physics},
    VOLUME = {126},
      YEAR = {1996},
    NUMBER = {1},
     PAGES = {202--228},
      ISSN = {0021-9991},
   MRCLASS = {65M06 (65Y05)},
  MRNUMBER = {1391627},
       URL = {https://doi.org/10.1006/jcph.1996.0130},
}

@article {bal-shu_00a,
    AUTHOR = {Balsara, Dinshaw S. and Shu, Chi-Wang},
     TITLE = {Monotonicity preserving weighted essentially non-oscillatory
              schemes with increasingly high order of accuracy},
   JOURNAL = {J. Comput. Phys.},
  FJOURNAL = {Journal of Computational Physics},
    VOLUME = {160},
      YEAR = {2000},
    NUMBER = {2},
     PAGES = {405--452},
      ISSN = {0021-9991},
   MRCLASS = {65L06 (76M20)},
  MRNUMBER = {1763821},
       URL = {https://doi.org/10.1006/jcph.2000.6443},
}

@article {ren-etal_03a,
    AUTHOR = {Ren, Yu-Xin and Liu, Miao'er and Zhang, Hanxin},
     TITLE = {A characteristic-wise hybrid compact-{WENO} scheme for solving
              hyperbolic conservation laws},
   JOURNAL = {J. Comput. Phys.},
  FJOURNAL = {Journal of Computational Physics},
    VOLUME = {192},
      YEAR = {2003},
    NUMBER = {2},
     PAGES = {365--386},
      ISSN = {0021-9991},
   MRCLASS = {65M25 (65M06 76L05 76M20 76N15)},
  MRNUMBER = {2020743},
       URL = {https://doi.org/10.1016/j.jcp.2003.07.006},
}

@article {bor-etal_08a,
    AUTHOR = {Borges, Rafael and Carmona, Monique and Costa, Bruno and Don,
              Wai Sun},
     TITLE = {An improved weighted essentially non-oscillatory scheme for
              hyperbolic conservation laws},
   JOURNAL = {J. Comput. Phys.},
  FJOURNAL = {Journal of Computational Physics},
    VOLUME = {227},
      YEAR = {2008},
    NUMBER = {6},
     PAGES = {3191--3211},
      ISSN = {0021-9991},
   MRCLASS = {65M06 (35L65)},
  MRNUMBER = {2392730},
MRREVIEWER = {Allaberen Ashyralyev},
       URL = {https://doi.org/10.1016/j.jcp.2007.11.038},
}

@article {bal-etal_20a,
    AUTHOR = {Balsara, Dinshaw S. and Garain, Sudip and Florinski, Vladimir
              and Boscheri, Walter},
     TITLE = {An efficient class of {WENO} schemes with adaptive order for
              unstructured meshes},
   JOURNAL = {J. Comput. Phys.},
  FJOURNAL = {Journal of Computational Physics},
    VOLUME = {404},
      YEAR = {2020},
     PAGES = {109062, 32},
      ISSN = {0021-9991},
   MRCLASS = {65N06},
  MRNUMBER = {4044720},
       DOI = {10.1016/j.jcp.2019.109062},
       URL = {https://doi.org/10.1016/j.jcp.2019.109062},
}

@article {bal-etal_16a,
    AUTHOR = {Balsara, Dinshaw S. and Garain, Sudip and Shu, Chi-Wang},
     TITLE = {An efficient class of {WENO} schemes with adaptive order},
   JOURNAL = {J. Comput. Phys.},
  FJOURNAL = {Journal of Computational Physics},
    VOLUME = {326},
      YEAR = {2016},
     PAGES = {780--804},
      ISSN = {0021-9991},
   MRCLASS = {65M06},
  MRNUMBER = {3554165},
       DOI = {10.1016/j.jcp.2016.09.009},
       URL = {https://doi.org/10.1016/j.jcp.2016.09.009},
}

@incollection {pup_03a,
    AUTHOR = {Puppo, Gabriella},
     TITLE = {Adaptive application of characteristic projection for central
              schemes},
 BOOKTITLE = {Hyperbolic problems: theory, numerics, applications},
     PAGES = {819--829},
 PUBLISHER = {Springer, Berlin},
      YEAR = {2003},
   MRCLASS = {65M25},
  MRNUMBER = {2053230},
}

@article {jun-etal_19a,
    AUTHOR = {Peng, Jun and Zhai, Chuanlei and Ni, Guoxi and Yong, Heng and
              Shen, Yiqing},
     TITLE = {An adaptive characteristic-wise reconstruction {WENO}-{Z}
              scheme for gas dynamic euler equations},
   JOURNAL = {Comput. \& Fluids},
  FJOURNAL = {Computers \& Fluids. An International Journal},
    VOLUME = {179},
      YEAR = {2019},
     PAGES = {34--51},
      ISSN = {0045-7930},
   MRCLASS = {65M06 (76M20 76N15)},
  MRNUMBER = {3872384},
       DOI = {10.1016/j.compfluid.2018.08.008},
       URL = {https://doi.org/10.1016/j.compfluid.2018.08.008},
}

@article{don-bor_13a,
title = "Accuracy of the weighted essentially non-oscillatory
conservative finite difference schemes",
journal = "Journal of Computational Physics",
volume = "250",
number = "Supplement C",
pages = "347 - 372",
year = "2013",
issn = "0021-9991",
doi = "https://doi.org/10.1016/j.jcp.2013.05.018",
url = "http://www.sciencedirect.com/science/article/pii/S0021999113003501",
author = "Wai-Sun Don and Rafael Borges"
}

@article{cas-etal_11a,
title = "High order weighted essentially non-oscillatory WENO-Z
schemes for hyperbolic conservation laws",
journal = "Journal of Computational Physics",
volume = "230",
number = "5",
pages = "1766 - 1792",
year = "2011",
issn = "0021-9991",
doi = "https://doi.org/10.1016/j.jcp.2010.11.028",
url = "http://www.sciencedirect.com/science/article/pii/S0021999110006431",
author = "Marcos Castro and Bruno Costa and Wai Sun Don"
}

@article{shu-osh_89a,
title = "Efficient implementation of essentially non-oscillatory shock-capturing schemes, II",
journal = "Journal of Computational Physics",
volume = "83",
number = "1",
pages = "32 - 78",
year = "1989",
issn = "0021-9991",
doi = "https://doi.org/10.1016/0021-9991(89)90222-2",
url = "http://www.sciencedirect.com/science/article/pii/0021999189902222",
author = "Chi-Wang Shu and Stanley Osher"
}

@article {kum-cha_18a,
    AUTHOR = {Kumar, Rakesh and Chandrashekar, Praveen},
     TITLE = {Simple smoothness indicator and multi-level adaptive order
              {WENO} scheme for hyperbolic conservation laws},
   JOURNAL = {J. Comput. Phys.},
  FJOURNAL = {Journal of Computational Physics},
    VOLUME = {375},
      YEAR = {2018},
     PAGES = {1059--1090},
      ISSN = {0021-9991},
   MRCLASS = {65M06 (76M20)},
  MRNUMBER = {3874573},
       DOI = {10.1016/j.jcp.2018.09.027},
       URL = {https://doi.org/10.1016/j.jcp.2018.09.027},
}

@article{Arb-etal_17a,
author = {Arbogast, T. and Huang, C. and Zhao, X.},
title = {Accuracy of WENO and Adaptive Order WENO Reconstructions for Solving Conservation Laws},
journal = {SIAM Journal on Numerical Analysis},
volume = {56},
number = {3},
pages = {1818-1847},
year = {2018},
doi = {10.1137/17M1154758},

URL = { 
        https://doi.org/10.1137/17M1154758
    
},
eprint = { 
        https://doi.org/10.1137/17M1154758
    
}

}

@article {ger-etal_09a,
    AUTHOR = {Gerolymos, G. A. and S\'en\'echal, D. and Vallet, I.},
     TITLE = {Very-high-order {WENO} schemes},
   JOURNAL = {J. Comput. Phys.},
  FJOURNAL = {Journal of Computational Physics},
    VOLUME = {228},
      YEAR = {2009},
    NUMBER = {23},
     PAGES = {8481--8524},
      ISSN = {0021-9991},
   MRCLASS = {65M06},
  MRNUMBER = {2558763},
       URL = {https://doi.org/10.1016/j.jcp.2009.07.039},
}

@article{har-etal_97a,
title = "Uniformly High Order Accurate Essentially Non-oscillatory Schemes, III",
journal = "Journal of Computational Physics",
volume = "71",
number = "1",
pages = "231-303",
year = "1987",
issn = "0021-9991",
doi = "https://doi.org/10.1006/jcph.1996.5632",
url = "http://www.sciencedirect.com/science/article/pii/S0021999196956326",
author = "Ami Harten and Bjorn Engquist and Stanley Osher and Sukumar R. Chakravarthy"
}

@article{liu-etal_94a,
title = "Weighted Essentially Non-oscillatory Schemes",
journal = "Journal of Computational Physics",
volume = "115",
number = "1",
pages = "200 - 212",
year = "1994",
issn = "0021-9991",
doi = "https://doi.org/10.1006/jcph.1994.1187",
url = "http://www.sciencedirect.com/science/article/pii/S0021999184711879",
author = "Xu-Dong Liu and Stanley Osher and Tony Chan"
}

@article {cra-sem_16a,
    AUTHOR = {Cravero, I. and Semplice, M.},
     TITLE = {On the accuracy of {WENO} and {CWENO} reconstructions of third
              order on nonuniform meshes},
   JOURNAL = {J. Sci. Comput.},
  FJOURNAL = {Journal of Scientific Computing},
    VOLUME = {67},
      YEAR = {2016},
    NUMBER = {3},
     PAGES = {1219--1246},
      ISSN = {0885-7474},
   MRCLASS = {65M06 (65M08)},
  MRNUMBER = {3493501},
MRREVIEWER = {Erc\'\i lia Sousa},
       URL = {https://doi.org/10.1007/s10915-015-0123-3},
}

@article {cas-sem_19a,
    AUTHOR = {Castro, Manuel J. and Semplice, Matteo},
     TITLE = {Third- and fourth-order well-balanced schemes for the shallow
              water equations based on the {CWENO} reconstruction},
   JOURNAL = {Internat. J. Numer. Methods Fluids},
  FJOURNAL = {International Journal for Numerical Methods in Fluids},
    VOLUME = {89},
      YEAR = {2019},
    NUMBER = {8},
     PAGES = {304--325},
      ISSN = {0271-2091},
   MRCLASS = {65M08 (76M12)},
  MRNUMBER = {3911483},
       DOI = {10.1002/fld.4700},
       URL = {https://doi.org/10.1002/fld.4700},
}

@article {sem-vis_20a,
    AUTHOR = {Semplice, M. and Visconti, G.},
     TITLE = {Efficient {I}mplementation of {A}daptive {O}rder
              {R}econstructions},
   JOURNAL = {J. Sci. Comput.},
  FJOURNAL = {Journal of Scientific Computing},
    VOLUME = {83},
      YEAR = {2020},
    NUMBER = {1},
     PAGES = {Paper No. 6, 27},
      ISSN = {0885-7474},
   MRCLASS = {65M08 (65M12 76M12)},
  MRNUMBER = {4078386},
       DOI = {10.1007/s10915-020-01156-6},
       URL = {https://doi.org/10.1007/s10915-020-01156-6},
}

@article {kol_14a,
    AUTHOR = {Kolb, Oliver},
     TITLE = {On the full and global accuracy of a compact third order
              {WENO} scheme},
   JOURNAL = {SIAM J. Numer. Anal.},
  FJOURNAL = {SIAM Journal on Numerical Analysis},
    VOLUME = {52},
      YEAR = {2014},
    NUMBER = {5},
     PAGES = {2335--2355},
      ISSN = {0036-1429},
   MRCLASS = {65M08 (65M12)},
  MRNUMBER = {3262606},
MRREVIEWER = {Qing Fang},
       URL = {https://doi.org/10.1137/130947568},
}

@article{ryu2006equation,
  title={Equation of state in numerical relativistic hydrodynamics},
  author={Ryu, Dongsu and Chattopadhyay, Indranil and Choi, Eunwoo},
  journal={The Astrophysical Journal Supplement Series},
  volume={166},
  number={1},
  pages={410},
  year={2006},
  publisher={IOP Publishing}
}

@article{Duan2020,
   author = {Junming Duan and Huazhong Tang},
   doi = {10.4208/aamm.OA-2019-0124},
   issn = {2070-0733},
   issue = {1},
   journal = {Advances in Applied Mathematics and Mechanics},
   month = {6},
   pages = {1-29},
   title = {High-Order Accurate Entropy Stable Finite Difference Schemes for One- and Two-Dimensional Special Relativistic Hydrodynamics},
   volume = {12},
   year = {2020}
}

@InCollection{landau1987,
  author	= {Landau, L. D. and Lifshitz, E.M.},
  booktitle	= {Fluid Mechanics},
  doi		= {10.1016/b978-0-08-033933-7.50023-4},
  edition	= {Second},
  isbn		= {978-0-08-033933-7},
  pages		= {505--514},
  publisher	= {Elsevier},
  title		= {Chapter 15: Relativistic Fluid Dynamics},
  year		= {1987}
}

@Article{	  begelman1984theory,
  author	= {Begelman, Mitchell C. and Blandford, Roger D. and Rees,
		  Martin J.},
  doi		= {10.1103/RevModPhys.56.255},
  issn		= {00346861},
  journal	= {Reviews of Modern Physics},
  number	= {2},
  pages		= {255--351},
  publisher	= {APS},
  title		= {{Theory of extragalactic radio sources}},
  volume	= {56},
  year		= {1984}
}

@Article{	  mirabel1999sources,
  archiveprefix	= {arXiv},
  arxivid	= {astro-ph/9902062},
  author	= {Mirabel, I. F. and Rodr{\'{i}}guez, L. F.},
  doi		= {10.1146/annurev.astro.37.1.409},
  eprint	= {9902062},
  issn		= {0066-4146},
  journal	= {Annual Review of Astronomy and Astrophysics},
  number	= {1},
  pages		= {409--443},
  primaryclass	= {astro-ph},
  publisher	= {Annual Reviews 4139 El Camino Way, PO Box 10139, Palo
		  Alto, CA 94303-0139, USA},
  title		= {{Sources of Relativistic Jets in the Galaxy}},
  volume	= {37},
  year		= {1999}
}

@Article{	  zensus1997parsec,
  author	= {Zensus, J. Anton},
  doi		= {10.1146/annurev.astro.35.1.607},
  issn		= {0066-4146},
  journal	= {Annual Review of Astronomy and Astrophysics},
  number	= {1},
  pages		= {607--636},
  publisher	= {Annual Reviews 4139 El Camino Way, PO Box 10139, Palo
		  Alto, CA 94303-0139, USA},
  title		= {{Parsec-Scale Jets in Extragalactic Radio Sources}},
  volume	= {35},
  year		= {1997}
}

@Book{bottcher2012relativistic,
  address	= {Weinheim, Germany},
  author	= {B{\"{o}}ttcher, Markus and Harris, Daniel E. and
		  Krawczynski, Henric},
  booktitle	= {Relativistic Jets from Active Galactic Nuclei},
  doi		= {10.1002/9783527641741},
  editor	= {B{\"{o}}ttcher, Markus and Harris, Daniel E. and
		  Krawczynski, Henric},
  isbn		= {9783527410378},
  publisher	= {Wiley-VCH Verlag GmbH \& Co. KGaA},
  title		= {{Relativistic Jets from Active Galactic Nuclei}},
  volume	= {1},
  year		= {2012}
}

@Book{anile,
  author	= {Anile, A.M},
  publisher	= {Cambridge University Press, Cambridge},
  title		= {Relativistic Fluids and Magneto-Fluids: With Applications
		  in Astrophysics and Plasma Physics},
  year		= {1989}
}

@Article{	  deepak2019entropy,
  author	= {Bhoriya, Deepak and Kumar, Harish},
  journal	= {Zeitschrift f{\"{u}}r angewandte Mathematik und Physik},
  doi		= {10.1007/s00033-020-1250-8},
  issn		= {0044-2275},
  number	= {1},
  pages		= {29},
  title		= {{Entropy-stable schemes for relativistic hydrodynamics
		  equations}},
  volume	= {71},
  year		= {2020}
}

@Article{	  dai1997iterative,
  author	= {Dai, Wenlong and Woodward, Paul R.},
  doi		= {10.1137/S1064827595282234},
  issn		= {10648275},
  journal	= {SIAM Journal of Scientific Computing},
  number	= {4},
  pages		= {982--995},
  publisher	= {SIAM},
  title		= {{An iterative Riemann solver for relativistic
		  hydrodynamics}},
  volume	= {18},
  year		= {1997}
}

@Article{	  mignone2005piecewise,
  archiveprefix	= {arXiv},
  arxivid	= {astro-ph/0505200},
  author	= {Mignone, A. and Plewa, T. and Bodo, G.},
  doi		= {10.1086/430905},
  eprint	= {0505200},
  issn		= {0067-0049},
  journal	= {The Astrophysical Journal Supplement Series},
  number	= {1},
  pages		= {199--219},
  primaryclass	= {astro-ph},
  publisher	= {IOP Publishing},
  title		= {{The Piecewise Parabolic Method for Multidimensional
		  Relativistic Fluid Dynamics}},
  volume	= {160},
  year		= {2005}
}

@Article{	  wilson1972numerical,
  author	= {Wilson, James R.},
  doi		= {10.1086/151434},
  issn		= {0004-637X},
  journal	= {The Astrophysical Journal},
  pages		= {431},
  title		= {{Numerical Study of Fluid Flow in a Kerr Space}},
  volume	= {173},
  year		= {1972}
}

@Article{mignone2005hllc,
  archiveprefix	= {arXiv},
  arxivid	= {astro-ph/0506414},
  author	= {Mignone, A. and Bodo, G.},
  doi		= {10.1111/j.1365-2966.2005.09546.x},
  eprint	= {0506414},
  issn		= {00358711},
  journal	= {Monthly Notices of the Royal Astronomical Society},
  number	= {1},
  pages		= {126--136},
  primaryclass	= {astro-ph},
  title		= {{An HLLC Riemann solver for relativistic flows -I.
		  Hydrodynamics}},
  volume	= {364},
  year		= {2005}
}

@Article{	  del2002efficient,
  author	= {{Del Zanna}, L. and Bucciantini, N. and Londrillo, P.},
  doi		= {10.1051/0004-6361:20021641},
  issn		= {00046361},
  journal	= {Astronomy and Astrophysics},
  number	= {2},
  pages		= {397--413},
  title		= {{An efficient shock-capturing central-type scheme for
		  multidimensional relativistic flows II.
		  Magnetohydrodynamics}},
  volume	= {400},
  year		= {2003}
}

@Article{	  weno2007tchekhovskoy,
  archiveprefix	= {arXiv},
  arxivid	= {0704.2608},
  author	= {Tchekhovskoy, Alexander and McKinney, Jonathan C. and
		  Narayan, Ramesh},
  doi		= {10.1111/j.1365-2966.2007.11876.x},
  eprint	= {0704.2608},
  issn		= {00358711},
  journal	= {Monthly Notices of the Royal Astronomical Society},
  number	= {2},
  pages		= {469--497},
  title		= {{wham: A WENO-based general relativistic numerical scheme
		  - I. Hydrodynamics}},
  volume	= {379},
  year		= {2007}
}

@Article{	  wu2015high,
  archiveprefix	= {arXiv},
  arxivid	= {1504.07707},
  author	= {Wu, Kailiang and Tang, Huazhong},
  doi		= {10.1016/j.jcp.2015.06.012},
  eprint	= {1504.07707},
  issn		= {10902716},
  journal	= {Journal of Computational Physics},
  pages		= {539--564},
  publisher	= {Elsevier},
  title		= {{High-order accurate physical-constraints-preserving
		  finite difference WENO schemes for special relativistic
		  hydrodynamics}},
  volume	= {298},
  year		= {2015}
}

@Article{	  marti1996extension,
  author	= {Mart{\'{i}}, Jos{\'{e}} Ma and M{\"{u}}ller, Ewald},
  doi		= {10.1006/jcph.1996.0001},
  issn		= {00219991},
  journal	= {Journal of Computational Physics},
  number	= {1},
  pages		= {1--14},
  title		= {{Extension of the piecewise parabolic method to
		  one-dimensional relativistic hydrodynamics}},
  volume	= {123},
  year		= {1996}
}

@Article{	  aloy1999genesis,
  author	= {Aloy, M. A. and Ibanez, J. M. and Marti, J. M. and Muller,
		  E.},
  doi		= {10.1086/313214},
  issn		= {0067-0049},
  journal	= {The Astrophysical Journal Supplement Series},
  number	= {1},
  pages		= {151--166},
  publisher	= {IOP Publishing},
  title		= {{GENESIS: A High‐Resolution Code for Three‐dimensional
		  Relativistic Hydrodynamics}},
  volume	= {122},
  year		= {1999}
}

@Article{	  marti2003numerical,
  author	= {Mart{\'{i}}, Jos{\'{e}} Maria and M{\"{u}}ller, Ewald},
  doi		= {10.12942/lrr-2003-7},
  issn		= {14338351},
  journal	= {Living Reviews in Relativity},
  number	= {1},
  pages		= {7},
  title		= {{Numerical hydrodynamics in special relativity}},
  volume	= {6},
  year		= {2003}
}

@Article{	  ibanez1999riemann,
  author	= {Ib{\'{a}}{\~{n}}ez, J. M.A. and Mart{\'{i}}, J. M.A.},
  doi		= {10.1016/S0377-0427(99)00158-2},
  issn		= {03770427},
  journal	= {Journal of Computational and Applied Mathematics},
  number	= {1-2},
  pages		= {173--211},
  publisher	= {Elsevier},
  title		= {{Riemann solvers in relativistic astrophysics}},
  volume	= {109},
  year		= {1999}
}

@Article{	  marti1994analytical,
  author	= {Mart{\'{i}}, Jos{\'{e}} M. and M{\"{u}}ller, Ewald},
  doi		= {10.1017/S0022112094003344},
  issn		= {14697645},
  journal	= {Journal of Fluid Mechanics},
  pages		= {317--333},
  publisher	= {Cambridge University Press},
  title		= {{The analytical solution of the Riemann problem in
		  relativistic hydrodynamics}},
  volume	= {258},
  year		= {1994}
}

@Article{	  marti1991numerical,
  author	= {Mart{\'{i}}, Jos{\'{e}} Ma and Ib{\~{n}}ez, Jos{\'{e}} Ma
		  and Miralles, Juan A.},
  doi		= {10.1103/PhysRevD.43.3794},
  issn		= {05562821},
  journal	= {Physical Review D},
  number	= {12},
  pages		= {3794--3801},
  publisher	= {American Physical Society},
  title		= {{Numerical relativistic hydrodynamics: Local
		  characteristic approach}},
  volume	= {43},
  year		= {1991}
}

@Article{	  nunez2016xtroem,
  author	= {{Nunez-de la Rosa}, Jonatan and Munz, Claus Dieter},
  doi		= {10.1093/mnras/stw999},
  issn		= {13652966},
  journal	= {Monthly Notices of the Royal Astronomical Society},
  number	= {1},
  pages		= {535--559},
  publisher	= {Oxford University Press},
  title		= {{XTROEM-FV: A new code for computational astrophysics
		  based on very high order finite-volume methods - II.
		  Relativistic hydro- and magnetohydrodynamics}},
  volume	= {460},
  year		= {2016}
}

@Article{		he2012,
	author = {He, Peng and Tang, Huazhong},
	doi = {10.4208/cicp.291010.180311a},
	journal = {Communications in Computational Physics},
	number = {1},
	pages = {114--146},
	title = {{An Adaptive Moving Mesh Method for Two-Dimensional Relativistic Hydrodynamics}},
	volume = {11},
	year = {2012}
}

@article{centrella1984planar,
	title={Planar numerical cosmology. II-The difference equations and numerical tests},
	author={Centrella, Joan and Wilson, James R},
	journal={Astrophysical Journal Supplement Series (ISSN 0067-0049), vol. 54, Feb. 1984, p. 229-249. Research supported by the University of Texas and Aspen Center for Physics.},
	volume={54},
	pages={229--249},
	year={1984}
}

@article{chen2022physical,
  title={A physical-constraint-preserving finite volume WENO method for special relativistic hydrodynamics on unstructured meshes},
  author={Chen, Yaping and Wu, Kailiang},
  journal={Journal of Computational Physics},
  volume={466},
  pages={111398},
  year={2022},
  publisher={Elsevier}
}

@article{bhoriya2026physical,
  title={Physical Constraint Preserving Alternative Finite-Difference WENO Scheme for Hyperbolic Systems with Non-conservative Products},
  author={Bhoriya, Deepak and Balsara, Dinshaw S and Chandrashekar, Praveen and Shu, Chi-Wang},
  journal={Journal of Scientific Computing},
  volume={107},
  number={1},
  pages={1},
  year={2026},
  publisher={Springer}
}

\end{document}